\documentclass{article}
\usepackage{amsmath, amsfonts, latexsym, amsrefs, amssymb, amsthm}
\usepackage{mathtools}

\usepackage{hyperref}

\usepackage{bm}

\usepackage[font={small}]{caption}

\usepackage{tocloft}

\usepackage{pstricks}

\usepackage{subcaption}

\allowdisplaybreaks

\usepackage{tikz}
\usepackage{tikz,fullpage}
\usetikzlibrary{shapes,arrows, petri, topaths, automata}

\usepackage{tabulary}
\usepackage{booktabs}

\usepackage{graphicx}
\usepackage{grffile}

\usepackage{enumerate}
\usepackage{mathrsfs}
\usepackage{wrapfig}

\usepackage{color}

\numberwithin{equation}{section}

\newcommand{\bZ}{{\mathbb Z}}

\newcommand{\beq}{\begin{equation}}
	\newcommand{\bEq}{\end{equation}}

\newcommand{\bv}{{\bf{v}}}
\newcommand{\bw}{{\bf{w}}}

\newcommand{\ord}{{\rm{ord}}}
\newcommand{\size}{{\rm{size}}}
\newcommand{\gh}{{\rm{gh}}}

\newcommand{\al}{\alpha}

\newcommand{\be}{\begin{equation}}
	\newcommand{\ee}{\end{equation}}

\newcommand{\e}{{\varepsilon}}

\newcommand{\fa}{{\mathfrak a}}
\newcommand{\fb}{{\mathfrak b}}

\renewcommand{\Dot}{\mathbf {Dot}}

\renewcommand{\cal}{\mathcal}
\newcommand{\wh}{\widehat}
\newcommand{\wt}{\widetilde}

\newcommand{\ii}{\mathrm{i}} 
\newcommand{\dd}{\mathrm{d}}

\renewcommand{\epsilon}{\varepsilon}
\renewcommand{\leq}{\leqslant}
\renewcommand{\geq}{\geqslant}

\renewcommand{\le}{\leq}
\renewcommand{\ge}{\geq}

\renewcommand{\P}{\mathbb{P}}
\newcommand{\E}{\mathbb{E}}
\newcommand{\R}{\mathbb{R}}
\newcommand{\C}{\mathbb{C}}
\newcommand{\N}{\mathbb{N}}
\newcommand{\Z}{\mathbb{Z}}

\newcommand{\dashed}{\text{diffusive}}
\newcommand{\Dashed}{\text{Diffusive}}
\newcommand{\self}{\text{self-energy}}
\newcommand{\selfs}{\text{self-energies}}

\newcommand{\Selfs}{\text{Self-energies}}
\newcommand{\renormal}{\text{self-energy renormalization}}

\DeclareMathOperator{\tr}{Tr}

\DeclareMathOperator{\re}{Re}
\DeclareMathOperator{\im}{Im}

\DeclareMathOperator{\OO}{O}

\theoremstyle{plain} 
\newtheorem{theorem}{Theorem}[section]
\newtheorem*{theorem*}{Theorem}
\newtheorem{lemma}[theorem]{Lemma}
\newtheorem{assumption}[theorem]{Assumption}
\newtheorem*{lemma*}{Lemma}
\newtheorem{corollary}[theorem]{Corollary}
\newtheorem*{corollary*}{Corollary}

\newtheorem*{proposition*}{Proposition}
\newtheorem{claim}[theorem]{Claim}
\newtheorem*{claim*}{Claim}
\newtheorem{notation}[theorem]{Notation}

\newtheorem{definition}[theorem]{Definition}
\newtheorem*{definition*}{Definition}
\theoremstyle{remark}
\newtheorem{example}[theorem]{Example}
\newtheorem*{example*}{Example}
\newtheorem{remark}[theorem]{Remark}

\newtheorem*{remark*}{Remark}
\newtheorem*{remarks*}{Remarks}
\newtheorem{strategy}[theorem]{Strategy}

\usepackage{dsfont}
\usepackage{stmaryrd}

\newcommand{\Sdelta}{{\Sigma}_T}
\newcommand{\Sdeltak}{{\Sigma}_{T,k}}
\newcommand{\wtSdelta}{\Sigma}
\newcommand{\wtSdeltan}{\Sigma^{(n)}}

\newcommand{\Sele}{{\mathcal E}}

\newcommand{\Selek}{{\mathcal E}}

\newcommand{\incomp}{{\text{$T$-equation}}}

\newcommand{\PT}{\mathcal R_T}
\newcommand{\wtPT}{\widetilde{\mathcal R}_T}
\newcommand{\PTn}{\mathcal R_T^{(n)}}
\newcommand{\PTk}{{\mathcal R_{T,k}}}
\newcommand{\PIT}{\mathcal R_{IT}}
\newcommand{\wtPIT}{\widetilde{\mathcal R}_{IT}}
\newcommand{\whPIT}{\widehat{\mathcal R}_{IT}}
\newcommand{\PITn}{{\mathcal R_{IT}^{(n)}}}
\newcommand{\PITk}{{\mathcal R_{IT,k}}}

\newcommand{\PGn}{\mathcal R^{(n)}}
\newcommand{\AGn}{\mathcal A^{(>n)}_{ho}}
\newcommand{\QGn}{\mathcal Q^{(n)}}

\newcommand{\QT}{\mathcal Q_T}
\newcommand{\QTn}{\mathcal Q_T^{(n)}}
\newcommand{\QTk}{{\mathcal Q_{T,k}}}
\newcommand{\QIT}{\mathcal Q_{IT}}
\newcommand{\whQIT}{\widehat{\mathcal Q}_{IT}}
\newcommand{\QITn}{{\mathcal Q_{IT}^{(n)}}}
\newcommand{\QITk}{{\mathcal Q_{IT,k}}}

\newcommand{\AT}{\mathcal A_T}
\newcommand{\ATn}{\mathcal A_T^{(>n)}}

\newcommand{\AIT}{\mathcal A_{IT}}

\newcommand{\AITn}{{\mathcal A_{IT}^{(>n)}}}

\newcommand{\Pol}{{\mathbf{M}}}

\newcommand{\Iso}{{\mathcal I}}
\newcommand{\Err}{{\mathcal Err}}

\newcommand{\nonuni}{{\text{non-universal $T$-expansion}}}
\newcommand{\Nonuni}{{\text{Non-universal $T$-expansion}}}

\def\bZ{{\mathbb Z}}

\newcommand{\cM}{{\cal M}}

\def\@empty{}

\def\author#1{\par
	{\centering{\authorfont#1}\par\vspace*{0.05in}}
}

\def\titlefont{\fontsize{13}{15}\bfseries\boldmath\selectfont\centering{}}
\def\authorfont{\fontsize{13}{15}}
\def\abstractfont{\fontsize{8}{10}}

\let\affiliationfont\rhfont

\def\address#1{\par
	{\centering{\affiliationfont#1\par}}\par\vspace*{11pt}
}

\def\body{
	\setcounter{footnote}{0}
	\def\thefootnote{\alph{footnote}}
	\def\@makefnmark{{$^{\rm \@thefnmark}$}}
}

\def\title#1{
	\thispagestyle{plain}
	\vspace*{-14pt}
	\vskip 79pt
	{\centering{\titlefont #1\par}}%
	\vskip 1em
}

\renewenvironment{abstract}{\par%
	\vspace*{6pt}\noindent 
	\abstractfont
	\noindent\leftskip18pt\rightskip18pt
}{%
	\par}

\usepackage{tocloft}

\makeatletter
\renewcommand{\section}{\@startsection
	{section}
	{1}
	{0mm}
	{-2\baselineskip}
	{2\baselineskip}
	{\normalfont\large\scshape\centering}} 
\makeatother

\makeatletter
\renewcommand{\subsection}{\@startsection
	{subsection}
	{2}
	{0mm}
	{-\baselineskip}
	{0.5\baselineskip}
	{\normalfont\bf\itshape}} 
\makeatother

\newcommand{\tnorm}[1]{{\left\vert\kern-0.25ex\left\vert\kern-0.25ex\left\vert #1 
		\right\vert\kern-0.25ex\right\vert\kern-0.25ex\right\vert}}

\begin{document}

%
%
\title{Delocalization and quantum diffusion of random band matrices in high dimensions II: $T$-expansion}
{\let\thefootnote\relax\footnotetext{\noindent 
		The work of F.Y. is partially supported by the Wharton Dean’s Fund for Postdoctoral Research. 
		The work of H.-T. Y. is partially supported by the NSF grant DMS-1855509 and a Simons Investigator award. 
		The work of J.Y. is partially supported by the NSF grant DMS-1802861. 
}}
\vspace{1cm}
\noindent \begin{minipage}[b]{0.32\textwidth}
	\author{Fan Yang }
	\address{University of Pennsylvania\\
		fyang75@wharton.upenn.edu}
\end{minipage}
\begin{minipage}[b]{0.32\textwidth}
	\author{Horng-Tzer Yau}
	\address{Harvard University\\
		htyau@math.harvard.edu}
\end{minipage}
\begin{minipage}[b]{0.32\textwidth}
	\author{Jun Yin}
	\address{University of California, Los Angeles\\
		jyin@math.ucla.edu}
\end{minipage}

\begin{abstract}

We consider the Green's function $G(z):=(H-z)^{-1}$ of Hermitian random band matrix $H$ on the $d$-dimensional lattice $(\Z/L\Z)^d$. The entries $h_{xy}=\overline h_{yx}$ of $H$ are independent centered complex Gaussian random variables with variances $s_{xy}=\mathbb E|h_{xy}|^2$, which satisfy a banded profile so that $s_{xy}$ is negligible if $|x-y|$ exceeds the band width $W$. For any fixed $n\in \N$,  we construct an expansion of the $T$-variable, $T_{xy}=|m|^2 \sum_{\alpha}s_{x\alpha}|G_{\alpha y}|^2$, with an error $\OO(W^{-nd/2})$, and use it to prove a local law on the Green's function. This $T$-expansion was the main tool to prove the delocalization and quantum diffusion of random band matrices for dimensions $d\ge 8$ in part I \cite{PartI_high} of this series. 
 

\end{abstract}

\tableofcontents

\vspace{20pt}

\newpage
\section{Introduction}
Random band matrices $H=(h_{xy})$ model interacting quantum systems on a lattice of scale $L$ with random transition amplitudes effective up to a short scale $W\ll L$. In this paper, we consider a general Hermitian random band matrix ensemble on the $d$-dimensional integer lattice $\Z_L^d:=\{1,2, \cdots, L\}^d$ with $N:= L^d$ many lattice sites. The entries of $H$ are independent centered random variables  up to the Hermitian condition $h_{xy}=\overline h_{yx}$. The variance $s_{xy}:=\mathbb E|h_{xy}|^2$ typically decays with the distance between $x$ and $y$ on a characteristic length scale $W$, called the {\it band width} of $H$, and is negligible when $|x-y|\gg W$. We require that $s_{xy}$ satisfy the normalization condition  
\be\label{fxy}
\sum_{x}s_{xy}=\sum_{y}s_{xy}=1.
\ee
It is well-known that under the condition \eqref{fxy}, the global eigenvalue distribution of $H$ converges weakly to the Wigner's semicircle law supported in $[-2,2]$. 

As $W$ varies, random band matrices naturally interpolate between the random Schr\"odinger operator \cite{Anderson} and mean-field (generalized) Wigner matrix ensembles \cite{Wigner}. One important conjecture about random band matrices is that a sharp {\it localization-delocalization transition} occurs at some critical band width $W_c$. An eigenvector is said to be \emph{localized} if most of its weight resides in a region of scale $\ell \ll L$, and \emph{delocalized} if $\ell\sim L$. In physics, this length scale $\ell$ is called the \emph{localization length}, which generally depends on the energy level of the eigenvector. We restrict ourselves to the bulk eigenvectors with eigenvalues in $(-2+ \kappa, 2-\kappa)$ for a small constant $\kappa>0$. It is conjectured that there exists $W_c\in [1, L]$ such that if $W\ll W_c$, the bulk eigenvectors of $H$ are localized, while if $W\gg W_c$, the bulk eigenvectors of $H$ are delocalized. Based on simulations \cite{ConJ-Ref1, ConJ-Ref2, ConJ-Ref4, ConJ-Ref6} and non-rigorous supersymmetric arguments \cite{fy}, the critical band width of one-dimensional random band matrices is conjectured to be of order $W_c\sim \sqrt L$. In higher dimensions, the critical band width is conjectured to be $W_c\sim \sqrt{\log L}$ when $d=2$ and $W_c\sim 1$ when $d\ge 3$; see \cite{Spencer1,Spencer2,Spencer3, PB_review} for more details about these conjectures. 


Many partial results have been proved rigorously concerning the localization-delocalization conjecture of random band matrices in dimension $d=1$ \cite{BaoErd2015,Semicircle, ErdKno2013,ErdKno2011,delocal,HeMa2018,BouErdYauYin2017,PartI,PartII,PartIII,Sch2009,PelSchShaSod,Sod2010,SchMT,Sch1,Sch2,Sch3,1Dchara,Sch2014} and dimensions $d\ge 2$ \cite{ErdKno2013,ErdKno2011,delocal, HeMa2018,PartIII}. We refer the reader to the introduction of \cite{PartI_high} for a brief review.   
In this series (i.e. part I \cite{PartI_high} and the current paper), we prove that, as long as $W\ge L^\e$ for a small constant $\e>0$, most bulk eigenvectors of $H$ have localization lengths comparable to $L$. This gives a positive answer to the delocalization conjecture (in the weak delocalization sense) for dimensions $d \ge 8$ under the slightly stronger assumption $W\ge L^\e$ (vs. $W\ge C$). 

In part I \cite{PartI_high}, we described the main structure of the proof of the delocalization conjecture and quantum diffusion of random band matrices. This paper will provide some key results used in \cite{PartI_high}. Our proof is based on the resolvent (or Green's function) of $H$ defined by 
\be\nonumber
G(z)=(H-z)^{-1},\quad z\in \C_+:=\{x\in \C: \im z>0\},
\ee
and the $T$-matrix defined by \cite{delocal} 
\be\label{defGT}
T_{xy}(z):=|m|^2\sum_\al s_{x\al}|G_{\al y}(z)|^2,\quad x,y \in \Z_L^d.
\ee
An important tool is an expansion of the $T$-matrix, called the \emph{$T$-expansion}, with errors of order $W^{-nd/2}$ for $n\in \N$.
In \cite{PartI_high}, we defined basic graph operations in constructing the $T$-expansion (see Section \ref{sec_basiclocal}), studied the local expansion strategy, and explored a key $\renormal$ property. In this paper, we will describe the subtle global expansion strategy (see Sections \ref{sec pre} and \ref{sec global}) and complete the construction of the $T$-expansion up to any fixed order. 
The $T$-expansion thus obtained gives a precise description of the fluctuation of the $T$-matrix by decomposing it into a sum of terms with sophisticated but sufficiently nice structures. Furthermore, given an $n$-th order $T$-expansion for any fixed $n\in \N$, we will use it to prove an almost sharp local law on $G(z)$ (cf. Theorem \ref{thm ptree}), which is also a key input for \cite{PartI_high}. For the convenience of the reader, we will repeat some notations and definitions of \cite{PartI_high} in Sections \ref{sec notation}--\ref{sec_basiclocal} so that  this paper is relatively self-contained.

\subsection{The model and $T$-expansion} 

We will consider $d$-dimensional random band matrices indexed by a cube of linear size \(L\) in \(\mathbb{Z}^{d}\), i.e., 
\be\label{ZLd}
\Z_L^d:=\left( \Z\cap ( -L/2 , L/2]\right) ^d. 
\ee
We will view $\Z_L^d$ as a torus and denote  by  $[x-y]_L$ the representative of $x-y$ in $\Z_L^d$, i.e.,  
\be\label{representativeL}[x-y]_L:= \left[(x-y)+L\Z^d\right]\cap \Z_L^d.\ee
Clearly, $\|x-y\|_L:=\| [x-y]_L \|$ is a {periodic} distance on $\Z_L^d$ for any norm $\|\cdot\|$ on $\Z^d$. For definiteness, we use $\ell^\infty$-norm in this paper, i.e. $\|x-y\|_L:=\|[x-y]_L\|_\infty$. In this paper, we consider the following class of $d$-dimensional random band matrices. 

\begin{assumption}[Random band matrix] \label{assmH}
Fix any $d\in \N$. For $L\gg W\gg 1$ and $N:=L^d$, we assume that $ H\equiv H_{d,f,W,L}$ is an $N\times N$ complex Hermitian random matrix whose entries $(\re h_{xy}, \im   h_{xy}: x,y \in \Z_L^d)$ are independent Gaussian random variables (up to symmetry $H=H^\dagger$) such that  
\be\label{bandcw0}
\mathbb E h_{xy} = 0, \quad \E (\re h_{xy})^2 =  \E (\im h_{xy})^2 = s_{xy}/2, \quad x , y \in \bZ_L^d,
\ee
where the variances $s_{xy}$ satisfy that
\be\label{sxyf}s_{xy}= f_{W,L}\left( [x-y]_L \right)\ee
for a positive symmetric function $f_{W,L}$ satisfying Assumption \ref{var profile} below. Then, we say that $H$ is a $d$-dimensional  random band matrix with linear size $L$, band width $W$ and variance profile $f_{W,L}$. Denote the variance matrix by $S : = (s_{xy})_{x,y\in \Z_L^d}$, which is a doubly stochastic symmetric $N \times N$ matrix. 
	%
\end{assumption}   

\begin{assumption}[Variance profile]\label{var profile}
	We assume that $f_{W,L}:\Z_L^d\to \mathbb R_+$ is a positive symmetric function on $\Z_L^d$ that can be expressed by the Fourier transform 
	\be\label{choicef}
	f_{W,L}(x):= \frac{1}{(2\pi)^d Z_{W,L}}\int \psi(Wp)e^{\ii p\cdot x} \dd p.  \ee
	Here $\Z_{W,L}$ is the  normalization constant so that $\sum_{x\in \Z_L^d} f_{W,L}(x)=1$, and $\psi\in C^\infty(\R^d)$ is a symmetric smooth function independent of $W$ and $L$ and satisfies the following properties:
	\begin{itemize}
		\item[(i)] $\psi(0)=1$ and $\|\psi\|_\infty \le 1$;  
		\item[(ii)] $\psi(p)\le \max\{1 - c_\psi |p|^2 , 1-c_\psi  \}$ for a constant $c_\psi>0$;
		\item[(iii)]  $\psi$ is in the Schwartz space, i.e.,
		\be\label{schwarzpsi} \lim_{|p|\to \infty}(1+|p|)^{k}|\psi^{(l)}(p)| =0, \quad \text{for any }k,l\in \N.\ee
	\end{itemize}
\end{assumption}

Clearly, $f_{W,L}$ is of order $\OO(W^{-d})$
and decays faster than any polynomial, that is, for any fixed $k\in \N$, there exists a constant $C_k>0$ so that
\be\label{subpoly}
|f_{W,L}(x)|\le C_k W^{-d}\left( {\|x\|_L}/{W}\right)^{-k}.
\ee
Hence the variance profile $S$ defined in \eqref{sxyf}  has a banded structure, namely,  for any constants $\tau,D>0$,
\be\label{app compact f}
\mathbf 1_{|x-y|\ge W^{1+\tau}}|s_{xy}|\le W^{-D}.
\ee
Combining \eqref{schwarzpsi} and \eqref{subpoly} with the Poisson summation formula, we obtain that  
\be\label{bandcw1} 
Z_{W,L} =   \psi(0) + \OO(W^{-D})=1+ \OO(W^{-D}),
\ee
for any large constant $D>0$ as long as $L\ge W^{1+\e}$ for a constant $\e>0$.

The diagonal resolvent entries of $G(z)$ are expected to be given by the Stieltjes transform of Wigner's semicircle law,
\be\label{msc}
m(z):=\frac{-z+\sqrt{z^2-4}}{2} = \frac{1}{2\pi}\int_{-2}^2 \frac{\sqrt{4-\xi^2}}{\xi-z}\dd\xi,\quad z\in \C_+.
\ee
On the other hand, it was proved in \cite{PartI_high} that off-diagonal resolvent entries can be approximated by a \emph{diffusive kernel} $\Theta$ defined by  
\be\label{Thetapm}
\Theta(z):=\frac{|m(z)|^2 S}{1-|m(z)|^2 S} ,
\ee
for $z=E+\ii \eta$ with $E\in (-2+\kappa,2-\kappa)$ and $\eta\ge W^{2+\e}/L^2$ for a constant $\e>0$. In this case, it is well-known that (see, e.g., \cite{delocal,PartIII} for a proof of \eqref{thetaxy})
\be\label{thetaxy}
\Theta_{xy}(z) \le   \frac{W^\tau \mathbf 1_{|x-y|\le   \eta^{-1/2}W^{1+\tau}}}{W^2\left(\|x-y\|_{ L}+W\right)^{d-2}}   + \frac{1}{ \left(\|x-y\|_{ L}+W\right)^{D}}  \le W^{\tau} B_{xy},
\ee
for any small constant $\tau>0$ and large constant $D>0$, where we have abbreviated that 
\be\label{defnBxy}B_{xy}:=W^{-2}\left(\|x-y\|_{ L}+W\right)^{-d+2}.\ee 
In part I of this series \cite{PartI_high}, we proved the following local law in Theorem \ref{main thm} and quantum diffusion in Theorem \ref{main thm2}. In addition, Theorem \ref{main thm} was used in \cite{PartI_high} to prove the delocalization of random band matrices in dimensions $d\ge 8$. We remark that all results in this paper hold only for large enough $W$ and $L$, and, for simplicity, we do not repeat it in  our statements.

\begin{theorem}[Theorem 1.4 of \cite{PartI_high}] 
	\label{main thm}
Fix $d\ge 8$ and any small constants $ c_w, \e, \kappa>0$. Suppose that $L^{c_w} \le W \le L$, and $H$ is a $d$-dimensional random band matrix satisfying Assumptions \ref{assmH} and \ref{var profile}. Then for any small constant $\tau>0$ and large constant $D>0$, we have the following estimate for $z=E+\ii \eta$ and all $x,y \in \Z_{L}^d$: 
	\be\label{locallaw}
	\P\bigg(\sup_{E\in (-2+\kappa, 2- \kappa)}\sup_{W^{2+\e}/L^{2}\le \eta\le 1} |G_{xy} (z) -m(z)\delta_{xy}|^2 \le  W^\tau B_{xy}\bigg) \ge 1- L^{-D} . 
	\ee
\end{theorem}

\begin{theorem}[Theorem 1.5 of \cite{PartI_high}] 
	\label{main thm2}
	Suppose the assumptions of Theorem \ref{main thm} hold. 
	Fix any small constant $\e>0$ and large constant $M\in \N$. Then for all $x,y \in \Z_{L}^d$ and $z=E+\ii \eta$ with $E\in (-2+\kappa,2-\kappa)$ and $W^{2+\e}/L^{2}\le \eta\le 1$, we have that 
	\be\label{E_locallaw}
	\begin{split}
		\E T_{xy} &=  \left[  \Theta^{(M)} \left(|m|^2 + \cal G^{(M)} \right)\right]_{xy}  + \OO( W^{-Md/2} ) . 
	\end{split}
	\ee
	Here $\Theta^{(M)}$ is the $M$-th order \emph{renormalized diffusive matrix}
	\be\label{theta_renormal}\Theta^{(M)}:= \frac{1}{1-|m|^2  S\left(1+\wtSdelta^{(M)}\right)}|m|^2S,\ee
	and it  satisfies the bound 
	\be\label{Self_theta} \big|\Theta^{(M)}_{xy} \big| \le W^{\tau}B_{xy}, \ee 
	for any constant $\tau>0$. Furthermore, the \emph{self-energy correction} $\wtSdelta^{(M)}$ is given by $\wtSdelta^{(M)}(z):=\sum_{l=4}^M \Sele_l(z)$ where $\{\Sele_l\}_{l=4}^M$ is a sequence of deterministic matrices satisfying the following properties:  
	\be\label{two_properties0}
	\Selek_{l} (x, x+a) =  \Selek_{l} (0,a), \quad   \Selek_{l} (0, a) = \Selek_{l}(0,-a), \quad \forall \ x,a\in \Z_L^d,
	\ee 
	and for any constant $\tau>0$,
	\be\label{4th_property0}
	\left|  (\Selek_{l})_{0x}(z) \right| \le  W^{-(l-4)d/2+\tau}B_{0x}^2, \quad \forall \ x\in \Z_L^d,\ \eta\in [W^{2+\e}/L^{2}, 1],
	\ee
	\be\label{3rd_property0}
	\Big|\sum_{x\in \Z_L^d} (\Selek_{l})_{0x}(z)\Big|   \le  \eta W^{-(l-2)d/2 +\tau } , \quad \forall  \ \eta\in [W^{2+\e}/L^{2},1].
	\ee
	Here in \eqref{two_properties0} and throughout the rest of this paper, we use $\cal A_{xy}$ and $ \cal A(x,y)$ interchangeably for any matrix $\cal A$.  The $M$-th order \emph{local correction} $ \cal G^{(M)}$ satisfies that
	\be\label{bound_calG}
	\big|\cal G_{xy}^{(M)}\big|\le W^\tau B_{xy}^{3/2} , 
	\ee
	for any constant $\tau>0$. 
\end{theorem}
 
%
%


As remarked in \cite{PartI_high}, the proof in this paper can be adapted to non-Gaussian band matrices after some technical modifications. 
The proof for the real symmetric Gaussian case is similar to the complex case except that the number of terms will double in every expansion step (which is due to the fact that $\E h_{xy}^2 = 0$ in the complex case but not in the real case). 
The condition $d\ge 8$ can also be improved, but it seems to require substantial improvements in estimating the terms in the $T$-expansion to reach the physical dimension $d=3$. We will deal with these improvements in forthcoming papers.



The proofs of Theorem \ref{main thm} and Theorem \ref{main thm2} in \cite{PartI_high} depend on main results of this paper. More precisely, we will prove Lemmas 5.2, 5.4, 5.6 and 5.7 of \cite{PartI_high}. We summarize our main results in Section \ref{subsec_main}: in Theorem \ref{incomplete Texp}, we will construct the \emph{$T$-equation} defined in Definition \ref{def incompgenuni}; as a corollary of Theorem \ref{incomplete Texp}, in Corollary \ref{lem completeTexp}, we will construct the \emph{$T$-expansion} defined in Definition \ref{defn genuni}; in Theorem \ref{thm ptree}, we will prove a local law on $G(z)$ using the $T$-expansion.


Define $T$-variables with three subscripts by 
\be\label{general_T}
T_{x,yy'}:=|m|^2\sum_\al s_{x\al}G_{\al y}\overline G_{\al y'},\quad \text{and}\quad T_{yy',x}:=|m|^2\sum_\al G_{ y\al }\overline G_{y' \al}s_{\al x}. 
\ee
By definition, the $T$-variable in \eqref{defGT} can be written as $T_{xy}\equiv T_{x,yy}$. For any $x\in \Z_L^d$, we define $\E_x$ as the partial expectation with respect to the $x$-th row and column of $H$, i.e., $\E_x(\cdot) := \E(\cdot|H^{(x)}),$ where $H^{(x)}$ denotes the $(N-1)\times(N-1)$ minor of $H$ obtained by removing the $x$-th row and column. For simplicity, in this paper we will use the notations 
$$P_x :=\E_x , \quad Q_x := 1-\E_x.$$
Then we write 
$$T_{x,yy'}=|m|^2\sum_{\al} s_{x\al} P_\al \left(G_{\al y}\overline G_{\al y'}\right)  + |m|^2\sum_{\al} s_{x\al} Q_\al \left(G_{\al y}\overline G_{\al y'}\right) . $$
The $Q_\al$ term is a fluctuation term. We will calculate the $P_\al$ term using Gaussian integration by parts to derive an expression of the form 
\be\nonumber
T_{x,yy'} =m \overline G_{yy'} \Theta_{xy} +  (\text{higher order term}) + (\text{fluctuation term}),
\ee
where the higher order term is a sum of expressions of order $\OO(W^{-3d/2})$ (see \eqref{seconduniversal} for the explicit form). We call this expansion a second order $T$-expansion. However, this expansion is useful only when $L$ is not too large compared to $W$, because the error is bounded in terms of a power of $W$ instead of $L$. 

In order to handle the case $L\ge W^C$ for any large constant $C>0$, we need to obtain much finer expansions of the $T$-matrix with arbitrarily high order errors. In this paper, we will construct a sequence of $n$-th order $T$-expansions with leading terms $\Theta^{(n)}$ and errors of order $\OO(W^{-(n+1)d/2})$. Roughly speaking, we will construct an expansion of the form  
\be\label{T-exp-intro}
T_{x,yy'} =m \overline G_{yy'}  \Theta^{(n)}_{xy} + (\text{recollision term}) + (\text{higher order term}) + (\text{fluctuation term}) + (\text{error term}),
\ee
where the recollision term is a sum of expressions with coincidences in the summation indices, the higher order term is a sum of expressions of order $\OO(W^{-(n+1)d/2})$, the fluctuation term is a sum of expressions that can be written into the form $\sum_\al Q_\al(\cdot)$, and the error term is negligible for all of our proofs. 
In the expansion process, we need to give a precise construction of the $\selfs$ $\Sele_l$, and show that they lead to the renormalized diffusive matrix $\Theta^{(n)}$ defined in \eqref{theta_renormal}. In addition, we also need to track the recollision, higher order and fluctuation terms very carefully, so that they satisfy a key structural property---\emph{doubly connected property} in Definition \ref{def 2net}---which we will explain now. 

In this paper, we represent all expressions using  graphs with edge labels  $G$, $\Theta$ and $S$ as in the standard graphic language in quantum physics. These graphs naturally have \emph{two-level structures}: a pair of atoms are at most $\OO(W)$ apart from each other if they are connected by a path of $S$ edges; otherwise the distance between them can vary  up to $L$. We define quotient graphs with the equivalence relation that two vertices are equivalent if they are connected by a path of $S$ edges. Following notations in \cite{PartIII,PartI_high}, we call equivalence classes of vertices \emph{molecules} (cf. Definition \ref{def_poly}) and corresponding quotient graphs \emph{molecular graphs} (cf. Definition \ref{def moleg}). The (original) subgraph inside a molecule is called the {\it local structure} of this molecule, and the structure of a molecular graph is said to be {\it global}. Local structures are easy to track while global structures are our main focus. In particular, the \emph{doubly connected property} defined in Definition \ref{def 2net} is the most important global structural property of this paper. Roughly speaking, a doubly connected graph contains a spanning tree of $\Theta$ edges and a spanning tree of $G$ edges that connect all molecules together. This property is needed to derive effective  estimates on the terms in \eqref{T-exp-intro}. By \eqref{thetaxy} and \eqref{locallaw}, the typical sizes of $\Theta_{xy}$ and $G_{xy}$ are of order $B_{xy}$ and \smash{$B_{xy}^{1/2}$}, respectively. The doubly connected property of a graph ensures that in each sum of an index, we have at least a product of a $\Theta$ factor and a $G$ factor, so that the sum can be controlled by the row sums of \smash{$B_{xy}^{3/2}$}, which are bounded by $\OO(W^{-d/2}$) independently of $L$. This is one of the reasons why we can treat random band matrices with large $L\ge W^C$. 

The main difficulty in constructing the $T$-expansion is that the doubly connectedness is not necessarily preserved in arbitrary graph expansions. With terminologies that will be introduced in Section \ref{sec_basiclocal}, \emph{local expansions} in Section \ref{sec localexp} will always preserve the doubly connectedness and have been discussed in details in part I \cite{PartI_high}. On the other hand, global expansions in Section \ref{subsec global} may break the doubly connected structure of a graph. Roughly speaking, a global expansion involves a substitution of a $T$-variable with a lower order $T$-expansion. We will see that a global expansion does not break the doubly connected property if and only if it involves a substitution of a \emph{redundant} $T$-variable, where a $T$-variable is said to be redundant in a graph if after removing it, the remaining graph is still doubly connected (cf. Definition \ref{defn pivotal}). In this paper, we find a sophisticated global expansion strategy such that we only need to expand a redundant $T$-variable in every global expansion until getting the expansion \eqref{T-exp-intro}. Hence doubly connected structures of the graphs in \eqref{T-exp-intro} will follow immediately from our construction. The global expansion strategy is based on a deeper study of subtle structures of the molecular graphs, where we identify two key graphical properties: the sequentially pre-deterministic property (cf. Definition \ref{def seqPDG}) and the globally standard property (cf. Definition \ref{defn gs}).   
We refer the reader to Section \ref{sec_idea} for a more detailed discussion of some key ideas.


\medskip

The rest of this paper is organized as follows. In Section \ref{sec notation}, we introduce some graphical notations that will be used in the definition of the $T$-expansion. In Section \ref{sec outline}, we define two main concepts of this paper---the $T$-expansion and $\incomp$. Then we will state main results of this paper, that is, Theorem \ref{incomplete Texp} and Corollary \ref{lem completeTexp} regarding the construction of the $\incomp$ and $T$-expansion up to arbitrarily high order, and, Theorem \ref{thm ptree} giving a sharp local law on $G(z)$ as a consequence of the $n$-th order $T$-expansion for any fixed $n\in \N$. In Section \ref{sec_basiclocal}, we introduce basic graph operations that are used to construct the $\incomp$ and $T$-expansion, and in Section \ref{sec pre}, we explore the \emph{pre-deterministic property} of graphs appearing in our expansions. With the tools in Sections \ref{sec_basiclocal} and \ref{sec pre}, we prove Theorem \ref{incomplete Texp} and Corollary \ref{lem completeTexp} in Section \ref{sec global}. In Section \ref{sec_pflocal}, we give the proof of Theorem \ref{thm ptree} based on three lemmas, Lemma \ref{eta1case0}, Lemma \ref{lem: ini bound} and Lemma \ref{lemma ptree}. Lemma \ref{eta1case0} and Lemma \ref{lemma ptree} will be proved in Section \ref{sec ptree}, and Lemma \ref{lem: ini bound} will be proved in Section \ref{sec ini_bound}.

\medskip
\noindent{\bf Acknowledgements.} We would like to thank Changji Xu for helpful discussions.

\section{Graphical tools}\label{sec notation}

To be prepared for the definition of the $\incomp$ and $T$-expansion, we introduce the concept of atomic and molecular graphs, and some important graphical properties that will be satisfied by our graphs. The graphical notations in this section have been defined in \cite{PartI_high}, and we repeat them in this section for completeness of this paper.

We introduce the following two deterministic matrices that will appear in expansions:
\be\label{Thetapm2}
S^+(z):=\frac{m^2(z) S}{1-m^2(z) S}, \quad S^-(z):=\overline S^+(z) . 
\ee
For the reason of defining these two matrices, we refer the reader to Definitions \ref{Ow-def} and \ref{GG-def} below. The matrix $S^+(z)$ satisfies the following Lemma \ref{lem deter}, which is a folklore result. A formal proof is given in equation (4.21) of \cite{PartII}. 
For simplicity of notations, throughout the rest of this paper, we abbreviate 
\be\label{Japanesebracket} |x-y|\equiv \|x-y\|_L,\quad \langle x-y \rangle := \|x-y\|_L + W.\ee

\begin{lemma} \label{lem deter}
	Suppose the assumptions of Theorem \ref{main thm} hold and $\eta\ge W^{2+\e}/L^2$ for a small constant $\e>0$. Then for any small constant $\tau>0$ and large constant $D>0$, we have that
	\be\label{S+xy}|S^\pm_{xy}(z)| \lesssim    W^{-d}\mathbf 1_{|x-y| \le W^{1+\tau}} + \OO\left(\langle x-y\rangle^{-D}\right). \ee 
\end{lemma}

\subsection{Atomic graphs} \label{unisec}


In an $n$-th order $T$-expansion, the number of terms will grow super-exponentially with respect to $n$ (there are about $ n^{Cn}$ many terms).  Moreover, each term has a complicated structure, which is represented by a graph in this paper. Our goal is to expand the $T$-variable $T_{\fa,\fb_1 \fb_2}$ for some fixed indices $\fa, \fb_1,\fb_2 \in \Z_L^d$. We represent these three indices by special vertices
\be\label{rep_abc}\fa\equiv \otimes, \quad \fb_1\equiv \oplus, \quad \fb_2\equiv \ominus,\ee
in graphs. In other words, we use $\fa,\fb_1,\fb_2$ in expressions, and draw them as $\otimes,\oplus,\ominus$ in graphs. Some of these indices can be equal to each other. For example, if $\fb_1=\fb_2$, then we only have two special vertices $\otimes$ and $\oplus$. 
We first introduce atomic graphs, and the concept of subgraphs. 

  \begin{definition}[Atomic graphs] \label{def_graph1} 
	Given a standard oriented graph with vertices and edges, we assign the following  structures and call the  resulting graph an atomic graph.  
	
	\begin{itemize}
		%
		%
		
		\item {\bf Atoms:} We will call the vertices atoms  (vs. molecules in Definition \ref{def_poly} below). Each graph has some external atoms and internal atoms. 
		The external atoms represent external indices whose values are fixed, while internal atoms represent summation indices that will be summed over. In particular, each graph in expansions of $T_{\fa,\fb_1\fb_2}$ has $\otimes$, $\oplus$ and $\ominus$ atoms. By fixing the value of an internal atom, it will become an external atom; by summing over an external atom, it will become an internal atom. 
		
		\item {\bf Regular weights:}  A regular weight on an atom $x $ represents a $G_{xx}$ or $\overline G_{xx}$ factor. Each regular weight has a charge, where ``$+$" charge indicates that the weight is a $G$ factor, represented by a blue solid $\Delta$, and ``$-$" charge indicates that the weight is a $\overline G$ factor, represented by a red solid $\Delta$. 
				
		\item \noindent{\bf Light weights}:  Corresponding to regular weights defined above, we define the light weights on atom $x$ representing $G_{xx}-m$ and $\overline G_{xx}-\overline m$ factors. They are drawn as blue or red hollow $\Delta$ in graphs depending on their charges.

		\item {\bf Edges:} The edges are divided into the following types. 
		
		\begin{enumerate}

			\item{\bf  Solid edges:} A solid edge represents a $G$ factor. More precisely, 
			\begin{itemize}
				\item each oriented edge from atom $\alpha$ to atom $\beta$ with $+$ charge represents a $G_{\al\beta}$ factor; 
				\item each oriented edge from atom $\alpha$ to atom $\beta$ with $-$ charge represents a $\overline G_{\al\beta}$ factor.
			\end{itemize}
			Plus $G$ edges will be drawn as blue solid edges, while minus $G$ edges will be drawn as red solid edges. In this paper, whenever we say ``$G$ edges", we mean both the plus and minus $G$ edges.

			\item {\bf Waved edges:} We have neutral black, positive blue and negative red waved edges:  
			\begin{itemize}
				\item a neutral waved edge connecting atoms $x$ and $y$ represents an $s_{xy}$ factor; 
				
				\item a blue waved edge of positive charge between atoms $x$ and $y$ represents an $S^+_{xy}$ factor;
				
				\item a red waved edge of negative charge between atoms $x$ and $y$ represents an $S^-_{xy}$ factor.
			\end{itemize}

			\item {\bf $\Dashed$ edges:} A $\dashed$ edge connecting atoms $x$ and $y$ represents a ${\Theta}_{xy}$ factor;
			we draw it as a double-line edge between atoms $x$ and $y$. 			
			
			\item {\bf Dotted edges:} A dotted line connecting atoms $\al$ and $\beta$  represents the factor $\mathbf 1_{\al=\beta}\equiv \delta_{\al\beta}$; a dotted line with a cross ($\times$) represents the factor $\mathbf 1_{\al\ne \beta} \equiv  1-\delta_{\al\beta} $. There is at most one dotted or $\times$-dotted edge between each pair of atoms. 
By definition, a $\times$-dotted edge between the two ending atoms of a $G$ edge indicates that this $G$ edge is off-diagonal. We also allow for dotted edges between external atoms.   
		\end{enumerate}
		The orientations of non-solid edges do not matter. Edges between internal atoms are called \emph{internal edges}, while edges with at least one end at an external atom are called \emph{external edges}.

		\item{\bf $P$ and $Q$ labels:} Some solid edges and weights may have a label $P_x$ or $Q_x$ for an atom $x$ in the graph. Moreover, every edge or weight can have at most one $P$ or $Q$ label.

		\item{\bf Coefficients:} There is a coefficient (which is a polynomial of $m$, $m^{-1}$, $(1-m^2)^{-1}$ and their complex conjugates) associated with each graph.  
		
		
	\end{itemize}
	
\end{definition}

Along the proof, we will introduce some other types of weights and edges.

\begin{definition}[Sugraphs]\label{def_sub}
	A graph $\cal G_1$ is said to be a subgraph of $\cal G_2$, denoted by $\cal G_1\subset \cal G_2$, if every graphical component of $\cal G_1$ is also in $\cal G_2$. Moreover, $\cal G_1 $ is a proper subgraph of $\cal G_2$ if  $\cal G_1\subset \cal G_2$ and $\cal G_1\ne \cal G_2$. Given a subset $\cal S$ of atoms in a graph $\cal G$, the subgraph $\cal G|_{\cal S}$ induced on $\cal S$ 
	refers to the subgraph of $\cal G$ with atoms in $\cal S$ as vertices, the edges between these atoms, and the weights on these atoms.  
\end{definition}

%

We assign a \emph{value} to each graph as follows.

\begin{definition}[Values of graphs]\label{ValG} 
Given an atomic graph $\mathcal G$, we define its value, denoted by $\llbracket \cal G\rrbracket$, as an expression obtained as follows. We first take the product of 
	all the edges, all the weights and the coefficient of the graph $\cal G$. 
	Then for the edges and weights with the same $P_x$ or $Q_x$ label, we take the product of them and apply $P_x$ or $Q_x$ to them. Finally, we sum over all the internal indices represented by the internal atoms. The values of the external indices are fixed by their given values. 
	For a linear combination of graphs  $\sum_i c_i \cal G_i$, where $\{c_i\}$ is a sequence of coefficients and $\{\cal G_i\}$ is a sequence of graphs, we naturally define its value by 
	$$ \Big\llbracket\sum_i c_i  \cal G_i\Big\rrbracket=\sum_i c_i \left\llbracket  \cal G_i\right\rrbracket.$$
	For simplicity, we will abuse the notation by identifying  a  graph (which is a geometric object) with its value (which is an analytic expression). 
\end{definition}

Next, we introduce the concepts of \emph{regular} and \emph{normal regular graphs}.

\begin{definition}[Normal regular graphs]  \label{defnlvl0} 
	We say an atomic graph $\cal G$ is \emph{regular} if it satisfies the following properties:
	\begin{itemize}
		\item[(i)] it is a connected graph that contains at most $\OO(1)$ many atoms and edges;
		\item[(ii)] all internal atoms are connected together through paths of waved and $\dashed$ edges;
		\item[(iii)] there are no dotted edges between internal atoms.
	\end{itemize}
	Moreover, we say a regular graph is \emph{normal} if it satisfies the following additional property:
	\begin{itemize}
		\item[(iv)] any pair of atoms $\al$ and $\beta$ in the graph are connected by a $\times$-dotted edge \emph{if and only if} they are connected by a $G$ edge.
	\end{itemize}
\end{definition}

By this definition, every $G$ edge in a normal regular graph is off-diagonal, while all diagonal $G$ factors will be represented by weights. Given a normal regular graph, we define its scaling order as follows. 

   \begin{definition} [Scaling order] \label{def scaling}
	Given a normal regular  graph $\cal G$, we define its scaling order as 
	\begin{align}
		\text{ord}(\cal G): = & \#\{\text{off-diagonal }  G  \text{ edges}\} + \#\{\text{light weights}\} + 2\#\{ \text{waved edges}\}  + 2 \#\{\text{$\dashed$ edges}\} \nonumber\\
		&  - 2\left[ \#\{\text{internal atoms}\}- \#\{\text{dotted edges}\} \right]. \label{eq_deforderrandom2}
	\end{align}
	Here every dotted edge in a normal regular graph means that an internal atom is equal to an external atom, so we lose one free summation index. The concept of scaling order can be also defined for subgraphs. 
\end{definition}  

In the following proof, whenever we say the order of a graph, we are referring to its scaling order. We emphasize that in general the scaling order does not imply the real size of the graph value directly. In order to establish such a connection, we need to introduce the \emph{doubly connected property} in Definition \ref{def 2net} below.  

\subsection{$\Selfs$ and labelled $\dashed$ edges}\label{sec_selfc}

One of the most important concepts in the $T$-expansion is the $\self$ introduced in \cite{PartI_high}. More precisely, an $l$-th order $\self$ is a linear combination of deterministic graphs satisfying the following definition.

%

\begin{definition}[$\Selfs$]\label{collection elements}
Fix any $l\in \N$. $\Sele_l(z)$ is a deterministic matrix depending on $m(z)$, $S$, $S^\pm(z)$ and $\Theta(z)$ only, and satisfying the following properties for $z= E+ \ii \eta$ with $E\in (-2+\kappa, 2- \kappa)$ and $\eta\in[ W^{2+\e}/L^{2}, 1]$.
	
	\begin{itemize}
		\item[(i)] For any $x,y \in \Z_L^d$, $  (\Selek_{l})_{xy}$ is a sum of at most $C_l$ many deterministic graphs of scaling order $l$ and with external atoms $x$ and $y$. Here $C_l$ is a large constant depending on $l$. Some graphs, say $\cal G$, in $\Sele_l$ can be diagonal matrices satisfying $  \cal G_{xy}= \cal G_{xx} \delta_{xy}$, i.e. there is a dotted edge between the atoms $x$ and $y$.  
		
		\item[(ii)] $ \Selek_{l} (z)$ satisfies the properties \eqref{two_properties0}--\eqref{3rd_property0}.
		
	\end{itemize} 
	Then we call $\Selek_l$ the  $l$-th order \emph{$\self$} and graphically we will use a square, $\square$, with a label $l$ and between  $x$ and $y$ to represent $(\Sele_l)_{xy}$. 
\end{definition}

By Definition \ref{def scaling}, the scaling order of a deterministic graph can only be even. Moreover, every nontrivial $\self$ $\Selek_{l}$ in this paper has scaling order $\ge 4$. Hence we always have 
\be\label{trivial conv}
\Sele_1=\Sele_2=\Sele_3=0, \quad \text{and}\quad \Sele_{2l+1}:=0,\quad l\in \N.
\ee
Using the properties \eqref{two_properties0}--\eqref{3rd_property0}, we have proved the following crucial estimates in \cite{PartI_high}. These estimates are necessary for bounding the graphs in the $T$-expansion and $\incomp$.  

\begin{lemma}[Lemma 6.2 of \cite{PartI_high}]\label{lem redundantagain}
 Fix $d\ge 6$. Let $\Selek_{2l} $ be a $\self$ satisfying Definition \ref{collection elements}. We have that for any constant $\tau>0$ and $x,y \in \Z_L^d$,
	\begin{align}\label{redundant again}
		\Big| \sum_{\al}\Theta_{x \al} (\Selek_{2l})_{\al y}\Big| \le \frac{W^\tau}{W^{(l-1)d}\langle x-y\rangle^d } .
	\end{align}
	Let $\Sele_{2k_1}, \ \Sele_{2k_2},\ \cdots , \ \Sele_{2k_l} $ be a sequence of $\selfs$ satisfying Definition \ref{collection elements}. We have that for any constant $\tau>0$ and $x,y \in \Z_L^d$,
	\be\label{BRB} 
	\left|\left(\Theta \Sele_{2k_1}\Theta  \Sele_{2k_2}\Theta \cdots \Theta  \Sele_{2k_l}\Theta\right)_{xy}\right|\le W^{-(k-2)d/2+\tau}B_{xy}  , \ee
	where $k:=\sum_{i=1}^l  2k_i -2(l-1)$ is the scaling order of $\left(\Theta \Sele_{2k_1}\Theta  \Sele_{2k_2}\Theta \cdots \Theta  \Sele_{2k_l}\Theta\right)_{xy}$.
\end{lemma}

By \eqref{BRB}, \smash{$\left(\Theta \Sele_{2k_1}\Theta  \Sele_{2k_2}\Theta \cdots \Theta  \Sele_{2k_l}\Theta\right)_{xy}$} has the same decay with respect to $|x-y|$ as $\Theta_{xy}$ except for an extra $W^{-(k-2)d/2}$ factor. Hence we will regard it as another type of diffusive edge. The proof of Lemma \ref{lem redundantagain} uses a summation by parts argument. In particular, the $\eta$ factor in \eqref{3rd_property0} plays an essential role in the sense that it will cancel the following factor from a row sum of $\Theta$:
\be\label{sumTheta}\sum_{y}\Theta_{xy}(z)=\frac{|m(z)|^2 }{1-|m(z)|^2 }\sim \eta^{-1}.\ee
In \cite{PartI_high}, we refer to \eqref{3rd_property0} as a \emph{sum zero property}.

\begin{definition}[Labelled $\dashed$ edges] \label{def_graph comp} 
Given $l$ $\selfs$ $\Sele_{2k_i}$, $ i =1,2,\cdots,l$, we represent the entry
	\be\label{eq label} \left(\Theta \Sele_{2k_1}\Theta  \Sele_{2k_2}\Theta \cdots \Theta  \Sele_{2k_l}\Theta\right)_{xy}\ee
	by a labelled $\dashed$ edge between atoms $x$ and $y$ with label $(k;  {2k_1}, \cdots , {2k_l})$, where $k:=\sum_{i=1}^l 2 k_i -2(l-1)$ is the scaling order of this edge.  
	In graphs, every labelled $\dashed$ edge is drawn as one single double-line edge with a label but  without any internal structure as in the following figure: 
	\begin{center}
		\includegraphics[width=11cm]{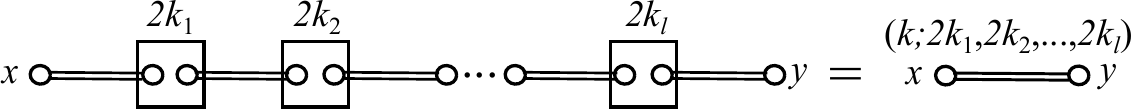}
	\end{center}
\end{definition}

In this paper, whenever we refer to $\dashed$ edges, we mean both the $\dashed$ and labelled $\dashed$ edges. The scaling order of a normal regular graph with labelled $\dashed$ edges can be equivalently calculated as
\begin{align}
	\text{ord}(\cal G) &:= \#\{\text{off-diagonal }  G  \text{ edges}\} + \#\{\text{light weights}\} + 2\#\{ \text{waved edges}\}  + 2 \#\{\text{$\dashed$ edges}\} \nonumber\\
	&+\sum_k k\cdot   \#\{k\text{-th order labelled $\dashed$ edges}\}  -  2\left[ \#\{\text{internal atoms}\}- \#\{\text{dotted edges}\} \right] .\label{eq_deforderrandom3}
\end{align}
In other words, a $k$-th order labelled $\dashed$ edge is simply counted as an edge of scaling order $k$, and there is no need to count its internal structures using Definition \ref{def scaling}.

\subsection{Molecular graphs and doubly connected property}\label{sec doublenet}
 
Each of our graphs has a two-level structure, that is, a local structure varying on scales of order $W$, and a global structure varying on scales up to $L$. To explain this, we introduce the following concept of molecules. 


\begin{definition}[Molecules]\label{def_poly}
We partition the set of all atoms into a union of disjoint sets $\{\text{all atoms}\}=\cup_j \cal M_j$, where every $\cal M_j$ is called a molecule. Two internal atoms belong to the same molecule if and only if they are connected by a path of neutral/plus/minus \emph{waved edges} and \emph{dotted edges}. 
	Every external atom will be by definition called an external molecule (such as $\otimes$, $\oplus$ and $\ominus$ molecules).  
	An edge is said to be inside a molecule if its ending atoms belong to this molecule.
\end{definition}
By \eqref{subpoly} and \eqref{S+xy}, if two atoms $x$ and $y$ are in the same molecule, then we essentially have $|x-y|\le W^{1+\tau}$ up to a negligible error $\OO(W^{-D})$. Given an atomic graph, we will call the subgraph inside a molecule the {\it local structure} of this molecule. On the other hand, the {\it global structure} of an atomic graph refers to its \emph{molecular graph}, which is the quotient graph with each molecule regarded as a vertex. 

\begin{definition}[Molecular graphs] \label{def moleg}
 Molecular graphs are graphs consisting of 
	\begin{itemize}
		\item external molecules which represent the external atoms (such as the $\otimes$, $\oplus$ and $\ominus$ molecules);
		\item internal molecules;
		\item blue and red solid edges, which represent the plus and minus $G$ edges between molecules;
		\item $\dashed$ edges between molecules;
		\item dotted edges between external and internal molecules.
	\end{itemize}
	Given any atomic graph $\cal G$, we define its molecular version, called $ \cal G_{\cal M}$, in the following way:
	\begin{itemize}
		\item every molecule of $\cal G$ is represented by a vertex in $ \cal G_{\cal M}$;
		
		\item every blue or red solid edge of $\cal G$ between atoms in different molecules is represented by a blue or red solid edge between these two molecules in $ \cal G_{\cal M}$;

		\item every $\dashed$ edge of $\cal G$ between atoms in different molecules is represented by a $\dashed$ edge between these two molecules in $ \cal G_{\cal M}$;
		
		\item every dotted edge of $\cal G$ between an external atom and an internal atom is represented by a dotted edge between the corresponding external and internal molecules;
		
		
		\item we discard all the other components in $\cal G$ (including weights, $\times$-dotted edges, and edges inside molecules).
	\end{itemize}
\end{definition}



In this paper, molecular graphs are used solely to help with the analysis of graph structures, while all graph expansions are  applied to atomic graphs only. In the proof, we assume that there is automatically a molecular graph associated with each atomic graph. In general, local structures are almost trivial to bound, while global structures are much harder to analyze. This is one of the main reasons why we want to introduce the concept of molecular graphs---we want to get rid of all local structures and focus on the more intricate global structures instead.

The following \emph{doubly connected property} is crucial for our proof, because it allows us to establish a direct connection between the scaling order of a graph and a bound on its value. All graphs in the $T$-expansion and $\incomp$ will satisfy this property (cf. Definitions \ref{defn genuni} and \ref{def incompgenuni}). 

\begin{definition} [Doubly connected property]\label{def 2net}  
A subgraph $\cal G$ without external molecules is said to be doubly connected if its molecular graph $\cal G_{\cal M}$ satisfies the following property. There exists a collection, say $\cal B_{black}$, of $\dashed$ edges, and another collection, say $\cal B_{blue}$, of  either blue solid or  $\dashed$ edges  such that (a) $\cal B_{black}\cap \cal B_{blue}=\emptyset$, and (b) both $\cal B_{black}$ and $\cal B_{blue}$ contain a spanning tree that connects all molecules in the graph. For simplicity of notations, we call the $\dashed$ edges in $\cal B_{black}$ as black edges, and the blue solid and $\dashed$ edges in $\cal B_{blue}$ as blue edges. Correspondingly, $\cal B_{black}$ and $\cal B_{blue}$ are referred to as black net and blue net, respectively, where a ``net" refers to a subset of edges that contains a spanning tree.  

A graph $\cal G$ with external molecules is called doubly connected if its subgraph with all external molecules removed is doubly connected,  i.e. the spanning trees in the two nets are not required to contain the external molecules.
\end{definition}
Note that red solid edges are not tracked in the doubly connected property. In fact, the connectivity of red solid edges will be broken in our expansion procedures. Hence we will often consider molecular graphs with all red solid edges removed. 
Doubly connected graphs satisfy some important bounds that depend explicitly on their scaling orders; see Lemma \ref{no dot} below.



\section{Main results}\label{sec outline}


In this section, we define the concepts of $T$-expansion and $\incomp$, which were introduced in \cite[Section 2]{PartI_high}. 
Then we will state our main results in Section \ref{subsec_main}. In this paper, we will focus on expansions of the $T$-variable $T_{\fa,\fb_1 \fb_2}$, while expansions of $T_{\fb_1\fb_2,\fa}$ can be obtained immediately by considering the transposition of $T_{\fa,\fb_1 \fb_2}$.

\subsection{$T$-expansion}\label{sec Texp2}

In this subsection, we define the $T$-expansion using the notations introduced in Section \ref{sec notation}. In addition, we also introduce the following two types of graphs. 
 
\begin{definition}[Recollision graphs and $Q$-graphs]\label{Def_recoll}
(i) We say a graph is a $\oplus$/$\ominus$-recollision graph, if there is at least one dotted edge connecting $\oplus$ or $\ominus$ to an internal atom. In other words, a recollision graph represents an expression where we set at least one summation index to be equal to $\fb_1$ or $\fb_2$.
 	
\medskip
\noindent(ii) We say a graph is a $Q$-graph if all  $G$ edges and $G$ weights in the graph have the same $Q_x$ label with a specific atom $x$, i.e., 
all $Q$ labels in the graph are the same $Q_x$ for some atom $x$. In other words, the value of a $Q$-graph can be written as $Q_x(\Gamma)$ for an external atom $x$ or $\sum_x Q_x(\Gamma)$ for an internal atom $x$, where $\Gamma$ is a graph without $Q$ labels. 
\end{definition}

For any $n\in \N$, we now define the $n$-th order $T$-expansion. Its graphs actually satisfy several more delicate properties to be stated later in 
Definition \ref{def genuni2}.

\begin{definition} [$n$-th order $T$-expansion]\label{defn genuni}
	Fix any $n\in \N$ and a large constant $D>n$. For $\fa, \fb_1,\fb_2 \in \Z_L^d$, an $n$-th order $T$-expansion of $T_{\fa,\fb_1\fb_2}$ with $D$-th order error is an expression of the following form: 
	\be\label{mlevelTgdef}
	\begin{split}
		T_{\fa,\fb_1 \fb_2}&= m  \Theta_{\fa \fb_1}\overline G_{\fb_1\fb_2} +m   (\Theta \Sdelta^{(n)}\Theta)_{\fa\fb_1} \overline G_{\fb_1\fb_2} \\
		&+ (\PTn)_{\fa,\fb_1 \fb_2} +  (\ATn)_{\fa,\fb_1\fb_2}  + (\QTn)_{\fa,\fb_1\fb_2}  +  (\Err_{n,D})_{\fa,\fb_1\fb_2}.
	\end{split}
	\ee
	The graphs on the right side depend only on $n$, $D$, $m(z)$, $S$, $S^{\pm}(z)$, $\Theta(z)$ and $G(z)$, but do not depend on $W$, $L$ and $d$ explicitly. Moreover, they satisfy the following properties with $C_n$ and $C_D$ denoting large constants depending on $n$ and $D$, respectively.
	\begin{enumerate}
		\item  The graphs on the right side are normal regular graphs (recall Definition \ref{defnlvl0}) with external atoms $\otimes\equiv \fa$, $\oplus\equiv \fb_1$ and $\ominus\equiv \fb_2$, and with at most $C_D$ many atoms.
		
		\item 
		A diffusive edge in any graph on the right-hand side of \eqref{mlevelTgdef} is either a $\dashed$ edge or a labelled $\dashed$ edge of the form \eqref{eq label} with $4\le 2k_i \le n$.
		
		\item $\Sdelta^{(n)}$ is a sum of at most $C_n$ many deterministic normal regular graphs. 
		We decompose it according to the scaling order as 
		\be\label{decompose Sdelta0} \Sdelta^{(n)} = \sum_{k\le n} \Sdeltak.\ee
		Moreover, we have a sequence of $\selfs$ $\Sele_{k}$ satisfying Definition \ref{collection elements} for $4\le k \le n$ such that $\Sdeltak$ can be written into the form 
		\be\label{chain S2k}
		\Sdeltak=\Sele_{k} + \sum_{ l =2}^k \sum\limits_{ \mathbf k=(k_1,\cdots, k_l) \in \Omega_{k}^{(l)} } \Sele_{k_1}\Theta  \Sele_{k_2}\Theta \cdots \Theta  \Sele_{k_l} . 
		\ee
		Here all deterministic graphs with $l=1$ are included into $\Sele_k$ so that the summation starts with $l=2$. Moreover,  $\Omega_{k}^{(l)} \subset \N^l$ is the subset of vectors $\mathbf k$ satisfying that 
		\be\label{omegakl}
		4\le k_i \le k-1,\quad \text{ and }\quad  \sum_{i=1}^l  k_i  -2(l-1)=k. 
		\ee
		The second condition of \eqref{omegakl} guarantees that the subgraph $(\Sele_{k_1}\Theta  \Sele_{k_2}\Theta \cdots \Theta  \Sele_{k_l})_{xy}$ has scaling order $k$.

		\item $(\PTn)_{\fa,\fb_1\fb_2}$ is a sum of at most $C_n$ many $\oplus/\ominus$-recollision graphs of scaling order $\le n$ and without any $P/Q$ labels.  Moreover, it can be decomposed as 
		\be\label{decomposeP}(\PTn)_{\fa,\fb_1\fb_2} =\sum_{ k =3}^{ n}(\PTk)_{\fa,\fb_1\fb_2},\ee
		where each $(\PTk)_{\fa,\fb_1\fb_2}$ is a sum of the $\oplus/\ominus$-recollision graphs of scaling order $k$ in $(\PTn)_{\fa,\fb_1\fb_2}$.

		\item  $(\ATn)_{\fa,\fb_1\fb_2}$ is a sum of at most $C_D$ many graphs of scaling order $> n$ and without any  $P/Q$ labels.

		\item $(\QTn)_{\fa,\fb_1\fb_2} $ is a sum of at most $C_D$ many $Q$-graphs.  Moreover, it can be decomposed as 
		\be\label{decomposeQ} (\QTn)_{\fa,\fb_1\fb_2}= \sum_{k=2}^{n} (\QTk)_{\fa,\fb_1\fb_2} + (\cal Q^{(>n)}_T)_{\fa,\fb_1\fb_2},\ee
		where  $(\QTk)_{\fa,\fb_1\fb_2}$ is a sum of the $Q$-graphs in $ (\QTn)_{\fa,\fb_1\fb_2}$ of scaling order $k$, and $(\cal Q^{(>n)}_T)_{\fa,\fb_1\fb_2}$ is a sum of all the $Q$-graphs in $ (\QTn)_{\fa,\fb_1\fb_2}$ of scaling order $>n$.

		\item $\Sdeltak$, $\PTk$ and $\QTk$ are independent  of $n$.
		
		\item $(\Err_{n,D})_{\fa,\fb_1\fb_2}$ is a sum of at most $C_D$ many graphs, each of which has scaling order $> D$ and may contain some $P/Q$ labels in it. 
		
		\item In every graph of $(\PTn)_{\fa,\fb_1\fb_2}$, $(\ATn)_{\fa,\fb_1\fb_2}$, $(\QTn)_{\fa,\fb_1\fb_2} $ and $(\Err_{n,D})_{\fa,\fb_1\fb_2}$,  there is a unique $\dashed$ edge connected to $\otimes$. Furthermore,  there is at least an edge, which is either blue solid or $\dashed$ or dotted,  connected to  $\oplus$, and there is at least an edge, which is either red solid or $\dashed$ or dotted,  connected to $\ominus$.
		
		\item Every graph in $(\PTn)_{\fa,\fb_1\fb_2}$, $(\ATn)_{\fa,\fb_1\fb_2}$, $(\QTn)_{\fa,\fb_1\fb_2}$ and $(\Err_{n,D})_{\fa,\fb_1\fb_2}$ is doubly connected in the sense of Definition \ref{def 2net}. 
		
		
		%
		
	\end{enumerate}
	The graphs on the right-hand side of \eqref{mlevelTgdef} satisfy some additional properties given in Definition \ref{def genuni2} below. 
\end{definition}

%




By \eqref{BRB}, the term $\left(\Theta \Sdeltak \Theta\right)_{\fa\fb_1}$ can be bounded as 
\be\label{intro_redagain}|\left(\Theta \Sdeltak \Theta\right)_{\fa\fb_1} |\le W^{-(k-2)d/2+\tau}B_{\fa\fb_1} \ee
for any constant $\tau>0$. This bound shows that when $\fb_1=\fb_2=\fb$, the second term on the right-hand side of \eqref{mlevelTgdef} can be bounded by  $m\overline G_{\fb\fb}(\Theta \Sdelta^{(n)} \Theta)_{\fa\fb}=\OO(W^\tau B_{\fa\fb})$, which is necessary for the local law \eqref{locallaw} to hold.

In \cite{PartI_high}, the following second order $T$-expansion is given explicitly in Lemma 2.5: 
\begin{align}\label{seconduniversal}
	& T_{\fa,\fb_1\fb_2}= 
	m \overline G_{\fb_1\fb_2}  \Theta_{\fa\fb_1}   + (\AT^{(>2)})_{\fa,\fb_1\fb_2} + (\QT^{(2)})_{\fa,\fb_1\fb_2}  ,
\end{align}
where 
\begin{align}
	(\AT^{(>2)})_{\fa,\fb_1\fb_2}& :  =  m\sum_{x,y} \Theta_{\fa x} s_{xy} (G_{yy}-m) G_{x \fb_1}\overline G_{x \fb_2}  + m \sum_{x,y} \Theta_{\fa x} s_{xy}  ( \overline G_{xx} - \overline m) G_{y \fb_1}\overline G_{y \fb_2}, \label{Aho>2}  \\
	(\QT^{(2)})_{\fa,\fb_1\fb_2}& := \sum_x Q_x\left(\Theta_{\fa x} G_{x \fb_1}   \overline G_{x \fb_2} \right) - m Q_{\fb_1} \left( \Theta_{\fa\fb_1} \overline G_{\fb_1 \fb_2}\right)  \nonumber\\
	& -  m\sum_{x,y} Q_x\left[\Theta_{\fa x} s_{xy} (G_{yy}-m) G_{x \fb_1} \overline G_{x \fb_2} \right] - m \sum_{x,y} Q_x\left[ \Theta_{\fa x} s_{xy}  \overline G_{xx} G_{y\fb_1}\overline G_{y\fb_2} \right] .\label{QT>2}
\end{align}
We still need to expand $(\AT^{(>2)})_{\fa,\fb_1\fb_2}$ and $(\QT^{(2)})_{\fa,\fb_1\fb_2} $ into sums of normal regular graphs using the operation in Definition \ref{dot-def} below, but we omit this minor issue here. 

\subsection{$\incomp$}

In this subsection, we define the $n$-th order $\incomp$ for any fixed $n\in \N$. 

\begin{definition} [$n$-th order $\incomp$]\label{def incompgenuni}
Fix any $n\in \N$ and a large constant $D>n$. For $\fa, \fb_1,\fb_2 \in \Z_L^d$, an $n$-th order $\incomp$ of $T_{\fa,\fb_1\fb_2}$ with $D$-th order error is an expression of the following form: 
	\be \label{mlevelT incomplete}
	\begin{split}
		T_{\fa,\fb_1 \fb_2} &= m  \Theta_{\fa \fb_1}\overline G_{\fb_1\fb_2}   +\sum_x (\Theta \wtSdeltan)_{\fa x} T_{x,\fb_1\fb_2} \\
		&+ (\PITn)_{\fa,\fb_1 \fb_2} +  (\AITn)_{\fa,\fb_1 \fb_2} + (\QITn)_{\fa,\fb_1 \fb_2} + (\Err'_{n,D})_{\fa,\fb_1\fb_2}  ,
	\end{split}
	\ee
	where the graphs on the right-hand side depend only on $n$, $D$, $m(z)$, $S$, $S^{\pm}(z)$, $\Theta(z)$ and $G(z)$, but do not depend on $W$, $L$ and $d$ explicitly. 
	Moreover, they satisfy the following properties.
	\begin{enumerate}
		\item  $(\PITn)_{\fa,\fb_1 \fb_2}$, $(\AITn)_{\fa,\fb_1 \fb_2}$, $(\QITn)_{\fa,\fb_1 \fb_2}$ and $(\Err'_{n,D})_{\fa,\fb_1\fb_2}$ respectively satisfy the same properties as $(\PTn)_{\fa,\fb_1 \fb_2}$, $(\ATn)_{\fa,\fb_1 \fb_2}$, $(\QTn)_{\fa,\fb_1 \fb_2}$ and $(\Err_{n,D})_{\fa,\fb_1\fb_2}$ in Definition \ref{defn genuni}. Furthermore,  $(\PITn)_{\fa,\fb_1 \fb_2}$ and $(\QITn)_{\fa,\fb_1 \fb_2}$ can be decomposed as
		\be\label{decomposePIT}(\PITn)_{\fa,\fb_1\fb_2} =\sum_{ k =3}^{ n}(\PITk)_{\fa,\fb_1\fb_2},\ee
		and
		\be\label{decomposeQIT} (\QITn)_{\fa,\fb_1\fb_2}= \sum_{k=2}^{n} (\QITk)_{\fa,\fb_1\fb_2} + (\cal Q^{(>n)}_{IT})_{\fa,\fb_1\fb_2},\ee
		where  $(\PITk)_{\fa,\fb_1\fb_2}$ is a sum of the $\oplus/\ominus$-recollision graphs in $(\PITn)_{\fa,\fb_1\fb_2}$ of scaling order $k$,  $(\QITk)_{\fa,\fb_1\fb_2}$ is a sum of the $Q$-graphs in $(\QITn)_{\fa,\fb_1 \fb_2}$ of scaling order $k$, and $(\cal Q^{(>n)}_{IT})_{\fa,\fb_1\fb_2}$ is a sum of the $Q$-graphs in  $(\cal Q^{(n)}_{IT})_{\fa,\fb_1\fb_2}$ of scaling order $>n$.  Moreover, $\PITk$ and $\QITk$ are independent of $n$.

		\item $\wtSdeltan$ can be decomposed according to the scaling order as 
		\be\label{decompose Sdelta} 
		\wtSdeltan = \Sele_n + \sum_{l=4}^{n-1} \Sele_{l},
		\ee
		where $\Sele_{l}$ (which is the same as that appeared in \eqref{chain S2k}), $1\le l \le n-1$, is a sequence of $\selfs$ satisfying Definition \ref{collection elements}. For any $x,y \in \Z_L^d$, $(\Sele_{n})_{xy}$ is a sum of at most $C_n$ many deterministic graphs of scaling order $n$ and with external atoms $x$ and $y$. 
		
		\item A diffusive edge in $\Sele_n$, $(\PITn)_{\fa,\fb_1 \fb_2}$, $(\AITn)_{\fa,\fb_1 \fb_2}$, $(\QITn)_{\fa,\fb_1 \fb_2}$ and $(\Err'_{n,D})_{\fa,\fb_1\fb_2}$ is either a $\dashed$ edge or a labelled $\dashed$ edge of the form \eqref{eq label} with $4\le 2k_i \le n-1$. 
		

		\item Every graph in $(\PITn)_{\fa,\fb_1\fb_2}$, $(\AITn)_{\fa,\fb_1\fb_2}$, $(\QITn)_{\fa,\fb_1\fb_2} $ and $(\Err_{n,D}')_{\fa,\fb_1\fb_2}$ can be written as 
		$$ \sum_{x}\Theta_{\fa x} {\cal G}_{x,\fb_1\fb_2},$$ 
		where $ {\cal G}_{x,\fb_1\fb_2}$ is a normal regular graph with external atoms $x$, $\fb_1$ and $\fb_2$. Moreover, $ {\cal G}_{x,\fb_1\fb_2}$ has at least an edge, which is either blue solid or $\dashed$ or dotted,  connected to  $\oplus$, and at least an edge, which is either red solid or $\dashed$ or dotted, connected to $\ominus$.
		
		

		\item Every graph in $\Sele_n$, $(\PITn)_{\fa,\fb_1\fb_2}$, $(\AITn)_{\fa,\fb_1\fb_2}$, $(\QITn)_{\fa,\fb_1\fb_2}$ and $(\Err'_{n,D})_{\fa,\fb_1\fb_2}$ is doubly connected in the sense of Definition \ref{def 2net}. 
	\end{enumerate}
	The graphs on the right side of \eqref{mlevelT incomplete} satisfy some additional properties given in Definition  \ref{def genuni2} below.
\end{definition} 
\begin{remark}
The form of \eqref{mlevelT incomplete} is different from \eqref{mlevelTgdef} only in the second term on the right-hand side. It can be regarded as a linear equation of the $T$-variable $T_{x,\fb_1\fb_2}$. If we move the second term on the right-hand side of \eqref{mlevelT incomplete} to the left-hand side and multiply both sides with $(1-\Theta\wtSdeltan)^{-1}$, we can get the $n$-th order $T$-equation (see Section \ref{sec lastglobal} for more details). Furthermore, taking expectation and choosing $\fb_1=\fb_2$ and $n=M$, we will get \eqref{E_locallaw} as discussed in \cite{PartI_high}.  We remark that using the language of Feynman diagrams, $\wtSdelta_n$ in \eqref{decompose Sdelta} is a sum of \emph{one-particle irreducible graphs}, while \smash{$\Sdelta^{(n)}$} in \eqref{decompose Sdelta0} is a sum of \emph{one-particle reducible graphs}. We also remark that \smash{$\Theta\Sdelta^{(n)}\Theta$} is exactly the sum of scaling order $\le n$ graphs in the Taylor expansion of $\Theta^{(n)}=(1-\Theta\wtSdeltan)^{-1}\Theta$.

We will construct a sequence of $\incomp$s inductively. In particular, before constructing the $n$-th order $\incomp$, we have obtained the $k$-th order $T$-expansion and proved the properties \eqref{two_properties0}--\eqref{3rd_property0} for $\Sele_k$ for all $4\le k\le n-1$. On the other hand, $ \Sele_n$ is a new sum of deterministic graphs obtained in the $n$-th order $\incomp$, whose properties  \eqref{two_properties0}--\eqref{3rd_property0} are yet to be shown. Moreover, we perform the expansions in a careful way so that the doubly connected property of the graphs in $\Sele_n$ will follow easily from the construction. 
\end{remark}


\subsection{Main results}\label{subsec_main}

We are ready to state the main results of this paper. First, given an $n$-th order $T$-expansion, we can get a local law in Theorem \ref{thm ptree} for random band matrices satisfying \eqref{Lcondition1}. For simplicity of presentation, we will adopt the following convention of stochastic domination introduced in \cite{EKY_Average}. 

\begin{definition}[Stochastic domination and high probability event]\label{stoch_domination}
	(i) Let
	\[\xi=\big(\xi^{(W)}(u):W\in\mathbb N, u\in U^{(W)}\big),\hskip 10pt \zeta=\big(\zeta^{(W)}(u):W\in\mathbb N, u\in U^{(W)}\big),\]
	be two families of non-negative random variables, where $U^{(W)}$ is a possibly $W$-dependent parameter set. We say $\xi$ is stochastically dominated by $\zeta$, uniformly in $u$, if for any fixed (small) $\tau>0$ and (large) $D>0$, 
	\[\mathbb P\bigg[\bigcup_{u\in U^{(W)}}\left\{\xi^{(W)}(u)>W^\tau\zeta^{(W)}(u)\right\}\bigg]\le W^{-D},\]
	for large enough $W\ge W_0(\tau, D)$, and we will use the notation $\xi\prec\zeta$.	If some complex family $\xi$ satisfies $|\xi|\prec\zeta$, then we will also write $\xi \prec \zeta$ or $\xi=\OO_\prec(\zeta)$. 
	
	\vspace{5pt}
	\noindent (ii) As a convention, for two deterministic non-negative quantities $\xi$ and $\zeta$, we will write $\xi\prec\zeta$ if and only if $\xi\le W^\tau \zeta$ for any constant $\tau>0$. 
	
	\vspace{5pt}
	\noindent (iii) We say that an event $\Xi$ holds with high probability (w.h.p.) if for any constant $D>0$, $\mathbb P(\Xi)\ge 1- W^{-D}$ for large enough $W$. More generally, we say that an event $\Omega$ holds $w.h.p.$ in $\Xi$ if for any constant $D>0$,
	$\P( \Xi\setminus \Omega)\le W^{-D}$ for large enough $W$.
\end{definition}

\begin{theorem}[Local law]\label{thm ptree}
	Fix any $n\in \N$. Under the assumptions of Theorem \ref{main thm}, suppose we have an $n$-th order $T$-expansion satisfying Definition \ref{defn genuni}. Assume that $L$ satisfies 
	\be\label{Lcondition1}  {L^2}/{W^2}  \le W^{(n-1)d/2-c_0} \ee
	for some constant $c_0>0$. Then for any constant $\e>0$, the local law 
	\be\label{locallaw1}
	|G_{xy} (z) -m(z)\delta_{xy}|^2 \prec B_{xy}
	\ee
	holds uniformly in all $z= E+\ii\eta$ with $E\in (-2+\kappa,2-\kappa)$ and $\eta \in [W^{2+\e}/L^{2},1]$.
\end{theorem}


Next we show that we can construct the $\incomp$ and $T$-expansion up to any fixed order $M\in \N$.

\begin{theorem}[Construction of the $\incomp$]   \label{incomplete Texp} 
Given any $M\in \N$, 
we can construct a sequence of $n$-th order $\incomp$s satisfying Definition \ref{def incompgenuni} for all $2\le n \le M$.
\end{theorem} 

Solving the $n$-th order $T$-equation, we can get an $n$-th order $T$-expansion. 

\begin{corollary}[Construction of the $T$-expansion]  \label{lem completeTexp} 
Assume that Theorem \ref{incomplete Texp} holds. Then we can construct a sequence of $n$-th order $T$-expansions satisfying Definition \ref{defn genuni} for all $2\le n\le M-1$. If we further assume that $\Sele_M$ satisfies \eqref{two_properties0}--\eqref{3rd_property0}, then we can also construct the $M$-th order $T$-expansion. 
\end{corollary}


In the proof of Theorem \ref{incomplete Texp}, we will construct a sequence of $\incomp$s inductively. Suppose we have constructed the $\incomp$ up to order $n$. Then by Corollary \ref{lem completeTexp}, we have an $(n-1)$-th order $T$-expansion. Using this $T$-expansion, we obtain from Theorem \ref{thm ptree} that the local law \eqref{locallaw1} holds when $L$ satisfies the condition (with $n$ in \eqref{Lcondition1} replaced by $n-1$)
\be\nonumber 
{L^2}/{W^2}  \le W^{(n-2)d/2-c_0} .
\ee
Using the local law \eqref{locallaw1}, it was proved in \cite{PartI_high} that $\Sele_n$ constructed in the $n$-th order $\incomp$ in Theorem \ref{incomplete Texp} is actually an $n$-th order $\self$.
  
\begin{lemma}[Lemma 5.8 of \cite{PartI_high}]\label{cancellation property}
The deterministic matrix $\Sele_n$ constructed in the $n$-th order $\incomp$ in Theorem \ref{incomplete Texp} satisfies the properties \eqref{two_properties0}--\eqref{3rd_property0} with $l=n$.
\end{lemma}

After showing that $ \Sele_{n} $ is a $\self$ satisfying Definition \ref{collection elements}, we can use Corollary \ref{lem completeTexp}
again to get the $n$-th order $T$-expansion. Then we use the $n$-th order $T$-expansion to construct the $(n+1)$-th order $\incomp$ and prove the properties \eqref{two_properties0}--\eqref{3rd_property0} for $\Sele_{n+1}$. Continuing this process, we can construct the $\incomp$ up to any fixed order. One contribution of this paper is that we develop a sophisticated strategy to construct the $n$-th order $\incomp$ by using the $(n-1)$-th order $T$-expansion. The reader can refer to Section \ref{sec_idea} for a discussion of some key ideas in the proof. 



\subsection{Main ideas}\label{sec_idea}

In this subsection, we briefly discuss some key ideas for the proof of Theorem \ref{incomplete Texp} in Sections \ref{sec_basiclocal}--\ref{sec global}. Similar ideas will also be used in the proof of Theorem \ref{thm ptree} in Sections \ref{sec_pflocal}--\ref{sec ini_bound}. We remark that the following discussions are only for heuristic purposes, and they are not completely rigorous regarding some technical details. 

\subsubsection{Basic strategy} 

As discussed above, we will construct a sequence of $\incomp$s \eqref{mlevelT incomplete} inductively. We start with the second order $T$-expansion in \eqref{seconduniversal}. As an induction hypothesis, suppose we have obtained the $(n -1)$-th order $T$-expansion.
We then construct the $n$-th order $\incomp$ by repeating the following two steps: \emph{local expansions} and \emph{global expansions}. Roughly speaking, given \eqref{seconduniversal}, we first apply local expansions in Section \ref{sec localexp} to expand \smash{$\AT^{(>2)}$} into a sum of \emph{locally standard} graphs in which every blue solid edge ($G$ entry) is paired with a red solid edge ($\overline G$ entry) to form a $T$-variable.  
Then in every locally standard graph, we choose a $T$-variable and substitute it with the $(n-1)$-th order $T$-expansion. This process is called a global expansion. If a resulting graph can be included into one of the terms on the right-hand side of \eqref{mlevelT incomplete}, then we will not expand it anymore. Otherwise, we will further expand it using local and global expansions. Our basic strategy is to repeat the above process until we get \eqref{mlevelT incomplete}. The main difficulty with this strategy is to maintain the doubly connectedness of graphs throughout the expansions, which will be discussed in next few subsections.   


%
%
%

The above expansion strategy involves inserting the $(n-1)$-th order $T$-expansion into the second order $T$-equation and then repeating local and global expansions to obtain the $n$-th order $T$-equation. This is different from the following ``naturally inductive"  strategy:  given an $(n-1)$-th order $\incomp$, we further expand the $n$-th order graphs in \smash{$\AIT^{(>n-1)}$} into graphs satisfying Definitions \ref{def incompgenuni} plus higher order graphs of scaling order $>n$. However, this procedure fails because there may be some non-expandable graphs in \smash{$\AIT^{(>n-1)}$}, i.e. graphs that cannot be expanded without breaking the doubly connected property. 
The existence of such graphs does not indicate that the concept of $\incomp$ fails, because there may be other graphs that cancel this graph up to a small error (although it is almost impossible to know what these graphs are).  
On the other hand, our global expansion strategy explicitly maintains the $\selfs$ (especially their delicate cancellations given by the sum zero property \eqref{3rd_property0}) by replacing each $T$-variable with a lower order $T$-expansion. In particular, this allows us to plug into terms \smash{$ \left(\Theta \Sdeltak \Theta\right)_{xy} $} as a whole, which are well-behaved by \eqref{intro_redagain}.

%

\subsubsection{Redundant edges and pre-deterministic property}

Local expansions preserve the doubly connected property (cf. Lemma \ref{lem localgood}), and they were dealt with in \cite{PartI_high}.
On the other hand, global expansions may break doubly connected structures of our graphs, and hence are much harder to deal with. We will see that a global expansion will preserve the doubly connected property if and only if we expand a \emph{redundant} blue solid edge. (Here an expansion of a blue solid edge refers to an expansion of a $T$-variable containing this edge.) As will be defined in Definition \ref{defn pivotal}, a blue solid edge is redundant if after removing it, the remaining graph is still doubly connected. A non-redundant blue solid edge is called \emph{pivotal}. To preserve the doubly connected structure of every graph, we will follow the strict rule that we  only \emph{expand a redundant blue solid edge in a global expansion}. 

Under the above rule, if we expand blue solid edges in an arbitrary order, then we may get graphs that are locally standard and only contain pivotal blue solid edges, namely, \emph{non-expandable} graphs. These graphs may not be included into one of the six terms on the right-hand side of \eqref{mlevelT incomplete}. 
To avoid such ``bad" graphs, we need to impose structural properties that are stronger than the doubly connected property. 

The leading term of a global expansion corresponds to replacing a blue solid edge with a $\dashed$ edge (i.e. the first term on the right-hand side of \eqref{mlevelTgdef}). Due to this observation, we see that in order to expand a graph, say $\cal G$, into deterministic graphs (together with some recollision, $Q$, higher order or error graphs) using global expansions under the above rule about redundant edges, we need the following property. All the blue solid edges in $\cal G$ can be replaced by $\dashed$ edges one by one according to an order, such that at each step we are dealing with a redundant edge. We call this property a \emph{pre-deterministic property}  (cf. Definition \ref{def PDG}), and the corresponding order of blue solid edges a \emph{pre-deterministic order}. The term ``pre-deterministic" refers to that pre-deterministic graphs will finally contribute to deterministic graphs (i.e. graphs in the second term on the right-hand side of \eqref{mlevelT incomplete}) after the prescribed expansion procedure. Since there is at least one redundant edge in a pre-deterministic graph, that is, the first edge in the pre-deterministic order, a pre-deterministic graph is always expandable unless it is deterministic.  

We remark that the pre-deterministic property is defined by only looking at leading terms of global expansions (i.e. replacing a $T$-variable with a graph in the first two terms on the right-hand side of \eqref{mlevelTgdef} with $n$ replaced by $n-1$). 
If we consider subleading graphs obtained by inserting recollision, $Q$ or higher order graphs, the pre-deterministic property may be broken in a global expansion. These subleading graphs actually satisfy a slightly weaker sequentially pre-deterministic property, which we introduce now.

\subsubsection{Isolated subgraphs and sequentially pre-deterministic property}

In a doubly connected molecular graph (by convention, blue and black nets are selected), every subset of molecules are connected to other molecules through at least an edge in the blue net and an edge in the black net. We define an isolated subgraph to be a subgraph induced on a subset of molecules that are connected to other molecules exactly by two edges in the molecular graph with all red solid edges removed (because red solid edges are not used in the definition of doubly connected property). It is easy to see that if a graph contains a proper  isolated subgraph, then this graph cannot be pre-deterministic
 because of the pivotal external blue solid edge connected with this isolated subgraph. In general, global expansions may create isolated subgraphs, and hence break the pre-deterministic property. Hence we need to introduce a weaker property, called the \emph{sequentially pre-deterministic} (SPD) property. 

A SPD graph is a doubly connected graph satisfying Definition \ref{def seqPDG} below. Roughly speaking, the term ``sequentially" refers to that a SPD graph contains a \emph{sequence of isolated subgraphs}, say $\Iso_k \subset \Iso_{k-1} \subset \cdots \subset \Iso_1$, as shown in the following figure:
\be\label{deter_comps}
\parbox[c]{0.65\linewidth}{\includegraphics[width=10cm]{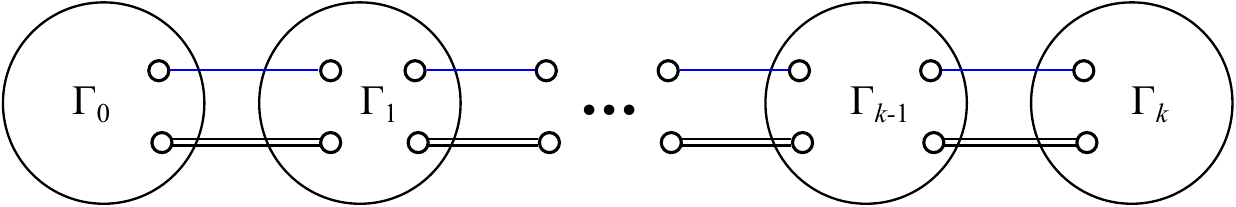}} 
\ee
Here we draw a molecular graph without red solid edges, and inside the black circles are some subgraphs. $\Gamma_k$ is the minimal isolated subgraph (MIS) $\Iso_k$, $\Gamma_k \cup \Gamma_{k-1}$ plus the two edges between them form the isolated subgraph $\Iso_{k-1}$, and $\Gamma_k \cup \Gamma_{k-1} \cup \cdots \cup \Gamma_{k-i}$ plus the edges between these components form the isolated subgraph $\Iso_{k-i}$. $\Gamma_0$ is a subgraph containing internal molecules that are connected with external edges, and we do not treat it as an isolated subgraph. The term ``pre-deterministic" refers to that the MIS of a SPD graph is pre-deterministic.   

As a key step of our proof, we will show that the SPD property will be preserved both in local expansions (cf. Lemma \ref{lem localgood2}) and global expansions (cf. Lemma \ref{lem globalgood}) if we always expand the \emph{first blue solid edge} in a pre-deterministic order of the MIS of every SPD graph. Hence continuing expansions, we can finally expand $T_{\fa, \fb_1\fb_2}$ into a sum of SPD graphs, which, by definition, imply the doubly property of them. However, unlike pre-deterministic graphs, a SPD graph is not always expandable.    
For example, if \eqref{deter_comps} corresponds to a locally standard atomic graph and all its components $\Gamma_i$ do not contain any blue solid edges, then the atomic graph only contains pivotal blue solid edges and cannot be expanded without breaking the doubly connected property.

\subsubsection{Globally standard property} 

Our goal is thus to find a graphical property as a compromise between the pre-deterministic and SPD properties, so that it is preserved in both local and global expansions and also guarantees that we will not get non-expandable graphs. We will identify such a property in Definition \ref{defn gs}, called the \emph{globally standard property}. Roughly speaking, a globally standard graph is a SPD graph such that its MIS is connected with at least two external red solid edges.  First, using the fact that the SPD property is preserved by local and global expansions, 
it is not hard to show that the globally standard property is also preserved by both local and global expansions (as long as we always expand the first blue solid edge in a pre-deterministic order of the MIS). Second, a globally standard graph is always expandable. To see why, we consider the setting \eqref{deter_comps}. If $\Gamma_k$ contains at least one blue solid edge, then it contains a redundant edge due to its pre-deterministic property. Otherwise, $\Gamma_k$ is connected with at least two external red solid edges and at most one blue solid edge. 
This shows that we cannot have a perfect pairing between blue and red solid edges, so the graph is not locally standard. Then we can apply local expansions to the graph. 


Now our expansion strategy roughly proceeds as follows. We first apply local expansions to \eqref{seconduniversal} to get locally and globally standard graphs. Then we expand the first blue solid edge in a pre-deterministic order of the MIS in every graph using 
the $(n-1)$-th order $T$-expansion. Next we apply local expansions to the resulting graphs and get some new locally and globally standard graphs. In every such graph, we further expand the first blue solid edge in a pre-deterministic order of the MIS. Repeating this process for sufficiently many times, we finally obtain a linear combination of graphs that can be written into the form \eqref{mlevelT incomplete}. Throughout the expansion process, the doubly connected property is preserved, so all the graphs in \eqref{mlevelT incomplete} are also doubly connected.

\section{Basic graph operations}\label{sec_basiclocal}
  
A graph operation $\cal O[\cal G]$ on a graph $\cal G$ is a linear combination of new graphs such that the graph value of  $\cal G$ 
is \emph{unchanged}, i.e. $\llbracket \cal O[\cal G]\rrbracket=\llbracket \cal G\rrbracket$. All graph operations are linear, that is,
\be\label{linear graphO} \cal O\Big[\sum_i c_i  \cal G_i\Big] =\sum_i c_i \cal O\left[ \cal G_i\right].
\ee
The graphs operations in Sections \ref{sec_dotted}--\ref{subsec global} have been defined in \cite[Section 3]{PartI_high}, while the $Q$-expansions in Section \ref{sec defnQ} are new. 
  
\subsection{Dotted edges operations}\label{sec_dotted}

Recall that a dotted edge between atoms $\al$ and $\beta$ represents a $\delta_{\al\beta}$ factor. We will identify internal atoms  connected by dotted edges, but we will {\it not} identify an external and an internal atom due to their different roles in graphs. We define the following simple merging operation regarding dotted edges. 
\begin{definition}[Merging operation]\label{def merge}
	Given a graph $\cal G$ that contains dotted edges between different atoms, we define an operator $\cal O_{merge}$ in the following way: $\cal O_{merge}[\cal G]$ is a graph obtained by merging every pair of internal atoms, say $\al$ and $\beta$, that are connected by a path of dotted edges into a single internal atom, say $\gamma$. Moreover, the weights and edges attached to $\al$ and $\beta$ 
	are now attached to the atom $\gamma$ 
	in $\cal O_{merge}[\cal G]$. In particular, the $G$ edges between $\al$ and $\beta$ become weights on $\gamma$, and the waved and $\dashed$ edges between $\al$ and $\beta$ become self-loops on $\gamma$. 
\end{definition}
Given any regular graph, we can rewrite it into a linear combination of normal regular graphs using the following \emph{dotted edge partition} operation. 


\begin{definition}[Dotted edge partition] \label{dot-def}
	Given a regular graph $\cal G$, for any pair of atoms $\al$ and $\beta$, if there is at least one $G$ edge but no $\times$-dotted edge between them, then we write 
	$$1=\mathbf 1_{\al=\beta} + \mathbf 1_{\al \ne \beta} ;$$
	if there is a $\times$-dotted edge but no $G$ edge between them, then we write 
	$$\mathbf 1_{\al\ne \beta} =1 - \mathbf 1_{\al = \beta}.$$
	Expanding the product of all these sums on the right-hand sides, we can expand $\cal G$ as
	\be\label{odot}\cal O_{dot}[\cal G] := \sum \cal O_{merge}[{\Dot} \cdot \cal G],\ee
	where each ${\Dot}$ is a product of dotted and $\times$-dotted edges together with a $+$ or $-$ sign. 
	In ${\Dot} \cdot \cal G$, if there is a $\times$-dotted edge between $\al$ and $\beta$, then the $G$ edges between them are off-diagonal; otherwise, the $G$ edges between them become weights after the merging operation. If $\Dot$ is ``inconsistent" (i.e., two atoms are connected by a  $\times$-dotted edge and a path of dotted edges), then we trivially have $\llbracket \Dot \cdot \cal G\rrbracket=0$. Thus we will drop all inconsistent graphs.  Finally, if the graph $\cal G$ is already normal, then $\cal O_{dot}$ acting on $\cal G$ is a null operation and we let $\cal O_{dot}[\cal G]:=\cal G$.

\end{definition}
From this definition, we see that $\cal O_{dot}[\cal G]$ of a regular graph $\cal G$ is a sum of normal regular graphs. 

\subsection{Local expansions} \label{sec localexp}

In this section, we define the basic local expansions. We call them ``local" because these expansions do not create new molecules, that is, every molecule in a new graph is obtained by adding new atoms to existing molecules or merging some molecules in the original graph.
It is easy to see that $\cal O_{dot}$ is a local expansion. We emphasize that local expansions can change the global structure, but they will preserve the doubly connected property as we will show in Lemma \ref{lem localgood}.

As introduced in \cite{PartI_high}, there are four basic local expansions, including one weight expansion and three edge expansions. We now define them one by one.  

\begin{definition} [Weight expansion operator]\label{Ow-def} \ 
	Given a normal regular graph $\cal G$ containing an atom $x$, if there is no weight or light weight on $x$, then 
	we trivially define $\cal O_{weight}^{(x)}[\cal G]:=\cal G$. Otherwise,  we define ${\cal O}_{weight}^{(x)}[\cal G]$ in the following way. 
	
	\medskip
	
	\noindent{\bf (i) Removing regular weights:} Suppose there are regular $G_{xx}$ or $\overline G_{xx}$ weights on $x$. Then we rewrite 
	$$G_{xx}=m+ (G_{xx}-m),\quad \text{or}\quad \overline G_{xx}=\overline m+ (\overline G_{xx}-\overline m).$$
	Expanding the product of all these sums, we can rewrite $\cal G$ into a linear combination of normal regular graphs containing only light weights on $x$. We denote this graph expansion operator as $ {\cal O}^{(x), 1}_{weight}$.  
	
	\medskip
	
	\noindent{\bf (ii) Expanding light weight:} If $\cal G$ has a light weight $G_{xx}-m$ of positive charge on $x$ and is of  the form $\cal G= (G_{xx}-m)f (G)$,  then we define the light weight expansion on $x$ by  
	\begin{align} 
		{\cal O}^{(x),2}_{weight} \left [  \cal G \right] &:=  m \sum_{  \al} s_{x\al}  (G_{xx}-m) (G_{\al\al}-m)f (G) +m \sum_{  \al,\beta}S^{+}_{x\al} s_{\al \beta}  (G_{\al\al}-m) (G_{\beta\beta}-m)f (G) \nonumber\\
		&-  m  \sum_{ \al} s_{x \al} G_{\al x}\partial_{ h_{ \al x}} f (G) -  m \sum_{ \al,\beta} S^{+}_{x\al}s_{\al \beta} G_{\beta \al}\partial_{ h_{ \beta\al}} f(G) + \cal Q_w ,\label{Owx}\end{align}
	where 
	$\cal Q_w$ is a sum of $Q$-graphs,
	\begin{align} 
		\cal Q_w &:=  Q_x \left[(G_{xx}  -  m) f(G)\right] + \sum_{\al} Q_\al \left[  S^{+}_{x\al}(G_{\al\al}  -  m) f(G)\right] \nonumber\\
		& - m  Q_x\Big[ \sum_{\al} s_{x\al}  (G_{\al\al}-m)G_{xx} f(G) \Big]  - m \sum_\al Q_\al\Big[ \sum_{ \beta}S_{x\al}^{+}s_{\al\beta}(G_{\beta\beta}-m)G_{\al\al} f(G) \Big] \nonumber\\
		&+  m  Q_x  \Big[  \sum_\al s_{x \al}  G_{\al x}\partial_{ h_{\al x}} f(G)\Big]+ m \sum_\al Q_\al \Big[  \sum_{ \beta} S^{+}_{x\al}s_{\al \beta} G_{\beta \al}\partial_{ h_{\beta \al}} f(G)\Big]. \nonumber
	\end{align} 
	If $\cal G= (\overline G_{xx}-\overline m)f (G)$, then we define 
	\be\label{Owx neg}
	{\cal O}^{(x),2}_{weight}\left [  \cal G \right] := \overline{ {\cal O}^{(x),2}_{weight} \left [  (G_{xx}  -  m) \overline{f(G)} \right ] } ,
	\ee
	where the right-hand side is defined using \eqref{Owx}. When there are more than one light weights on $x$, we pick any one of them and apply \eqref{Owx} or \eqref{Owx neg}. 
	
	Combining the above two graph operators, we define 
	\be\label{eq defOdot}
	{\cal O}_{weight}^{(x)}[ \cal G]:={\cal O}^{(x),1}_{weight} \circ\cal O_{dot} \circ  {\cal O}^{(x),2}_{weight} \circ {\cal O}^{(x),1}_{weight}  [\cal G], \ee
	where $\cal O_{dot}$ is applied to ensure that the resulting graphs after applying ${\cal O}_{weight}^{(x)}$ are normal regular. 
	
\end{definition}

We introduce the following notion of degree of atoms: 
\be\label{degree}
{\deg}(x) := \# \{\text{solid edges connected with $x$}\} . 
\ee
In other words, we only count the plus and minus $G$ edges connected with an atom regarding its ``degree". We define the following three edge expansions on atoms of nonzero degrees.


 \begin{definition} [Multi-edge expansion operator]\label{multi-def} 
 Given a normal regular graph $\cal G$, if there are no solid edges connected with the atom $x$, then we trivially define \smash{$\cal O_{multi-e}^{(x)}[\cal G]:=\cal G$}. Otherwise, we define \smash{${\cal O}_{multi-e}^{(x)}$} in the following way. 
 	Suppose the graph takes the form
 	\be\label{multi setting}
 	\cal G := \prod_{i=1}^{k_1}G_{x y_i}  \cdot  \prod_{i=1}^{k_2}\overline G_{x y'_i} \cdot \prod_{i=1}^{k_3} G_{ w_i x} \cdot \prod_{i=1}^{k_4}\overline G_{ w'_i x} \cdot f(G),
 	\ee
 	where $f (G) $ does not contain $G$ edges attached to $x$, and the atoms $y_i,$ $y'_i$, $w_i$ and $w'_i$ are all not equal to $x$.  
 	
 	\medskip
 	\noindent{(i)} If $k_1\ge 1$, then we define the multi-edge expansion on $x$ as 
 	\begin{align} 
 		& \wh{\cal O}_{multi-e}^{(x)} \left[\cal G\right] : = \sum_{i=1}^{k_2}|m|^2  \left( \sum_\al s_{x\al }G_{\al y_1} \overline G_{\al y'_i}\right)\frac{\cal G}{G_{x y_1} \overline G_{xy_i'}}   + \sum_{i=1}^{k_3} m^2 \left(\sum_\al s_{x\al }G_{\al y_1} G_{w_i \al} \right)\frac{\cal G}{G_{xy_1}G_{w_i x}} \nonumber \\
 		& + \sum_{i=1}^{k_2}m  (\overline G_{xx} -\overline m) \left( \sum_\al s_{x\al }G_{\al y_1} \overline G_{\al y'_i}\right)\frac{\cal G}{G_{x y_1} \overline G_{xy_i'}}   + \sum_{i=1}^{k_3} m  (G_{xx}-m) \left(\sum_\al s_{x\al }G_{\al y_1} G_{w_i \al} \right)\frac{\cal G}{G_{xy_1}G_{w_i x}} \nonumber \\
 		& +  m   \sum_\al s_{x\al }\left(G_{\al \al}-m\right) \cal G  +(k_1-1) m  \sum_\al s_{x\al } G_{x \al} G_{\al y_1}\frac{ \cal G}{G_{xy_1}}  + k_4 m   \sum_\al s_{x\al }\overline G_{\al x} G_{\al y_1}  \frac{\cal G}{G_{xy_1}}  \nonumber\\
 		& - m    \sum_\al s_{x\al } \frac{\cal G}{G_{x y_1}f(G)}G_{\al y_1}\partial_{ h_{\al x}}f (G)  +\cal Q_{multi-e} .\label{Oe1x}
 	\end{align}
 	On the right-hand side of \eqref{Oe1x}, the first two terms are main terms with the same scaling order as $\cal G$, but the degree of atom $x$ is reduced by 2 and a new atom $\al$ of degree 2 is created; the third to fifth terms contain one more light weight and hence are of  strictly higher scaling order than $\cal G$; the sixth to eighth terms contain at least one more off-diagonal $G$ edge and hence are of strictly higher scaling order than $\cal G$.  The last term $\cal Q_{multi-e} $ is a sum of $Q$-graphs defined by 
 	\begin{align*}
 		\cal Q_{multi-e} &:= Q_x \left( \cal G\right)  -  \sum_{i=1}^{k_2}m Q_x \left[ \overline G_{xx}  \left( \sum_\al s_{x\al }G_{\al y_1} \overline G_{\al y'_i}\right)\frac{\cal G}{G_{x y_1} \overline G_{xy_i'}} \right] \\
 		& - \sum_{i=1}^{k_3} m Q_x \left[ G_{xx} \left(\sum_\al s_{x\al }G_{\al y_1} G_{w_i \al} \right)\frac{\cal G}{G_{xy_1}G_{w_i x}}\right] -  m Q_x \left[  \sum_\al s_{x\al }\left(G_{\al \al}-m\right) \cal G\right] \\
 		&-(k_1-1) m Q_x \left[  \sum_\al s_{x\al }G_{x \al} G_{\al y_1}  \frac{ \cal G}{G_{xy_1}}\right] - k_4 m Q_x \left[  \sum_\al s_{x\al }\overline G_{\al x} G_{\al y_1}  \frac{\cal G}{G_{xy_1}}\right] \\
 		& +m Q_x \left[  \sum_\al s_{x\al } \frac{\cal G}{G_{x y_1}f(G)}G_{\al y_1}\partial_{h_{\al x}}f(G)  \right].
 	\end{align*} 
 	
 	\noindent{(ii)} If $k_1=0$ and $k_2\ge 1$, then we define 
 	$$\wh{\cal O}_{multi-e}^{(x)} \left[\cal G\right]:= \overline{  \wh{\cal O}_{multi-e}^{(x)} \left[\prod_{i=1}^{k_2}G_{x y'_i} \cdot \prod_{i=1}^{k_3} \overline  G_{ w_i x} \cdot \prod_{i=1}^{k_4} G_{ w'_i x} \cdot \overline{f (G)}\right] },$$
 	where the right-hand side can be defined using (i).
 	
 	\medskip
 	
 	\noindent{(iii)} If $k_1=k_2=0$ and $k_3 \ge 1$, then we can define $ \wh{\cal O}_{multi-e}^{(x)} \left[\cal G\right]$ by exchanging the order of matrix indices in (i). More precisely, we define
 	\begin{align} 
 		\wh{\cal O}_{multi-e}^{(x)} \left[\cal G\right] & := \sum_{i=1}^{k_4}|m|^2  \left( \sum_\al s_{x\al }G_{w_1\al} \overline G_{w'_i\al}\right)\frac{\cal G}{G_{w_1x} \overline G_{w_i'x}}   + \sum_{i=1}^{k_4}m  (\overline G_{xx} -\overline m) \left( \sum_\al s_{x\al }G_{w_1 \al } \overline G_{ w'_i \al}\right)\frac{\cal G}{G_{w_1x} \overline G_{w_i' x}}    \nonumber \\
 		& +  m   \sum_\al s_{x\al }\left(G_{\al \al}-m\right) \cal G  +(k_3-1) m  \sum_\al s_{x\al }G_{w_1\al} G_{\al x} \frac{ \cal G}{G_{w_1x}} \nonumber\\
 		&  - m    \sum_\al s_{x\al } \frac{\cal G}{G_{w_1x}f(G)}G_{  w_1 \al}\partial_{ h_{x\al}}f(G)  +\cal Q_{multi-e} ,\label{Oe1x3}
 	\end{align}
 	where 
 	\begin{align*}
 		\cal Q_{multi-e} &:= Q_x \left( \cal G\right)  -  \sum_{i=1}^{k_4}m Q_x \left[ \overline G_{xx}  \left( \sum_\al s_{x\al }G_{w_1 \al } \overline G_{ w'_i \al}\right)\frac{\cal G}{G_{w_1 x} \overline G_{w_i' x}} \right]  -  m Q_x \left[  \sum_\al s_{x\al }\left(G_{\al \al}-m\right) \cal G\right] \\
 		&  -(k_3-1) m Q_x \left[  \sum_\al s_{x\al }G_{w_1 \al } G_{\al x} \frac{ \cal G}{G_{w_1 x}}\right]  +m Q_x \left[  \sum_\al s_{x\al } \frac{\cal G}{G_{w_1x}f(G)}G_{ w_1\al}\partial_{ h_{x\al }}f(G) \right] .
 	\end{align*} 
 	
 	\noindent{(iv)} If $k_1=k_2=k_3=0$ and $k_4\ge 1$, then we define 
 	$$\wh{\cal O}_{multi-e}^{(x)} \left[\cal G\right]:= \overline{  \wh{\cal O}_{multi-e}^{(x)} \left[  \prod_{i=1}^{k_4} G_{ w'_i x} \cdot \overline{f (G)}\right] },$$
 	where the right-hand side can be defined using (iii).
 	
 	Finally, applying the $\cal O_{dot}$ in Definition \ref{dot-def}, we define
 	$${\cal O}_{multi-e}^{(x)}[ \cal G]:= \cal O_{dot} \circ  \wh{\cal O}_{multi-e}^{(x)}   [\cal G]. $$ 
 \end{definition}

   \begin{definition}[$G G$ expansion operator] \label{GG-def}
	Suppose a normal regular graph $\cal G$ takes the form $\cal G=G_{xy}   G_{y' x } f(G)$ with $y,y' \ne x$, where $f (G) $ does not contain $G$ edges attached to $x$. 
	Then we define
	\begin{align}
		\wh{\cal O}_{GG}^{(x)} [\cal G] & := 
		m S^{+}_{xy} G_{y' y} f(G)  + m \sum_\al  s_{x\al} (G_{\al \al}-m) \cal G + m \sum_{\al,\beta}  S^{+}_{x\al}  s_{\al\beta} (G_{\beta \beta}-m) G_{\al y}   G_{y'\al} f(G) \nonumber \\
		&  + m(G_{xx }-m)   \sum_\al s_{x\al}G_{\al y}   G_{y'\al} f(G) + m \sum_{\al,\beta}  S^+_{x\al} s_{\al\beta}  (G_{\al\al }-m) G_{\beta y}   G_{y'\beta} f(G)  \nonumber\\
		& - m \sum_{ \al}  s_{x\al}G_{\al y} G_{y' x} \partial_{ h_{\al x}}f(G) - m \sum_{\al,\beta} S^{+}_{x\al} s_{\al\beta}G_{\beta y} G_{y' \al} \partial_{ h_{\beta \al}}f(G) + \cal Q_{GG} .\label{Oe2x}
	\end{align}
	On the right hand side of \eqref{Oe2x}, the first term is the main term which is either of the same scaling order as $\cal G$ if $y=y'$ or of strictly higher scaling order if $y\ne y'$; the second to fifth terms contain one more light weight and hence are of strictly higher scaling order than $\cal G$; the sixth and seventh terms contain at least one more off-diagonal $G$ edge and hence are of strictly higher scaling order than $\cal G$; $\cal Q_{GG} $ is a sum of $Q$-graphs defined by
	\begin{align}
		\cal Q_{GG}&:=  Q_x \left(\cal G\right)+  \sum_\al Q_\al\Big[ S^+_{x\al}  G_{\al y}   G_{y'\al} f(G) \Big]  - m Q_y\Big[S^{+}_{xy}  G_{y' y} f(G)\Big] - m Q_x\Big[\sum_\al  s_{x\al} (G_{\al\al}-m) \cal G\Big]  \nonumber\\
		&- m \sum_\al Q_\al\Big[\sum_{ \beta}  S^{+}_{x\al}  s_{\al\beta} (G_{\beta \beta}-m)G_{\al y}   G_{y'\al} f(G)\Big] - mQ_x\Big[ G_{xx }  \sum_\al s_{x\al} G_{\al y}   G_{y'\al} f(G)\Big] \nonumber\\
		&- m\sum_\al Q_\al\Big[ \sum_{ \beta}  S^{+}_{x\al}  s_{\al\beta} G_{\al\al }  G_{\beta y}   G_{y'\beta} f(G)\Big]  + m Q_x\Big[\sum_{ \al}  s_{x\al} G_{\al y} G_{y' x} \partial_{ h_{ \al x}}f(G)\Big] \\
		& + m\sum_\al Q_\al\Big[ \sum_{\beta} S^{+}_{x\al} s_{\al\beta}G_{\beta y} G_{y' \al} \partial_{h_{\beta \al}}f (G)\Big] .\nonumber 
	\end{align}
	If $\cal G$ takes the form $\cal G=\overline G_{xy}  \overline G_{y' x }  f(G)$, then we define 
	$$\wh{\cal O}_{GG}^{(x)}\left[\cal G\right]:=\overline{ \wh{\cal O}_{GG}^{(x)}\left[G_{xy}   G_{y' x }  \overline{f(G)}\right] },$$
	where the right-hand side is defined using \eqref{Oe2x}. Finally, applying the $\cal O_{dot}$ in Definition \ref{dot-def}, we define
	$${\cal O}_{GG}^{(x)}[ \cal G]:= \cal O_{dot} \circ  \wh{\cal O}_{GG}^{(x)}   [\cal G]. $$  
\end{definition}

\begin{definition} [$G\overline G$ expansion operator]\label{GGbar-def} 
	Suppose a normal regular graph $\cal G$ takes the form $\cal G= G_{xy}  \overline G_{xy'} f(G)$ with $y,y' \ne x$, where $f (G) $ does not contain $G$ edges attached to $x$.
	Then by taking $k_1=k_2=1$ and $k_3=k_4=0$ in Definition \ref{multi-def}, we define  
	\begin{align}
		\wh{\cal O}_{G\overline G}^{(x)}[\cal G] &:=   |m|^2  \sum_\al s_{x\al }G_{\al y} \overline G_{\al y'} f(G) + m \sum_\al s_{x\al }\left(G_{\al \al} -m \right) \cal G \nonumber  \\
		&  +  m (\overline G_{xx} - \overline m)   \sum_\al s_{x\al }G_{\al y} \overline G_{\al y'} f(G)  - m  \sum_\al s_{x\al }G_{\al y}  \overline G_{xy'} \partial_{ h_{\al x}} f(G) + \cal Q_{G\overline G},\label{Oe3x}
	\end{align}
	where $\cal Q_{G\overline G} $ is defined by
	\begin{align*}
		\cal Q_{G\overline G}&:= Q_x \left(\cal G\right) -  m Q_x\Big[ \sum_\al s_{x\al }\overline G_{xx}G_{\al y} \overline G_{\al y'} f(G) \Big] - m Q_x\Big[\sum_\al s_{x\al }\left(G_{\al \al} -m \right) \cal G\Big] \\
		& + mQ_x\Big[  \sum_\al s_{x\al }G_{\al y}  \overline G_{xy'} \partial_{ h_{\al x}} f(G) \Big].
	\end{align*}
	If $\cal G$ takes the form $\cal G= G_{yx}  \overline G_{y'x} f(G)$, we define $\cal O_{G\overline G}(x)[ \cal G] $ by taking $k_1=k_2=0$ and $k_3=k_4=1$ in Definition \ref{multi-def}, and we omit the explicit expression. Finally, applying the $\cal O_{dot}$ in Definition \ref{dot-def}, we define
	$${\cal O}_{G\overline G}^{(x)}[ \cal G]:= \cal O_{dot} \circ  \wh{\cal O}_{G\overline G}^{(x)}   [\cal G]. $$  
\end{definition}





As shown in \cite{PartI_high}, the operations $\cal O^{(x)}_{weight}$, $\cal O^{(x)}_{multi-e}$, $\cal O^{(x)}_{GG}$ and $\cal O^{(x)}_{G\overline G}$ are all obtained from Gaussian integration by parts. Their basic properties have been summarized in Section 3 of \cite{PartI_high}. In particular, it is easy to see that the new atoms $\al$ and $\beta$ appearing in these expansions are connected to atom $x$ through a path of waved edges, so they are included into the existing molecule containing atom $x$. Hence Definitions \ref{Ow-def}--\ref{GGbar-def} all give local expansions. Applying these expansions repeatedly, we can expand any regular graph into a linear combination of \emph{locally standard graphs},  $Q$-graphs and higher order graphs. To define locally standard graphs, we first introduce the following concept of \emph{standard neutral atoms}.

\begin{definition}[{Standard neutral atoms}]\label{def SNA}
	Consider an internal atom $x$ of degree 2 in a graph. We say the two edges connected with $x$ are \emph{mismatched} if they are of  following forms: 
	\begin{center}
		\parbox[c]{0.8\linewidth}{\includegraphics[width=13cm]{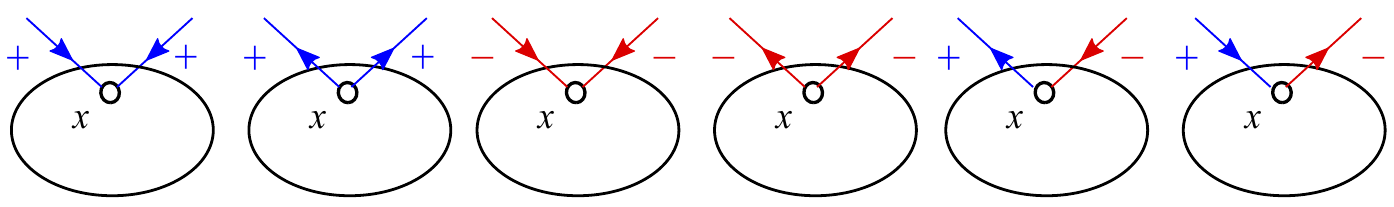}} 
	\end{center}
	Otherwise the two edges are \emph{matched} and of the following forms:
	\begin{center}
		\parbox[c]{0.56\linewidth}{\includegraphics[width=9cm]{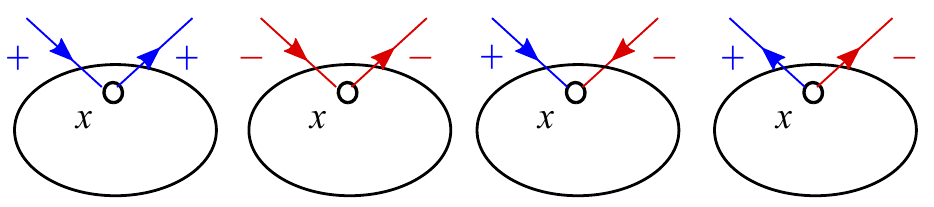}} 
	\end{center}
	An atom is said to be \emph{standard neutral} if it is only connected to three edges besides possible $\times$-dotted edges: two {matched} $G$ edges of {opposite charges} and one waved $S$ edge. 
\end{definition}

The edges connected with a standard neutral atom almost form a $T$-variable (but not an exact $T$-variable because of $\times$-dotted edges; see Section \ref{subsec global} for more details).

 \begin{definition} [Locally standard graphs] \label{deflvl1} 
	A graph $\cal G$ is  \emph{locally standard}  if 
	\begin{itemize}
		\item[(i)] it is a normal regular graph without $P/Q$ labels; 
		
		
		
		\item[(ii)] it has no weights or light weights;
		
		\item[(iii)] the degree of any internal atom is $0$ or $2$; 
		
		\item[(iv)] all degree 2 internal atoms are standard neutral. 
		
		\end{itemize}
\end{definition}

Given a regular graph $\cal G$, we first apply $\cal O_{dot}$ to expand it into a linear combination of normal regular graphs, then apply the weight expansion to remove all weights, and then apply the multi-edge, $GG$ and $G\overline G$ expansions one by one to remove all atoms that are not standard neutral. In the end, we can expand $\cal G$ into a sum of locally standard graphs, recollision graphs, $Q$-graphs and higher order graphs. The following basic lemma on local expansions has been proved in \cite{PartI_high}.

\begin{lemma}[Lemma 3.22 of \cite{PartI_high}] \label{lvl1 lemma}
Let $\cal G_{\fa, \fb_1\fb_2}$ be a normal regular graph without solid edges connected with $\fa$. Then for any fixed $n\in \N$, we have that
\begin{align}\label{expand lvl1}
\cal G_{\fa, \fb_1\fb_2} =  (\cal G_{local})_{\fa, \fb_1\fb_2}  + \PGn_{\fa,\fb_1\fb_2} +  (\AGn)_{\fa,\fb_1\fb_2} + \QGn_{\fa,\fb_1\fb_2} ,
\end{align}
where $(\cal G_{local})_{\fa, \fb_1\fb_2} $ is a sum of $\OO(1)$ many locally standard graphs, \smash{$ \PGn_{\fa,\fb_1\fb_2} $} is a sum of $\OO(1)$ many $\oplus$/$\ominus$-recollision graphs, \smash{$(\AGn)_{\fa,\fb_1\fb_2}$} is a sum of $\OO(1)$ many graphs of scaling order $> n$, and \smash{$(\QGn)_{\fa,\fb_1\fb_2} $} is a sum of $\OO(1)$ many $Q$-graphs.  Every molecule in a graph on the right-hand side is obtained by merging some molecules in the original graph $\cal G_{\fa, \fb_1\fb_2}$. 
\end{lemma}


If there are no dotted or waved edges added  between different molecules, then the molecules in a new graph are the same as those  in $\cal G_{\fa, \fb_1\fb_2}$. In general, there may be newly added dotted edges (due to the dotted edge partition $\cal O_{dot}$) or waved edges (due to the first term on the right-hand side of \eqref{Oe2x}), so the molecules in a new graph are obtained from merging the molecules connected by dotted or waved edges. Now we show that local expansions do not break the doubly-connected property.

\begin{lemma}\label{lem localgood}
Suppose a normal regular graph $\cal G$ 
is doubly-connected in the sense of Definition \ref{def 2net}. Then applying any expansion in Definitions \ref{dot-def}, \ref{Ow-def}, \ref{multi-def}, \ref{GG-def} or \ref{GGbar-def} to $\cal G$, the new graphs (including the $Q$-graphs) are all doubly connected normal regular graphs. 
 \end{lemma}
 \begin{proof}
First, all blue solid and (labelled) $\dashed$ edges between molecules are all unchanged in a dotted edge partition. The only changes are that in new molecular graphs, some molecules are merged into one molecule due to newly added dotted edges between different molecules. It is obvious that such merging operations will not break the doubly connected property. 

Then we consider the weight expansion in Definition \ref{Ow-def}. It is trivial to see that \smash{${\cal O}^{(x),1}_{weight}$} does not change molecular graphs at all. Hence it remains to show that \smash{${\cal O}^{(x),2}_{weight}[\cal G]$} of a doubly connected graph $\cal G$ is still a sum of doubly connected graphs. Let $\cal G_1$ be one of the new graphs. Since new atoms are all included into the molecule containing $x$, the molecules in the molecular graph of $\cal G_1$ are the same as those in the original molecular graph of $\cal G$. Moreover, (labelled) $\dashed$ edges between molecules are also not affected, i.e. the black net does not change. It remains to consider the plus $G$ edges between molecules in $\cal G_1$. From \eqref{Owx}, it is easy to see that any plus $G$ edge in the molecular graph of $\cal G$ between two different molecules, say $\cal M_1$ and $\cal M_2$, either is not affected or becomes a connected path of two plus $G$ edges between $\cal M_1$ and $\cal M_2$ due to the partial derivatives $\partial_{h_{ \al x}}$ or $\partial_{h_{\beta\al}}$. (For example, if the atoms $x$, $\al$ and $\beta$ are in molecule $\cal M_3$, then the edge between $\cal M_1$ and $\cal M_2$ becomes two plus $G$ edges from $\cal M_1$ to $\cal M_3$ and from $\cal M_3$ to $\cal M_2$.) In either case, the path connectivity of the blue net is not affected, and hence the doubly connected property preserves in $\cal G_1$.


For the edge expansions in Definitions \ref{multi-def}, \ref{GG-def} and \ref{GGbar-def}, the argument is exactly the same. 
 \end{proof}

\subsection{Global expansions}\label{subsec global}

In this section, we define global expansions. As an induction hypothesis, suppose we have obtained the $(n-1)$-th order $T$-expansion.  Then a global expansion of a locally standard graph, say $\cal G$, consists of the following three steps:
\begin{itemize}
	\item[(i)] we choose a standard neutral atom in $\cal G$;
	\item[(ii)] we replace the $T$-variable containing the atom in (i) by the $(n-1)$-th order $T$-expansion;
	\item[(iii)] we apply $Q$-expansions to every resulting graph with $Q$-labels from (ii).
\end{itemize}
We call this procedure ``global" because it may create new molecules in our graphs. 
Unlike local expansions, a global expansion may break the doubly connectedness of graphs. To avoid this issue, 
we need to follow a specific criterion in choosing the standard neutral atom in (i). This will be discussed in Section \ref{sec pre}. 
In this subsection, we define step (ii) in the above procedure, and postpone the definition of $Q$-expansions in step (iii) to the next subsection.

Picking a standard neutral atom, say $\al$, in a locally standard graph, the edges connected to it take one of the following forms:
\begin{equation}\label{12in4T}
	t_{x,y_1 y_2} :=|m|^2\sum_\al s_{x\al}G_{\al y_1}\overline G_{\al y_2}\mathbf 1_{\al\ne y_1}\mathbf 1_{\al\ne y_2},\quad \text{or}\quad t_{y_1y_2, x}:=|m|^2\sum_\al G_{ y_1\al }\overline G_{y_2 \al}s_{\al x}\mathbf 1_{\al\ne y_1}\mathbf 1_{\al\ne y_2} .
\end{equation}
Then we apply the $(n-1)$-th order $T$-expansion in \eqref{mlevelTgdef} to these variables in the following way: 
\begin{align*}
	t_{x,y_1 y_2} & = m  \Theta_{xy_1}\overline G_{y_1y_2} +m (\Theta \Sdelta^{(n-1)}\Theta)_{xy_1} \overline G_{y_1y_2} + (\PT^{(n-1)})_{x,y_1 y_2} +  (\AT^{(>n-1)})_{x,y_1y_2}  + (\QT^{(n-1)})_{x,y_1y_2} \nonumber\\
	& +  (\Err_{n-1,D})_{x,y_1y_2}  - \sum_\al s_{x\al}G_{\al y_1}\overline G_{\al y_2}\left(\mathbf 1_{\al\ne y_1}\mathbf 1_{\al = y_2}+\mathbf 1_{\al=y_1}\mathbf 1_{\al \ne y_2} +\mathbf 1_{\al = y_1}\mathbf 1_{\al = y_2}\right) .
\end{align*}
The last term on the right-hand side gives one (since $\mathbf 1_{\al\ne y_1}\mathbf 1_{\al = y_2} =\mathbf 1_{\al=y_1}\mathbf 1_{\al \ne y_2}=0$ if $y_1= y_2$) or two (since $\mathbf 1_{\al = y_1}\mathbf 1_{\al = y_2} =0$ if $y_1\ne y_2$) recollision graphs, so we combine it with {$ (\PT^{(n-1)})_{x,y_1 y_2}$} and denote the resulting term by {$(\wtPT^{(n-1)})_{x,y_1 y_2}$}. Hence we get the following expansion formula of the $t$-variable $t_{x,y_1 y_2}$:
\begin{align}
	t_{x,y_1 y_2} &=   m  \Theta_{xy_1}\overline G_{y_1y_2} +m   (\Theta \Sdelta^{(n-1)}\Theta)_{xy_1} \overline G_{y_1y_2}  +  (\wtPT^{(n-1)})_{x,y_1 y_2} +  (\AT^{(>n-1)})_{x,y_1y_2}  \nonumber\\
	& + (\QT^{(n-1)})_{x,y_1y_2} +  (\Err_{n-1,D})_{x,y_1y_2} .\label{replaceT}
\end{align}
The expansion of $t_{y_1y_2, x}$ can be obtained by exchanging the order of every pair of matrix indices in the above equation. 

\subsection{$Q$-expansions}\label{sec defnQ}

If we replace a $t$-variable with a graph in $\QT^{(n-1)}$, then we will get a graph of the form  
\be\label{QG}\cal G_0=\sum_x \Gamma_0 Q_x (\wt\Gamma_0) ,\ee
where both the graphs $\Gamma_0$ and $\wt\Gamma_0$ do not contain $P/Q$ labels. The $Q$-expansions defined  in this subsection aim to expand the above graph into a sum of $Q$-graphs and graphs without $P/Q$ labels.  By Definition \ref{Def_recoll}, $Q$-graphs refer to graphs where all edges and weights have the same $Q$-label. The graph $\cal G_0$ does not satisfy this property, so we will not call it a $Q$-graph to avoid confusion. 

\medskip
\noindent{\bf Step 1:} First, we apply \smash{${\cal O}^{(x),1}_{weight} \circ\cal O_{dot}$} to $\cal G_0$ in \eqref{QG} and get a sum of normal regular graphs without regular weights on atom $x$. Now for any new graph, say $\cal G=\sum_x \Gamma Q_x (\wt \Gamma), $ if $\Gamma$ does not contain any light weight or edges attached to $x$, then we apply the following operations.

Let $H^{(x)}$ be the $(N-1)\times(N-1)$ minor of $H$ obtained by removing the $x$-th row and column of $H$, and define the resolvent minor $G^{(x)}(z):=(H^{(x)}-z)^{-1}$. The following resolvent identity is well-known:
\be\label{resol_exp0}G_{x_1x_2}=G_{x_1x_2}^{(x)}+\frac{G_{x_1 x}G_{xx_2}}{G_{xx}},\quad x_1,x_2\in \Z_L^d,\ee
which follows from the Schur complement formula. Correspondingly, we introduce a new type of weights, $(G_{xx})^{-1}$ and $(\overline G_{xx})^{-1}$, and a new label $(x)$ to solid edges and weights. More precisely, if a weight on $x$ has a label ``$(-1)$", then this weight represents a $(G_{xx})^{-1}$ or $(\overline G_{xx})^{-1}$ factor depending on its charge; if an edge or a weight has a label $(x)$, then it represents a $G^{(x)}$ entry. 

Applying \eqref{resol_exp0} to expand the resolvent entries in $\Gamma$ one by one, we can get that
\be\label{decompose_gamma} \Gamma=\Gamma^{(x)}+\sum_{\omega} \Gamma_\omega,\ee 
where $\Gamma^{(x)}$ is a graph whose weights and solid edges all have the $(x)$ label, and hence is independent of the $x$-th row and column of $H$. The second term on the right-hand side of \eqref{decompose_gamma} is a sum of $\OO(1)$ many graphs. Moreover, every $\Gamma_\omega$ has a strictly higher scaling order than $\Gamma$, at least two new solid edges connected with $x$, and some weights with label $(-1)$ on $x$. Using \eqref{decompose_gamma}, we can expand $\cal G$ as
\be\label{QG2}
\cal G=\sum_x \Gamma Q_x (\wt \Gamma)=\sum_{\omega}\sum_x \Gamma_\omega Q_x(\wt \Gamma) + \sum_x Q_x ( \Gamma \wt\Gamma )- \sum_{\omega}\sum_x Q_x (\Gamma_\omega \wt\Gamma ),
\ee
where the second and third terms are sums of $Q$-graphs.

Next, we expand graphs $\Gamma_\omega$ in \eqref{QG2} into graphs without $G^{(x)}$, $(G_{xx})^{-1}$ or $(\overline G_{xx})^{-1}$ entries. First, we apply the following expansion to the $G^{(x)}$ entries in $\Gamma_{\omega}$: 
\be\label{resol_reverse}G_{x_1x_2}^{(x)}= G_{x_1x_2} - \frac{G_{x_1 x}G_{xx_2}}{G_{xx}}. \ee
In this way, we can write $\Gamma_{\omega}$ as a sum of graphs \smash{$\sum_{\zeta} \Gamma_{\omega,\zeta}$}, where $\Gamma_{\omega,\zeta}$ does not contain any $G^{(x)}$ entries. Second, we expand the $(G_{xx})^{-1}$ and $(\overline G_{xx})^{-1}$ weights in $\Gamma_{\omega,\zeta}$ using the Taylor expansion
\be\label{weight_taylor}\frac{1}{G_{xx}}=\frac{1}{m} + \sum_{k=1}^{D}\frac1m\left(-\frac{G_{xx}-m}{m}\right)^k + \cal W^{(x)}_{D},\quad \cal W^{(x)}_{D}:=\sum_{k>D}\left(-\frac{G_{yy}-m}{m}\right)^k . \ee
We will regard $\cal W^{(x)}_{D}$ and $\overline{\cal W}^{(x)}_{D}$ as a new type of weights of scaling order $D+1$. Expanding the product of all Taylor expansions of the $(G_{xx})^{-1}$ and $(\overline G_{xx})^{-1}$ weights, we can write every $\Gamma_{\omega,\zeta}$ into a sum of graphs that do not contain weights with label $(-1)$. 
Finally, applying \smash{${\cal O}^{(x),1}_{weight} \circ\cal O_{dot}$} again, we will get a sum of normal regular graphs without regular weights on $x$. 

Now we summarize Step 1 in the following lemma. 
\begin{lemma}\label{Q_lemma1}
Given a normal regular graph taking the form \eqref{QG}, we can expand it into a sum of $\OO(1)$ many graphs: 
\be\label{Q_eq1}\cal G_0 = \sum_{\omega} \Gamma_\omega Q_x(\wt \Gamma_{\omega}) + \sum_{\zeta} Q_x \left( \cal G_\zeta \right) + \cal G_{err},\ee
where $\Gamma_{\omega}$, $\wt \Gamma_{\omega}$ and $\cal G_\zeta $ are normal regular graphs without $P/Q$, ${(x)}$ or $(-1)$ labels, and $\cal G_{err}$ is a sum of normal regular graphs of scaling order $>D$.  
If $\Gamma_0$ does not contain any weights or edges attached to $x$, then every $\Gamma_\omega$ has a strictly higher scaling order than $\Gamma_0$ and contains at least one atom which either is connected to $x$ through a solid edge or has been merged with $x$.

\end{lemma}
\begin{proof} 
We only need to prove the second statement. Suppose $\Gamma_0$ does not contain any weights or edges attached to $x$. Then there exists at least one $G_{x_1x_2}$ or $\overline G_{x_1x_2}$ entry in $\Gamma_0$ that has been replaced by two solid edges $G_{x_1 x}G_{xx_2}$ or $\overline G_{x_1 x}\overline G_{x x_2}$ (together with a $(G_{xx})^{-1}$ or $(\overline G_{xx})^{-1}$ weight). These two solid edges will either appear in $\Gamma_{\omega}$ or become weights if $x_1$ or $x_2$ is identified with $x$ due to newly added dotted edges in $\cal O_{dot}$ operations. 
\end{proof}
If $\Gamma_0$ does not contain any weights or edges attached to $x$, then after Step 1 the structures of $\Gamma_\omega$ in \eqref{Q_eq1} are already good enough for our purpose. In Steps 2 and 3, we aim to remove the $Q_x$ label in $Q_x(\wt \Gamma_{\omega})$.

\medskip   

\noindent{\bf Step 2:} In this step and Step 3, we will remove the $Q_x$ label in  
$$\Gamma_\omega Q_x(\wt \Gamma_{\omega})=\Gamma_\omega \wt \Gamma_{\omega} - \Gamma_\omega P_x(\wt \Gamma_{\omega}).$$
It suffices to write $P_x (\wt \Gamma_{\omega})$ into a sum of graphs without the $P_x$ label. To this end, we will use the following lemma. 

\begin{lemma}
Let $f$ be a differentiable function of $G$. 
We have that
\begin{align}
	  P_x \left[(G_{xx}  -  m) f(G)\right] &=  P_x\Big[ m^2 \sum_{\al} s_{x\al} (G_{\al\al}-m)  f(G) + m \sum_{\al} s_{x\al} (G_{\al\al}-m) (G_{xx}-m)f(G)\Big] \nonumber\\
	& -  P_x  \Big[  m \sum_\al s_{x\al} G_{\al x}\partial_{  h_{\al x}} f(G)\Big]. \label{gHw}
\end{align}
Given a graph $\cal G$ taking the form \eqref{multi setting}, we have the following identity if $k_1\ge 1$:
\begin{align} 
	  P_x[\cal G]   &= \sum_{i=1}^{k_2}|m|^2 P_x \left[  \left( \sum_\al s_{x\al }G_{\al y_1} \overline G_{\al y'_i}\right)\frac{\cal G}{G_{x y_1} \overline G_{xy_i'}} \right] + \sum_{i=1}^{k_3} m^2 P_x \left[ \left(\sum_\al s_{x\al }G_{\al y_1} G_{w_i \al} \right)\frac{\cal G}{G_{xy_1}G_{w_i x}}\right] \nonumber \\
	 & + \sum_{i=1}^{k_2}m P_x \left[ \left(\overline G_{xx} -\overline m\right) \left( \sum_\al s_{x\al }G_{\al y_1} \overline G_{\al y'_i}\right)\frac{\cal G}{G_{x y_1} \overline G_{xy_i'}} \right] \nonumber\\
	 &+ \sum_{i=1}^{k_3} m P_x \left[ (G_{xx}-m) \left(\sum_\al s_{x\al }G_{\al y_1} G_{w_i \al} \right)\frac{\cal G}{G_{xy_1}G_{w_i x}}\right] \nonumber \\
	&+  m P_x \left[  \sum_\al s_{x\al }\left(G_{\al \al}-m\right) \cal G\right] +(k_1-1) m P_x \left[  \sum_\al s_{x\al }G_{\al y_1} G_{x \al} \frac{ \cal G}{G_{xy_1}}\right] \nonumber\\
	&+ k_4 m P_x \left[  \sum_\al s_{x\al }G_{\al y_1} \overline G_{\al x} \frac{\cal G}{G_{xy_1}}\right]  - m P_x \left[  \sum_\al s_{x\al }\frac{\cal G}{G_{x y_1}f(G)}G_{\al y_1}\partial_{ h_{\al x}}f (G)\right].\label{gH0} 
\end{align} 
\end{lemma}
\begin{proof}
These two identities both follow from Gaussian integration by parts, and have been proved in Lemma 3.5 and Lemma 3.10 of \cite{PartI_high}.
\end{proof}

Similar to Definition \ref{Ow-def}, we can define an operator $\wt{\cal O}^{(x)}_{weight}$ acting on a graph $P_x(\cal G)$ as 
$$	\wt{\cal O}_{weight}^{(x)}[ P_x(\cal G)]:={\cal O}^{(x),1}_{weight} \circ\cal O_{dot} \circ  \wt{\cal O}^{(x),2}_{weight} \circ {\cal O}^{(x),1}_{weight}  [P_x(\cal G)],$$
where $\wt{\cal O}^{(x),2}_{weight}$ is an operator defined from \eqref{gHw}. Also similar to Definition \ref{multi-def}, we can define an operator \smash{$\wt{\cal O}^{(x)}_{multi-e}$} acting on a graph $P_x(\cal G)$ as 
$$\wt{\cal O}_{multi-e}^{(x)}[ P_x(\cal G) ]:= \cal O_{dot} \circ  \wt{\cal O}_{multi-e}^{(x),1} [P_x(\cal G)], $$ 
where $\wt{\cal O}_{multi-e}^{(x),1}$ is an operator defined from \eqref{gH0} (the $k_1=0$ case can be handled in the same way as cases (ii)--(iv) of Definition \ref{multi-def} and we omit the details).


If a normal regular graph $P_x(\cal G)$ contains weights or solid edges attached to $x$, then we will apply \smash{$\wt{\cal O}_{weight}^{(x)}$ or $\wt{\cal O}^{(x)}_{multi-e}$} to it. We repeat these operations until all graphs do not contain weights or solid edges attached to $x$, and hence obtain the following lemma.

\begin{lemma}\label{Q_lemma2}
Suppose $\cal G$ is a normal regular graph without $P/Q$, $(x)$ or $(-1)$ labels. 
Then we can expand $P_x(\cal G)$ into a sum of $\OO(1)$ many graphs: 
$$P_x(\cal G) = \sum_{\omega}  P_x(\cal G_{\omega}) +  \cal G_{err}\ ,$$
where $\cal G_{\omega}$ are normal regular graphs without weights or edges attached to $x$ (and without $P/Q$, $(x)$ or $(-1)$ labels), and $\cal G_{err}$ is a sum of normal regular graphs of scaling order $>D$. 
\end{lemma}
\begin{proof}
We only describe the main argument for the proof without writing down all the details. Given a normal regular graph $P_x(\cal G)$, it is easy to check that every graph in \smash{$\wt{\cal O}_{weight}^{(x)}[P_x(\cal G)]$ or $\wt{\cal O}^{(x)}_{multi-e}[P_x(\cal G)]$}, say $P_x(\cal G_1)$, satisfies one of the following properties:
\begin{itemize}
	\item $P_x(\cal G_1)$ has a strictly higher scaling order than $P_x(\cal G_0)$;
	\item $P_x(\cal G_1)$ contains strictly fewer weights or edges attached to $x$. 
\end{itemize}
Hence applying the operations \smash{$\wt{\cal O}_{weight}^{(x)}$ or $\wt{\cal O}^{(x)}_{multi-e}$} for $C_{\cal G, D}$ many times,  each resulting graph either is of scaling order $>D$, or has no weights or edges attached to $x$. Here $C_{\cal G, D}\in \N$ is a large constant depending only on $D$ and the number of weights and edges attached to atom $x$ in $\cal G$.
\end{proof}

With Lemma \ref{Q_lemma2}, in Step 2 we can expand $Q_x(\wt \Gamma_{\omega})$ as
\be\label{PxGamma} Q_x(\wt \Gamma_{\omega})=   \wt \Gamma_{\omega} + \sum_{\zeta}P_x(\wt \Gamma_{\omega,\zeta}) + \cal G_{err},\ee
where $\Gamma_{\omega,\xi}$ are normal regular graphs without weights or edges attached to $x$, and $\cal G_{err}$ is a sum of graphs of scaling order $>D$.

\medskip
\noindent{\bf Step 3:} In this step, we expand the graphs $P_x(\wt \Gamma_{\omega,\zeta})$ in \eqref{PxGamma}. Suppose $\cal G$ is a normal regular graph without weights or edges attached to $x$. Using the resolvent identity \eqref{resol_exp0}, we can write $\cal G$ into a similar form as in \eqref{decompose_gamma}:
$$\cal G = \cal G^{(x)} + \sum_\omega \cal G_{\omega},$$
where $\cal G^{(x)}$ is a graph whose weights and solid edges all have the same $(x)$ label, and every graph $\cal G_{\omega}$ has a strictly higher scaling order than $\cal G$, at least two new solid edges connected with $x$, and some weights with label $(-1)$ on $x$. Then we can expand $P_x(\cal G)$ as 
$$P_x(\cal G) = \cal G- \sum_\omega \cal G_{\omega} + \sum_\omega P_x(\cal G_{\omega}).$$
Next, as in Step 1, we apply \eqref{resol_reverse} and \eqref{weight_taylor} to remove all $G^{(x)}$, $(G_{xx})^{-1}$ and $(\overline G_{xx})^{-1}$ entries from $\cal G_{\omega}$. Finally, we apply \smash{${\cal O}^{(x),1}_{weight} \circ\cal O_{dot}$} to all resulting graphs to get a sum of normal regular graphs without regular weights on $x$. In sum, we obtain the following result. 

\begin{lemma}\label{Q_lemma3}
Suppose $\cal G$ is a normal regular graph without $P/Q$, $(x)$ or $(-1)$ labels. Moreover, suppose $\cal G$ has no weights or edges attached to $x$. Then we can expand $P_x(\cal G)$ into a sum of $\OO(1)$ many graphs: 
$$P_x(\cal G) = \cal G+ \sum_\xi  \cal G_{\xi}  + \sum_\gamma  P_x(\wt{\cal G}_{\gamma}) + \cal G_{err},$$
where \smash{$\cal G_{\xi}$ and $\wt{\cal G}_{\gamma}$} are normal regular graphs without $P/Q$, $(x)$ or $(-1)$ labels, and $\cal G_{err}$ is a sum of normal regular graphs of scaling order $>D$. Moreover, $\cal G_\xi$ and $\wt{\cal G}_{\gamma}$ are of strictly higher scaling order than $\cal G$, and have light weights or edges attached to the atom $x$.
\end{lemma}

With this lemma, we can expand $P_x(\wt \Gamma_{\omega,\zeta}) $ in \eqref{PxGamma} as
$$P_x(\wt \Gamma_{\omega,\zeta}) =\wt \Gamma_{\omega,\zeta}+ \sum_\xi  \Gamma_{\omega,\zeta,\xi} + \sum_\gamma  P_x (\wt \Gamma_{\omega,\zeta,\gamma} ) +\cal G_{err},$$
where $\Gamma_{\omega,\zeta,\xi}$ and $\wt \Gamma_{\omega,\zeta,\gamma} $ respectively satisfy the same properties as 
$ \cal G_{\xi}$ and $\wt{\cal G}_{\gamma} $ in Lemma \ref{Q_lemma1}. Then we apply Step 2 to the graphs $P_x(\wt \Gamma_{\omega,\zeta,\gamma} )$ again. Repeating Steps 2 and 3 for $\OO(1)$ many times, we can finally expand $P_x (\wt \Gamma_{\omega})$ into a sum of graphs without $P_x$ labels plus graphs of scaling order $>D$. 

\medskip
 
Combining Steps 1--3, we obtain the following lemma on $Q$-expansions. 

\begin{lemma}[$Q$-expansions]\label{Q_lemma}
Suppose $\Gamma_0$ and $\wt\Gamma_0$ in \eqref{QG} are normal regular graphs without $P/Q$, $(x)$ or $(-1)$ labels. 
Then $\cal G_0$ can be expanded into a sum of $\OO(1)$ many graphs:
\be\label{G0Q} \cal G_0 = \sum_\omega \cal G_\omega + \sum_\zeta Q_x(\wt{\cal G}_\zeta)  + \cal G_{err},\ee
where $\cal G_{\omega}$ and $\wt{\cal G}_\zeta$ are normal regular graphs without $P/Q$, $(x)$ or $(-1)$ labels, and $\cal G_{err}$ is a sum of normal regular graphs of scaling order $>D$. Moreover, these graphs satisfy the following properties.
\begin{itemize}
	\item[(i)] Every graph on the right-hand side of \eqref{G0Q} has a scaling order $\ge \ord(\cal G_0)$.
	
	\item[(ii)] If there is a new atom in a graph on the right-hand side of \eqref{G0Q}, then it is connected to $x$ through a path of waved edges. Hence the expansion \eqref{G0Q} is a local expansion. 
	
	\item[(iii)] If $\cal G_0$ is doubly connected, then all graphs on the right-hand side of \eqref{G0Q} are also doubly connected.
	
	\item[(iv)] If $\Gamma_0$ does not contain any weights or edges attached to $x$, then the scaling order of  $\cal G_\omega$ is  strictly higher than $\ord(\cal G_0)$ for every $\omega$.
 Furthermore,  $\cal G_\omega$ contains at least one atom that belongs to the original graph $\Gamma_0$ and satisfies the following property: it is connected to $x$ through a solid edge or it has been merged with $x$.

\end{itemize}
\end{lemma}
\begin{proof}
This lemma is an easy consequence of Lemmas \ref{Q_lemma1}, \ref{Q_lemma2} and \ref{Q_lemma3}. We now give the main arguments without writing down all the details.

First, we show that the expansion \eqref{G0Q} exists. Given $\cal G_0$, we first expand it as in \eqref{Q_eq1}. Then we repeat Steps 2 and 3 to remove the $Q_x$ label in \smash{$\Gamma_\omega Q_x(\wt \Gamma_{\omega}) $}. By Lemmas \ref{Q_lemma2} and \ref{Q_lemma3}, after an iteration of Steps 2 and 3, every new graph either has no $P/Q$ labels or has a strictly higher scaling order than the original graph. Hence after repeating Steps 2 and 3 for $D$ many times, we can expand \smash{$\Gamma_\omega Q_x(\wt \Gamma_{\omega})$} into a sum of graphs without $P/Q$ labels and graphs of scaling orders $>D$. This gives \eqref{G0Q}.

The properties (i) and (ii) are trivial. The proof of (iii) is similar to the one for Lemma \ref{lem localgood}. The property (iv) follows from Lemma \ref{Q_lemma1}.
\end{proof}
 
The property (iv) of Lemma \ref{Q_lemma} is crucial to our global expansion strategy as we will explain in Section \ref{sec global}. Graphs with $(x)$ and $(-1)$ labels only appear in $Q$-expansions, and will not appear in any other part of the expansion process. Hence for the rest of this paper, unless mentioned otherwise, our graphs are assumed to have no $(x)$ and $(-1)$ labels.  
 
%
%

 \section{Pre-deterministic property}\label{sec pre}


The choice of the standard neutral atom is one of the most delicate parts of the global expansion defined in Section \ref{subsec global}, and this subsection is devoted to this issue. Our goal is to define a rule to choose the $t$ variable for every locally standard graph, so that the  doubly connected property is preserved by global expansions. In this section, we will define a specific order of blue solid edges, called the \emph{pre-deterministic order}, and we will show that the doubly connected property is preserved if we expand blue solid edges according to this order. 
As a convention, expansion of a blue solid edge refers to the expansion of a $t$-variable containing this edge.

In this section, we mostly consider molecular graphs with all red solid edges removed, because these edges are not used in the doubly connected property.

\subsection{Isolated subgraphs}

Using Definition \ref{def 2net}, it is trivial to prove the following simple property. 
 
\begin{claim}\label{trivial_claim1}
Let $\cal G$ be a doubly connected graph satisfying Definition \ref{def 2net}. Then in its molecular graph $\cal G_{\cal M}$ with all red solid edges removed, any subset of internal molecules are connected to other molecules (including external molecules) through at least one blue solid edge and one $\dashed$ edge, or at least two $\dashed$ edges. 
\end{claim}

Then we define an isolated subgraph to be a subgraph induced on a subset of molecules that are connected to other molecules through \emph{exactly two edges} in a molecular graph without red solid edges. 

\begin{definition}[Isolated subgraphs]\label{defn iso}
Let $\cal G$ be a doubly connected graph and $\cal G_{\cal M}$ be its molecular graph with all red solid edges removed. 
\begin{itemize}
\item {\bf Isolated subsets of molecule.} A subset of internal molecules in $\cal G$, say $\Pol$, is \emph{isolated} if and only if $\Pol$ is connected to $\Pol^c$ exactly by two edges 
in $\cal G_{\cal M}$. Here the complement $\Pol^c$ contains the internal molecules that are not in $\Pol$ and all external molecules. 


\item {\bf Isolated subgraph.} An isolated subgraph of $\cal G$ is a subgraph induced on an isolated subset of molecules. Here by the ``subgraph induced on a subset of molecules", we mean the subgraph induced on the atoms in these molecules (recall Definition \ref{def_sub}).
\end{itemize}
Given any isolated subset of molecules $\Pol$, we can define an isolated subgraph induced on it, denoted by $\Iso_{\Pol}$. 
Conversely, given an isolated subgraph $\Iso$, the molecules in it form an isolated subset of molecules, denoted by $\Pol(\Iso)$. Furthermore, we define the concept of proper isolated subgraphs.
\begin{itemize}

\item{\bf Proper isolated subgraph.} Given any subset of molecules (which is not necessarily isolated), say $\cal S$, $\Pol_1\subset \cal S$ is called a \emph{proper isolated subset} if $\Pol_1\ne \cal S$ and $\Pol_1$ is an isolated subset of molecules in $\cal G$. Moreover, 
a proper isolated subgraph of ${\cal G}|_{\cal S}$ refers to a subgraph induced on a proper isolated subset of molecules in $\cal S$.  
\end{itemize}

 \end{definition}


We call the set of all internal molecules the \emph{maximal subset of internal molecules}, denoted by $\Pol_{\max}$. Correspondingly, the subgraph of $\cal G$ induced on $\Pol_{\max}$ is called the \emph{maximal subgraph}, 
denoted by $\cal G_{\max}$. Whenever we say a proper isolated subset of molecules (resp. isolated subgraph) without specifying the superset, we mean a proper isolated subset (resp. isolated subgraph) of $\Pol_{\max}$ (resp. $\cal G_{\max}$).
Given any subset of molecules $\cal S$, we say an atom is inside $\cal S$ if this atom belongs to a molecule in $\cal S$, and we say an edge is inside $\cal S$ if its ending atoms are inside $\cal S$. Given two disjoint subsets of molecules $\cal S_1$ and $\cal S_2$, an edge between $\cal S_1$ and $\cal S_2$ refers to an edge between an atom in $\cal S_1$ and an atom in $\cal S_2$.

Every graph in our expansions has at least one edge between an internal atom and an external atom (recall property (ix) in Definition \ref{defn genuni}). Then from Definition \ref{defn iso}, it is easy to get the following property. 

\begin{claim}\label{trivial_claim2}
Let $\cal G$ be a doubly connected graph such that there is at least one edge between its internal and external molecules in the molecular graph $\cal G_{\cal M}$ without red solid edges. Then for any two isolated subsets of molecules $\Pol_1$ and $\Pol_2$ in $\cal G$, we have that either one is a subset of the other or $\Pol_1\cap \Pol_2=\emptyset$. 
\end{claim}
\begin{proof}
Suppose $\Pol_1\not\subset\Pol_2 $, $\Pol_2\not\subset \Pol_1$ and $\Pol_1\cap \Pol_2\ne \emptyset$. First, assume that $\Pol_1\cap \Pol_2$ is connected to $(\Pol_1\cup \Pol_2)^c$ through $a$ many edges with $a\in \{1,2\}$ in $\cal G_{\cal M}$. Since $\Pol_1$ and $\Pol_2$ are isolated, $\Pol_1\setminus \Pol_2$ connects to $\Pol_1^c$ through $2-a$ many edges and $\Pol_2\setminus \Pol_1$ connects to $\Pol_2^c$ through $2-a$ many edges. Applying Claim \ref{trivial_claim1}, we get that $\Pol_1\setminus \Pol_2$ connects to $\Pol_1\cap\Pol_2$ through at least $a$ many edges. This implies that $\Pol_2$ connects to $\Pol_2^c$ through at least $a+2\ge 3$ many edges, contradicting the fact that $\Pol_2$ is isolated. 

Now we assume that $\Pol_1\cap \Pol_2$ is not connected to $(\Pol_1\cup \Pol_2)^c$. Suppose $\Pol_1\cap \Pol_2$ connects to $\Pol_1\setminus \Pol_2$ through $a$ many edges and to $\Pol_2\setminus \Pol_1$ through $b$ many edges. By Claim \ref{trivial_claim1} and the fact that $\Pol_1$ and $\Pol_2$ are isolated, we have $0\le a\le 2$, $0\le b\le 2$ and $a+b\ge 2$. If $a=0$, then $\Pol_1\setminus \Pol_2$ connects to $\Pol_1^c$ through two edges and $\Pol_1\cap \Pol_2$ connects to $\Pol_2\setminus \Pol_1$ through two edges, contradicting the fact that $\Pol_1$ is isolated. If $a=2$, then $\Pol_2\setminus \Pol_1$ is not connected to $\Pol_2^c$ since $\Pol_2$ is isolated. This implies that $b=2$ by applying Claim \ref{trivial_claim1} to $\Pol_2\setminus \Pol_1$. Since $\Pol_1$ is isolated, $\Pol_1\setminus \Pol_2$ is also not connected to $\Pol_1^c$. Hence we get that $\Pol_1\cup \Pol_2$ is disconnected from $(\Pol_1\cup \Pol_2)^c$, which contradicts the assumption on $\cal G_{\cal M}$. Finally, we consider the case where $a=b=1$, $\Pol_1\setminus \Pol_2$ connects to $\Pol_1^c$ through one edge and $\Pol_2\setminus \Pol_1$ connects to $\Pol_2^c$ through one edge. However, it is not hard to see that this case contradicts the doubly connected property.
\end{proof}

This claim shows that there is a natural partial order of isolated subgraphs in a doubly connected graph satisfying Claim \ref{trivial_claim2}. In particular, we can define maximal and minimal subgraphs.
\begin{definition}\label{defn_min_iso}
(i) An isolated subgraph (resp. isolated subset of molecules) is said to be \emph{minimal} if it has no proper isolated subgraphs (resp. proper isolated subsets). As a convention, if a graph $\cal G$ does not contain any proper isolated subgraph, then the minimal isolated subgraph (MIS) refers to the maximal subgraph $\cal G_{\max}$.  

\vspace{5pt}

\noindent (ii) Given a subset of molecules $\cal S$, an isolated subset $\Pol$ of $\cal S$ is said to be \emph{maximal} if it is not a proper subset of any other proper isolated subset of $\cal S$. In this case, $\Iso_{\Pol}$ is called a  maximal isolated subgraph of $ {\cal G}|_{\cal S}$ induced on $\cal S$. 
\end{definition}

 \subsection{Redundant edges} 
 
For the global expansions defined in Section \ref{subsec global}, it is easy to see that expanding a blue solid edge inside a molecule will not break the doubly connected property. On the other hand, the doubly connected property may or may not be broken when we expand a blue solid edge between molecules, depending on whether this edge is \emph{pivotal} or \emph{redundant}. 

 \begin{definition}[Pivotal edges and redundant edges] \label{defn pivotal}
 We say a blue solid edge in a doubly connected graph $\cal G$ is \emph{pivotal} if after removing it, the remaining graph is not doubly connected anymore. Otherwise, this blue solid edge is said to be \emph{redundant}. 
 \end{definition}
 

\begin{example}
In \eqref{eq pivotal}, the blue solid edge $b$ connecting an isolated subgraph $\Iso$ to its complement is pivotal: 
\be  \label{eq pivotal}
\parbox[c]{0.6\linewidth}
{\includegraphics[width=9cm]{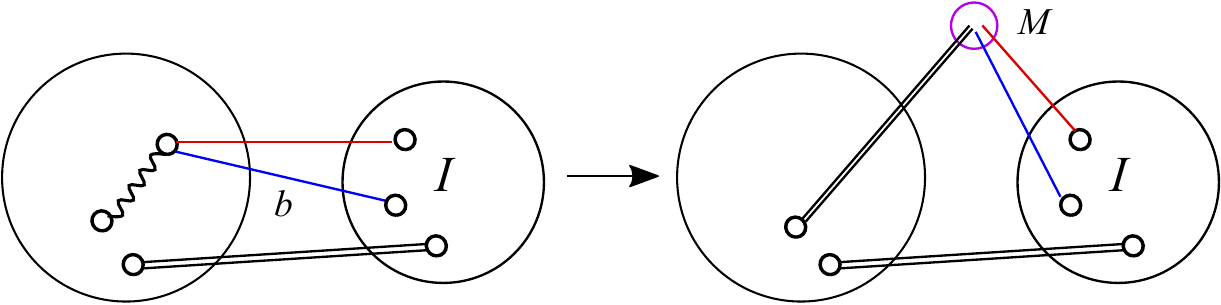}} 
\ee
In these graphs, inside the black circles are some subgraphs whose details are not drawn, and we did not indicate the directions of the solid edges for conciseness. If we replace the $T$-variable containing edge $b$ in the first graph by a graph in \eqref{Aho>2}, then we get the second graph, where the purple circle denotes a new molecule $\cal M$.  
It is easy to see that the second graph is not doubly connected anymore: in order to have a connected black net $\cal B_{black}$, we have to put the two $\dashed$ edges into $\cal B_{black}$, but then there is no edge in the blue net $\cal B_{blue}$ that connects $\cal M$ and the molecules in $\Iso$ to other molecules. 
\end{example}


By definition, blue solid edges inside molecules are all redundant. Now we show that a redundant edge can be expanded without breaking the doubly connected property. 
\begin{lemma}\label{trivial_claim3}
Suppose $\cal G$ is a doubly connected graph containing a $t_{x,y_1y_2}$ variable defined in \eqref{12in4T}, where the blue solid edge $G_{\al y_1}$ is redundant. If we replace $t_{x,y_1y_2}$ with a graph on the right-hand side of \eqref{replaceT} (including any recollision, $Q$, higher order or error graph), 
then the resulting graph is still doubly connected.
\end{lemma}
\begin{proof}
We denote the blue solid edge $G_{\al y_1}$ as $b$. If we replace $t_{x,y_1y_2}$ with $ m  \Theta_{xy_1}\overline G_{y_1y_2}$ or a graph in \smash{$m (\Theta \Sdelta^{(n-1)}\Theta)_{xy_1} \overline G_{y_1y_2}$}, then it corresponds to replacing the edge $b$ with a $\dashed$ edge in the molecular graph. Obviously, this will not break the doubly connected property because a $\dashed$ edge can be also used in the blue net. Now suppose we replace $t_{x,y_1y_2}$ with a graph, say \smash{$\wt{\cal G}_{x,y_1y_2}$}, in other terms. Since $b$ is redundant, $\cal G$ has a black net $\cal B_{black}$ and a blue net $\cal B_{blue}$ satisfying Definition \ref{def 2net} such that $\cal B_{blue}\setminus\{b\}$ is still a blue net. By property (x) of Definition \ref{defn genuni}, \smash{$\wt{\cal G}_{x,y_1y_2}$} also contains a black net \smash{$\wt{\cal B}_{black}$} and a blue net \smash{$\wt{\cal B}_{blue}$} satisfying Definition \ref{def 2net}. Moreover, by property (ix) of Definition \ref{defn genuni}, the atom $x$ connects to the internal molecules of \smash{$\wt{\cal G}_{x,y_1y_2}$} through a $\dashed$ edge, say \smash{$\wt b_1$}, while the atom $y_1$ connects to the internal molecules of $\wt{\cal G}_{x,y_1y_2}$ through a blue solid or $\dashed$ or dotted edge, say $\wt b_2$. Then $\cal B_{black}\cup \wt{\cal B}_{black} \cup \{\wt b_1\}$ and $(\cal B_{blue}\setminus\{b\})\cup \wt{\cal B}_{blue} \cup \{\wt b_2\}$ are respectively the black net and blue net of the resulting graph if $\wt b_2$ is not a dotted edge. If $\wt b_2$ is a dotted edge, then the blue net can be chosen as $(\cal B_{blue}\setminus\{b\})\cup \wt{\cal B}_{blue}$ after merging the molecules connected by $\wt b_2$. In either case, the resulting graph is still doubly connected.
\end{proof}

\subsection{Pre-deterministic property} 

In the proof of Theorem \ref{incomplete Texp}, we need to expand $T_{\fa, \fb_1\fb_2}$ into a sum of graphs that satisfy the properties in Definition \ref{def incompgenuni}. Hence for a graph that is not a recollision, higher order or $Q$ graph, we want its maximal subgraph to be deterministic, so that the only random parts in the graph are two plus and minus $G$ edges connected with $\oplus$ and $\ominus$. 
On the other hand, since we want to maintain the doubly connected property during expansions, we are only allowed to expand redundant $G$ edges. Hence, not all doubly connected graphs with non-deterministic maximal subgraphs can be expanded without breaking the doubly connected property. For example, in a doubly connected graph where both the black and blue nets are trees, all the blue solid edges  between molecules are pivotal.  

Now we want to identify a graphical property with which we can tell whether a graph is potentially expanded into deterministic graphs in the end without breaking the doubly connected property. We can obtain a necessary condition by looking at the leading terms of global expansions, where we replace $t_{x,y_1y_2}$ in \eqref{12in4T} with $m \Theta_{xy_1}\overline G_{y_1y_2}$. This corresponds to changing a plus $G$ edge into a $\dashed$ edge in the molecular graph. Hence we need to be able to change all plus $G$ edges into $\dashed$ edges one by one, such that the plus $G$ edge for every expansion is redundant. This leads to the following definition. 

\begin{definition} [Pre-deterministic property] \label{def PDG} 
We say a doubly connected graph $\cal G$ is \emph{pre-deterministic} if the following property holds. There exists an order of all  internal blue solid edges, denoted by $b_1 \preceq b_2 \preceq \cdots$, such that for any $k$, after changing the edges $b_1, \cdots, b_{k-1} $ into $\dashed$ edges, the blue solid edge $b_k$ becomes redundant. (Recall that internal edges refer to edges that do not connect to external atoms.) We will call this order of blue solid edges a \emph{pre-deterministic order}.
\end{definition}

A pre-deterministic order is not necessarily unique. Moreover, we can always choose an order where blue solid edges inside molecules precede blue solid edges between molecules. If all internal blue solid edges in a graph $\cal G$ are redundant, then $\cal G$ is obviously a pre-deterministic graph. On the other hand, Definition \ref{def PDG} covers more general graphs. We use the following example to explain why we define pre-deterministic graphs in the above way, and why the order of blue solid edges is important: 
\begin{center}
\includegraphics[width=14cm]{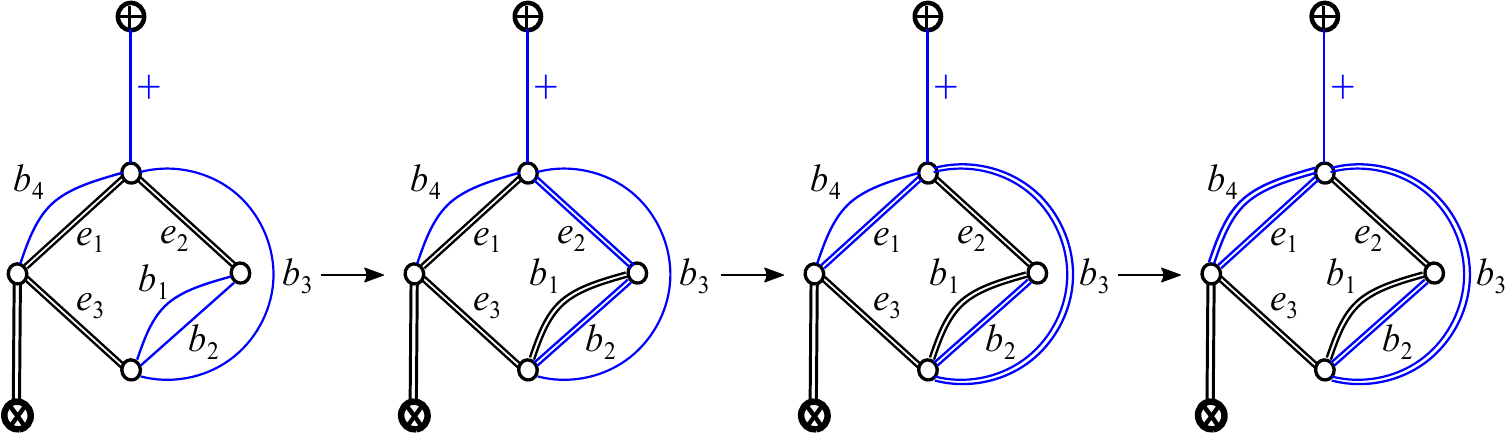}
\end{center}
Here all graphs are molecular graphs without red solid edges. Note that the edges $b_1$ and $b_2$ are redundant in the first graph, but $b_3$ and $b_4$ are not. Now we show that this graph is actually pre-deterministic. First, we replace $b_1$ and $b_2$ with $\dashed$ edges. In the second graph, if we choose the black net as $\{e_1, e_3, b_1\}$ and the blue net as $\{e_2,b_2,b_3,b_4\}$, then the edge $b_3$ is redundant, and we can replace it with a $\dashed$ edge in the third graph. In the third graph, if we choose the black net as $\{e_2, e_3, b_1\}$ and the blue net as $\{e_1,b_2,b_3,b_4\}$, then $b_4$ becomes redundant, and we can replace it with a $\dashed$ edge in the fourth graph. This argument gives a pre-deterministic order $b_1\preceq b_2\preceq b_3\preceq b_4$.

\subsection{Sequentially pre-deterministic property}

Not all graphs in our expansions satisfy the pre-deterministic property. Instead, we need to consider the following extension. 

\begin{definition}[Sequentially pre-deterministic property]\label{def seqPDG}
A graph $\cal G$ is said to be sequentially pre-deterministic (SPD) if its molecular graph $\cal G_{\cal M}$ without red solid edges satisfies the following properties.
\begin{itemize}
\item[(i)] $\cal G_{\cal M}$ is doubly connected in the sense of Definition \ref{def 2net}.

\item[(ii)] Any two proper isolated subgraphs $\Iso_1$ and $\Iso_2$ satisfy either $\Iso_1\subset \Iso_2$ or $\Iso_2\subset \Iso_1$.


\item[(iii)] 
The minimal isolated subgraph of $\cal G_{\cal M}$ is pre-deterministic. Moreover, given an isolated subgraph $\Iso$, if we replace its maximal isolated subgraph, say $\Iso'$, plus the two external edges with a single $\dashed$ edge, then $\Iso$ becomes pre-deterministic. (Here by ``external edges", we mean the blue solid or $\dashed$ edges connecting $\Pol(\Iso')$ to the two molecules, say $\cal M$ and $\cal M'$, in the complement of $\Pol(\Iso')$; by ``replace", we mean that we remove  $\Iso'$ and its two external edges, and add a $\dashed$ edge between $\cal M$ and $\cal M'$. If $\cal M=\cal M'$, then this $\dashed$ edge is inside $\cal M$ and will not appear in the molecular graph.) Finally, if we replace the maximal isolated subgraph of $(\cal G_{\cal M})_{\max}$ (recall that $(\cal G_{\cal M})_{\max}$ is the maximal subgraph induced on all internal molecules) plus the two external edges with a single $\dashed$ edge, then $(\cal G_{\cal M})_{\max}$ becomes pre-deterministic. 
\end{itemize}

\end{definition}

Now we discuss the meanings of properties (ii) and (iii). The property (ii) means that $\cal G$ contains at most one sequence of proper isolated subgraphs, say 
\be\label{seq_iso}
\Iso_k \subset \Iso_{k-1} \subset \cdots \subset \Iso_1 \subset  \cal G_{\max},
\ee
where $\Iso_1$ is the maximal isolated subgraph of $\cal G_{\max}$ and for any $j$, $\Iso_{j+1}$ is the maximal isolated subgraph of $\Iso_{j}$. In particular, any subgraph has at most one maximal isolated subgraph, so property (iii) is well-defined. 
Suppose 
that corresponding to \eqref{seq_iso}, the sequence of proper isolated subsets of molecules is
\be\label{chain_mole}
\Pol_k \subset \Pol_{k-1} \subset \cdots \subset \Pol_1 \subset \Pol_{\max},\quad \Pol_i:= \Pol(\Iso_i), \ \ 1\le i \le k.
\ee
We draw a sequence of isolated subgraphs in the following figure, where inside each black circle is a subgraph $\Gamma_i$, which will be called a \emph{component} in our proof, and we only draw the $\dashed$ and blue solid edges between each isolated subgraph and its complement:
\be\label{chain_mole_sub}
\parbox[c]{0.7\linewidth}{\includegraphics[width=11cm]{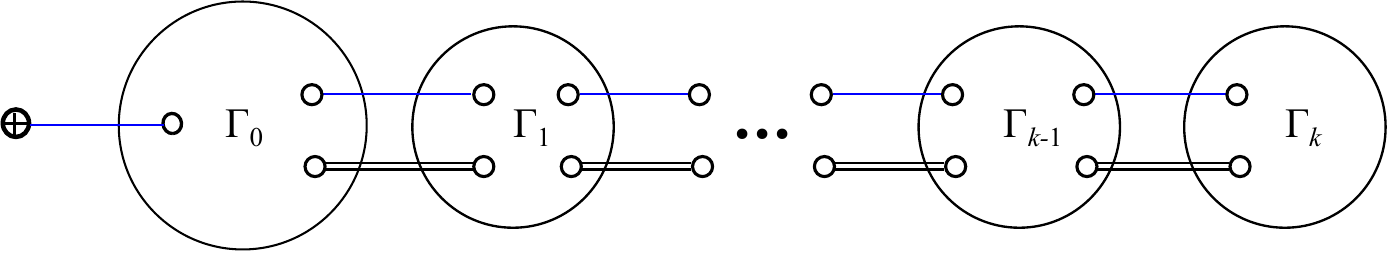}} 
\ee
Here for definiteness, we draw a graph with one external blue solid edge, but generally there may be zero or more than one external edges. Moreover, the blue solid edges between components can be replaced by blue $\dashed$ edges. In \eqref{chain_mole_sub}, the component $\Gamma_k$ is the minimal isolated subgraph $\Iso_k$, the component $\Gamma_{i}$, $1\le i \le k-1$, is the subgraph induced on $\Pol_{i}\setminus \Pol_{i+1}$, and $\Gamma_0$ is the subgraph induced on $\Pol_{\max}\setminus\Pol_1$. 
One can see that the isolated subgraphs in a SPD graph form a simple chain structure. Then property (iii) means that if we replace the MIS $\Gamma_k$ and its two external edges with a $\dashed$ edge, then the resulting graph is still SPD. In the new graph, replacing the MIS $\Gamma_{k-1}$ and its two external edges with a $\dashed$ edge again gives a SPD graph. Continuing this process, after replacing $\Gamma_1$ and its two external edges with a $\dashed$ edge, $\Gamma_0$ becomes pre-deterministic.

The reason why we define the SPD property in the above way is due to our global expansion strategy below (cf. Strategy \ref{strat_global}). We now describe one possible scenario to obtain a SPD graph. Initially, there is no proper isolated subgraph in a graph obtained from  local expansions, and its maximal subgraph is pre-deterministic. Then we find the first redundant blue solid edge in a pre-deterministic order, replace a $t$-variable containing it with a graph, say $\wt{\cal G}_1$, on the right-hand side of \eqref{replaceT}, and get a new graph $\cal G_1$. If $\wt{\cal G}_1$ contains a pre-deterministic maximal subgraph, then this subgraph becomes a pre-deterministic isolated subgraph $\Iso_1$ in graph $\cal G_1$. Next we find the first blue solid edge in a pre-deterministic order of $\Iso_1$, replace it with a graph, say $\wt{\cal G}_2$, on the right-hand side of \eqref{replaceT}, and get a new graph. Continuing this process, we get a SPD graph with a sequence of isolated subgraphs in \eqref{seq_iso}. At each step we expand the first blue solid edge in the minimal isolated subgraph. Moreover, replacing $\Iso_{i+1}$ and its two external edges with a $\dashed$ edge corresponds to replacing the first blue solid edge in a pre-deterministic order of $\Iso_{i}$ by a $\dashed$ edge, which still gives a pre-deterministic $\Iso_i$ by Definition \ref{def PDG}. This explains why we define property (iii) in such a way. 


We now show that the SPD property is preserved under local expansions.

\begin{lemma}\label{lem localgood2}
Let $\cal G$ be a SPD graph. Then applying any expansion in Definitions \ref{dot-def}, \ref{Ow-def}, \ref{multi-def}, \ref{GG-def} and \ref{GGbar-def} to $\cal G$, the new graphs (including the $Q$-graphs) are all sequentially pre-deterministic. 
 \end{lemma}
 
\begin{proof}
By Lemma \ref{lem localgood}, all new graphs are still doubly connected. It remains to check the properties (ii) and (iii) in Definition \ref{def seqPDG} for any new graph, say $\wt {\cal G}$. Let $\wt{\cal G}_{\cal M}$ be its molecular graph without red solid edges.

By Lemma \ref{lvl1 lemma}, every molecule in $\wt{\cal G}$ is obtained by merging some molecules in $\cal G$. We claim that any proper isolated subset of molecules, say $\Pol$, in \smash{$\wt{\cal G}$} is obtained from merging a proper isolated subset of molecules in $\cal G$. In fact, suppose the molecules of $\Pol$ are obtained from merging a subset of molecules $\cal S$ that is not isolated in $\cal G$. Then $\cal S$ is connected to $\cal S^c$ through at least three edges, say $b_1,$ $ b_2$ and $b_3$, in the molecular graph $\cal G_{\cal M}$ without red solid edges. For $i =1,2,3 $, if $b_i$ is a $\dashed$ edge, then it is not affected in expansions and is still an external edge of $\Pol$ in $\wt{\cal G}_{\cal M}$. If $b_i$ is a plus $G$ edge in $\cal G_{\cal M}$ between molecules $\cal M_1$ and $\cal M_2$, then either $b_i$ is still a plus $G$ edge between $\cal M_1$ and $\cal M_2$ in $\wt{\cal G}_{\cal M}$, or $b_i$ becomes a connected path of two plus $G$ edges, say $b_{i,1}$ and $b_{i,2}$, between $\cal M_1$ and $\cal M_2$. 
In the former case, $b_i$ is an external edge of $\Pol$ in \smash{$\wt{\cal G}_{\cal M}$}, and in the latter case, either $b_{i,1}$ or $b_{i,2}$ is an external edge of $\Pol$ in \smash{$\wt{\cal G}_{\cal M}$}. In sum, we see that $\Pol$ is connected to $\Pol^c$ through at least three edges in \smash{$\wt{\cal G}_{\cal M}$}, which contradicts the assumption that $\Pol$ is isolated. 

Now  with the above claim, we prove the property (ii) of Definition \ref{def seqPDG} for $\wt{\cal G}$. Suppose two proper isolated subsets of molecules $\Pol_1$ and $\Pol_2$ in \smash{$\wt{\cal G}$} are respectively obtained from merging two proper isolated subsets of molecules $\Pol_1'$ and $\Pol_2'$ in $\cal G$. Then we have either $\Pol_1 \subset \Pol_2$ if $ \Pol_1'\subset \Pol_2'$, or $\Pol_2 \subset \Pol_1$ if $ \Pol_2'\subset \Pol_1'$. 

It remains to prove the property (iii) of Definition \ref{def seqPDG} for $\wt{\cal G}$. In all expansions, it is not hard to see that
there are only two types of operations that may affect the doubly connected structures:
\begin{itemize}
\item[(A)] merging a pair of molecules due to a newly added dotted or waved edge between them;

\item[(B)] replacing a plus $G$ edge $b_0$ between molecules $\cal M$ and $\cal M'$ with a path of two plus $G$ edges---edge $b$ between $\cal M$ and $\cal M''$ and edge $b'$ between $\cal M'$ and $\cal M''$---due to partial derivatives $\partial_{h_{\al x}}$ or $\partial_{h_{ \beta\al}}$, where $\cal M''$ is a different molecule from $\cal M$ and $\cal M'$.
\end{itemize}
All the other operations only act on local structures within molecules, and hence are irrelevant for our proof. Moreover, in case (B), we have assumed that $\cal M''$ is different from $\cal M$ and $\cal M'$, because otherwise the molecular graph without red solid edges will be unchanged. To conclude the proof, it suffices to show that both operations (A) and (B) do not break the property (iii) of Definition \ref{def seqPDG}.

Suppose the sequence of proper isolated subgraphs in $\cal G$ is given by \eqref{seq_iso} and takes the form \eqref{chain_mole_sub}. In case (A), it is easy to see that merging a pair of molecules in the same component does not break the SPD property. Now suppose we merge a pair of molecules in different components $\Gamma_{i}$ and $\Gamma_j$ with $i<j$. 
With a slight abuse of notation, we still denote the subgraph induced on $\Pol_l$ by $\Iso_l$ in \smash{$\wt {\cal G}$}. Then the sequence of isolated subgraphs in $\wt{\cal G}$ is
$$ \Iso_k \subset \Iso_{k-1} \subset \cdots \subset \Iso_{j+1} \subset \Iso_i \subset \Iso_{i-1} \subset \cdots \subset \Iso_1. $$
In particular, $\Iso_l$ is not isolated in $\wt{\cal G}$ anymore if $j+1<l <i$. 
By the SPD property of $\cal G$, we immediately obtain the following property of $\wt{\cal G}$: for $l\ge j+1$ or $l\le i-1$, if we replace $\Iso_{l+1}$ and its two external edges with a $\dashed$ edge in the molecular graph of $\wt{\cal G}$, then $\Iso_l$ becomes pre-deterministic. 
It remains to show that after replacing $\Iso_{j+1}$ and its two external edges with a $\dashed$ edge, the subgraph $\Iso_i$ becomes pre-deterministic. This is given by the following lemma, whose proof is postponed to Appendix \ref{sec PDG}.

\begin{lemma}\label{hard lemmdot}
In the setting of \eqref{chain_mole_sub}, we merge a pair of molecules in different components $\Gamma_{i}$ and $\Gamma_j$ with $i<j$. Then after replacing $\Iso_{j+1}$ and its two external edges with a $\dashed$ edge in the molecular graph without red solid edges, the subgraph $\Iso_i$ becomes pre-deterministic. 
\end{lemma}

Next we consider case (B). First, we assume that $\cal M$, $\cal M'$ and $\cal M''$ are all inside $\Gamma_i$ for some $0\le i \le k$. 
Then all components are unchanged after operation (B), except the component $\Gamma_i$. By definition, after replacing $\Iso_{i+1}$ and its two external edges with a $\dashed$ edge, $\Iso_i$ becomes pre-deterministic in $\cal G$, and suppose a pre-deterministic order of $\Iso_i$ is $e_1 \preceq \cdots \preceq e_{\ell-1} \preceq b_0 \preceq e_{\ell+1} \preceq \cdots$. Then in \smash{$\wt{\cal G}$}, after replacing $\Iso_{i+1}$ and its two external edges with a $\dashed$ edge, $e_1\preceq \cdots\preceq e_{\ell-1}\preceq b\preceq b' \preceq  e_{\ell+1}\preceq\cdots$ is a pre-deterministic order in $\Iso_i$. Using this fact, we can readily check that $\wt{\cal G}$ is still SPD.

Second, we assume that $\cal M\in \Gamma_i$, $\cal M'\in \Gamma_{i+1}$ and (1) $\cal M''\in \Gamma_i$, or (2) $\cal M''\in \Gamma_{i+1}$. Suppose we have replaced $\Iso_{i+2}$ and its two external edges with a $\dashed$ edge in $\Iso_{i+1}$ and get the following graphs:
\begin{center}
\includegraphics[width=13cm]{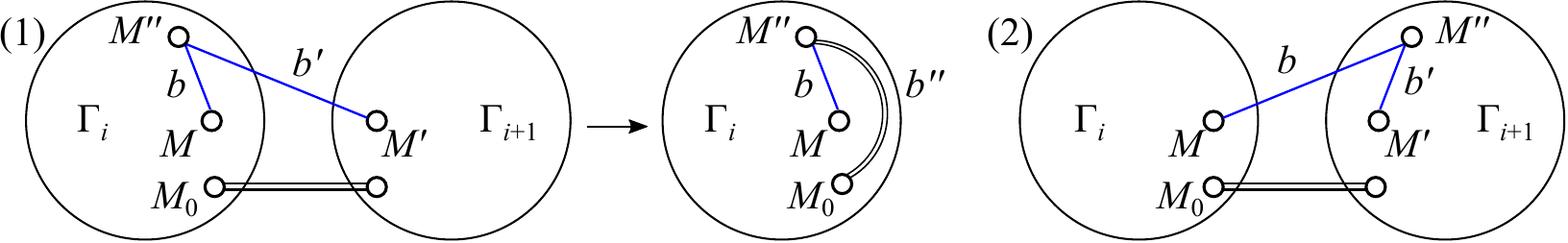}
\end{center}

In case (1), $\Iso_{i+1}$ is pre-deterministic because it is unchanged after the operation (B). In the second graph, we replace $\Iso_{i+1}$ and its two external edges with a $\dashed$ edge $b''$. We show that the edge $b$ is now redundant in $\Iso_{i}$. By the  doubly connected property of $\cal G$, removing the subgraph $\Iso_{i+1}$ and its two external edges still gives a doubly connected graph. Correspondingly, in the first graph of case (1), if we remove $\Gamma_i$, its two external edges and edge $b$, we still get a doubly connected graph. Hence the edge $b$ is redundant in the second graph of case (1). Now if we replace $b$ with a $\dashed$ edge, then the path of $\dashed$ edges $b''$ and $b$ will play the role of a single $\dashed$ edge between $\cal M_0$ and $\cal M$ in $\Iso_i$, and hence $\Iso_{i}$ becomes pre-deterministic by the property (iii) of Definition \ref{def seqPDG} for the original graph $\cal G$.

In case (2), $\Iso_{i+1}$ is still pre-deterministic by choosing the edge $b'$ as the last edge in a pre-deterministic order. Moreover, after replacing $\Iso_{i+1}$ and its two external edges with a $\dashed$ edge, $\Iso_{i+1}$ becomes pre-deterministic by the property (iii) of Definition \ref{def seqPDG} for the original graph $\cal G$. 

It remains to consider the case where $\cal M''$ is in a different component from both $\cal M$ and $\cal M'$. 

\begin{lemma}\label{hard lemmQ}
In the setting of \eqref{chain_mole_sub}, suppose a blue solid edge between molecules $\cal M$ and $\cal M'$ is replaced by a path of two blue solid edges from $\cal M$ to $\cal M''$ and from $\cal M''$ to $\cal M'$. Let the new graph be \smash{$\wt{\cal G}$}, and with a slight abuse of notation, we still denote the subgraphs induced on $\Pol_l$ by $\Iso_l$ in $\wt {\cal G}$. Suppose $\cal M''$ is inside component $\Gamma_j$.  Then we have the following five cases for the molecular graph $\wt {\cal G}_{\cal M}$ without red solid edges.
\begin{itemize}
\item[(i)] If $\cal M,\cal M' \in \Gamma_i$ and $i<j$, then $\Iso_{j+1}$ is the maximal isolated subgraph of $\Iso_i$. Moreover, after replacing $\Iso_{j+1}$ and its two external edges with a $\dashed$ edge, the subgraph $\Iso_i$ becomes pre-deterministic.
 
\item[(ii)] If $\cal M,\cal M' \in \Gamma_i$ and $i>j$, then $\Iso_{i+1}$ is the maximal isolated subgraph of $\Iso_j$. Moreover, after replacing $\Iso_{i+1}$ and its two external edges with a $\dashed$ edge, the subgraph $\Iso_j$ becomes pre-deterministic.
 
 \item[(iii)] If $\cal M\in \Gamma_i$, $\cal M'\in \Gamma_{i+1}$ and $i<j$, then $\Iso_{i+1}$ is the maximal isolated subgraph of $\Iso_i$, and $\Iso_{j+1}$ is the maximal isolated subgraph of $\Iso_{i+1}$. Moreover, after replacing $\Iso_{j+1}$ and its two external edges with a $\dashed$ edge, the subgraph $\Iso_{i+1}$ becomes pre-deterministic.
 
 \item[(iv)] If $\cal M\in \Gamma_i$, $\cal M'\in \Gamma_{i+1}$ and $i>j$, then $\Iso_{i+1}$ is the maximal isolated subgraph of $\Iso_j$. Moreover, after replacing $\Iso_{i+1}$ and its two external edges with a $\dashed$ edge, the subgraph $\Iso_{j}$ becomes pre-deterministic.

\item[(v)] If $\cal M$ is an external molecule and $\cal M'\in \Gamma_0$, then $\Iso_{j+1}$ is the maximal proper isolated subgraph. Moreover, after replacing $\Iso_{j+1}$ and its two external edges with a $\dashed$ edge, the maximal subgraph \smash{$\wt{\cal G}_{\max}$} becomes pre-deterministic.
\end{itemize}
\end{lemma}


The proof of Lemma \ref{hard lemmQ} will be postponed to Appendix \ref{sec PDG}. Similar to case (A), combining Lemma \ref{hard lemmQ} with the SPD property of the original graph $\cal G$, we can readily show that the new graph obtained from the operation (B) is still SPD. We omit the details. This concludes Lemma \ref{lem localgood2}.
 \end{proof}
Similar to Lemma \ref{lem localgood2}, we have the following result for $Q$-expansions.   

\begin{lemma}\label{lem localgoodQ}
Let $\cal G_0$ in \eqref{QG} be a SPD graph. Then applying the $Q$-expansions in Section \ref{sec defnQ}, the new graphs (including the $Q$-graphs) are all SPD. 
\end{lemma}
\begin{proof}
The proof of this lemma is similar to the one for Lemma \ref{lem localgood2} by using Lemmas \ref{hard lemmdot} and \ref{hard lemmQ}, because $Q$-expansions are also local expansions as those in Lemma \ref{lem localgood2}. We omit the details. 
\end{proof}

\section{Proof of Theorem \ref{incomplete Texp} and Corollary \ref{lem completeTexp}}\label{sec global}


\subsection{Proof of Corollary \ref{lem completeTexp}}\label{sec lastglobal}

Given the $n$-th order $\incomp$ constructed in Theorem \ref{incomplete Texp}, we can solve \eqref{mlevelT incomplete} to get that
\begin{align}
	T_{\fa,\fb_1\fb_2} &=  m \left( \frac{1}{1-\Theta \wtSdelta^{(n)}}\Theta\right)_{\fa\fb_1} \overline G_{\fb_1\fb_2} \nonumber \\
	&   + \sum_x \left( \frac{1}{1-\Theta \wtSdelta^{(n)}}\right)_{\fa x} \left[  (\PIT^{(n)})_{x,\fb_1 \fb_2} + (\AITn)_{x,\fb_1 \fb_2} + (\QITn)_{x,\fb_1 \fb_2}+ (\Err'_{n,D})_{x,\fb_1\fb_2} \right] . \label{solving_T}
\end{align}	
By the property (iv) of Definition \ref{def incompgenuni}, we can write that 
$$(\PIT^{(n)})_{x,\fb_1 \fb_2} = \sum_y \Theta_{xy} (\cal G_R^{(n)})_{y,\fb_1 \fb_2},\quad (\AITn)_{x,\fb_1 \fb_2} = \sum_y \Theta_{xy}(\cal G_A^{(>n)})_{y,\fb_1 \fb_2} ,$$
$$ (\QITn)_{x,\fb_1 \fb_2}= \sum_y \Theta_{xy}(\cal G_Q^{(n)})_{y,\fb_1 \fb_2},\quad (\Err'_{n,D})_{x,\fb_1\fb_2}=\sum_y \Theta_{xy}(\cal G_{err}^{(n,D)})_{y,\fb_1 \fb_2} ,$$
for some sums of graphs $\cal G_R^{(n)}$, $\cal G_A^{(>n)}$, $\cal G_Q^{(n)}$ and $\cal G_{err}^{(n,D)}$. Using the definition \eqref{theta_renormal}, we can write \eqref{solving_T} as 
\begin{align}
  T_{\fa,\fb_1\fb_2} &=  m \Theta^{(n)}_{\fa\fb_1} \overline G_{\fb_1\fb_2}   + \sum_x \Theta^{(n)}_{\fa x} \left[(\cal G_R^{(n)})_{x,\fb_1 \fb_2} + (\cal G_A^{(>n)})_{x,\fb_1 \fb_2}+  (\cal G_Q^{(n)})_{x,\fb_1 \fb_2} +   (\cal G_{err}^{(n,D)})_{x,\fb_1 \fb_2}\right] . \label{solving_T2}
\end{align}	
We expand $\Theta^{(n)}$ as 
\be\label{expand_thetan}
\Theta^{(n)}= \sum_{k=0}^D ( \Theta \wtSdelta^{(n)})^{k}\Theta + \Theta^{(n)}_{err},\quad \Theta^{(n)}_{err}:= \sum_{k>D} ( \Theta \wtSdelta^{(n)})^{k}\Theta.
\ee
The terms $( \Theta \wtSdelta^{(n)})^{k}\Theta$ can be expanded into sums of labelled $\dashed$ edges, and we regard  \smash{$(\Theta^{(n)}_{err})_{xy}$} as a new type of $\dashed$ edge of scaling order $\ge 2(D+2)$ between atoms $x$ and $y$. Then we plug the expansion \eqref{expand_thetan} into \eqref{solving_T2} and rearrange the resulting graphs as follows. First, all graphs containing \smash{$\Theta^{(n)}_{err}$} and graphs from \smash{$\sum_x \Theta^{(n)}_{\fa x} (\cal G_{err}^{(n,D)})_{x,\fb_1 \fb_2}$} will be included into $(\Err_{n,D})_{\fa,\fb_1\fb_2}$. For other graphs, 
\begin{itemize}
\item the term \smash{$m \Theta^{(n)}_{\fa\fb_1} \overline G_{\fb_1\fb_2}$} will give the first two terms on the right-hand side of \eqref{mlevelTgdef} and some higher order graphs in $ (\ATn)_{\fa,\fb_1\fb_2}$; 

\item the term \smash{$\sum_x \Theta^{(n)}_{\fa x} (\cal G_R^{(n)})_{x,\fb_1 \fb_2}$} will give $(\PTn)_{\fa,\fb_1\fb_2}$ and some higher order graphs in $ (\ATn)_{\fa,\fb_1\fb_2}$;  

\item the term \smash{$\sum_x \Theta^{(n)}_{\fa x} (\cal G_A^{(>n)})_{x,\fb_1 \fb_2}$} will give higher order graphs in $ (\ATn)_{\fa,\fb_1\fb_2}$;  

\item the term \smash{$\sum_x \Theta^{(n)}_{\fa x} (\cal G_Q^{(n)})_{x,\fb_1 \fb_2}$} will give graphs in $(\QTn)_{\fa,\fb_1\fb_2}$. 
\end{itemize}
This concludes Corollary \ref{lem completeTexp}.

\subsection{Globally standard graphs}\label{sec_global_standard}

For the rest of this section, we focus on proving Theorem \ref{incomplete Texp}. In this subsection, we introduce the concept of \emph{globally standard graphs}. By taking into account the external red solid edges connected with isolated subgraphs, we define the following concept of weakly and strongly isolated subgraphs.

\begin{definition}[Weakly and strongly isolated subgraphs]\label{def_weakstrong}
An isolated subset of molecules $\Pol$ is said to be \emph{strongly isolated} if there is \emph{at most one} red solid edge between $\Pol$ and $\Pol^c$, and the subgraph $\Iso_\Pol$ induced on $\Pol$ is called a strongly isolated subgraph. 
Otherwise, $\Pol$ and $\Iso_\Pol$ are said to be weakly isolated (i.e. a weakly isolated subgraph is an isolated subgraph that is not 
strongly isolated).


\end{definition}

%

With Definition \ref{def seqPDG} and Definition \ref{def_weakstrong}, we define globally standard graphs.

\begin{definition}[Globally standard graphs]\label{defn gs}
A graph is said to be \emph{globally standard} if it is SPD, and every proper isolated subgraph of it is weakly isolated. 
\end{definition}


The reason for introducing globally standard graphs is to avoid \emph{non-expandable graphs}, i.e. graphs that cannot be expanded without breaking the doubly connected property. As an example, we consider the following two graphs, where $\Gamma_1$ is strongly isolated in (a) and weakly isolated in (b):
\begin{center}
	\includegraphics[width=9cm]{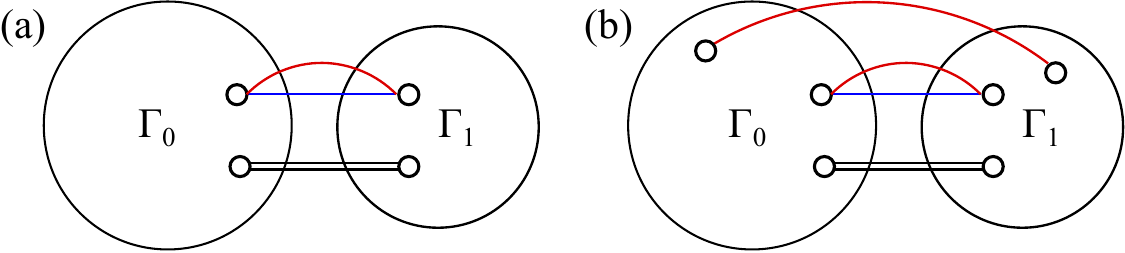}
\end{center}
Suppose both graphs (a) and (b) are SPD with $\Gamma_1$ being the minimal isolated subgraph. If $\Gamma_1$ contains blue solid edges, then we can either perform local expansions or expand the first blue solid edge in a pre-deterministic order of $\Gamma_1$. Now suppose $\Gamma_1$ does not contain any internal blue solid edge, then it is not locally standard in graph (b) because there are more external red solid edges than blue solid edges. Hence we can perform local expansions to graph (b). On the other hand, if both $\Gamma_0$ and $\Gamma_1$ in (a) are deterministic, then graph (a) is locally standard and only contains a pivotal blue solid edge, so it cannot be expanded without breaking the doubly connected property. Such a graph may not be included into one of the six terms on the right-hand side of \eqref{mlevelT incomplete}. 
In sum, the SPD property in Definition \ref{defn gs} gives a canonical order of blue solid edges to expand, while the weakly isolated property guarantees that our expansions will not give non-expandable graphs. 


Now we state some additional properties satisfied by the graphs in the $T$-expansion and $\incomp$. Recall the definition of minimal isolated subgraphs (MIS) in Definition \ref{defn_min_iso}.

\begin{definition} [$T$-expansion and $\incomp$: additional properties]\label{def genuni2}
The graphs in Definition \ref{defn genuni} satisfy the following additional properties. 

\begin{enumerate}

\item Every graph in $(\PTn)_{\fa,\fb_1\fb_2}$ is globally standard. 

\item Every graph in $  (\ATn)_{\fa,\fb_1\fb_2} $ is SPD. 

\item Every $Q$-graph in $(\QTn)_{\fa,\fb_1\fb_2} $ is SPD. 
Moreover, the atom in the $Q$-label of a $Q$-graph belongs to the MIS, i.e. all solid edges and weights have the same $Q$-label $Q_x$ for an atom $x$ inside the MIS. 

\end{enumerate}
$(\PITn)_{\fa,\fb_1 \fb_2}$, $(\AITn)_{\fa,\fb_1\fb_2}$ and $ (\QITn)_{\fa,\fb_1\fb_2}  $ in Definition \ref{def incompgenuni} also satisfy these properties.
\end{definition}

By this definition, the higher order and $Q$ graphs in the $T$-expansion and $\incomp$ may contain strongly isolated subgraphs, but the recollision graphs do not. The heuristic reason is as follows. We only expand globally standard graphs, and we want to show that an expansion of a globally standard graph still gives globally standard graphs plus some graphs that are not needed to be expanded further (i.e. recollision, $Q$, higher order and error graphs). Suppose we substitute a $t$-variable with a recollision graph $\cal G$ in \smash{$\PT^{(n-1)}$}. In order for the resulting graph to be globally standard, $\cal G$ cannot contain any strongly isolated subgraph. 
This explains why recollision graphs have to be weakly isolated. If we substitute a $T$-variable with a graph in \smash{$\AT^{(n-1)}$}, then the resulting graph will be of scaling order $>n$ and we will not expand this graph further. Hence we do not require this graph to be globally standard. Finally, if we substitute a $T$-variable with a $Q$-graph $\cal G$ in \smash{$\QT^{(n-1)}$} and get a graph, say $\cal G_0$, then we need to perform $Q$-expansions to $\cal G_0$. Suppose a $Q$-expansion of $\cal G_0$ is given by \eqref{G0Q}. By the property (iv) of Lemma \ref{Q_lemma}, there will be solid edges connecting the atom in the $Q$-label of $\cal G$, say $x$, to old atoms in the original graph. Since $x$ belongs to the MIS of $\cal G$, one can see that the strongly isolated subgraphs in $\cal G_0$ before the $Q$-expansion are not strongly isolated anymore in graphs $\cal G_\omega$ after the $Q$-expansion. The rigorous argument will be given in the proof of Lemma \ref{lem globalgood} below.

With Lemma \ref{lem localgood2}, we can easily show that local expansions preserve the globally standard property. 

\begin{lemma}\label{lem localgood3}
Let $\cal G$ be a globally standard graph without any $P/Q$ labels. 
\begin{enumerate}
\item Apply any expansion in Definitions \ref{dot-def}, \ref{Ow-def}, \ref{multi-def}, \ref{GG-def} and \ref{GGbar-def} to $\cal G$. The new graphs (including the $Q$-graphs) are all globally standard. 

\item Apply any expansion in Definitions \ref{Ow-def}, \ref{multi-def}, \ref{GG-def} and \ref{GGbar-def} on an atom in the MIS of $\cal G$. In each new $Q$-graph, the atom in the $Q$-label also belongs to the MIS. 

\end{enumerate}
\end{lemma}
\begin{proof}
By Lemma \ref{lem localgood2}, the new graphs are all SPD. Let $\wt{\cal G}$ be one of the new graphs. As shown in the proof of Lemma \ref{lem localgood2}, any proper isolated subset of molecules $\Pol$ in $\wt{\cal G}$ is obtained from merging a proper isolated subset of molecules, say $\Pol_0$, in $\cal G$. Moreover, it easy to see that the number of external red solid edges of $\Pol$ in $\wt{\cal G}$ is larger than or equal to the number of external red solid edges of $\Pol_0$ in ${\cal G}$. Since $\Pol_0$ is weakly isolated in ${\cal G}$, we get that $\Pol$ is also weakly isolated in $\wt{\cal G}$. This concludes (i). 
The property (ii) can be checked readily using the Definitions \ref{Ow-def}, \ref{multi-def}, \ref{GG-def} and \ref{GGbar-def}.
%
\end{proof}

The following lemma is key to our proof. Roughly speaking, it shows that the globally standard property is preserved if we expand the \emph{first} blue solid edge in a pre-deterministic order of the MIS. 


\begin{lemma}\label{lem globalgood}
Suppose $\cal G$ is a globally standard graph without $P/Q$ labels. Let $\Iso_k$ be the MIS of $\cal G$. Consider a $t_{x,y_1y_2}$ variable in $\cal G$ defined by \eqref{12in4T}, so that the blue solid edge $G_{\al y_1}$ is 
the first blue solid edge in a pre-deterministic order of $\Iso_k$.  
  

\begin{itemize}
\item[(1)] If we replace $t_{x,y_1y_2}$ with a graph in the first two terms on the right-hand side of \eqref{replaceT},
then the resulting graph, say \smash{$\wt{\cal G}$}, has no $P/Q$ labels and is globally standard. 
Moreover, \smash{$\wt{\cal G}$} either has a strictly higher scaling order than $\cal G$, or is obtained by replacing $t_{x,y_1y_2}$ with $ m\Theta_{xy_1}\overline G_{y_1y_2}$.

\item[(2)] If we replace $t_{x,y_1y_2}$ with a graph $(\cal G_{R})_{x,y_1 y_2}$ in $(\wtPT^{(n-1)})_{x,y_1 y_2}$, then the resulting graph has no $P/Q$ labels, is globally standard and has a scaling order $\ge \ord(\cal G) +1 $. 

\item[(3)] If we replace $t_{x,y_1y_2}$ with a graph $(\cal G_{A})_{x,y_1y_2}$ in $(\AT^{(>n-1)})_{x,y_1 y_2}$, then the resulting graph  has no $P/Q$ labels, is SPD and has a scaling order $\ge \ord(\cal G)+n-2$. 

\item[(4)] Suppose we replace $t_{x,y_1y_2}$ with a graph $(\cal G_{Q})_{x,y_1y_2}$ in $(\QT^{(n-1)})_{x,y_1 y_2}$ and get a graph $\wt{ \cal G} $. Then applying the $Q$-expansions, we can expand it into a sum of $\OO(1)$ many graphs: 
\be\label{mlevelTgdef3_Q}
\wt{ \cal G} = \sum_\omega \cal G_\omega  + \cal Q   + \cal G_{err} ,
\ee
where every $\cal G_\omega$ has no $P/Q$ labels, is globally standard and has a scaling order $\ge \ord(\cal G) +1 $; 
$\cal Q$ is a sum of $Q$-graphs, each of which is SPD and has a MIS containing the atom in the $Q$-label; 
$\cal G_{err}$ is a sum of doubly connected graphs of scaling orders $> D$.

\item[(5)] If we replace $t_{x,y_1y_2}$ with a graph $(\cal G_{err})_{x,y_1y_2}$ in $(\Err_{n-1,D})_{x,y_1 y_2}$, then the resulting graph is doubly connected and has a scaling order $\ge \ord(\cal G)+D-1$. 
\end{itemize}
\end{lemma}
\begin{proof}
The proofs of (1), (3) and (5) are simple by using the properties in Definitions \ref{defn genuni} and \ref{def genuni2}. We focus on the proofs of (2) and (4). 

\vspace{5pt}
\noindent {\bf Proof of (2):} We denote the new graph by $\wt{\cal G}$. Then we have the following graph (a), where we only show a case which has a dotted edge connected with $y_1$ in $(\cal G_{R})_{x,y_1 y_2}$. There are also cases with a dotted edge connected with $y_2$. With a slight abuse of notation, we have used $x$ and $y_1$ to denote the respective molecules that contain atoms $x$ and $y_1$. 
\be\label{eq_recolli1}
\parbox[c]{0.9\linewidth}{\includegraphics[width=15cm]{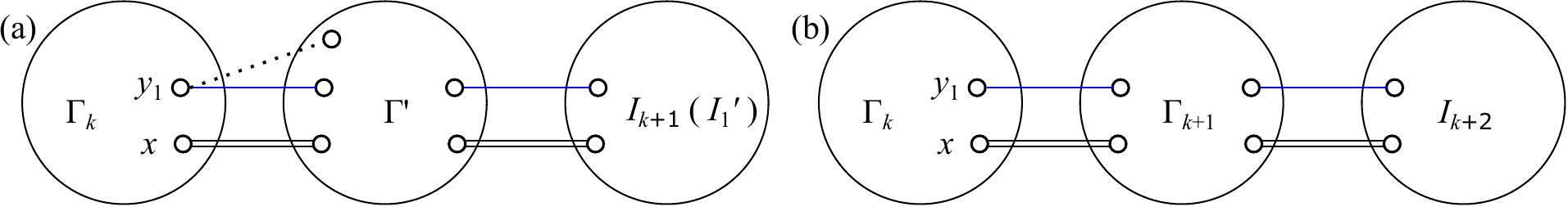}}
\ee
The graph (a) gives the isolated molecular subgraph $\Iso_k$ in $\wt{\cal G}$ with all red solid edges removed, where we have not merged the molecules connected by the dotted edge.  Inside the back circles are some subgraphs, where $\Gamma_k$ contains the molecules in $\Iso_k$ of the original graph $\cal G$, $\Iso'_1=\Iso_{k+1}$ is the first proper isolated subgraph in $(\cal G_{R})_{x,y_1 y_2}$ and also the $(k+1)$-th proper isolated subgraph in $\wt{\cal G}$, and $\Gamma'$ contains the molecules in $(\cal G_{R})_{x,y_1 y_2}$ that are not in $\Iso'_1$. 
The graph (b) is obtained by removing the dotted edge from graph (a) and renaming the subgraphs.

First, $\ord(\wt{\cal G})\ge \ord(\cal G) +1 $ follows immediately from the fact that $(\cal G_{R})_{x,y_1 y_2}$ has scaling order $\ge 3$.
Second, using the fact that both $\cal G$ and $(\cal G_R)_{x,y_1y_2}$ have no proper strongly isolated subgraphs, we immediately get that \smash{$\wt{\cal G}$} also does not contain any proper strongly isolated subgraph. 
It remains to prove the SPD property of $\wt{\cal G}$. With Lemma \ref{hard lemmdot}, it suffices to show that graph (b) in \eqref{eq_recolli1}, denoted by $\wt{\cal G}_b$, is SPD. In this graph, we have renamed $\Iso_{k+1}$ in (a) as $\Iso_{k+2}$ and $\Gamma'$ as $\Gamma_{k+1}$. Notice that $\Gamma_{k+1}$, $\Iso_{k+2}$ and the two edges between them form the isolated subgraph $\Iso_{k+1}$ in \smash{$\wt{\cal G}_b$}.
Now we verify the properties (i)--(iii) of Definition \ref{def seqPDG} for $\wt{\cal G}_b$. The property (i) follows from Lemma \ref{trivial_claim3}. 
Using the SPD property of $(\cal G_{R})_{x,y_1y_2}$ given by Definition \ref{def genuni2}, we get that there is at most one sequence of isolated graphs in $\Iso_k$. Together with the SPD property of $\cal G$ and the fact that $\Iso_k$ is the MIS of $\cal G$, it concludes the property (ii) for $\wt{\cal G}_b$. 
It remains to prove the property (iii). Suppose the sequence of proper isolated subgraphs in $\wt{\cal G}_b$ is 
\be\label{pf_chain_iso}\Iso_{l}\subset \cdots \subset \Iso_{k+2}\subset \Iso_{k+1} \subset  \Iso_k \subset \cdots \subset \Iso_1,\quad \text{for some} \ \ l\ge k,\ee
and for simplicity of notation, we denote the maximal subgraph by $\Iso_0$.   
For $j\ge k+1$, if we replace $\Iso_{j+1}$ and its two external edges with a $\dashed$ edge, then $\Iso_j$ becomes pre-deterministic because $(\cal G_{R})_{x,y_1y_2}$ is SPD. 
For $j\le k-1$, if we replace $\Iso_{j+1}$ and its two external edges with a $\dashed$ edge, then $\Iso_j$ becomes pre-deterministic because $\cal G$ is SPD. Finally, if we replace $\Iso_{k+1}$ and its two external edges with a $\dashed$ edge, then it corresponds to changing the first blue solid edge in a pre-deterministic order of $\Iso_k$ into a $\dashed$ edge. Hence $\Iso_k$ becomes pre-deterministic by Definition \ref{def PDG}. In sum, we get that $\wt{\cal G}_b$ is SPD, which concludes statement (2) of Lemma \ref{lem globalgood}. 

\vspace{5pt}
\noindent {\bf Proof of (4):} Using a similar argument as the one for $\wt{\cal G}_b$, we get that $\wt {\cal G}$ is SPD. Then the SPD property of the graphs on the right-hand side of \eqref{mlevelTgdef3_Q} follow from Lemma \ref{lem localgoodQ}. The conditions on scaling orders and the fact that the atom in the $Q$-label of a $Q$-graph belongs to the MIS follow from Lemma \ref{Q_lemma}. It remains to prove the weakly isolated property of (possible) proper isolated subgraphs in $\cal G_{\omega}$.


We assume that $(\cal G_{Q})_{x,y_1y_2}$ indeed contains strongly isolated subgraphs, since otherwise the proof will be trivial. Then we have the following molecular graph without red solid edges: 
\begin{equation}\nonumber
	\parbox[c]{0.8\linewidth}{\center
		\includegraphics[width=12cm]{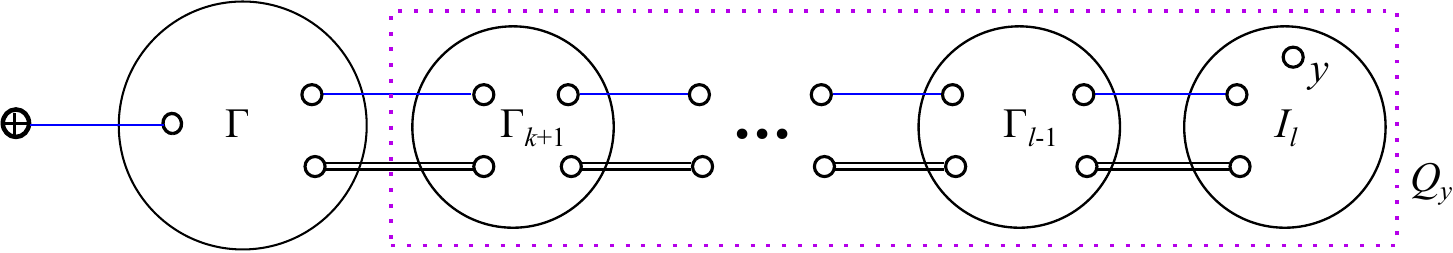}}
\end{equation}
Here all solid edges and weights inside the purple dashed rectangle have the same $Q_y$ label for an atom $y$ in the MIS, denoted by $\Iso_l$ for some $l\ge k$. Inside the back circles are some subgraphs, where $\Gamma$ contains the molecules in the original graph $\cal G$, and $\Gamma_i$'s are components corresponding to the isolated subgraphs in $(\cal G_{Q})_{x,y_1y_2}$. 
Again we denote the isolated subgraphs of $\wt{\cal G}$ by \eqref{pf_chain_iso}. Then by property (iv) of Lemma \ref{Q_lemma}, every $\cal G_{\omega}$ has atoms in $\Gamma$ or external atoms so that (a) they connect to $y$ through red solid edges, (b) they connect to $y$ through blue solid edges, or (c) they have been merged with $y$. In cases (b) and (c), the subgraphs $\Iso_{i}$, $k+1\le i \le l$, are not isolated anymore. In case (a), we have the following graph (or some variants of it, where we are not trying to draw all possible cases):
\begin{equation}\nonumber
	\parbox[c]{0.8\linewidth}{\center
		\includegraphics[width=11cm]{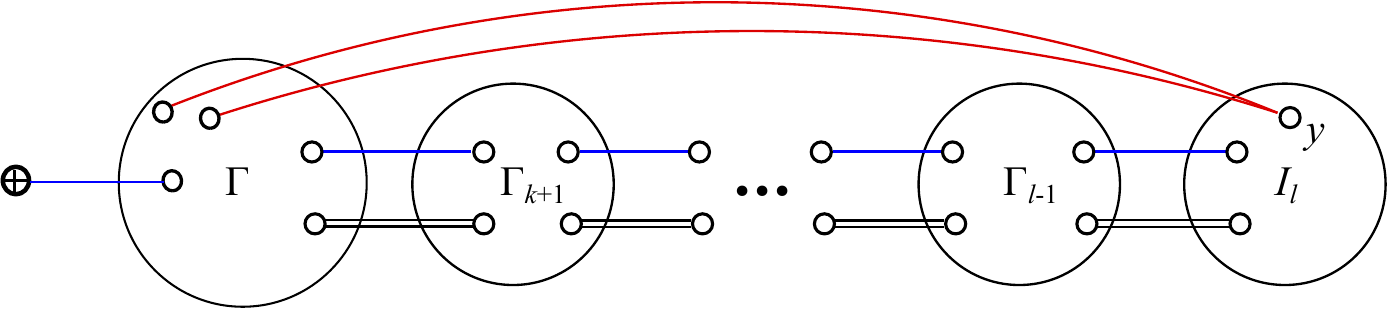}}
\end{equation}
In this case, any subgraph $\Iso_{i}$, $k+1\le i \le l$, is weakly isolated if it is still an isolated subgraph in $\cal G_\omega$. In sum, we see that $\cal G_{\omega}$ does not contain proper strongly isolated subgraphs, and hence is globally standard. 
\end{proof}


\subsection{Global expansion strategy} \label{sec expandlvl40}


Lemma \ref{lem globalgood} gives a canonical choice of the standard neutral atom in a global expansion of a globally standard graph, that is, we choose an ending atom of the first blue solid edge in a pre-deterministic order of the MIS. With this choice, we define the global expansion strategy for the proof of Theorem \ref{incomplete Texp} in this subsection.

Fix any $n\in \N$. As an induction hypothesis, suppose we have obtained the $(n-1)$-th order $T$-expansion. 
Then we define the following stopping rules. We stop the expansion of a graph if it is a normal regular graph and satisfies at least one of the following properties:
\begin{itemize}
\item[(S1)] it is a $\oplus$/$\ominus$-recollision graph; 

\item[(S2)] its scaling order is at least $n+1$;

\item[(S3)] it is a $Q$-graph; 

\item[(S4)] it is \emph{non-expandable}, that is, it is locally standard and has no redundant blue solid edge. 
\end{itemize}
If a graph has a deterministic maximal subgraph, then it is non-expandable. On the other hand, in a non-expandable graph that is not deterministic, we cannot expand a plus $G$ edge without breaking the doubly connected property. Now we give our global expansion strategy. 



\begin{strategy}[Global expansion strategy]\label{strat_global}
Fix any $n\in \N$ and large $D>n$. Given the above stopping rules (S1)--(S4), we apply the following strategy.
 \vspace{5pt}

\noindent{\bf Step 0}: We start with the second order $T$-expansion \eqref{seconduniversal}, and apply local expansions to obtain a linear combination of new graphs, each of which either satisfies the stopping rules (S1)--(S3) already, or is locally standard. At this step, there is only one internal molecule in every graph, which is trivially globally standard. 

\vspace{5pt}
\noindent{\bf Step 1}: Given a globally standard input graph, we perform the local expansions in Definitions \ref{Ow-def}, \ref{multi-def}, \ref{GG-def} and \ref{GGbar-def} on atoms in the minimal isolated subgraph (MIS). We send the resulting graphs that already satisfy the stopping rules (S1)--(S4) to the outputs. Every remaining graph is globally standard by Lemma \ref{lem localgood3}, and its MIS 
is locally standard (i.e. the MIS has no weights and every atom in it either has degree 0 or is a standard neutral atom by Definition \ref{deflvl1}).

\vspace{5pt}
\noindent{\bf Step 2}: Given a globally standard input graph $\cal G$ with a locally standard MIS, say $\Iso_k$, we find a $t_{x,y_1y_2}$ or $t_{y_1y_2,x}$ variable defined in \eqref{12in4T}, so that $\al$ is a standard neutral atom in $\Iso_k$ and the blue solid edge $G_{\al y_1}$ or $G_{y_1 \al}$ is the first blue solid edge in a pre-deterministic order of $\Iso_k$. 
If we cannot find such a $t$-variable, then we stop expanding the input graph.

\vspace{5pt}
\noindent{\bf Step 3}: We apply the global expansions defined in Section \ref{subsec global} to the $t_{x,y_1y_2}$ or $t_{y_1y_2,x}$ variable chosen in Step 2. More precisely, we first replace $t_{x,y_1y_2}$ (or $t_{y_1y_2,x}$) with \eqref{replaceT} (or an expression obtained by taking transpositions of all $G$ entries in \eqref{replaceT}). Then we apply the $Q$-expansions defined in Section \ref{sec defnQ} to graphs that take the form \eqref{QG}, where both $\Gamma_0$ and \smash{$\wt\Gamma_0$} are non-trivial. 
We send the resulting graphs that already satisfy the stopping rules (S1)--(S4) to the outputs. The remaining graphs are all globally standard by Lemma \ref{lem globalgood}, and we sent them back to Step 1. 
\end{strategy}

Using the definition of globally standard graphs, it is easy to show that if an output graph from the above strategy does not satisfy the stopping rules (S1)--(S3), then it must have a deterministic maximal subgraph. 
\begin{lemma}\label{lemm expansionstrat_nonstop}
Let $\cal G_{\fa,\fb_1\fb_2}$ be an output graph from the global expansion strategy \ref{strat_global}. If $\cal G_{\fa,\fb_1\fb_2}$ does not satisfy stopping rules (S1)-(S3), then it has a deterministic maximal subgraph with a doubly connected structure. 
\end{lemma}
\begin{proof}
By Lemma \ref{lem localgood3} and Lemma \ref{lem globalgood}, we know that $\cal G_{\fa,\fb_1\fb_2}$ is globally standard. Under the assumption of this lemma, $\cal G_{\fa,\fb_1\fb_2}$ either is non-expandable or does not contain a $t$-variable required by Step 2 of Strategy \ref{strat_global}. In either case, $\cal G_{\fa,\fb_1\fb_2}$ contains a locally standard MIS, say $\Iso_k$, which, by the pre-deterministic property of $\Iso_k$, does not contain any internal blue solid edge. 

Now we consider two cases. First, suppose $\Iso_k$ is a proper isolated subgraph. Due to the weakly isolated property, $\Iso_k$ is connected with at least two external red solid edges and at most one blue solid edge. Hence $\Iso_k$ contains a non-standard neutral atom, which gives a contradiction. Second, suppose $\Iso_k$ is indeed the maximal subgraph of $\cal G_{\fa,\fb_1\fb_2}$. Then $\Iso_k$ is locally standard and does  not contain any internal blue solid edges, so it is deterministic.
\end{proof}

With Lemma \ref{lem globalgood}, we show that Strategy \ref{strat_global} will stop after $\OO(1)$ many iterations.

\begin{lemma}\label{lemm expansionstrat}
The expansions in Strategy \ref{strat_global} will stop after at most $C_n$ many iterations of Steps 1--3, where $C_n$ denotes a large constant depending on $n$. All graphs from these expansions are doubly connected. If a graph satisfies the stopping rule (S1)/(S2)/(S3) and has scaling order $\le D$, then it satisfies the property (i)/(ii)/(iii) of Definition \ref{def genuni2}.   
\end{lemma}
\begin{proof}
Let $\cal G$ be an input graph for Step 3 of Strategy \ref{strat_global}. 
Its scaling order must be larger than or equal to the scaling orders of the graphs in \eqref{Aho>2}, i.e. $\ord(\cal G)\ge 3$. Hence the new graphs from case (3) of Lemma \ref{lem globalgood} are SPD and have scaling orders $\ge n + 1$, so it satisfies the property (ii) of Definition \ref{def genuni2}; the new graphs from case (5) of Lemma \ref{lem globalgood} are doubly connected and have scaling orders $\ge D + 2$. Furthermore, the new graphs from cases (1) and (2) of Lemma \ref{lem globalgood} are globally standard, so they satisfy the property (i)/(ii) of Definition \ref{def genuni2} if they satisfy the stopping rule (S1)/(S2). Finally, in equation \eqref{mlevelTgdef3_Q}, the graphs in $\cal Q$ satisfy the stopping rule (S3) and the property (iii) of Definition \ref{def genuni2}; the graphs in $\cal G_{err}$ are doubly connected and have scaling orders $\ge D+1$; the graphs $\cal G_{\omega}$ are globally standard, so they satisfy the property (i)/(ii) of Definition \ref{def genuni2} if they satisfy the stopping rule (S1)/(S2). Combining the above discussions with Lemma \ref{lem localgood3} for local expansions in Step 1 of Strategy \ref{strat_global}, we conclude the second and third statements. 
	
Now we prove the first statement. Let $\cal G_0$ be a graph from Step 0. We construct inductively a tree diagram $\cal T$ for the  expansions of $\cal G_0$ as follows. Let $\cal G_0$ be the root. Given a graph $\cal G$ represented by a vertex of the tree, its children are the graphs obtained from an iteration of Steps 1--3 of Strategy \ref{strat_global} acting on $\cal G$. Suppose a graph $\cal G$ satisfies a stopping rule, then we stop the expansion and $\cal G$ has no children on $\cal T$. Let the height of $\cal T$ be the maximum distance between the root and a leaf of $\cal T$. To show that the expansions will stop after $C_n$ many iterations of Steps 1--3, it is equivalent to show that $\cal T$ is a finite tree with height $\le C_n$.
	
Let $\cal G_0 \to \cal G_1 \to \cal G_2 \to \cdots \to \cal G_h$ be a self-avoiding path from the root to a leaf. Let $k_0=0$. After having defined $k_i$, let $k_{i+1}:=h\wedge \min\{j> k_i : \text{ord}(\cal G_j) > \text{ord}(\cal G_{k_i})\}$. Then the increasing sequence $\{k_0,k_1, k_2, \cdots\}$ has length at most $n$, because a graph of scaling order $\ge n+1$ satisfies the stopping rule (S2) and has no children. Moreover, we claim that $|k_{i+1}-k_i|\le n$. First, by Lemma \ref{lem globalgood}, if a new graph after an iteration of Steps 1--3 has the same scaling order as the input graph, then it must have one fewer blue solid edge than the input graph. Second, by the property (ii) of Definition \ref{defnlvl0}, the number of internal atoms in $\cal G_{k_i}$ is less than or equal to the number of waved and $\dashed$ edges, which gives that the number of blue solid edges in $\cal G_{k_i}$ is at most $n$. These two facts together conclude the claim. In sum, we have shown that $h\le n^2$, i.e. the height of $\cal T$ is at most $n^2$.   
\end{proof}

\subsection{Proof of Theorem \ref{incomplete Texp}}\label{sec expandlvl4}

In this subsection, we prove Theorem \ref{incomplete Texp} using the results in Sections \ref{sec_global_standard} and \ref{sec expandlvl40}. To summarize, we have proved the following facts.
\begin{itemize}
\item[(I)] The expansions of $T_{\fa,\fb_1\fb_2}$ will stop and give $\OO(1)$ many new graphs (Lemma \ref{lemm expansionstrat}). 


\item[(II)] If a new graph from the expansions satisfies the stopping rule (S1)/(S2)/(S3), then it satisfies the property (i)/(ii)/(iii) of Definition \ref{def genuni2} (Lemma \ref{lemm expansionstrat}).

\item[(III)] If a new graph from the expansions does not satisfy the stopping rules (S1)--(S3), then it has a deterministic doubly connected maximal subgraph (Lemma \ref{lemm expansionstrat_nonstop}).
\end{itemize}

Now we complete the proof of Theorem \ref{incomplete Texp} using induction. We have obtained the second order $T$-expansion and $\incomp$ in \eqref{seconduniversal}. Fix any $ 2\le n\le M$. Suppose for all $2\le k \le n-1$, we have constructed the $k$-th order $\incomp$ using Strategy \ref{strat_global}, shown the properties \eqref{two_properties0}--\eqref{3rd_property0} of $\Sele_{n-1}$ by Lemma \ref{cancellation property}, and obtained the $k$-th order $T$-equation using Corollary \ref{lem completeTexp}. Now applying Strategy \ref{strat_global}, we obtain the following expansion of $T_{\fa,\fb_1\fb_2}$: 
\begin{align} 
 T_{\fa,\fb_1 \fb_2}&= m  \Theta_{\fa \fb_1} \overline G_{\fb_1\fb_2} + \sum_{x} (\Theta \wtSdelta^{(n)})_{\fa x} t_{x,\fb_1\fb_2} + (\wtPIT^{(n)})_{\fa,\fb_1 \fb_2} +  (\AIT^{(>n)})_{\fa,\fb_1\fb_2}  + (\QIT^{(n)})_{\fa,\fb_1\fb_2}+  (\Err'_{n,D})_{\fa,\fb_1\fb_2} .\nonumber
\end{align}
Denoting
$$
	(\PIT^{(n)})_{\fa,\fb_1 \fb_2}=	(\wtPIT^{(n)})_{\fa,\fb_1 \fb_2} -   \sum_{x } (\Theta  \wtSdelta^{(n)})_{\fa x} |m|^2 \sum_\al s_{x\al} G_{\al\fb_1}\overline G_{\al\fb_2}\left(1- \mathbf 1_{\al \ne \fb_1}\mathbf 1_{\al\ne \fb_2} \right),
$$
the above equation can be rewritten as 
\begin{align} 
T_{\fa,\fb_1 \fb_2}&= m  \Theta_{\fa \fb_1} \overline G_{\fb_1\fb_2} + \sum_{x } (\Theta \wtSdelta^{(n)})_{\fa x} T_{x,\fb_1\fb_2} +  (\PIT^{(n)})_{\fa,\fb_1 \fb_2} +  (\AIT^{(>n)})_{\fa,\fb_1\fb_2} \nonumber\\
&+ (\QIT^{(n)})_{\fa,\fb_1\fb_2}+  (\Err'_{n,D})_{\fa,\fb_1\fb_2}. \label{add_incompT}
\end{align}
By the above three facts (I)--(III), we have shown that \eqref{add_incompT} satisfies Definitions \ref{def incompgenuni} and \ref{def genuni2}, except for the following properties:
\begin{itemize}

\item in the decomposition \eqref{decomposePIT} of $\PIT^{(n)}$, $\cal R_{IT,l}$, $4\le l\le n-1$, are the same expressions as those in lower order $\incomp$s; 

\item in the decomposition \eqref{decomposeQIT} of $\QIT^{(n)}$, $\cal Q_{IT,l}$, $4\le l\le n-1$, are the same expressions as those in lower order $\incomp$s; 

\item in the decomposition \eqref{decompose Sdelta} of $\wtSdeltan$, $\Sele_{l}$, $4\le l\le n-1$, are the same $\selfs$ as those in lower order $\incomp$s.




\end{itemize}

We decompose the graphs $\wtSdeltan$, $\PIT^{(n)}$ and $\QIT^{(n)}$ in \eqref{add_incompT} according to scaling orders as  
\be\label{same Sdelta0}\wtSdeltan=\wh\wtSdelta^{(n-1)}+ \Sele_{n},\quad \PIT^{(n)} = \whPIT^{(n-1)} + \cal R_{IT,n} ,\quad  \QIT^{(n)} = \whQIT^{(n-1)} +  \cal Q_{IT,n} + \QIT^{(>n)}, \ee
where $\wh\wtSdelta^{(n-1)}$, $\whPIT^{(n-1)}$ and $\whQIT^{(n-1)}$ are sums of graphs of scaling orders $\le n-1$, $ \Sele_{n}$, $\cal R_{IT,n}$ and $\cal Q_{IT,n}$ are sums of graphs of scaling order $n$, and \smash{$\QIT^{(>n)}$} is a sum of $Q$-graphs of scaling order $>n$. To conclude the proof, it suffices to prove that 
\be\label{same Sdelta} 
\wh\wtSdelta^{(n-1)}= \wtSdelta^{(n-1)} ,\quad \whPIT^{(n-1)}=\PIT^{(n-1)},\quad \whQIT^{(n-1)}=\QIT^{(n-1)}, 
\ee
if we perform the expansions in a \emph{proper way}. 
In order to have \eqref{same Sdelta}, we perform almost the same $(n-1)$-th and $n$-th order expansions of $T_{\fa,\fb_1 \fb_2}$ in parallel, with the only difference being the cutoff order in the stopping rule: the cutoff order is $n-1$ for the $(n-1)$-th order expansion strategy, and $n$ for the $n$-th order expansion strategy. We now give a more precise description of the expansion procedure. 

After the $r$-th step expansion (where one step refers to a local expansion in Section \ref{sec localexp}, a $Q$-expansion in Section \ref{sec defnQ} or a substitution in global expansions), we denote by \smash{$\mathfrak G^{(n-1)}_r$ and $\mathfrak G^{(n)}_r$} the collections of scaling order $\le n-1$ graphs in the $(n-1)$-th and $n$-th order expansion processes, respectively. We trivially have \smash{$\mathfrak G^{(n-1)}_0=\mathfrak G^{(n)}_0$}. Suppose \smash{$\mathfrak G^{(n-1)}_r=\mathfrak G^{(n)}_r$} for an $r\in \N$. Then given any graph in \smash{$\cal G \in \mathfrak G^{(n-1)}_r$}, we perform the $(r+1)$-th step expansion as follows.
\begin{itemize}


\item If we apply a local expansion to $\cal G$ in the $(n-1)$-th order expansion process, then we apply the same expansion to the same weight or edge of $\cal G$ in the $n$-th order expansion process. 

\item If $\cal G$ is of the form \eqref{QG}, then we apply exactly the same $Q$-expansion to it in both the $(n-1)$-th and $n$-th order expansion  processes. 

\item Suppose in the $(n-1)$-th order expansion processes, we replace a $t_{x,y_1y_2}$ variable in $\cal G$ with an $(n-2)$-th order $T$-expansion (i.e. \eqref{replaceT} with $n-1$ replaced by $n-2$). Then in the $n$-th order expansion, we replace the {same $t_{x,y_1y_2}$ variable} with \eqref{replaceT}. By \eqref{decompose Sdelta0}, \eqref{decomposeP} and \eqref{decomposeQ}, we have that
$$\Sdelta^{(n-1)}-\Sdelta^{(n-2)}=\Sigma_{T,n-1},\quad \PT^{(n-1)}-\PT^{(n-2)}=\cal R_{T,n-1},$$ 
$$  \QT^{(k-1)}-\QT^{(k-2)}=\cal Q_{T,k-1}+\QT^{(>k-1)}- \QT^{(>k-2)}. $$
Hence the $(n-2)$-th order version of \eqref{replaceT} is different from \eqref{replaceT} only in the following terms: $$m(\Theta\Sigma_{T,n-1}\Theta)_{xy_1}\overline G_{y_1y_2}, \quad  (\cal R_{T,n-1})_{x,y_1y_2}, \quad  (\cal Q_{T,n-1}+\QT^{(>n-1)}- \QT^{(>n-2)})_{x,y_1y_2},$$ 
$$(\AT^{(>n-2)})_{x,y_1y_2}, \quad (\AT^{(>n-1)})_{x,y_1y_2}, \quad (\Err_{n-1,D})_{x,y_1y_2}, \quad (\Err_{n,D})_{x,y_1y_2}.$$  
However, if we have replaced $t_{x,y_1y_2}$ with a graph in one of these terms, then the resulting graph must be of scaling order at least $n$. Such graphs will not be included into $\mathfrak G^{(n-1)}_{r+1}$ or $\mathfrak G^{(n)}_{r+1}$.

\end{itemize}
In sum, we see that $\mathfrak G^{(n-1)}_{r+1}=\mathfrak G^{(n)}_{r+1}$ as long as we perform the $n$-th order expansions in the above way. By induction in $r$, we get that \eqref{same Sdelta} holds, which concludes the proof.

 \section{Proof of Theorem \ref{thm ptree}} \label{sec_pflocal}

Recall the following standard large deviation estimates in Lemma \ref{lem G<T}, which show that the resolvent entries are controlled $T$-variables. The bound \eqref{offG largedev} was proved in equation (3.20) of \cite{PartIII}, and \eqref{diagG largedev} was proved in Lemma 5.3 of \cite{delocal}. Given a matrix $M$, we will use $\|M\|_{\max}=\max_{i,j}|M_{ij}|$ to denote its maximum norm. 

\begin{lemma}\label{lem G<T}
	Suppose for a constant $\delta_0>0$ and deterministic parameter $W^{-d/2}\le \Phi\le W^{-\delta_0}$ we have that
	\be\label{initialGT} 
	\|G(z)-m(z)\|_{\max}\prec W^{-\delta_0},\quad \|T\|_{\max} \prec \Phi^2 ,
	\ee
	uniformly in $z\in \mathbf D$ for a subset $\mathbf D\subset \C_+$. Then 
	\be\label{offG largedev} \mathbf 1_{x\ne y} |G_{xy}(z)|^2  \prec T_{xy}(z)\ee
	uniformly in $x, y \in \Z_L^d$ and $z\in \mathbf D$, and
	\be\label{diagG largedev} |G_{xx}(z)-m(z)| \prec \Phi
	\ee
	uniformly in $x\in \Z_L^d$ and $z\in \mathbf D$.
\end{lemma}

Now we state the three main ingredients, Lemmas \ref{eta1case0}, \ref{lem: ini bound} and \ref{lemma ptree}, for the proof of Theorem  \ref{thm ptree}. First, we have an initial estimate when $\eta=1$, whose proof will be given in Section \ref{sec ptree}. 

 \begin{lemma}[Initial estimate] \label{eta1case0} 
 	Under the assumptions of Theorem \ref{thm ptree}, for any $z=E+\ii\eta$ with $E\in (-2+\kappa,2-\kappa)$ and $\eta=1$, we have that 
 	\be\label{locallaw eta1}
 	|G_{xy} (z) -m(z)\delta_{xy}|^2\prec  B_{xy}  ,\quad \forall \ x,y \in \Z_L^d.  
 	\ee
 \end{lemma}

Second, we have a continuity estimate in Lemma \ref{lem: ini bound}, whose proof will be given in Section \ref{sec ini_bound}. It allows us to get some a priori estimates on $G(z)$ from the local law \eqref{locallaw1} on $G(\wt z)$ for $\wt z$ with a larger imaginary part $\im \wt z = W^{\e_0} \im z$, where $\e_0>0$ is a small constant.
 
\begin{lemma}[Continuity estimates]\label{lem: ini bound}
	Under the assumptions of Theorem \ref{thm ptree}, suppose that 
	\be\label{ulevel}
	|G_{xy} (\wt z) -m(\wt z)\delta_{xy}|^2\prec B_{xy}(\wt z),\quad \forall \ x,y \in \Z_L^d,
	\ee
	with $\wt z =E+ \ii \wt\eta$ for some $E\in (-2+\kappa,2-\kappa)$ and $\wt\eta\in [W^{2+\e}/L^2,1]$. Then we have that 
	\be \label{initial Txy2}
	\max_{x,x_0}\frac{1}{K^{d}}\sum_{y:|y-x_0|\le K} \left(|G_{xy}( z)|^2 + |G_{yx}( z)|^2\right) \prec \left(\frac{\wt \eta} {\eta}\right)^2 \frac{1}{W^4 K^{d-4}},
	\ee
	uniformly in $K \in[W,L/2]$ and $z=E+ \ii\eta$ with $W^{2+\e}/L^{2}\le \eta\le \wt \eta$. 
	Moreover, for any constant $\e_0\in (0,d/20)$, we have that
	\be\label{Gmax}
	\|G(z)-m(z)\|_{\max}\prec W^{-d/2+\e_0},
	\ee
	uniformly in $z=E+\ii \eta$ with $\max\{W^{-\e_0}\wt\eta, W^{2+\e}/L^{2}\} \le \eta \le \wt\eta$.
\end{lemma}

Compared with \eqref{locallaw1}, the $\ell^\infty$ bound \eqref{Gmax} is sharp up to a factor $W^{\e_0}$. The estimate \eqref{initial Txy2} is an averaged bound instead of an entrywise bound and the right-hand side of \eqref{initial Txy2} loses an $W^2/K^2$ factor when compared with the sharp averaged bound $W^{-2}K^{-(d-2)}$. We need to use the following lemma to improve the weaker estimates \eqref{initial Txy2} and \eqref{Gmax} to the stronger local law \eqref{locallaw1}. Its proof will be given in Section \ref{sec ptree}. Note that \eqref{initial Txy2} verifies the assumption \eqref{initial Txy222} as long as we have $W^{-\e_0}\wt\eta \le \eta\le \wt\eta$. 
 
 
 \begin{lemma}[Entrywise bounds on  $T$-variables]\label{lemma ptree}  
 	Suppose the assumptions of Theorem \ref{thm ptree} hold. Fix any $z=E+ \ii\eta$ with $E\in (-2+\kappa,2-\kappa)$ and $ \eta\in [W^{2+\e}/L^2,1]$. Suppose for some constant $\e_0>0$, \eqref{Gmax} and the following estimate hold:
 	\be \label{initial Txy222}
 	\max_{x,x_0} \frac1{K^{d}}\sum_{y:|y-x_0|\le K} \left(|G_{xy}( z)|^2 + |G_{yx}( z)|^2\right) \prec  \frac{W^{2\e_0}}{W^4 K^{d-4}} ,
 	\ee
 	for all $K \in[W,L/2]$. As long as $\e_0$ is sufficiently small (depending on $n$ and $c_0$ in \eqref{Lcondition1}), we have that 
 	\be\label{pth T}
 	T_{xy}(z) \prec B_{xy} ,\quad \forall \ x,y\in \Z_L^d.
 	\ee
 \end{lemma}

 Combining the above three lemmas, we can complete the proof of Theorem \ref{thm ptree} using a bootstrapping argument on a sequence of multiplicatively decreasing $\eta$ given in \eqref{def etak}. The details have been given in Section 5.2 of \cite{PartI_high}. For the reader's convenience, we repeat the main argument here. 

 \begin{proof}[Proof of Theorem \ref{thm ptree}]
 Fix any $E\in (-2+\kappa,2-\kappa)$. We define the following sequence of decreasing $\eta$ for a small constant $\e_0>0$: 
 	\be\label{def etak}
 	\eta_{k}:=\max(W^{- k\e_0},\; W^{2+\e}/L^2) ,\quad 0\le k \le \ell_{0},
 	\ee
 	where $\ell_0$ is the smallest integer such that $W^{-\ell_0\e_0}\le W^{2+\e}/L^2$. By definition, $\eta_{k+1}=W^{-\e_0}\eta_k$ for $k\le \ell_0-1$ and we always have $\eta_{\ell_0}=W^{2+\e}/L^2$. Now we prove Theorem \ref{thm ptree} through an induction on $k$. First, by Lemma \ref{eta1case0},  \eqref{locallaw1} holds with $z= z_0=E+\ii\eta_0$. 
 	Then suppose \eqref{locallaw1} holds with $z= z_{k}:=E+ \ii \eta_k$ for some $0\le k \le \ell_0-1$. 
 	By Lemma \ref{lem: ini bound},  \eqref{Gmax} and \eqref{initial Txy222} hold for all $z=E+\ii \eta$ with $\eta_{k+1}\le \eta \le \eta_{k}$. Then applying Lemma \ref{lemma ptree}, we obtain that \eqref{pth T} holds with $z=z_{k+1}$. Now using Lemma \ref{lem G<T}, we can conclude \eqref{locallaw1} with $z=z_{k+1}$. Repeating the induction for $\ell_0$ steps, we obtain that 
 	\begin{itemize}
 		\item[(i)]\eqref{locallaw1} holds for all $z_k$ with $0\le k \le \ell_0$;
 		\item[(ii)]\eqref{Gmax} and \eqref{initial Txy222} hold for all $z=E+\ii \eta$ with $\eta_{\ell_0}\le \eta\le 1$.
 	\end{itemize}  
 	To conclude Theorem \ref{thm ptree}, we need to extend \eqref{locallaw1} uniformly to all $z=E+\ii \eta$ with $E\in (-2+\kappa,2-\kappa)$  and $\eta\in [\eta_{\ell_0}, 1]$. This is a standard perturbation and union bound argument, so we omit the details.  
%
 \end{proof}


\section{High moment estimate}\label{sec ptree}

In this section, we prove Lemma \ref{eta1case0} and Lemma \ref{lemma ptree}. Their proofs use the same idea, that is, we bound the high moment $\E  T_{xy} (z) ^p $ for any fixed $p\in \N$ using the $n$-th order $T$-expansion.
First, Lemma \ref{lemma ptree} follows immediately from the high moment estimate in Lemma \ref{lem highp1}, which will be proved in Section \ref{sec_pfp1}. 

\begin{lemma}\label{lem highp1}
Suppose the assumptions of Lemma \ref{lemma ptree} hold. Assume that
\be\label{initial_p}
T_{xy} \prec B_{xy}+\wt\Phi^2,\quad \forall \ x,y \in \Z_L^d,
\ee
for a deterministic parameter $\wt\Phi $ satisfying $0\le \wt\Phi \le W^{-\delta}$ for a constant $\delta>0$. Then for any fixed $p\in \N$, we have that
\be\label{locallawptree}
\E T_{xy} (z) ^p \prec  (B_{xy} +W^{-c} \wt\Phi^2)^p
\ee
for a constant $c>0$ depending only on $d$ and $c_0$ in \eqref{Lcondition1}. 
\end{lemma}

\begin{proof}[Proof of Lemma \ref{lemma ptree}]
By \eqref{Gmax}, we have that \eqref{initial_p} holds with $\wt\Phi=\wt\Phi_0:=W^{-d/2+\e_0}$. Then combining \eqref{locallawptree} with Markov's inequality, we obtain that
$$T_{xy} (z) \prec  B_{xy} +W^{-c} \wt\Phi_0^2 \ .$$
Hence \eqref{initial_p} holds with a smaller parameter $\wt\Phi=\wt\Phi_1:=W^{-c/2}\wt\Phi_0$. Applying Lemma \ref{lem highp1} and Markov's inequality again, we obtain that
$$T_{xy} (z) \prec  B_{xy} +W^{-c} \wt\Phi_1^2=B_{xy} +W^{-2c} \wt\Phi_0^2 \ .$$
Now for any fixed $D>0$, iterating the above argument for $\left\lceil D/c\right\rceil$ many times, we obtain that 
$$T_{xy} (z) \prec  B_{xy} +W^{-D} .$$
This concludes \eqref{pth T} as long as $D$ is large enough.
\end{proof}

\subsection{Estimates of doubly connected graphs}\label{sec bound double}

The proof of Lemma \ref{lem highp1} will use some important estimates on doubly connected graphs stated in this subsection. First,  inspired by the maximum bound \eqref{Gmax}, the weak averaged bound \eqref{initial Txy222} and the local law \eqref{locallaw1}, the following weak and strong norms were introduced in \cite{PartI_high}. 

\begin{definition}\label{Def PseudoG}
Given a $\Z_L^d\times \Z_L^d$ matrix $\cal A$ and some fixed $a,b>0$, we define its weak-$(a,b)$ norm as
$$\|\cal A\|_{w;(a,b)}:= W^{a d/2}\max_{x,y\in \Z_L^d} \left|\cal A_{xy}\right| + \sup_{ K \in [W, L/2]} \left( \frac{W}{K}\right)^b K^{ad/2} \max_{x, x_0\in \Z_L^d} \frac1{K^d}\sum_{y:| y - x_0 |\le K} \left(\left|\cal A_{xy}\right|+\left| \cal A_{yx}\right| \right),$$
and its strong-$(a,b)$ norm as
$$\|\cal A\|_{s;(a,b)}:= \max_{x,y\in \Z_L^d} \left( \frac{W}{\langle x-y\rangle}\right)^b \langle x-y\rangle^{ad/2} \left|\cal A_{xy}\right|.$$
\end{definition}

It is easy to check that $\|B\|_{s;(2,2)}\prec 1$,  $\|G(z)-m(z)I_N\|_{s;(1,1)}\prec 1$ if \eqref{locallaw1} holds, and $\|G(z)-m(z)I_N\|_{w;(1,2)}\prec W^{\e_0}$ if \eqref{Gmax} and \eqref{initial Txy222} hold. We now introduce a positive random variable $\Psi_{xy}(\tau, D)$ for a small constant $\tau>0$ and a large constant $D>0$, which was defined in \cite[Definition 3.4]{PartIII}:
\be\label{eq defPsi}
\Psi^2_{xy}\equiv \Psi^2_{xy}(\tau,D) :=W^{-D}+ \max\limits_{ \substack{|x_1-x| \le   W^{1+\tau}  \\ |y_1-y|\le  W^{1+\tau}}}s_{x_1y_1} + W^{-(2+2\tau)d}\sum_{ |x_1-x| \le  W^{1+\tau}}\sum_{ |y_1-y|\le  W^{1+\tau}} |G_{x_1y_1}|^2  .\ee
It is easy to check that  $\|\Psi(z)\|_{w;(1,2)} \prec \|G(z)-m(z)I_N\|_{w;(1,2)} + 1$ and  $\|\Psi(z)\|_{s;(1,1)} \prec \|G(z)-m(z)I_N\|_{s;(1,1)} + 1$ as long as $D$ is large enough. The motivation for introducing the $\Psi$ variable is as follows: given two atoms $x_1,x_2\in \Z_L^d$, suppose $y_1\ne y_2$ satisfy that  
\be\label{y12x12}|y_1-x_1|\le W^{1+\tau/2} ,\quad |y_2-x_2|\le W^{1+\tau/2}.\ee 
Then under the setting of Lemma \ref{lem G<T}, by applying \eqref{offG largedev} twice we obtain that 
\begin{align}
	|G_{y_1y_2}(z)|^2&\prec T_{y_1y_2}(z) =|m|^2 s_{y_1y_2}|G_{y_2y_2}(z)|^2 +|m|^2 \sum_{\al\ne y_2}s_{y_1 \al}|G_{y_2\al }(\overline z)|^2 \nonumber\\
	&\prec   s_{y_1y_2} +  \sum_{\al\ne y_2}s_{y_1\al} T_{ y_2 \al}(\overline z)  \le s_{y_1y_2} + \sum_{\al ,\beta} s_{y_1\al}s_{y_2\beta} |G_{\al\beta }(z)|^2 \nonumber\\
	&\le W^{-D}+s_{y_1y_2} + W^{-2d}\sum_{|\al-y_1|\le W^{1+\tau/2}}\sum_{|\beta-y_2|\le W^{1+\tau/2}}  |G_{\al\beta}(z)|^2  \le W^{2d\tau} \Psi_{x_1x_2}^2(\tau,D),\label{Gpsi}
\end{align}
where we also used \eqref{subpoly} and the identity $G_{xy}(z)=\overline{G_{yx}(\overline z)}$ in the derivation. In particular, if $y_1$ and $y_2$ are in the same molecules as $x_1$ and $x_2$, respectively, then \eqref{y12x12} holds, since otherwise the graph will be smaller than $W^{-D}$ for any fixed $D>0$ by \eqref{subpoly} and \eqref{S+xy}. Hence \eqref{Gpsi} shows that all the $G$ edges between two molecules containing atoms $x_1$ and $x_2$ can be bounded by the same variable $\Psi_{x_1x_2}$. This fact will be convenient for our proof.

In this paper, we only use weak or strong-$(a,b)$ norms with $a\le 2$. In this case, it is not hard to check that the strong-$(a,b)$ norm is strictly stronger than the weak-$(a,b)$ norm. By Definition \ref{Def PseudoG}, we immediately get the bounds
\be\label{eq PseudoG2} 
\max_{x,y\in \Z_L^d} \left|\cal A_{xy}\right| \le W^{-a d/2}\|\cal A\|_{w;(a,b)},
\ee
\be\label{eq PseudoG}
\max_{x, x_0\in \Z_L^d} \frac1{ K^{d}}\sum_{y:| y - x_0 |\le K} \left(\left|\cal A_{xy}\right|+\left| \cal A_{yx}\right| \right)\le   \frac{1}{W^b K^{ad/2-b}}  \|\cal A\|_{w;(a,b)},\quad \text{for all $  K \in [W, L/2]$,}
\ee
\be\label{eq unifG}
\left|\cal A_{xy}\right| \le    \frac{1}{W^b \langle x-y\rangle^{ad/2-b}} \|\cal A\|_{s;(a,b)}.
\ee
Using these estimates, we can easily prove the following bounds.

\begin{claim}\label{claim_basic}
Let $a$ and $b$ be two positive constants satisfying that  
	\be\label{a+b}
	a d/2 - b - 2\ge 0.
	\ee
Given any two matrices $\cal A^{(1)}$ and $\cal A^{(2)}$, we have that 
\be\label{keyobs3}
\sum_{x}\cal A^{(1)}_{x \beta} \cdot \prod_{i=1}^k B_{x y_i} \prec W^{-ad/2 } \Gamma(y_1,\cdots, y_k) \cdot \|\cal A^{(1)}\|_{w;(a,b)},
\ee
\be\label{keyobs2.2}
\sum_{x} \cal A^{(1)}_{x \al}\cal A^{(2)} _{x \beta} \cdot \prod_{i=1}^k B_{x y_i} \prec W^{-a d/2 }\Gamma(y_1,\cdots, y_k) \cdot \frac{1}{W^b\langle x-y\rangle^{ad/2-b}} \cdot \|\cal A^{(1)}\|_{s; (a,b)}\|\cal A^{(2)}\|_{s; (a,b)},
\ee
\be\label{keyobs2}
\sum_{x} \cal A^{(1)}_{x \al}\cal A^{(2)} _{x \beta} \cdot  \prod_{i=1}^k B_{x y_i} \prec W^{-a d/2 }\Gamma(y_1,\cdots, y_k)  {\cal A}_{\al\beta} \cdot \|\cal A^{(1)}\|_{w; (a,b)}\|\cal A^{(2)}\|_{w; (a,b)},
\ee
where $ {\cal A}_{\al\beta}$ is a positive variable satisfying $\|\cal A\|_{w;(a,b)}\le 1$ and $\Gamma(y_1,\cdots, y_k)$ is defined by 
\be\label{defn_Gamma}\Gamma(y_1,\cdots, y_k):= \sum_{i=1}^k \prod_{j\ne i}B_{y_{i}y_j}.\ee
\end{claim}
\begin{proof}
The proof of this claim is basic by using \eqref{eq PseudoG2}--\eqref{eq unifG}, so we omit the details. The reader can also refer to the proof of Claim 6.11 in \cite{PartI_high} for a formal proof. 
\end{proof}

The following lemma gives useful estimates on doubly connected graphs.
\begin{lemma}[Lemma 6.10 of \cite{PartI_high}]\label{no dot}
Suppose $d\ge 8$ and $\cal G$ is a doubly connected normal regular graph without external atoms. Pick any two atoms of $\cal G$ and fix their values $x , y\in \Z_L^d$. Then the resulting graph $\cal G_{xy}$ satisfies that 
\be\label{bound 2net1}
\left|\cal G_{xy}\right| \prec W^{ - \left(n_{xy}-3\right)d/2 } B_{xy} \cal A_{xy} \cdot \|G(z)-m(z)I_N\|_{w;(1,2)}^{n_{xy}-2} \ ,
\ee
and 
\be\label{bound 2net1 strong}
\left|\cal G_{xy}\right| \prec W^{ -  \left(n_{xy}-3\right)d/2} B_{xy}^{3/2}\cdot \|G(z)-m(z)I_N\|_{s;(1,1)}^{n_{xy}-2}\ ,
\ee
where  $n_{xy}:=\ord(\cal G_{xy})$ is the scaling order of $\cal G_{xy}$ and  $\cal A_{xy}$ is some positive variable satisfying $\|\cal A\|_{w;(1,2)}\prec 1$. If we fix an atom $x \in \cal G$, then the resulting graph  $\cal G_{x}$ satisfies that  
\be\label{bound 2net1 singlex}
\left|\cal G_{x}\right| \prec W^{ - \ord(\cal G_{x}) \cdot d/2 }  \cdot \|G(z)-m(z)I_N\|_{w;(1,2)}^{n_{xy}-2} .
\ee
The above bounds hold also for the graph $ {\cal G}^{{\rm abs}}$, which is obtained by replacing every component (including edges, weights and coefficients) in $\cal G$ with its absolute value and ignoring all $P$ or $Q$ labels (if any). 
\end{lemma}

Deterministic doubly connected graphs satisfy better bounds than Lemma \ref{no dot}, because all edges in their blue nets are now (labelled) $\dashed$ edges, whose strong-$(2,2)$ norms are bounded by $\OO_\prec(1)$.


\begin{lemma}[Corollary 6.12 of \cite{PartI_high}]\label{lem Rdouble}
	Suppose $d\ge 6$ and $\cal G$ is a deterministic doubly connected normal regular graph without external atoms. Pick any two atoms of $\cal G$ and fix their values $x , y\in \Z_L^d$. Then the resulting graph $\cal G_{xy}$ satisfies that
	\be\label{bound 2net1B}
	\left|\cal G_{xy}\right| \prec W^{ - \left(n_{xy} -4\right) d/2} B_{xy}^2 ,
	\ee
	where $n_{xy}:=\ord(\cal G_{xy})$ is the scaling order of $\cal G_{xy}$. This bound holds also for the graph $ {\cal G}^{{\rm abs}}$, which is obtained by replacing every component (including edges and coefficients) in $\cal G$ with its absolute value. 
\end{lemma}

\subsection{Proof of Lemma \ref{lem highp1}} \label{sec_pfp1}
In this subsection, we prove Lemma \ref{lem highp1} using the $n$-th order $T$-expansion.
 
\begin{proof}[Proof of Lemma \ref{lem highp1}]
We rename the indices $x$ and $y$ as $\fa$ and $\fb$, respectively. Moreover, following the notation in \eqref{rep_abc}, we represent $\fa$ and $\fb$ by $\otimes$ and $\oplus$ in graphs. Applying \eqref{mlevelTgdef} to a $T_{\fa\fb}$ in $\E T_{\fa\fb}^p$, we obtain that 
 \be\label{mlevelT n-1}
 \begin{split}
\E T_{\fa\fb}^p & = \E T_{\fa\fb}^{p-1}\left[m  \Theta_{\fa\fb} + m(\Theta \Sdelta^{(n)} \Theta)_{\fa\fb} \right] \overline G_{\fb\fb} + \E T_{\fa\fb}^{p-1}  (\PT^{(n)})_{\fa,\fb\fb}  \\
&+ \E T_{\fa\fb}^{p-1} (\AT^{(>n)})_{\fa,\fb\fb} + \E T_{\fa\fb}^{p-1}(\QT^{(n)})_{\fa,\fb\fb} + \E T_{\fa\fb}^{p-1} (\Err_{n,D})_{\fa,\fb\fb} . 
\end{split}
\ee
Using \eqref{thetaxy} and \eqref{intro_redagain}, we can bound the first term on the right-hand side as
\be\label{Holder1}
\E T_{\fa\fb}^{p-1}\left[m \Theta_{\fa\fb} +m(\Theta \Sdelta^{(n)} \Theta)_{\fa\fb} \right]\overline G_{\fb\fb} \prec B_{\fa\fb} \E T_{\fa\fb}^{p-1} .
\ee
Next we claim that 
\be\label{PTab}
 (\PT^{(n)})_{\fa,\fb\fb}  \prec W^{-d/2+\e_0} B_{\fa\fb}.
\ee
Let $\cal G_{\fa\fb}$ be a graph in $(\PT^{(n-1)})_{\fa\fb}$. By Definition \ref{defn genuni}, it has dotted edges connected with $\oplus$, has a (labelled) $\dashed$ edge connected with $\otimes$, is of scaling order $\ge 3$, and is doubly connected. Now merging $\oplus$ with the internal atoms that are connected to it through dotted edges, we can write $\cal G_{\fa\fb}$ as
$$\cal G_{\fa\fb}=\sum_x \Theta_{\fa x}(\cal G_0)_{x \fb},$$
where the $\Theta$ edge can also be replaced by a labelled $\dashed$ edge and $(\cal G_0)_{x\fb}$ is a graph satisfying the assumptions of Lemma \ref{no dot}. Then using \eqref{thetaxy} (or \eqref{BRB} if $\Theta_{\fa x}$ is replaced by a labelled $\dashed$ edge), \eqref{bound 2net1} and $\|G(z) - m(z)I_N\|_{w;(1,2)}\prec W^{\e_0}$ by the assumptions of Lemma \ref{lemma ptree}, we can bound that
\be\label{PTab1}
 |\cal G_{\fa\fb}|\prec W^{ (n_{x\fb} - 3) (-d/2+\e_0) + \e_0} \sum_{x} B_{\fa x}  B_{x\fb}\cal A_{x\fb} \prec W^{ (n_{x\fb} - 2) (-d/2+\e_0)}  B_{\fa  \fb} ,
 \ee
where $n_{x\fb}:=\ord((\cal G_0)_{x \fb})$, $\cal A_{x\fb}$ is a positive variable satisfying $\|\cal A\|_{w;(1,2)}\prec 1$, and in the last step we used  \eqref{keyobs3} with $(a,b)=(1,2)$. Using \eqref{PTab1} and the fact that $n_{x\fb}\ge 3$, we conclude \eqref{PTab}. Now with \eqref{PTab}, we can bound the second term on the right-hand side of \eqref{mlevelT n-1} as 
\be\label{Holder2}
\E T_{\fa\fb}^{p-1} (\PT^{(n)})_{\fa,\fb\fb}   \prec W^{-d/2+\e_0} B_{\fa\fb} \E T_{\fa\fb}^{p-1} .
\ee
It remains to bound the last three terms on the right-hand side of \eqref{mlevelT n-1}. This can be done with the following two claims.
\begin{claim}\label{lem ATab}
Under the setting of Lemma \ref{lem highp1}, we have that
\be\label{ATab}
 (\ATn)_{\fa,\fb\fb}  \prec  W^{(n-1)(-d/2+  \e_0) }  \frac{L^2}{W^2}(B_{\fa\fb} + \wt\Phi^2) ,
\ee
and
\be\label{Err_ab}
(\Err_{n,D})_{\fa,\fb\fb}  \prec  W^{D (-d/2 + \e_0)}  \frac{L^2}{W^2}(B_{\fa\fb} + \wt\Phi^2).
\ee
\end{claim}
\begin{claim}\label{lem QTab}
Under the setting of Lemma \ref{lem highp1}, we have that
\be\label{QTab}
\E T_{\fa\fb}^{p-1}(\QT^{(n)})_{\fa,\fb\fb} \prec \sum_{k=2}^{p}\left[W^{-d/4+\e_0/2} (B_{\fa\fb}+\wt\Phi^2)  \right]^k \E T_{\fa\fb}^{p-k} . 
\ee
\end{claim}
We postpone the proofs of these two claims until we complete the proof of Lemma \ref{lem highp1}. Using Claim \ref{lem ATab} and the condition \eqref{Lcondition1}, we get that
\be\label{Holder3}
\begin{split}
\E T_{\fa\fb}^{p-1} \left[(\ATn)_{\fa ,\fb\fb}+(\Err_{n,D})_{\fa,\fb\fb}\right]  & \prec W^{(n-1)(-d/2 + \e_0) }  \frac{L^2}{W^2}(B_{\fa\fb} + \wt\Phi^2) \E T_{\fa\fb}^{p-1} \\
&\le   W^{-c_0+(n-1)\e_0} (B_{\fa\fb} + \wt\Phi^2)  \E T_{\fa\fb}^{p-1} .
\end{split}
\ee
Combining \eqref{Holder1}, \eqref{Holder2}, \eqref{QTab} and \eqref{Holder3}, and using H{\"o}lder's inequality $ \E T_{\fa\fb}^{p-k}\le \left(\E T_{\fa\fb}^{p}\right)^{\frac{p-k}{p}}$, we can bound \eqref{mlevelT n-1} as
\be\label{Holder3.5}\E T_{\fa\fb}^p  \prec ( B_{\fa\fb} + W^{-c}\wt\Phi^2) \left(\E T_{\fa\fb}^{p}\right)^{\frac{p-1}{p}} +\sum_{k=2}^{p}\left[W^{-c} (B_{\fa\fb}+\wt\Phi^2)  \right]^k \left(\E T_{\fa\fb}^{p}\right)^{\frac{p-k}{p}}, \ee
where $c:=\min(c_0-(n-1)\e_0, d/4-\e_0/2)$ is a positive constant as long as $\e_0$ is sufficiently small. Applying Young's inequality to every term on the right-hand side, we get that
\be\label{Holder3.6} \E T_{\fa\fb}^p  \prec W^\e ( B_{\fa\fb} + W^{-c}\wt\Phi^2)^p + W^{-\e}\E T_{\fa\fb}^{p}  \quad \Rightarrow \quad \E T_{\fa\fb}^p  \prec W^\e ( B_{\fa\fb} + W^{-c}\wt\Phi^2)^p , \ee
for any constant $\e>0$. Since $\e$ is arbitrary, we conclude \eqref{locallawptree}.
\end{proof}

For convenience of our proof, we introduce the following notation of $\Omega$ variables, which satisfy the same properties needed for the proof as off-diagonal $G$ edges. 
\begin{notation}\label{Omega1}
We will use $\Omega$ to denote matrices of non-negative random variables satisfying $\|\Omega\|_{w;(1,2)}\prec W^{\e_0}$ and 
\be\label{Omega2}
\Omega_{x y} \prec B_{x y}^{1/2}+\wt\Phi .
\ee
In the setting of Lemma \ref{lem highp1}, $ |G_{xy}-m\delta_{xy}|$ and $ \Psi_{xy}$ in \eqref{eq defPsi} are both $\Omega$ variables.
\end{notation}


Now we claim the following useful fact: if $\Omega^{(1)}$ and $\Omega^{(2)}$ satisfy Notation \ref{Omega1}, then 
\be\label{keyobs4}
\begin{split}
& \Omega_{\al\beta}:=\frac{W^{d/2-\e_0}}{\Gamma(y_1,\cdots, y_k)}\sum_{x}  \Omega^{(1)}_{x  \al}\Omega ^{(2)}_{x  \beta}\cdot  \prod_{i=1}^k B_{x y_i} \ \ \text{also sastifies Notation \ref{Omega1}.}
\end{split}
\ee
First, notice that by \eqref{keyobs2}, we have $ \|\Omega\|_{w;(1,2)}\prec W^{\e_0}$. It remains to prove that $\Omega$ satisfies \eqref{Omega2}. Consider two regions $\cal I_1:=\{x : \langle x -\al\rangle \ge \langle x -\beta\rangle \}$ and  $\cal I_2:=\{x : \langle x - \al\rangle < \langle x -\beta\rangle  \}$. On $\cal I_1$, we have $\langle x -\al\rangle \ge \langle \al -\beta\rangle/2$, which gives that
$$\Omega^{(1)} _{x \al}   \prec B_{x \al}^{1/2}+\wt\Phi \lesssim B_{\al \beta}^{1/2}+\wt\Phi.$$
Together with \eqref{keyobs3}, it gives that
\begin{align*}
\sum_{x \in \cal I_1}  \Omega^{(1)}_{x \al}\Omega^{(2)} _{x \beta} \cdot  \prod_{i=1}^k B_{x y_i}   \prec (B_{\al \beta}^{1/2}+\wt\Phi)\sum_{x \in \cal I_1} \Omega^{(2)} _{x \beta} \cdot  \prod_{i=1}^k B_{x y_i}  \prec (B_{\al \beta}^{1/2}+\wt\Phi) \cdot W^{- d/2+\e_0}\Gamma(y_1,\cdots, y_k) .
\end{align*}
We have a similar bound for the sum over $\cal I_2$. This concludes \eqref{keyobs4}. In the proof, it is more convenient to use the  following form of \eqref{keyobs4}:
\be\label{keyobs5}
\sum_{x} \Omega^{(1)}_{x \al}\Omega^{(2)} _{x \beta}\cdot  \prod_{i=1}^k B_{x y_i}  = W^{- d/2+\e_0}\Gamma(y_1,\cdots, y_k)  \Omega_{\al\beta},
\ee
for some $\Omega$ satisfying Notation \ref{Omega1}.


\begin{proof}[Proof of Claim \ref{lem ATab}]
We only need to consider the graphs in $(\AT^{(>n)})_{\fa\fb}$ that are not $\oplus$-recollision graphs, because $\oplus$-recollision graphs have been shown to satisfy \eqref{PTab1}. Any such graph $\cal G_{\fa\fb}$ can be written into 
\be\label{Ho123}
\cal G_{\fa\fb}= \sum_{x,y_1, y_2}{\Theta}_{ \fa x} (\cal G_0)_{x,y_1y_2} G_{y_1\fb}\overline G_{y_2\fb},\ \ \cal G_{{\fa\fb}}= \sum_{x,y_1}{\Theta}_{\fa x} (\cal G_0)_{x y} |G_{y\fb}|^2 ,\ \  \cal G_{\fa\fb}= \sum_{x,y}{\Theta}_{\fa x} (\cal G_0)_{xy}\Theta_{y\fb},\ee
or some forms obtained by setting $x$ to be equal to $y$, $y_1$ or $y_2$ and by replacing the $\Theta$ edges with labelled $\dashed$ edges. Here the graphs $\cal G_0$ are doubly connected graphs without external atoms. Using \eqref{bound 2net1}, we can bound the second and third terms of \eqref{Ho123} as 
\begin{align} 
|\cal G_{\fa\fb}| &\prec  W^{(n_{xy}-3)(-d/2+\e_0)+\e_0}\sum_{x,y} B_{\fa x} B_{xy} \cal A_{xy}\left(B_{y\fb}+\wt\Phi^2\right) \lesssim W^{(n_{xy}-2)(-d/2+\e_0)}\sum_{y} B_{\fa y}\left(B_{y\fb}+\wt\Phi^2\right) \nonumber\\
&\lesssim W^{(n_{xy}-2)(-d/2+\e_0)}\left( \frac{1}{W^4\langle \fa-\fb\rangle^{d-4}} + \frac{L^2}{W^2}\wt\Phi^2\right)\lesssim W^{(n -1)(-d/2+\e_0)}\frac{L^2}{W^2}\left( B_{\fa\fb} + \wt\Phi^2\right), \label{calc1}
\end{align}
where $n_{xy}:=\ord( (\cal G_0)_{xy} )$. Here in the second and third steps we used \eqref{keyobs3} and the simple bounds
$$ 
\sum_{y} B_{\fa y} B_{y\fb}\lesssim \frac{1}{W^4\langle \fa-\fb\rangle^{d-4}}, \quad \sum_{y} B_{\fa y}\lesssim \frac{L^2}{W^2},$$
and in the fourth step we used $\ord( (\cal G_0)_{xy} )= \ord( \cal G_{\fa\fb} )\ge n+1$ and $\langle \fa-\fb\rangle\lesssim L$.
It remains to bound the first expression in \eqref{Ho123}. (Some variants of \eqref{Ho123} obtained by setting $x$ to be equal to $y$, $y_1$ or $y_2$ and by replacing the $\Theta$ edges with labelled $\dashed$ edges can be estimated in similar ways, so we omit the details.)

By \eqref{subpoly}, \eqref{thetaxy}, \eqref{S+xy} and \eqref{BRB}, we have the following maximum bounds on deterministic edges:
\be\label{Ssmall}\begin{split}
\max_{x,y}s_{xy} = & \OO   (W^{-d}), \quad  \max_{x,y}|S^{\pm}_{xy}|  = \OO(W^{-d}),\quad \max_{x,y}\Theta_{xy} \prec W^{-d}, \\
& \max_{x,y}\left|(\Theta \Sele_{2k_1}\Theta  \Sele_{2k_2}\Theta \cdots \Theta  \Sele_{2k_l}\Theta)_{xy}\right|  \prec W^{-kd /2 } ,
\end{split}
\ee
where $k:=\sum_{i=1}^l 2 k_i -2(l-1)$. For simplicity of presentation, we will use $\al\sim_{\cal M} \beta$ to mean that ``atoms $\al$ and $\beta$ belong to the same molecule". Suppose there are $\ell$ internal molecules $\cal M_i$, $1\le i \le \ell$, in $\cal G_0$. We choose one atom in each $\cal M_i$, say $x_i$, as a representative. 
For definiteness, we assume that $x$, $y_1$ and $y_2$ belong to different molecules. Otherwise, the proof will be easier and we omit the details. Throughout the following proof, we fix a small constant $\tau>0$ and a large constant $D>0$. For any $\al \sim_{\cal M}x_i$, it suffices to assume that  
\be\label{yixi}
|\al-x_i|\le W^{1+\tau/2},
\ee
because otherwise the graph is smaller than $W^{-D}$. 
Then under the assumption \eqref{yixi}, for $\al_i\sim_{\cal M} x_i$ and $\al_j\sim_{\cal M} x_j$, by \eqref{thetaxy}, \eqref{Gpsi} and \eqref{BRB} we have that
\be\label{intermole1}|G_{\al_i \al_j}|\prec W^{d\tau} \Psi_{x_i x_j}(\tau, D),\quad \Theta_{\al_i \al_j}\prec B_{y_i y_j} \lesssim W^{(d-2)\tau/2} B_{x_i x_j}.\ee
\be\label{intermole2} \left|(\Theta \Sele_{2k_1}\Theta  \Sele_{2k_2}\Theta \cdots \Theta  \Sele_{2k_l}\Theta)_{\al_i \al_j}\right| \prec W^{-(k-2)d /2 + (d-2)\tau/2} B_{x_i x_j}. \ee 
The above estimates show that we can bound edges between different molecules with $\Psi$ or $B$ variables that only contain the representative atoms $x_i$ in their indices.

We first bound the edges between different molecules. Due to the doubly connected property of $\cal G_0$, we can pick two spanning trees of the black and blue nets, which we refer to as the \emph{black tree} and \emph{blue tree}, respectively. We bound the \emph{internal edges} that are not in the two trees by their maximum norms: 
\begin{itemize}
	\item[(i)] we bound each solid edge that is not in the blue tree by $\OO_\prec(W^{-d/2+\e_0})$; 
	
	\item[(ii)] we bound each $\dashed$ edge that is not in the black and blue trees by $\OO_\prec(W^{-d})$; 
	
	\item[(iii)] we bound each labelled $\dashed$ edge that is not in the black and blue trees by $\OO_\prec(W^{-kd/2})$, where $k$ is the scaling order of this edge. 
\end{itemize}
We bound the external edges and edges of the two trees as follows:
\begin{itemize}
	\item[(iv)] the blue solid and $\dashed$ edges in the two trees and the external edges ${\Theta}_{ \fa x}$, $G_{y_1\fb}$ and $\overline G_{y_2\fb}$ are bounded using \eqref{intermole1} and \eqref{intermole2}.
\end{itemize}
In this way, we can bound that 
\be\label{reduce Gaux0}
|\cal G_{\fa\fb}| \prec W^{-n_1 d/2 + n_2 \tau + n_3\e_0}\sum_{x_1, \cdots, x_\ell} \Gamma_{global}(x_1, \cdots, x_\ell) \prod_{i=1}^\ell |\cal G_{x_i}^{(i)}| \ ,
\ee
where $W^{-n_1 d/2+ n_2\tau + n_3\e_0}$ is a factor coming from the above items (i)--(iv); $\Gamma_{global}$ is a product of blue solid edges that represent $\Psi$ entries, and (black or blue) double-line edges that represent $B$ entries; every \smash{$\cal G_{x_i}^{(i)}$} is the subgraph inside the molecule $\cal M_i$, which has $x_i$ as an external atom. The local structure \smash{$\cal G_{x_i}^{(i)}$} can be  bounded as follows: 
\begin{itemize}
	\item we bound each waved or $\dashed$ edge by $\OO_\prec(W^{-d })$ using the maximum bounds in \eqref{Ssmall}; 
	\item we bound each labelled $\dashed$ edge by $\OO_\prec (W^{-kd/2 })$, where $k$ is its scaling order; 
	\item we bound each off-diagonal $G$ edge and light $G$ weight by $\OO_\prec(W^{-d/2+\e_0})$; 
	\item each summation over an internal atom in $\cal M_i \setminus \{x_i\}$ gives a $W^{(1+\tau/2)d}$ factor due to \eqref{yixi}. 
\end{itemize}
Thus with the definition of scaling order in \eqref{eq_deforderrandom3}, we get that
\begin{align}\label{internal struc1}
	|\cal G_{x_i}^{(i)}|\prec W^{-\ord(\cal G^{(i)}_{x_i}) \cdot d/2 + l_i \e_0 + k_{i} d \tau /2 }, 
\end{align} 
where $l_i$ is the number of off-diagonal $G$ edges and light weights in $\cal G_{x_i}^{(i)}$, and $k_i$ is the number of internal atoms. Finally, for convenience of proof, we bound every $\dashed$ edge in the \emph{blue} (but not black) tree of $\Gamma_{global}(x_1, \cdots, x_\ell)$ as
\be\label{bound thetaweak} B_{x_i x_j}\le W^{-d/2}B_{x_i x_j} ^{1/2}.\ee
(This is merely because we can express both the $\Psi$ and $B^{1/2}$ entries using the Notation \ref{Omega1}.) Plugging \eqref{internal struc1} and \eqref{bound thetaweak} into \eqref{reduce Gaux0}, we obtain that
\be\label{bound G aux}
\cal G_{\fa\fb}\prec W^{ (n_{\fa\fb}-\ell - 1) ( -d/2 +  \e_0) + n_4\tau} (\cal G_{\fa\fb})_{aux},  
\ee
where $n_{\fa\fb}:= \ord( \cal G_{\fa\fb} ),$ $n_4:=n_2+\sum_{i=1}^\ell k_i d/2$, and $n_{\fa\fb}-\ell -1$ in the exponent can be obtained by counting the number of $W^{-d/2}$ factors from the above argument. Here we also used the following fact: by property (ii) of Definition \ref{defnlvl0}, the number of internal atoms in $\cal G_{\fa\fb}$ is at most the number of waved and $\dashed$ edges, and hence the number of $W^{\e_0}$ factors coming from off-diagonal $G$ edges and light weights is less than or equal to the number of $W^{-d/2}$ factors. $(\cal G_{\fa\fb})_{aux}$ is an auxiliary graph satisfying the following properties: 
\begin{enumerate}
	\item $(\cal G_{\fa\fb})_{aux}$ has two external atoms $\fa$ and $\fb$, and $\ell$ internal atoms $x_i$, $1\le i \le \ell$, which are the representative atoms of the internal molecules of $\cal G_{\fa\fb}$; 
	
	\item every double-line edge between atoms $x$ and $y$ represents a $B_{xy}$ factor, and every blue solid edge between atoms $x$ and $y$ represents a 
	variable satisfying the Notation \ref{Omega1};
	
	\item there is a spanning tree of black double-line edges, and $\otimes$ is connected to an internal atom through a double-line edge (which corresponds to the edge $\Theta_{\fa x}$ in $\cal G_{\fa\fb}$); 
	
	\item there is a spanning tree of blue solid edges, and $\oplus$ is connected to internal atoms through two blue solid edges (which correspond to the edges $G_{y_1\fb}$ and $\overline G_{y_2\fb}$ in $\cal G_{\fa\fb}$).
\end{enumerate}

Now we choose an atom connected to $\oplus$ as the root of the blue spanning tree. Then we sum over the internal atoms in $(\cal G_{\fa\fb})_{aux}$ from the leaves to this root. Without loss of generality, assume that we sum over the internal atoms according to the order $\sum_{x_\ell}\cdots \sum_{x_2}\sum_{x_1}$, such that atom $x_\ell$ is the root. For simplicity of notations, we denote all blue solid edges appearing in the proof by $\Omega$, including the $\Psi$ and $B^{1/2}$ edges in $(\cal G_{\fa\fb})_{aux}$ and new edges coming from applications of \eqref{keyobs4}. All these $\Omega$ variables satisfy the Notation \ref{Omega1}, and their exact forms may change from one line to another. 


We sum over the internal atoms one by one. For the sum over $x_1$, using \eqref{keyobs5} (if $x_1$ is connected with $\oplus$) or \eqref{keyobs3} (if  $x_1$ is not connected with $\oplus$), we can bound $ (\cal G_{\fa\fb})_{aux}$ as
\be\label{reduceaux add}
(\cal G_{\fa\fb})_{aux} \prec W^{-d/2+\e_0}\sum_{k=1}^{n_1} (\cal G^{(1)}_{\fa\fb,k})_{aux},\ee
where $(\cal G^{(1)}_{\fa\fb, k})_{aux}$ are auxiliary graphs obtained by replacing the edges connected to $x_1$ with the graphs on the right-hand side of \eqref{keyobs5} or \eqref{keyobs3}, and $n_1$ is the number of atoms connected to $x_1$ through double-line edges. More precisely, after summing over $x_1$, we get rid of the blue solid edge connected with $x_1$. If $\oplus$ and $x_1$ are connected through a blue solid edge in $(\cal G_{\fa\fb})_{aux} $, then in each new graph $\oplus$ is connected to the parent of $x_1$ on the blue tree through a blue solid edge. Moreover, suppose $w_{1},\cdots,w_{n_1}$ are the atoms connected to $x_1$ through double-line edges. Then corresponding to the $n_1$ terms in $\Gamma(w_{1},\cdots, w_{n_1})$, the atoms $w_{1},\cdots,w_{n_1}$ are now connected to one of them, say $w_{k}$, through double-line edges in new graphs, which leads to the $n_1$ different choices in \eqref{reduceaux add}. 
Now we observe that every new graph \smash{$(\cal G^{(1)}_{\fa\fb,k})_{aux}$} is doubly connected in the following sense: there is a black spanning tree connecting all the internal atoms and $\otimes$, and a disjoint blue spanning tree connecting all the internal atoms and $\oplus$. Furthermore, there are at least two blue solid edges connected with $\oplus$ in \smash{$(\cal G^{(1)}_{\fa\fb,k})_{aux}$}. 

Then we sum over the atoms $x_2, \cdots , x_{\ell-1}$ one by one using \eqref{keyobs5} or \eqref{keyobs3}. At each step, we gain a $W^{-d/2+\e_0}$ factor and reduce a graph into a sum of several new graphs, each of which has one fewer atom, a black spanning tree connecting all the internal atoms and $\otimes$, a blue spanning tree connecting all the internal atoms and $\oplus$, and two blue solid edges connected with $\oplus$. In the following figure, we give an example of the graph reduction process by summing over the internal atoms $x_1,x_2, x_3$. 
\begin{center}
	\includegraphics[width=10cm]{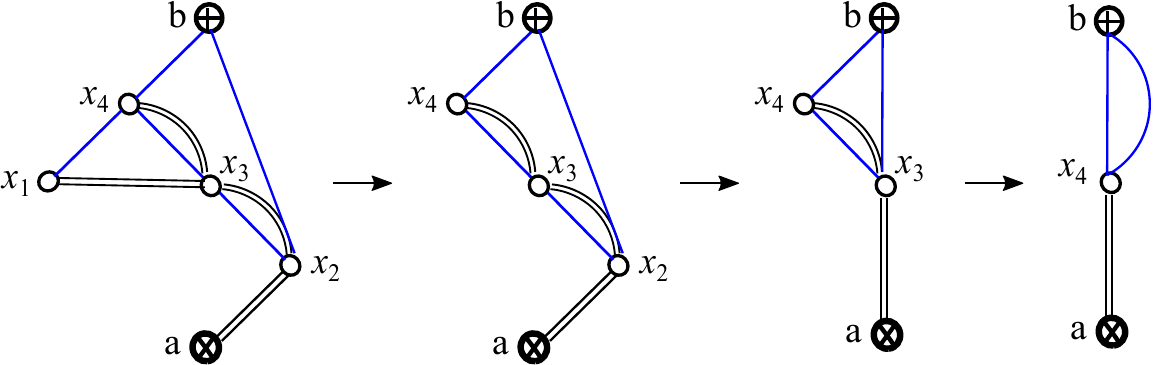}
\end{center}
Finally, we get a graph with only one internal atom $x_l$, one double-line edge between $\otimes$ and $x_l$, and two blue solid edges between $\oplus$ and $x_l$. Then using \eqref{Omega2}, we can bound \eqref{bound G aux} as 
\begin{align*}
\left|\cal G_{\fa\fb}\right|\prec W^{(n_{\fa\fb} -2)(-d/2+\e_0) + n_4\tau}\sum_{x_{\ell}}B_{\fa x_{\ell}}\left(B_{ x_{\ell}\fb} +\wt\Phi^2\right) \lesssim W^{(n-1)(-d/2+\e_0)+ n_4\tau}\frac{L^2}{W^2}\left( B_{\fa\fb} + \wt\Phi^2\right),
\end{align*}
where we used $n_{\fa\fb} \ge n +1$ and a similar argument as in \eqref{calc1}. Since $\tau$ is arbitrary, we can bound the first case in \eqref{Ho123} by the right-hand side of \eqref{ATab}. This concludes the proof.
\end{proof}

\begin{proof} [Proof of Claim \ref{lem QTab}]
We only consider the graphs in $(\QT^{(n)})_{\fa,\fb\fb}$ that are not $\oplus$-recollision graphs, because $\oplus$-recollision graphs are much easier to deal with and we omit the details. 
Let $\cal G_{\fa\fb}$ be such a $Q$-graph. We need to bound the graph \smash{$\E T_{\fa\fb}^{p-1}\cal G_{\fa\fb}$}. In the following proof, we call the $2(p-1)$ solid edges in $T_{\fa\fb}^{p-1}$ \emph{external solid edges}, and we regard the atoms $x$ in $T_{\fa\fb} =|m|^2\sum_x s_{\fa x}|G_{x\fb}|^2$ as an atom in the external $\otimes$ molecule. 

We perform the $Q$-expansions in Section \ref{sec defnQ} to $T_{\fa\fb}^{p-1}\cal G_{\fa\fb}$ and get that
\be\label{Holder5}
\E T_{\fa\fb}^{p-1}\cal G_{\fa\fb}=\E \sum_{k=2}^{p} T_{\fa\fb}^{p-k} \sum_{\omega} (\cal G^{(k)}_\omega)_{\fa\fb}  + \E \sum_{k=2}^{p} T_{\fa\fb}^{p-k} \sum_{\zeta} (\cal G^{(k)}_{err,\zeta})_{\fa\fb}.\ee
Every $(\cal G^{(k)}_\omega)_{\fa\fb}$ is a normal regular graph satisfying the following properties.
\begin{enumerate}
\item[(1)] $(\cal G^{(k)}_\omega)_{\fa\fb}$ contains no $P$/$Q$ labels.
\item[(2)] It 
is doubly connected.
\item[(3)] $\otimes$ is connected to internal molecules through a (labelled) $\dashed$ edge.
\item[(4)] Let the atom in the $Q$-label of $\cal G_{\fa\fb}$ be $\al$. 
There are $r$, $k-1\le r \le 2(k-1)$, external edges (coming from $(k-1)$ many $T_{\fa\fb}$ variables) that become paths of two external solid edges in the molecular graph from the $\otimes$ molecule to the molecule containing $\al$ and then to the $\oplus$ molecule.

\item[(5)] Besides the external solid edges in (4), the internal molecules are also connected to $\oplus$ through a $\dashed$ edge, or two blue and red solid edges (which correspond to the two solid edges connected with $\oplus$ in $\cal G_{\fa\fb}$). In the latter case, we call the two solid edges \emph{special solid edges}. 
\end{enumerate}
Every graph $(\cal G^{(k)}_{err,\zeta})_{\fa\fb}$ is of scaling order $> 2k+D$ and satisfies the above properties (2)--(5) (but it may contain $P$/$Q$ labels). For convenience of presentation, whenever we refer to external solid edges in the following proof, they do not include the special solid edges in (5). For the rest of the proof, we will focus on \smash{$(\cal G^{(k)}_\omega)_{\fa\fb} $}, while the graphs \smash{$(\cal G^{(k)}_{err,\zeta})_{\fa\fb}$} can be bounded in the same way. Moreover, we only consider the hardest case where (a) there are no dotted edges connected with $\oplus$ or $\otimes$ in \smash{$(\cal G^{(k)}_\omega)_{\fa\fb}$}; (b) in the above item (5), $\oplus$ is connected to the internal molecules through a blue solid edge and a red solid edge. Other cases where (a) or (b) does not hold 
are all easier to bound. 


We will follow similar arguments as in the above proof of Claim \ref{lem ATab}. More precisely, given any graph $(\cal G^{(k)}_\omega)_{\fa\fb}$, we first pick two spanning trees of the black and blue nets. Then we bound the internal edges that are not in the two trees by their maximum norms, and bound the local structures inside all internal molecules by \eqref{internal struc1}. Furthermore, we bound the external edges, special solid edges and the internal edges that are in the two trees using \eqref{intermole1}, \eqref{intermole2} and \eqref{bound thetaweak}. Again we can reduce the problem to bounding an auxiliary graph  $(\cal G_{\fa\fb})_{aux}$ satisfying the following properties:
 \begin{enumerate}
\item $(\cal G^{(k)}_\omega)_{\fa\fb}$ has two external atoms $\fa$ and $\fb$, and $\ell$ internal atoms $x_i$, $1\le i \le \ell$, which are the representative atoms of the internal molecules; 

\item each double-line edge between atoms $x$ and $y$ represents a $B_{x y}$ factor, and each blue solid edge between atoms $x$ and $y$ represents a variable satisfying the Notation \ref{Omega1};

\item each external or special solid edge between two molecules in $(\cal G^{(k)}_\omega)_{\fa\fb}$ is replaced by an external blue solid edge between the representative atoms of these molecules in $(\cal G_{\fa\fb})_{aux}$; 

\item the maximal subgraph of $(\cal G_{\fa\fb})_{aux}$ induced on the internal atoms has a spanning tree of black double-line edges and a spanning tree of blue solid edges;

\item in $(\cal G_{\fa\fb})_{aux}$, $\otimes$ is connected to the internal atoms through a black double-line edge, and $\oplus$ is connected to the internal atoms through two special blue solid edges;

\item in $(\cal G_{\fa\fb})_{aux}$, there are $2(k-1)-r$ blue solid edges between $\fa$ and $\fb$, $r$ blue solid edges between $\fa$ and $x_{\ell}$, and $r$ blue solid edges between $\fb$ and $x_{\ell}$, where $x_{\ell}$ is the representative atom of the molecule containing $\al$.

\end{enumerate}
Since $(\cal G^{(k)}_\omega)_{\fa\fb}  \prec ({\cal G}_{\fa\fb})_{aux}$, we only need to estimate $({\cal G}_{\fa\fb})_{aux}$. We choose the atom $x_{\ell}$, the representative atom of the molecule containing $\al$, as the root of the blue spanning tree. Then we sum over the internal atoms in $(\cal G_{\fa\fb})_{aux}$ from leaves to this root. Without loss of generality, assume that we sum over the internal atoms according to the order $\sum_{x_\ell}\cdots \sum_{x_2}\sum_{x_1}$. For simplicity of notations, we denote all blue solid edges appearing in the proof by $\Omega$, including the edges in the original graph $(\cal G_{\fa\fb})_{aux}$ and new edges coming from applications of \eqref{keyobs4}. All these $\Omega$ variables satisfy the Notation \ref{Omega1}, and their exact forms may change from one line to another. 
Suppose the two special blue solid edges in (v) connect $\oplus$ to atoms $x_i$ and $x_j$, respectively. We have the following two cases.

\medskip
\noindent{\bf Case 1:} Suppose the unique path on the blue spanning tree from $x_i$ to $x_j$ passes through $x_\ell$. Then, similar to the proof of Claim \ref{lem ATab}, we sum over the internal atoms one by one and reduce the graphs using \eqref{keyobs3} and \eqref{keyobs5}. Finally, we get the following graph with $2(k-1)-r$ blue solid edges between $\otimes$ and $\oplus$, and only one internal atom $x_{\ell}$, which is connected with $\otimes$ through a black double-line edge and $r$ external solid edges, and with $\oplus$ through two special solid edges and $r$ external solid edges. (To have a clearer picture, we draw two copies of the $\otimes$ and $\oplus$ atoms in the graph.)
\begin{center}
\includegraphics[width=2.7cm]{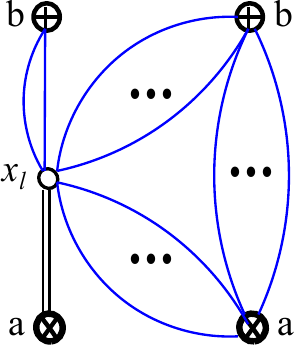}
\end{center}
Denote this graph by $(\wt{\cal G}_{\fa\fb})_{aux}$. To estimate $(\wt{\cal G}_{\fa\fb})_{aux}$, we consider two regions $\cal I_1:=\{x_\ell : \langle x_\ell  -\fa\rangle \ge \langle x_\ell -\fb\rangle \}$ and  $\cal I_2:=\{x : \langle x_\ell - \fa\rangle<\langle x_\ell -\fb\rangle  \}$. Then we can bound that 
\begin{align}
(\wt{\cal G}_{\fa\fb})_{aux}  &\le \Big(\sum_{x_\ell \in \cal I_1}+\sum_{x_\ell \in \cal I_2}\Big) \left( B_{\fa x_\ell} |\Omega_{x_\ell \fb}|^2\right)\left( |\Omega_{\fa x_\ell}|^r|\Omega_{ x_\ell \fb}|^r\right) | \Omega_{\fa\fb}|^{2k-2-r} \nonumber\\ 
&\prec \Big[\sum_{x_\ell \in \cal I_1} B_{\fa x_\ell}\left( B_{x_\ell \fb}+\wt\Phi^2\right)\left|\Omega_{ x_\ell \fb}\right|^r + \sum_{x_\ell \in \cal I_2}B_{\fa x_\ell}\left( B_{x_\ell \fb}+\wt\Phi^2\right)\left| \Omega_{\fa x_\ell}\right|^r \Big] \left(B^{1/2}_{\fa\fb}+\wt\Phi\right)^{2k-2},\label{Gaux1}
\end{align}
where in the second step we used \eqref{Omega2}, $\langle x_\ell -\fa\rangle \ge \langle \fa -\fb\rangle/2$ on $\cal I_1$ and $\langle x_\ell -\fb\rangle \ge \langle \fa -\fb\rangle/2$ on $\cal I_2$. We bound the sum over $\cal I_1$ as 
\be\label{Gaux2}
\begin{split}\sum_{x_\ell \in \cal I_1} B_{\fa x_\ell}\left( B_{x_\ell \fb}+\wt\Phi^2\right)\left|\Omega_{ x_\ell \fb}\right|^r &\prec W^{(r-1)(-d/2+\e_0)}\sum_{x_\ell \in \cal I_1} B_{\fa x_\ell}\left( B_{x_\ell \fb}+\wt\Phi^2\right) \Omega_{ x_\ell \fb}\\
& \prec W^{r(-d/2+\e_0)} \left( B_{\fa \fb}+\wt\Phi^2\right),\end{split}\ee
where we used \eqref{eq PseudoG2} in the first step, and \eqref{keyobs3} in the second step with $(a,b)=(1,2)$ and $\|\Omega\|_{w;(1,2)}\prec W^{\e_0}$. We have a similar bound for the sum over $\cal I_2$. Hence we get that
\begin{align}\label{eq QTcase1}
 (\wt{\cal G}_{\fa\fb})_{aux}  &\prec W^{r(-d/2+\e_0)} \left( B_{\fa \fb}+\wt\Phi^2\right)\left(B^{1/2}_{\fa\fb}+\wt\Phi\right)^{2k-2}\lesssim W^{r(-d/2+\e_0)} \left( B_{\fa \fb}+\wt\Phi^2\right)^{k}.
\end{align}

\noindent{\bf Case 2:} Suppose the unique path on the blue spanning tree from $x_i$ to $x_j$ does not pass through $x_\ell$. Then, similar to the proof of Claim \ref{lem ATab}, we sum over the internal atoms one by one and reduce the graphs using \eqref{keyobs3} and \eqref{keyobs5}. At certain step, we will get a graph where the two special blue solid edges are connected with the same internal atom, say $x_i$. Denote this graph by \smash{$(\wt{\cal G}_{\fa\fb})_{aux}$}. Now applying \eqref{Omega2} to the two special edges, we can bound $(\wt{\cal G}_{\fa\fb})_{aux}$ by 
\be\label{eq QTcase2} (\wt{\cal G}_{\fa\fb})_{aux} \prec (\wt{\cal G}_{\fa\fb}^{(1)})_{aux}+\wt\Phi^2 (\wt{\cal G}^{(2)}_{\fa\fb})_{aux},\ee
where $(\wt{\cal G}_{\fa\fb}^{(1)})_{aux}$ is obtained by replacing the two special edges with a double-line edge between $x_i$ and $\oplus$, and {$(\wt{\cal G}_{\fa\fb}^{(2)})_{aux}$} is obtained by removing the two special edges. In the following figure, we draw an example of graphs {$(\wt{\cal G}_{\fa\fb})_{aux}$, $(\wt{\cal G}_{\fa\fb}^{(1)})_{aux}$ and $(\wt{\cal G}_{\fa\fb}^{(2)})_{aux}$}.
\begin{equation} \nonumber 
\parbox[c]{0.9\linewidth}{\center
\includegraphics[width=14cm]{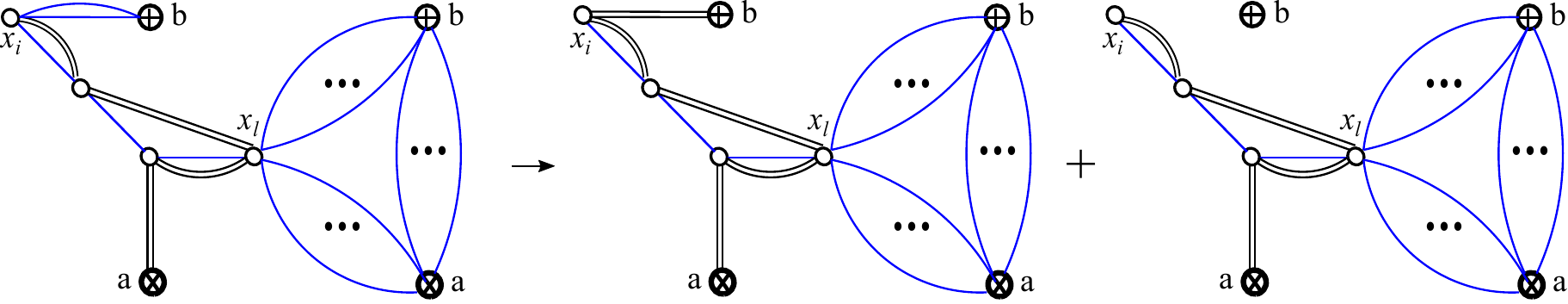}}
\end{equation}
There is a $\wt\Phi^2$ factor associated with the third graph $(\wt{\cal G}_{\fa\fb}^{(2)})_{aux}$, but we did not draw it. 

Similar to the proof of Claim \ref{lem ATab}, we sum over the internal atoms of \smash{$(\wt{\cal G}_{\fa\fb}^{(2)})_{aux}$} from leaves to the root $x_\ell$ on the blue tree and reduce the graphs using \eqref{keyobs3} at each step. 
We will finally get the following graph with only one internal atom $x_{\ell}$:
\be\nonumber
\parbox[c]{0.3\linewidth}{\center
\includegraphics[width=2.7cm]{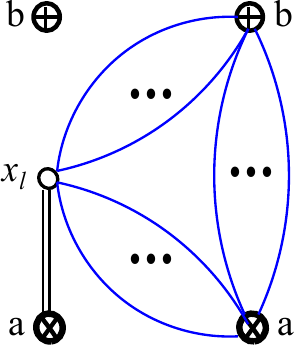}}
\ee
Denote this graph by $(\wh{\cal G}_{\fa\fb}^{(2)})_{aux}$. Using similar arguments as in \eqref{Gaux1} and \eqref{Gaux2}, we can bound that 
\begin{align}
(\wt{\cal G}^{(2)}_{\fa\fb})_{aux} \prec (\wh{\cal G}^{(2)}_{\fa\fb})_{aux} & \prec \Big(\sum_{x_\ell \in \cal I_1} B_{\fa x_\ell} \left|\Omega_{ x_\ell \fb}\right|^r + \sum_{x_\ell \in \cal I_2}B_{\fa x_\ell} \left| \Omega_{\fa x_\ell}\right|^r \Big) \left(B^{1/2}_{\fa\fb}+\wt\Phi\right)^{2k-2} \nonumber\\
&\prec \Big(\sum_{x_\ell \in \cal I_1} B_{\fa x_\ell}  \Omega_{ x_\ell \fb}  + \sum_{x_\ell \in \cal I_2}B_{\fa x_\ell} \Omega_{\fa x_\ell} \Big)W^{(r-1)(-d/2+\e_0)} \left(B^{1/2}_{\fa\fb}+\wt\Phi\right)^{2k-2} \nonumber\\ 
&\prec W^{r(-d/2+\e_0)}\left(B_{\fa\fb}+\wt\Phi^2\right)^{k-1} . \label{eq QTcase2.1}
\end{align}

To estimate the graph $(\wt{\cal G}_{\fa\fb}^{(1)})_{aux}$, we also sum over its internal atoms from leaves to the root $x_\ell$ on the blue tree and reduce the graphs using \eqref{keyobs3} at each step. After each summation, we will get a sum of several new graphs, each of which has a black spanning tree that connects all the internal atoms and the atoms $\otimes$ and $\oplus$. In the end, we get a sum of graphs with only one internal atom $x_{\ell}$. Each of these graphs must take one of the following three forms: 
\begin{center}
\includegraphics[width=11cm]{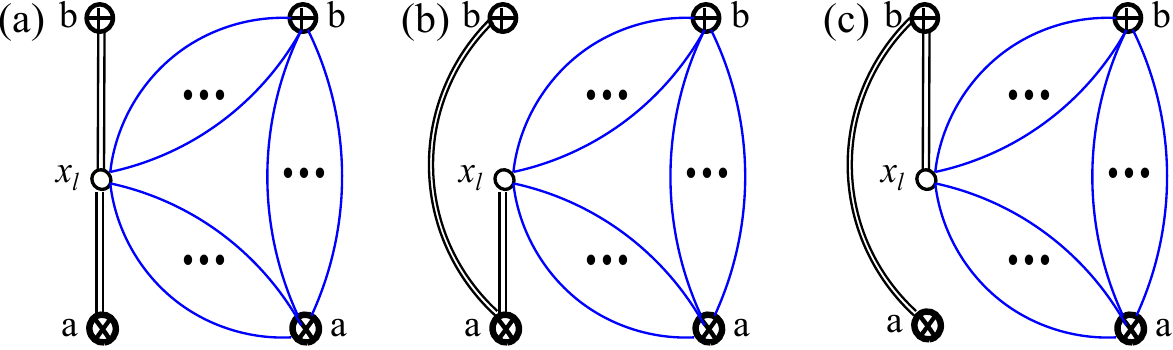}
\end{center}
Graph (a), denoted by $(\wh{\cal G}_{\fa\fb}^{(a)})_{aux}$, can be bounded using similar arguments as in \eqref{Gaux1} and \eqref{Gaux2}: 	
\begin{align*}
(\wh{\cal G}_{\fa\fb}^{(a)})_{aux}  \prec   W^{r(-d/2+\e_0)} B_{\fa \fb} \left( B_{\fa \fb}+\wt\Phi^2\right)^{k-1}.
\end{align*}
Graph (b), denoted by $(\wh{\cal G}_{\fa\fb}^{(b)})_{aux}$, can be bounded by 
\begin{align*} 
(\wh{\cal G}_{\fa\fb}^{(b)})_{aux}\prec B_{\fa\fb} (\wh{\cal G}_{\fa\fb}^{(2)})_{aux} \prec W^{r(-d/2+\e_0)}B_{\fa\fb} \left(B_{\fa\fb}+\wt\Phi^2\right)^{k-1}.
\end{align*}
It is trivial to see that this bound also holds for graph (c) denoted by $(\wh{\cal G}_{\fa\fb}^{(c)})_{aux}$. Combining these estimates, we obtain that 
\begin{align}
(\wt{\cal G}^{(1)}_{\fa\fb})_{aux} \prec (\wh{\cal G}^{(a)}_{\fa\fb})_{aux} + (\wh{\cal G}^{(b)}_{\fa\fb})_{aux}  + (\wh{\cal G}^{(c)}_{\fa\fb})_{aux}  \prec W^{r(-d/2+\e_0)}B_{\fa\fb}\left(B_{\fa\fb}+\wt\Phi^2\right)^{k-1} . \label{eq QTcase2.2}
\end{align}
Plugging \eqref{eq QTcase2.1} and \eqref{eq QTcase2.2} into \eqref{eq QTcase2}, we get the estimate \eqref{eq QTcase1} in case 2.

\medskip

Finally, combining the above two cases, we conclude from \eqref{eq QTcase1} that 
$$(\cal G^{(k)}_\omega)_{\fa\fb} \prec ({\cal G}_{\fa\fb})_{aux}  \prec (\wt{\cal G}_{\fa\fb})_{aux} \prec W^{r(-d/2+\e_0)} \left( B_{\fa \fb}+\wt\Phi^2\right)^{k}  \prec \left[W^{-d/4+\e_0/2}   \left( B_{\fa \fb}+\wt\Phi^2\right)\right]^{k},$$
where we used $k\ge 2$ and $r\ge k-1\ge k/2$ in the last step. Inserting this estimate into \eqref{Holder5}, we can conclude Claim \ref{lem QTab}.
\end{proof}

\subsection{Proof of Lemma \ref{eta1case0}}
The proof of Lemma \ref{eta1case0} is similar to (and actually much easier than) the one for Lemma \ref{lemma ptree}. 
In the $\eta=1$ case, 
 by \eqref{thetaxy} we have that for any small constant $\tau>0$ and large constant $D>0$, 
 \be\label{thetaxy eta1}\Theta_{xy}\le \frac{W^\tau}{W^d}\mathbf 1_{|x-y|\le W^{1+\tau}}+ \OO(W^{-D}).\ee
In other words, the $\Theta$ edge is typically a ``short edge" of length $\OO_\prec(W)$. Moreover, the following local law on $G(z)$ was proved in \cite{ErdYauYin2012Univ,Semicircle}. 
\begin{lemma}[Theorem 2.3 of \cite{Semicircle} and Theorem 2.1 of \cite{ErdYauYin2012Univ}]\label{semicircle}
Suppose the assumptions of Theorem \ref{main thm} hold. For any constants $\kappa, \e >0$, the local law 
\be\label{jxw}
\max_{x,y}|G_{xy}(z) - \delta_{xy}m(z)| \prec ( W^d\eta)^{-1/2} 
 \ee 
 holds uniformly in $z=E+\ii \eta$ with $E\in (-2+\kappa, 2- \kappa)$ and $\eta \in[W^{-d+\e},1]$.
\end{lemma}

With \eqref{jxw} as the a priori estimate, to conclude Lemma \ref{eta1case0}, it suffices to prove the following high moment estimate similar to \eqref{locallawptree}: if \eqref{initial_p} holds, then
\be\label{locallawptree222}
\E T_{xy} (z) ^p \prec  (B_{xy} +W^{-d/4} \wt\Phi^2)^p. 
\ee
To prove \eqref{locallawptree222}, we use the second order $T$-expansion \eqref{seconduniversal}. Using \eqref{thetaxy eta1}, \eqref{jxw} and \eqref{initial_p}, we can easily bound that 
$$(\AT^{(>2)})_{\fa,\fb\fb} \prec W^{-d/2}\sum_{x:|x-\fa|\le W^{1+\tau}}  \frac{W^\tau}{W^d} \left( B_{x \fb} + \wt\Phi^2\right) + W^{-D} \prec W^{-d/2 + 2d\tau}  \left( B_{\fa\fb} + \wt\Phi^2\right) .$$
Moreover, we can prove that 
$$ \E T_{\fa\fb}^{p-1}(\QT^{(2)})_{\fa,\fb\fb} \prec \sum_{k=2}^{p}\left[W^{-d/4+d\tau} \left(B_{xy}+\wt\Phi^2\right)  \right]^k \E T_{\fa\fb}^{p-k}  . $$
We omit the details since the proof is similar to (and actually much easier than) the one for Claim \ref{lem QTab}.
With the above two estimates, we can obtain an estimate similar to \eqref{Holder3.5}, which gives \eqref{locallawptree222} by the argument  in \eqref{Holder3.6}. 

\section{Continuity estimate}\label{sec ini_bound}
 
 \subsection{Proof of Lemma \ref{lem: ini bound}}

Define the subset of indices $\cal I=\{y:|y-x_0|\le K\}$ for $x_0\in \Z_L^d$ and $K\in [W,L/2]$. Then we define the $\mathcal I \times \mathcal I$ positive definite Hermitian matrix $A=(A_{yy'}: y,y'\in \cal I)$ by
\begin{align} \label{Ayy}
	A_{y'y}:= \frac{G_{y'y}(\wt z)- \overline {G_{yy'}(\wt z)}}{2\ii }.
\end{align}
In the proof of Lemma 5.3 in \cite{PartI_high}, we have obtained the following estimates:  
\be\label{vwrel3}
\sum_{y\in \mathcal I} |G_{xy}(z)|^2 \lesssim \sum_{y\in \mathcal I} |G_{xy}(\wt z)|^2 +  \left( \frac{\wt\eta}{\eta}\right)^2 \|A\|_{\ell^2\to \ell^2},
\ee
\be\label{vwrel4}
\sum_{y\in \mathcal I} |G_{yx}(z)|^2 \lesssim \sum_{y\in \mathcal I} |G_{yx}(\wt z)|^2 + \left( \frac{\wt\eta}{\eta}\right)^2  \|A\|_{\ell^2\to \ell^2} ,
\ee
with high probability.
Using \eqref{ulevel}, we can bound that
\be\label{vwrel5}\sum_{y\in \mathcal I}\left( |G_{xy}(\wt z)|^2+ |G_{yx}(\wt z)|^2\right) \prec 1+ \sum_{y\in \mathcal I}B_{xy}\lesssim \frac{K^2}{W^2}.\ee
In this section, we will prove the following bound. 

\begin{lemma}\label{lem normA}
Suppose the assumptions of Lemma \ref{lem: ini bound} hold. 
Then for any fixed $p\in \N$ and small constant $\e>0$, the matrix $A$ in \eqref{Ayy} satisfies that  
\be\label{Moments method}
\E \tr\left(A^{2p}\right)  \le K^d\left( W^\e \frac{K^4}{W^4}\right)^{2p-1}.
\ee
\end{lemma}
With Lemma \ref{lem normA}, we obtain that 
$$\E\|A\|_{\ell^2\to \ell^2}^{2p} \le  \E \tr\left(A^{2p}\right)\le K^d\left( W^\e\frac{K^4}{W^4}\right)^{2p-1}.$$
Since $p$ can be arbitrarily large, using Markov's inequality we obtain that 
\be\label{strong normA}
\|A\|_{\ell^2\to \ell^2} \prec \frac{K^4}{W^4}.
\ee
Inserting \eqref{vwrel5} and \eqref{strong normA} into \eqref{vwrel3} and \eqref{vwrel4}, we conclude \eqref{initial Txy2}. The proof of \eqref{Gmax} uses a standard perturbation and union bound argument, which has been given in \cite{PartI_high}. Hence we omit the details.

\subsection{Non-universal $T$-expansion}\label{sec nonT}

The rest of this section is devoted to proving Lemma \ref{lem normA}. Recall that the matrix $A$ is defined using $G(\wt z)$. Hence all $G$ edges in the following proof represent $G(\wt z)$ entries, and they satisfy the local law \eqref{ulevel}. The proof of Lemma \ref{lem normA} is based on an intricate expansion strategy of the graphs in $ \tr\left(A^{2p}\right)$, which is an extension of Strategy \ref{strat_global}. In this subsection, we establish the first tool for our proof, that is, the \emph{non-universal $T$-expansion}. We will expand the $n$-th order $T$-expansion into 
a sum of deterministic graphs (more precisely, graphs whose maximal subgraphs are deterministic) and $Q$-graphs. In other words, we will further expand the graphs in \smash{$(\PITn)_{\fa,\fb_1 \fb_2}$ and $ (\AITn)_{\fa,\fb_1 \fb_2}$} into deterministic graphs plus $Q$-graphs.
As we will see later, the expansion is called \emph{non-universal} because it depends on $L$ explicitly, and the graphs in the expansion are bounded properly only when the condition \eqref{Lcondition1} holds. 

To get a {\nonuni}, we apply Strategy \ref{strat_global} to \smash{$(\PITn)_{\fa,\fb_1 \fb_2}$} and \smash{$ (\AITn)_{\fa,\fb_1 \fb_2}$} with the following stopping rule: we will stop expanding a graph when it has a deterministic maximal subgraph, it is a $Q$-graph, or its scaling order is larger than a large enough constant $D>0$. However, there is one scenario which is not covered by Strategy \ref{strat_global}, that is, at a certain step we have to expand a pivotal edge and hence break the double connected property. This is the main difficulty with the proof. We will design an expansion strategy so that in the end all the graphs that are not doubly connected in the $\nonuni$ can be bounded properly under the condition \eqref{Lcondition1}. 


Although some graphs may not be doubly connected, all graphs in our expansions will automatically have a black net. On the other hand, the subset of blue and red solid edges may not contain a spanning tree and can be separated into disjoint components. For simplicity, we summarize these properties in the following definition. 


\begin{definition}\label{weak-double}
We say a subgraph $\cal G$ without external atoms has (at most) $k_0\equiv k_0(\cal G)$ many components if there exist three disjoint subsets of edges in the molecular graph $\cal G_{\cal M}$ satisfying the following properties. 
\begin{enumerate}
\item $\cal B_{black}$ is a \emph{black net} consisting of $\dashed$ edges. 

\item $\cal B_{blue}$ is a subset consisting of blue solid (i.e. plus $G$) and $\dashed$ edges.

\item $\cal R_{red}$ is a subset consisting of red solid (i.e. minus $G$) edges.

\item There are $k_{0}$ disjoint components in the molecular subgraph consisting of the edges in $\cal B_{blue}\cup \cal R_{red}$. (By definition, a component means a subgraph in which any two molecules are connected to each other by a path of edges in $\cal B_{blue}\cup \cal R_{red} $, and which is connected to no additional molecules in the rest of the graph.) 

\end{enumerate}
A graph $\cal G$ with external atoms is said to have $k_0$ many components if its maximal subgraph induced on internal atoms has $k_0$ many components. We will call the edges in $\cal B_{blue}$ and $\cal B_{red}$ as blue edges and red edges, respectively. 
\end{definition}

 

If a graph has only one component, then Lemma \ref{no dot} applies to it. In general, we have the following lemma for graphs with more than one component. 
\begin{lemma}\label{no dot weak0}
Under the assumptions of Lemma \ref{lem: ini bound}, 
let $\cal G$ be a normal regular graph without external atoms and with (at most) $k_0$ components in the sense of Definition \ref{weak-double}.
Pick any two atoms of $\cal G$ and fix their values $x , y\in \Z_L^d$. Then the resulting graph $\cal G_{xy}$ satisfies that
\be\label{bound 2net1 weak}
\left|\cal G_{xy}\right| \prec \left(\frac{L^2}{W^2}\right)^{k_0-1} \frac{W^{ - \left(n_{xy}-2\right)d/2 } }{\langle x-y\rangle^d }, 
\ee
where $n_{xy}:=\ord(\cal G_{xy})$ is the scaling order of $\cal G_{xy}$. 
This bound holds also for the graph $ {\cal G}^{{\rm abs}}$. 
 \end{lemma}
\begin{proof}
The proof is similar to the one for Lemma \ref{no dot} in \cite{PartI_high} and the one for Claim \ref{lem ATab} in Section \ref{sec ptree}. Hence we only explain the main difference instead of writing down all the details. 
First, we pick a spanning tree of black edges in $\cal B_{black}$, and a spanning tree of blue and red edges in $\cal B_{blue}\cup \cal R_{red}$ for each component. Then similar to the proof of Claim \ref{lem ATab}, we can reduce the problem to bounding an auxiliary graph $(\cal G_{xy})_{aux}$ satisfying the following properties:
 \begin{itemize}
\item $\cal G_{xy}$ has two external atoms $x$ and $y$, and $\ell$ internal atoms $x_i$, $1\le i \le \ell$, which are the representative atoms of the internal molecules; 

\item there is a spanning tree of black double-line edges connecting all atoms, where a double-line edge between atoms $\al$ and $\beta$ represents a $B_{\al\beta}$ factor; 

\item there is a spanning tree of blue solid edges connecting all atoms in each of the $k_0$ components, where a blue solid edge between atoms $\al$ and $\beta$ represents a $B^{1/2}_{\al\beta}$ factor.
\end{itemize}
In obtaining the auxiliary graph, we also bound the red solid edges in the spanning trees of the $k_0$ components by $B^{1/2}$ factors. This is different from the proof of Claim \ref{lem ATab}, where we bound all red solid edges by their maximum norms.

For any spanning tree of blue solid edges in a component of $(\cal G_{xy})_{aux}$, we choose an atom in it as the root. If a blue tree contains $x$ or $y$, then we choose $x$ or $y$ as the root. For definiteness, we assume that $x$ and $y$ belong to different blue trees in the following proof. The case where $x$ and $y$ belong to the same tree can be handled in a similar way, and we omit the details. Now we sum over the atoms in a spanning tree from leaves to the root. Similar to the proof of Claim \ref{lem ATab}, we bound each sum using \eqref{keyobs3} and \eqref{keyobs2.2} with $(a,b)=(1,1)$. After summing over all internal atoms in one component except for the root, we turn to the blue spanning tree of the next component and sum over its internal atoms from leaves to the chosen root. After summing over the non-root atoms in all components, we obtain a graph, say $(\wt{\cal G}_{xy})_{aux}$, satisfying the following properties: 
\begin{itemize}
\item it has two external atoms $x,\ y$ and $k_0-2$ internal atoms, which are the roots of the $k_0$ blue trees; 

\item it has a spanning tree of black $\dashed$ edges connecting all atoms.  

\end{itemize}
We sum over the internal atoms, say $x_1, \cdots, x_{k_0-2}$, of $(\wt{\cal G}_{xy})_{aux}$ one by one, and bound each sum using the simple estimate
\be\label{keyobs2.91}
\sum_{x_i} \prod_{i=1}^k B_{x_i y_k}\lesssim \frac{L^2}{W^2}\Gamma(y_1,\cdots, y_k).
\ee
In $(\wt{\cal G}_{xy})_{aux}$, we have a black spanning tree connecting all internal atoms. By \eqref{keyobs2.91}, after each summation, we get some new graphs, each of which has one fewer internal atom, an additional $L^2/W^2$ factor and a black spanning tree connecting all atoms. After summing over all the $k_0-2$ internal atoms, we get a graph where $x$ is connected to $y$ 
through a double-line edge, i.e., 
 $$(\wt{\cal G}_{xy})_{aux} \prec \left( \frac{L^2}{W^2}\right)^{k_0-2} B_{xy} \lesssim  \left( \frac{L^2}{W^2}\right)^{k_0-1} \frac{1}{\langle x-y\rangle^d} .$$
Finally, by counting the number of $W^{-d/2}$ factors carefully, we can obtain \eqref{bound 2net1 weak}.
\end{proof}

For our purpose, the weakest bound required for a graph $\cal G_{xy}$ satisfying the setting of Lemma \ref{no dot weak0} is $\langle x-y\rangle^{-d}$. Hence we need to make sure that the factor $W^{-\left(n_{xy}-2\right)d/2 }$ is small enough to cancel the $ ({L^2}/{W^2} )^{k_0-1}$ factor. By condition \eqref{Lcondition1}, if we have $n_{xy}\ge (n-1)\cdot (k_0-1)+2$, then \eqref{bound 2net1 weak} gives
$$\left|\cal G_{xy}\right| \prec 
\frac{W^{-(k_0-1)c_0}}{\langle x-y\rangle^d}.$$   
Thus in expansions, we need to make sure that the scaling order of every graph is sufficiently high depending on the number of components in it. For this purpose, we introduce the following helpful tool.

\begin{definition}[Ghost edges]
We use a green dashed edge between atoms $x$ and $y$ to represent a $W^2/L^2$ factor, and we call such an edge a \emph{ghost edge}. We do not count ghost edges when calculating the scaling order of a graph, i.e. the scaling order of a ghost edge is $0$. Moreover, the doubly connected property in Definition \ref{def 2net}, the SPD property in Definition \ref{def seqPDG}, and the globally standard property in Definition \ref{defn gs} can be extended to graphs with ghost edges by including these edges into blue nets. 
\end{definition}

To keep graph values unchanged, whenever we add a ghost edge to a graph, we always multiply it by $L^2/W^2$ to cancel this edge. In a doubly connected graph with $k_0$ ghost edges, removing all of them gives at most $k_0+1$ many components. Hence with Lemma \ref{no dot weak0}, we immediately obtain the following estimate. 
  \begin{lemma}\label{no dot weak}
Under the assumptions of Lemma \ref{lem: ini bound}, let $\cal G$ be a doubly connected normal regular graph without external atoms and with $k_{\gh}$ ghost edges. Pick any two atoms of $\cal G$ and fix their values $x , y\in \Z_L^d$. Then the resulting graph $\cal G_{xy}$ satisfies that
\be\label{bound 2net1 weak2}
\left|\cal G_{xy}\right| \prec \left( \frac{L^2}{W^2}\right)^{k_{\gh}} \frac{W^{ - \left(n_{xy}-2\right)d/2 } }{\langle x-y\rangle^d } , 
\ee
where $n_{xy}:=\ord(\cal G_{xy})$ is the scaling order of $\cal G_{xy}$.
This bound holds also for the graph $ {\cal G}^{{\rm abs}}$. 
\end{lemma}
 
Ghost edges are introduced to deal with expansions of pivotal blue solid edges (recall Definition \ref{defn pivotal}). We have avoided such expansions in the proof of Theorem \ref{incomplete Texp}, but they are necessary for constructing the $\nonuni$. Roughly speaking, when we expand a pivotal blue solid edge in a graph, we will add a ghost edge between its two ending atoms and multiply the graph by $L^2/W^2$ to keep the graph value unchanged. Then this pivotal edge becomes redundant, and we can make use of the tools developed in Sections \ref{sec pre} and \ref{sec global}. In this way, we do not need to rebuild a new set of graphical tools from scratch. However, different from  Strategy \ref{strat_global}, we also need to keep track of the scaling orders of our graphs carefully to make sure that they are high enough to cancel the $L^2/W^2$ factors associated with the ghost edges. To deal with this issue, we will introduce some new graphical tools. 


To define the $\nonuni$, we first define the \emph{generalized SPD property}, extending the SPD property in Definition \ref{def seqPDG}. For simplicity of notations, we call an isolated subgraph plus its external edges (including the ending atoms of them) the \emph{closure} of this isolated subgraph.

\begin{definition}[Generalized SPD property]\label{def seqPDG weak}
A graph $\cal G$ with ghost edges is said to be a generalized sequentially pre-deterministic graph if the following properties hold.
\begin{itemize}
\item[(i)] $\cal G$ is doubly connected by including its ghost edges into the blue net. 

\item[(ii)] 
$\cal G$ contains at most one sequence of proper isolated subgraphs with non-deterministic closures. (By definition, an isolated subgraph has a non-deterministic closure if this subgraph contains $G$ edges and weights inside it, or it is connected with at least one external $G$ edge.)  


\item[(iii)] The maximal subgraph $\cal G_{\max}$ (resp. an isolated subgraph $\Iso$) of $\cal G$ is pre-deterministic if it has no proper isolated subgraph with non-deterministic closure. Otherwise, let $\Iso'$ be its maximal isolated subgraph with non-deterministic closure. If we replace the closure of $\Iso'$ with a $\dashed$ edge in the molecular graph without red solid edges, then $\cal G_{\max}$ (resp. $\Iso$) becomes pre-deterministic. 
\end{itemize}

\end{definition}

The only difference from Definition \ref{def seqPDG} is that property (ii) of Definition \ref{def seqPDG weak} is restricted to the sequence of isolated subgraphs with non-deterministic closures. In other words, a proper isolated subgraph of $\cal G$ may have several other sequences of \emph{deterministic} isolated subgraphs. 
Now we are ready to state the main result of this subsection on $\nonuni$s.
\begin{lemma}[$\Nonuni$]\label{def nonuni-T}
Suppose the assumptions of Lemma \ref{lem: ini bound} hold. For any large constant $D>0$,  $T_{\fa,\fb_1 \fb_2}$ can be expanded into a sum of $\OO(1)$ many graphs: 
\be\label{mlevelTgdef weak}
\begin{split}
  T_{\fa,\fb_1 \fb_2} & =   m  \Theta_{\fa \fb_1}\overline G_{\fb_1\fb_2} +  \sum_{\mu} B^{(\mu)}_{\fa \fb_1}\overline G_{\fb_1\fb_2} f_\mu (G_{\fb_1\fb_1})+   \sum_{\nu} \wt B^{(\nu)}_{\fa \fb_2} G_{\fb_2\fb_1} \wt f_\nu(G_{\fb_2\fb_2})  \\
&+   \sum_{\mu} \mathscr B^{(\mu)}_{\fa , \fb_1 \fb_2}g_\mu(G_{\fb_1\fb_1},G_{\fb_2\fb_2},\overline G_{\fb_1\fb_2},  G_{\fb_2\fb_1}) + \cal Q_{\fa,\fb_1\fb_2}   + \Err_{\fa,\fb_1 \fb_2} .
\end{split}
\ee
The graphs on the right-hand side depend on $n$, $D$, $W$ and $L$ directly, and satisfy the following properties. 
\begin{enumerate}
\item $\Err_{\fa,\fb_1 \fb_2}$ is an error term satisfying $\Err_{\fa,\fb_1 \fb_2}\prec W^{-D}.$ 

\item $B^{(\mu)}_{\fa \fb_1}$ and $\wt B^{(\nu)}_{\fa \fb_2}$ are deterministic graphs satisfying  
\be\label{est_BwtB} | B^{(\mu)}_{\fa \fb_1}|\prec W^{-c_1}B_{\fa\fb_1},\quad |\wt B^{(\nu)}_{\fa \fb_2}|\prec W^{-c_1}B_{\fa\fb_2},\ee
under the condition \eqref{Lcondition1}, where $c_1:=\min(c_0,d/2)$. 

\item $\mathscr B^{(\mu)}_{\fa \fb_1}$ are deterministic graphs of the form  
$ \sum_{x}\wh B^{(\mu)}_{\fa x}\cal D^{(\mu)}_{x, \fb_1\fb_2},$
where $\wh B^{(\mu)}_{\fa x}$ satisfies  
\be\label{est_BwtB1} | \wh B^{(\mu)}_{\fa x}|\prec B_{\fa x}, \ee
and $\cal D^{(\mu)}_{x, \fb_1\fb_2}$ is doubly connected with $x$, $\fb_1$ and $\fb_2$ regarded as internal atoms and satisfies  
\be\label{ordG1}\ord(\cal D^{(\mu)}_{x, \fb_1\fb_2}) \ge k_\gh(\cal D^{(\mu)}_{x, \fb_1\fb_2}) \cdot (n-1) + 2.\ee
Here $k_\gh(\cal D^{(\mu)}_{x, \fb_1\fb_2}) $ denotes the number of ghost edges in $\cal D^{(\mu)}_{x, \fb_1\fb_2}$.

\item $f_\mu$ are monomials of $G_{\fb_1\fb_1}$ and $\overline G_{\fb_1\fb_1}$, $\wt f_\nu$ are monomials of  $G_{\fb_2\fb_2}$ and $\overline G_{\fb_2\fb_2}$, and $g_\mu$ are monomials of $G_{\fb_1\fb_1}$, $\overline G_{\fb_1\fb_1}$, $G_{\fb_2\fb_2}$, $\overline G_{\fb_2\fb_2}$, $\overline G_{\fb_1\fb_2}$ and $ G_{\fb_2\fb_1}$. 

\item $\cal Q_{\fa,\fb_1\fb_2} $ is a sum of $\OO(1)$ many $Q$-graphs, denoted by $\cal Q^{(\omega)}_{\fa,\fb_1\fb_2} $, satisfying the following properties.

\begin{enumerate}

\item $\cal Q^{(\omega)}_{\fa,\fb_1\fb_2} $ is a generalized SPD graph. 

\item The atom in the $Q$-label of $\cal Q^{(\omega)}_{\fa,\fb_1\fb_2} $ belongs to the MIS with non-deterministic closure, denoted by $\Iso_{\min}$. In other words, we can write that 
$$\cal Q^{(\omega)}_{\fa,\fb_1\fb_2} =\left( \frac{L^2}{W^2}\right)^{ k_{\gh} ( \cal Q^{(\omega)}_{\fa,\fb_1\fb_2} )}\sum_\al  Q_\al(\Gamma^{(\omega)}),$$ 
where $k_{\gh} ( \cal Q^{(\omega)}_{\fa,\fb_1\fb_2} )$ denotes the number of ghost edges in $\cal Q^{(\omega)}_{\fa,\fb_1\fb_2} $, $\Gamma^{(\omega)}$ is a graph without $P/Q$ labels, and the internal atom $\al$ is inside $\Iso_{\min}$. 

\item The scaling order of $\cal Q^{(\omega)}_{\fa,\fb_1\fb_2}$ satisfies that 
\be\label{ordG2} \ord( \cal Q^{(\omega)}_{\fa,\fb_1\fb_2} ) \ge k_\gh( \cal Q^{(\omega)}_{\fa,\fb_1\fb_2} ) \cdot (n-1) + 2.\ee

\end{enumerate}
\item Each $\cal Q^{(\omega)}_{\fa,\fb_1\fb_2}$ has a deterministic graph $\cal B^{(\omega)}_{\fa x}$ between $\otimes$ and an internal atom $x$, where $\cal B^{(\omega)}_{\fa x}$ satisfies 
\be\label{est_BwtB2}| \cal B^{(\omega)}_{\fa x}|\prec B_{\fa x}. \ee
Moreover, there is at least an edge, which is either blue solid or $\dashed$ or dotted, connected to $\oplus$, and there is at least an edge, which is either red solid or $\dashed$ or dotted, connected to $\ominus$. 
\end{enumerate} 
 \end{lemma}
To prove Lemma \ref{def nonuni-T} , we start with the $n$-th order $T$-expansion \eqref{mlevelTgdef}. 
The term $m(\Theta \Sdelta^{(n)}\Theta)_{\fa\fb_1} \overline G_{\fb_1\fb_2} $ can be included into the second term on the right-hand side of \eqref{mlevelTgdef weak}, \smash{$(\QT^{(n)})_{\fa,\fb_1\fb_2} $} can be included into $\cal Q_{\fa,\fb_1\fb_2} $, and $(\Err_{n,D})_{\fa,\fb_1\fb_2}$ can be included into $\Err_{\fa,\fb_1 \fb_2}$. It remains to expand the graphs in \smash{$(\PT^{(n)})_{\fa,\fb_1 \fb_2}$} and \smash{$(\AT^{(>n)})_{\fa,\fb_1\fb_2}$}. We again use the Strategy \ref{strat_global}, but with some modifications to handle the ghost edges and to keep track of the scaling order of every graph (see Strategy \ref{strat_global_weak} below). 

To describe the expansion strategy, we now define the stopping rules. Given a graph $\cal G$, we define  
\be\label{defn sizeG} \size(\cal G): =\left( \frac{L^2}{W^2}\right)^{k_\gh(\cal G)} W^{-\ord(\cal G) \cdot d/2},\ee
where $k_{\gh}(\cal G)$ denotes the number of ghost edges in $\cal G$. By Lemma \ref{no dot weak}, we have $|\cal G|\prec  \size(\cal G)$ if $\cal G$ is doubly connected with ghost edges. 
Given the large constant $D$ in Lemma \ref{def nonuni-T}, we stop expanding a graph $\cal G$ if it is normal regular and satisfies at least one of the following properties:
\begin{itemize}
\item[(T1)] 
$\cal G$ has a deterministic maximal subgraph that is locally standard (i.e. if there is an internal atom connected with external edges, it must be standard neutral);


\item[(T2)] $\size(\cal G)\le W^{-D}$;

\item[(T3)] $\cal G$ is a $Q$-graph. 

\end{itemize}
Then, similar to Strategy \ref{strat_global}, the core of the expansion strategy for the proof of Lemma \ref{def nonuni-T} is still to expand the plus $G$ edges according to a pre-deterministic order in the MIS with non-deterministic closure. However, there is a key difference that we have to expand a pivotal blue solid edge connected with an isolated subgraph at a certain step. To deal with this issue, we add a ghost edge between the ending atoms of the pivotal edge which we want to expand. Then this blue solid edge becomes redundant, and we can expand it in the same way as Step 3 of Strategy \ref{strat_global}. However, we need to make sure that every graph, say $\cal G$, from the expansion has a scaling order satisfying 
$$\ord(\cal G)\ge (n-1)\cdot k_{\gh}(\cal G)  + 2.$$
Now we define the following expansion strategy.
 
\begin{strategy}
\label{strat_global_weak}
Given a large constant $D>0$ and a generalized SPD graph $\cal G$ without $P/Q$ labels, we perform one step of expansion as follows. Let $\Iso_{k}$ be the MIS with non-deterministic closure. 
 
\medskip

\noindent{\bf Case 1}: If $\Iso_{k}$ is not locally standard, then we perform the local expansions in Definitions \ref{Ow-def}, \ref{multi-def}, \ref{GG-def} and \ref{GGbar-def} on an atom in $\Iso_{k}$. We send the resulting graphs that already satisfy the stopping rules (T1)--(T3) to the outputs. Every remaining graph has a locally standard MIS with non-deterministic closure and satisfies the generalized SPD property by Lemma \ref{lem localgood2}.

\medskip
\noindent{\bf Case 2}: Suppose $\Iso_{k}$ is locally standard. We find a $t_{x,y_1y_2}$ or $t_{y_1y_2,x}$ variable defined in \eqref{12in4T}, so that $\al$ is a standard neutral atom in $\Iso_{k}$ and the edge $G_{\al y_1}$ or $G_{y_1 \al}$ is the first blue solid edge in a pre-deterministic order of $\Iso_{k}$. Then we apply the expansion in Step 3 of Strategy \ref{strat_global} (except that we will use \eqref{replaceT} with $n-1$ replaced by $n$).


\medskip
\noindent{\bf Case 3}: Suppose $\Iso_{k}$ is deterministic, strongly isolated in $\Iso_{k-1}$ (cf. Definition \ref{def_weakstrong}), and locally standard. (In other words, $\Iso_{k}$ contains a standard neutral atom connected with two external blue and red solid edges.) 
Suppose $b_1$ is the edge $G_{\al y_1}$ or $G_{y_1\al}$ in \eqref{12in4T}, with $\al$ being the standard neutral atom in $\Iso_k$. We add a ghost edge between $\al$ and $y_1$ and multiply the graph by $L^2/W^2$. In the resulting graph, the edge $b_1$ becomes redundant if we include the newly added ghost edge into the blue net. Then we apply the expansion in Step 3 of Strategy \ref{strat_global} (except that we will use \eqref{replaceT} with $n-1$ replaced by $n$).
We will show that each $Q$-graph from this expansion satisfies the property (v) of Lemma \ref{def nonuni-T}, and each non-$Q$ graph, say $\cal G'$, satisfies 
$$\ord(\cal G')\ge (n-1)\cdot k_{\gh}(\cal G')  + 2,$$
so that $ \size(\cal G')$ is under control. 
\end{strategy}

We will show that by applying the Strategy \ref{strat_global_weak} repeatedly until all graphs satisfy the stopping rules (T1)--(T3), we can obtain the expansion \eqref{mlevelTgdef weak}. For this purpose, we define the \emph{generalized globally standard property}, extending the globally standard property in Definition \ref{defn gs}. 

\begin{definition}[Generalized globally standard graphs]\label{defn gs weak}
 We say a graph $\cal G$ is generalized globally standard (GGS) if it is a generalized SPD graph and satisfies at least one of the following properties. 
\begin{itemize}
\item[(A)] $\cal G$ satisfies \be\label{ord ghost0} \ord(\cal G)\ge (n-1)\cdot \left[k_{\gh}(\cal G)+1\right] + 2 .\ee

\item[(B)] $\cal G$ satisfies 
\be\label{ord ghost1}
 \ord(\cal G)\ge (n-1)\cdot  k_{\gh}(\cal G) + 2 . 
\ee
Let $\Iso_k$ be the minimal strongly isolated subgraph with non-deterministic closure. (In other words, $\Iso_k$ may contain weakly isolated subgraphs, but a strongly isolated subgraph of $\Iso_k$ must have a deterministic closure.) 
There exists a ghost edge, say $g_0$, inside $\Iso_k$ such that the following property holds. 
\begin{itemize}
\item[(B1)] Let $\Iso'$ be the minimal isolated subgraph of $\Iso_k$ that contains $g_0$. ($\Iso'$ can be $\Iso_k$ if $g_0$ is not inside any proper isolated subgraph of $\Iso_k$.) If we replace the closure of the maximal proper isolated subgraph of $\Iso'$ and all other blue solid edges in $\Iso'$ with $\dashed$ edges, then $g_0$ becomes redundant.
	
\end{itemize}

%
%
\end{itemize}
\end{definition}

Similar to the proof of Theorem \ref{incomplete Texp} in Section \ref{sec global}, Lemma \ref{def nonuni-T} follows from the following two  observations:
\begin{itemize}
\item if the input graph is GGS, then after one step of expansion in Strategy \ref{strat_global_weak}, all the resulting non-$Q$ graphs are still GGS and all the $Q$-graphs satisfy property (v) of Lemma \ref{def nonuni-T}; 
	
\item the expansion of $T_{\fa,\fb_1\fb_2}$ will stop after $\OO(1)$ many iterations of Strategy \ref{strat_global_weak}.
\end{itemize}
We postpone the detailed proof of Lemma \ref{def nonuni-T} to Appendix \ref{app_pf_nonuniversal}.

\subsection{Generalized doubly connected property}\label{sec gen2net}

In this subsection, we develop another tool for the proof of Lemma \ref{lem normA}, that is, a \emph{generalized doubly connected property} of graphs with external atoms, which extends the doubly connected property in Definition \ref{def 2net}. 

\begin{definition}[Generalized doubly connected property]\label{def 2netex}
A graph $\cal G$ with external molecules is said to be generalized doubly connected if its molecular graph satisfies the following property. There exist a collection, say $\cal B_{black}$, of $\dashed$ edges, and another collection, say $\cal B_{blue}$, of blue solid, $\dashed$ or ghost edges such that: (a)  $\cal B_{black}\cap \cal B_{blue}=\emptyset$, (b) every internal molecule connects to external molecules through two disjoint paths: a path of edges in $\cal B_{black}$ and a path of edges in $\cal B_{blue}$. 
\end{definition}

Given a generalized doubly connected graph $\cal G$, after removing its external molecules we \emph{may not} get a doubly connected graph in the sense of Definition \ref{def 2net} (i.e., the maximal subgraph of $\cal G$ consisting of internal molecules may not have a black net and a blue net). For example, the following graph is a generalized doubly connected graph with external molecules $a_1, a_2, a_3, a_4$, and internal molecules $x_1, x_2, x_3,x_4, x_5$, but it is not doubly connected.
\begin{center}
	\includegraphics[width=4cm]{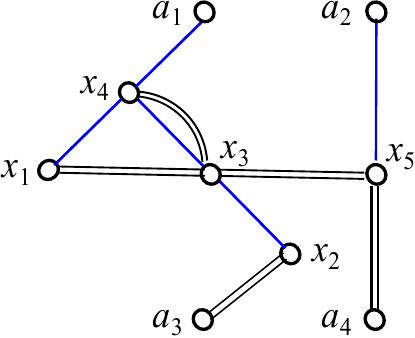}
\end{center}
However, by Definition \ref{def 2netex}, if we merge all external molecules in a generalized doubly connected graph into a single internal molecule, then we will get a doubly connected graph in the sense of Definition \ref{def 2net}.  




Generalized doubly connected graphs satisfy the following bound. 
  \begin{lemma}\label{general dot}
  Under the assumptions of Lemma \ref{lem: ini bound}, let $\cal G$ be a generalized doubly connected normal regular graph. 
  Then, it satisfies the bound
 \be\label{bound 2net}
 |\cal G| \prec \left( \frac{L^2}{W^2}\right)^{k_{\gh}(\cal G)}  W^{- \ord(\cal G) \cdot d/2 }.
 \ee
 \end{lemma}
\begin{proof}
The proof of this lemma is similar to the one for Lemma \ref{no dot weak0}. With the generalized doubly connected property of $\cal G$, we can choose disjoint black spanning trees and blue spanning trees with external molecules so that each internal molecule connects to an external molecule through a \emph{unique black path} on a black tree, and to an external molecule through a \emph{unique blue path} on a blue tree. We can again reduce the problem to bounding an auxiliary graph $\cal G_{aux}$ consisting of these black and blue trees. Then we sum over the internal atoms in $\cal G_{aux}$ from leaves of the blue trees to the roots of them. We bound the summation over each internal atom using  \eqref{keyobs3} with $(a,b)=(1,1)$ (if the blue edge in the summation is bounded by a $B^{1/2}$ entry), or \eqref{keyobs2.91}  (if the blue edge in the summation is bounded by $W^2/L^2$). After each summation, we get a sum of new graphs that are still generalized doubly connected. Finally, after summing over all internal atoms, we get the bound \eqref{bound 2net}. We omit the details of the proof.
\end{proof}

The following lemma will be used in the proof of Lemma \ref{lem normA}. It bounds the average of a deterministic generalized doubly connected graph over its external molecules in the box $\cal I$. 
 
\begin{lemma}\label{lemma_boundgen2net}
Under the assumptions of Lemma \ref{lem: ini bound}, let $\cal G$ be a deterministic generalized doubly connected regular graph with $2p$ external atoms $\fa_i$, $1\le i \le 2p$. Moreover, suppose $\cal G$ is a connected graph without using ghost edges, i.e., 
all (internal and external) molecules are connected to each other through diffusive edges. 
Then, we have that
\be\label{bound 2net weakdeter}
\frac{1}{\left(K^d\right)^{2p}}\sum_{\fa_1,\fa_2, \cdots, \fa_{2p}\in \cal I} \left|\cal G \left( \fa_1,\fa_2,\cdots, \fa_{2p}\right)\right| \prec \left(\frac{W^{d-4}}{K^{d-4}}\right)^{2p-1} \cdot \left( \frac{L^2}{W^2}\right)^{k_{\gh}(\cal G)}  W^{-\ord(\cal G)\cdot d/2},
\ee
where recall that $\cal I=\{y:|y-x_0|\le K\}$ for some $x_0\in \Z_L^d$. We emphasize that in calculating $\ord(\cal G)$, $\fa_i$, $1\le i \le 2p$, are still treated as external atoms, regardless of the sum $\sum_{\fa_1,\fa_2, \cdots, \fa_{2p}}$. 
\end{lemma}
\begin{proof}
We can assume that there are no dotted edges in $\cal G$. Otherwise, we merge the atoms connected by dotted edges and still get a graph satisfying the assumptions of this lemma and with at most $2p$ external atoms (there may be dotted edges between external atoms, so the graph after the merging operation may contain strictly fewer external atoms). Without loss of generality, we assume that the external atoms are $\fa_i$, $1\le i \le q$, for some $1\le q\le 2p$. Moreover, for simplicity, we can pick out the $( {L^2}/{W^2} )^{k_{\gh}(\cal G)}$ factor in the coefficient of $\cal G$ and do not keep it in the following proof.

Following the argument in the proof of Claim \ref{lem ATab}, we again reduce the problem to bounding a generalized doubly connected auxiliary graph $\cal G_{aux}$ obtained as follows: $\cal G_{aux}$ has a set of external atoms, say $\fa_1, \fa_2,\cdots, \fa_q$, and a set of internal atoms, say $x_1,x_2,\cdots, x_\ell$, which are the representative atoms of the internal molecules. Every $\dashed$ edge between different molecules of $\cal G$ is replaced by a double-line edge representing a $B$ factor in $\cal G_{aux}$. Every ghost edge in  $\cal G$ is still a ghost edge representing a $W^2/L^2$ factor in $\cal G_{aux}$.  To conclude the proof, it suffices to prove that
\be \label{boundaux_ext_weak}
\frac{1}{\left(K^d\right)^{2p}}\sum_{\fa_1,\fa_2, \cdots, \fa_{q}\in \cal I} \left|\cal G_{aux} \left( \fa_1,\fa_2,\cdots, \fa_{q}\right)\right| \prec \left(\frac{W^{d-4}}{K^{d-4}}\right)^{2p-1} W^{-\ord(\cal G_{aux})\cdot d/2} ,
\ee
where $\ord(\cal G_{aux})$ is defined by
$$\ord(\cal G_{aux}):=2\#\{\text{double-line edges}\} -2 \#\{\text{internal atoms}\}.$$

With the generalized doubly connected property of $\cal G_{aux}$, we can choose a collection of blue trees such that every internal atom connects to an external atom through a \emph{unique path} on a blue tree. These blue trees are disjoint, and each of them contains an external atom as the root. Now, we sum over all internal atoms from leaves of the blue trees to the roots until there is no internal atom in every graph. 
We introduce a new type of purple solid edges: a purple solid edge between atoms $\al$ and $\beta$ represents a \smash{$\wt B_{\al\beta}= W^{-4}\langle \al-\beta\rangle^{-(d-4)}$} factor. Purple solid edges are not included into either the black net or the blue net. In bounding each summation, we will use the following estimates, which can be proved with basic calculus: if $k\ge 1$ and $r\ge 0$, we have 
\be\label{keyobs4.3}
\begin{split}
& \sum_{x_i} \prod_{j=1}^k B_{ x_i y_j}\cdot \prod_{s=1}^r \wt B_{ x_i z_s} \cdot B_{x_i \fa} \\ 
&\lesssim \sum_{l=1}^k \prod_{j: j\ne l}B_{y_j y_{l}} \cdot \Big( \wt B_{y_l \fa} \prod_{s=1}^r \wt B_{z_s \fa}+ \wt B_{y_l \fa} \prod_{s=1}^r \wt B_{z_s y_l} + \sum_{t=1}^r  \wt B_{y_l z_t}  \wt B_{z_t \fa} \prod_{s:s\ne t} \wt B_{z_s z_t} \Big) ,
\end{split}
\ee
\be\label{keyobs4.3_add}
\sum_{x_i} \prod_{j=1}^k B_{ x_i y_j}\cdot \prod_{s=1}^r \wt B_{ x_i z_s} \cdot \frac{W^2}{L^2}\lesssim \sum_{l=1}^k \prod_{j: j\ne l}B_{y_j y_{l}} \cdot \Big(  \prod_{s=1}^r \wt B_{z_s y_l} + \sum_{t=1}^r  \wt B_{y_l z_t}  \prod_{s:s\ne t} \wt B_{z_s z_t} \Big) .
\ee
Note that the left-hand side is a star graph consisting of $k$ double-line edges in the black tree, a double-line or ghost edge in the blue tree, and $r$ purple solid edges connected with $x_i$, while every graph on the right-hand side is a connected graph consisting of $k-1$ double-line edges in the black tree and $r+1$ purple solid edges. 

Recall that the graph $\cal G_{aux}$ is generalized doubly connected and all its atoms are connected to each other through double-line edges. Then, using \eqref{keyobs4.3} or \eqref{keyobs4.3_add}, we can bound a summation over an internal atom by a sum of new graphs, each of which satisfies the following two properties:
\begin{itemize}
	\item[(i)] it is generalized doubly connected;
	\item[(ii)] its atoms are connected to each other through double-line and purple solid edges. 
\end{itemize} 
For every new graph, we bound its sum over a leaf of the blue tree by another set of new graphs satisfying the above properties (i) and (ii). Finally, after summing over all internal atoms, we obtain that 
\be\label{bound 2net weakdeter2} \cal G_{aux} \prec \sum_{\omega} W^{- \Delta_\omega d/2}\cal G_{\omega,aux} ,\quad  \Delta_\omega:=\ord(\cal G_{aux})- \ord(\cal G_{\omega,aux}).\ee
where every $\cal G_{\omega,aux}$ is a connected graph consisting of external atoms $\fa_1, \cdots, \fa_q$, and double-line and purple solid edges. Since $\cal G_{\omega,aux}$ only contains external atoms, its scaling order is 
\be\label{eq_scale_aux}
 \ord\left( \cal G_{\omega,aux}\right)=2\#\{\text{$\dashed$ and purple solid edges in }\cal G_{\omega,aux}\} 
 .\ee

Finally, it remains to bound 
\be\nonumber \frac{1}{\left(K^d\right)^{2p}}\sum_{\fa_1,\fa_2, \cdots, \fa_{q}\in \cal I} \cal G_{\omega,aux} \left( \fa_1,\fa_2,\cdots, \fa_{q}\right) .\ee
We perform the summations according to the order $\sum_{\fa_q}\cdots \sum_{\fa_2}\sum_{\fa_1}$, and we use the following estimate to bound each average:  
 \be\nonumber
 \frac{1}{K^d}\sum_{\fa_i \in \cal I} \prod_{j=1}^k \wt B_{y_j \fa_i} \lesssim  \frac{1}{W^4 K^{d-4}}\sum_{l=1}^k   \prod_{j:j\ne l}\wt B_{y_{l}y_j} = W^{-d}\cdot \frac{W^{d-4}}{K^{d-4}}\sum_{l=1}^k   \prod_{j\ne l}\wt B_{y_{l}y_j} .
\ee
This estimate shows that we can bound each average by a sum of new connected graphs with one fewer atom and one fewer edge, 
while gaining an extra factor $W^{-4}K^{-(d-4)}$. 
After taking averages over $\fa_1, \cdots, \fa_{q-2}$, we obtain a graph $(\wt B_{\fa_{q-1}\fa_q})^k$ with 
$$k :=  \ord\left( \cal G_{\omega, aux}\right)/2- (q-2).$$
Its average over $\fa_{q-1}$ can be bounded by 
 \be\nonumber
\frac{1}{K^d}\sum_{\fa_{q-1} \in \cal I} (\wt B_{\fa_{q-1}\fa_q})^k \lesssim W^{-kd} \frac{W^{d-4}}{  K^{d-4}}  .
\ee
Finally, the average over $\fa_q$ is equal to 1. In sum, we get that
\begin{align} \label{bound_aux_wgamma}
\frac{1}{\left(K^d\right)^{2p}}\sum_{\fa_1,\fa_2, \cdots, \fa_{q}\in \cal I} \cal G_{\omega,aux} \left( \fa_1,\fa_2,\cdots, \fa_{q}\right) \prec \left(\frac{1}{K^d}\right)^{2p-q}\left(\frac{W^{d-4}}{K^{d-4}}\right)^{q-1}  W^{-\ord(\cal G_{\omega,aux})\cdot d/2}  .
\end{align}
Plugging this estimate into \eqref{bound 2net weakdeter2}, we obtain \eqref{boundaux_ext_weak}, which concludes \eqref{bound 2net weakdeter}. 
\end{proof}

\subsection{Proof of Lemma \ref{lem normA}}

In this subsection, we complete the proof of Lemma \ref{lem normA}. 
Using \eqref{Ayy}, we can expand $ \E \tr\left( A^{2p}\right) $ as 
\begin{align}\label{eq_pgons}
\E \tr\left(A^{2p}\right) = \sum_{\fa_1,\cdots, \fa_{2p}\in \cal I}\sum_{s_1, \cdots, s_{2p} \in \{+,-\}} c(s_1, \cdots, s_{2p}) \prod_{i=1}^{2p}G^{s_i}_{\fa_i \fa_{i+1}},
\end{align}
where as a convention we let $\fa_{2p+1}:= \fa_1$, $c(s_1, \cdots, s_{2p})$ is a deterministic coefficient of order $\OO(1)$, and $s_i $, $i=1,\cdots, 2p$, denote $\pm$ signs. Here we adopted the notations
$$ G_{xy}^{+}:= G_{xy},\quad G_{xy}^{-}:= \overline G_{xy}.$$ 
Thus to conclude \eqref{Moments method}, it suffices to prove that for any $2p$-gon graph 
\be\label{pgon graphs} 
\cal G_{\fa_1,\cdots,\fa_{2p}}\equiv \cal G_{\fa_1,\cdots,\fa_{2p}}(s_1, \cdots, s_{2p}):=\prod_{i=1}^{2p}G^{s_i}_{\fa_i \fa_{i+1}},
\ee
we have 
\be\label{Moments method2}
\sum_{\fa_1,\cdots, \fa_{2p}\in \cal I} \E \cal G_{\fa_1,\cdots,\fa_{2p}} \le K^d\left( W^\e \frac{K^4}{W^4}\right)^{2p-1},
\ee
for any small constant $\e>0$. The proof of \eqref{Moments method2} is based on the results in Sections \ref{sec nonT} and \ref{sec gen2net}: we first use Lemma \ref{def nonuni-T} to expand $\cal G_{\fa_1,\cdots,\fa_{2p}}$ into a sum of deterministic generalized doubly connected graphs, and then use Lemma \ref{lemma_boundgen2net} to bound each of them.

Because of the estimates \eqref{est_BwtB}, \eqref{est_BwtB1} and \eqref{est_BwtB2}, it is more natural to treat graphs $B^{(\mu)}$, $\wt B^{(\nu)}$,  $\wh B^{(\mu)}$ and $\cal B^{(\omega)}$ in Lemma \ref{def nonuni-T} as new types of $\dashed$ edges in the following proof. 
It is not hard to see that Lemma \ref{general dot} and Lemma \ref{lemma_boundgen2net} still hold for generalized doubly connected graphs containing these new $\dashed$ edges, as long as we replace the factor $({L^2}/{W^2})^{k_{\gh}(\cal G)}  W^{-\ord(\cal G)\cdot d/2}$ by $\size(\cal G)$ defined in \eqref{defn sizeG}.

Corresponding to Definition \ref{def 2netex}, we can extend the generalized SPD property in Definition \ref{def seqPDG weak} and the globally standard property in Definition \ref{defn gs} as follows. 

\begin{definition}
\label{GSS_external}
\noindent {(i)} A graph $\cal G$ is said to satisfy the {\bf generalized SPD property with external molecules} if, after merging all external molecules of $\cal G$ into one single internal molecule, the resulting molecular graph satisfies Definition \ref{def seqPDG weak}. 

\vspace{5pt}
\noindent {(ii)} A graph $\cal G$ is said to be {\bf globally standard with external molecules} if it is generalized SPD with external molecules in the above sense, and every proper isolated subgraph with non-deterministic closure is weakly isolated.
\end{definition}

It is easy to see that Lemma \ref{lem normA} is an immediate consequence of the following lemma. 


\begin{lemma}\label{lemma pgon expand}
Suppose the assumptions of Lemma \ref{lem: ini bound} hold. Then for any large constant $D>0$, a $2p$-gon graph in \eqref{pgon graphs} can be expanded as 
\be\label{mlevelTgdef pgon}
\E \cal G_{\fa_1,\cdots,\fa_{2p}} =\sum_\omega \cal G^{(\omega)}_{\fa_1,\cdots,\fa_{2p}} + \OO_\prec(W^{-D}),
\ee
where the first term on the right-hand side is a sum of $\OO(1)$ many deterministic graphs \smash{$\cal G^{(\omega)}_{\fa_1,\cdots,\fa_{2p}}$} satisfying the assumptions of Lemma \ref{lemma_boundgen2net}. In particular, they satisfy the estimate 
\be\label{ordGkappa}
\frac{1}{\left(K^d\right)^{2p}}\sum_{\fa_1,\fa_2, \cdots, \fa_{2p}\in \cal I}   \cal G^{(\omega)}_{\fa_1,\cdots,\fa_{2p}} \prec \left(\frac{1}{W^4K^{d-4}}\right)^{2p-1}.
\ee

\end{lemma} 


\begin{proof}
The expansion \eqref{mlevelTgdef pgon} can be obtained by applying Strategy \ref{strat_global} with the following modifications: 
\begin{enumerate}
	\item we will use the $\nonuni$ \eqref{mlevelTgdef weak} instead of the $n$-th order $T$-expansion; 
	\item we will use the globally standard property with external molecules defined in Definition \ref{GSS_external};
	\item we will stop the expansion of a graph if it is deterministic, its size is less than $W^{-D}$ or it is a $Q$-graph.
\end{enumerate}  
The proof is actually easier than that for Theorem \ref{incomplete Texp}, because the $\nonuni$ takes a simpler form with only two types of graphs: graphs with deterministic maximal subgraphs and $Q$-graphs. Moreover, $Q$-graphs from the expansion of $\cal G_{\fa_1,\cdots,\fa_{2p}}$ will vanish after taking expectation, so we do not need to keep them. We now give more details. 

First, we apply local expansions to $\cal G_{\fa_1,\cdots,\fa_{2p}}$ to get a linear combination of locally standard graphs without $P/Q$ labels. The new atoms generated in this process are all included into the external molecules around $\fa_i$, $1\le i \le 2p$. Then we pick any $t_{x,y_1y_2}$ variable in one of the new graphs, say \smash{$\wt{\cal G}$}, and replace it with the $\nonuni$ \eqref{mlevelTgdef weak}: 
\be\label{mlevelTgdef weak3} 
\begin{split}
	t_{x,y_1 y_2}&=   m  \Theta_{xy_1}\overline G_{y_1y_2} +  \sum_{\mu} B^{(\mu)}_{xy_1}\overline G_{y_1y_2} f_\mu (G_{y_1y_1})+   \sum_{\nu} \wt B^{(\nu)}_{xy_2} G_{y_2y_1} \wt f_\nu(G_{y_2y_2})  \\
	&+   \sum_{\mu} \mathscr B^{(\mu)}_{x , y_1 y_2}g_\mu(G_{y_1y_1},G_{y_2y_2},\overline G_{y_1y_2},  G_{y_2y_1}) + \cal Q_{x,y_1y_2}   + \Err_{x,y_1y_2} .
\end{split}
\ee
Here we again include the graphs containing 
$t_{x, y_1 y_2}-T_{x,y_1y_2}$ into the second and third terms on the right-hand side of \eqref{mlevelTgdef weak3} (where we abuse the notation a little bit). If we replace $t_{x,y_1y_2}$ with a graph in the first four terms on the right-hand side of \eqref{mlevelTgdef weak3}, then we will get a generalized doubly connected graph which either does not contain any internal molecule or has a deterministic maximal subgraph. Moreover, the new graph either has one fewer blue solid edge or is of smaller size than $\size(\wt{\cal G})$ by a factor $W^{-c_1}$. If we replace $t_{x,y_1y_2}$ with a graph in $\Err_{x,y_1y_2}$, then the new graph is of size $\le W^{-D}$. 

Now suppose we replace $t_{x,y_1y_2}$ in $\wt {\cal G}$ with a $Q$-graph in $ \cal Q_{x,y_1y_2} $ and apply the $Q$-expansions. Then by Lemma \ref{lem globalgood}, the resulting non-$Q$ graphs are globally standard with external molecules in the sense of Definition \ref{GSS_external}. Here we draw two examples of such molecular graphs with $2p=6$: 
\begin{center}
\includegraphics[width=11cm]{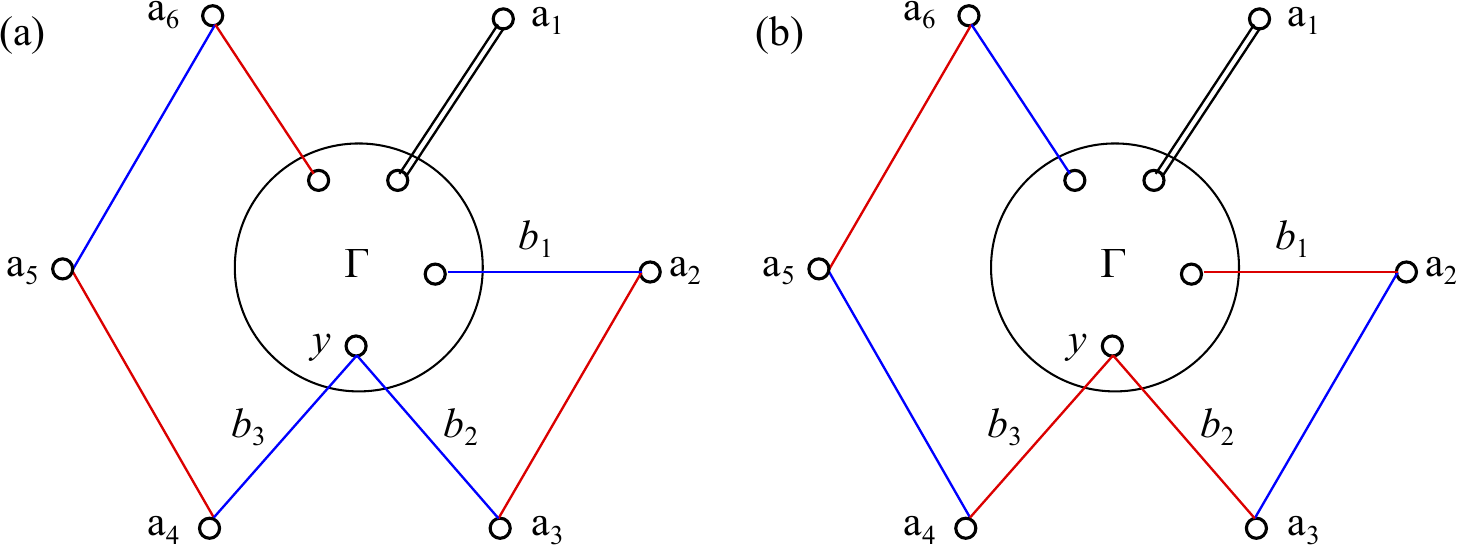}
\end{center}
Inside the black circle is a subgraph containing all internal molecules, and $y$ is the atom in the label of the $Q$-graph (where we used $y$ to label the molecule containing it). Graph (a) does not contain any isolated subgraph with non-deterministic closure and has a pre-deterministic order: internal blue solid edges $\preceq b_1\preceq b_2\preceq b_3$. Graph (b) may contain isolated subgraphs with non-deterministic closures, but they are weakly isolated due to the two red solid edges $b_2$ and $b_3$ connected with $y$. 
Now, for each resulting graph from $Q$-expansions, we either perform local expansions on atoms in the MIS as in Step 1 of Strategy \ref{strat_global}, or expand the first blue solid edge in a pre-deterministic order of the MIS with non-deterministic closure as in Step 3 of Strategy \ref{strat_global} by plugging into the non-universal $T$-expansion. 

Continuing the expansions, we will get a sum of graphs that are generalized doubly connected with external molecules, and whose internal molecules are only connected with deterministic edges. In other words, all the $G$ edges are inside or between external molecules as in the following example: 
\begin{center}
\includegraphics[width=5cm]{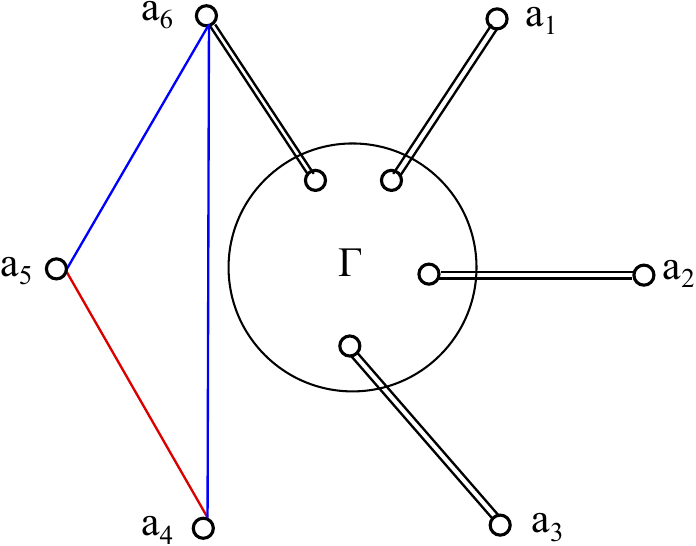}
\end{center}
Now, we repeat the previous expansions, that is, we first apply local expansions and then pick any $t_{x,y_1y_2}$ variable to expand using \eqref{mlevelTgdef weak3}. This may generate a new subset of internal molecules. 
Again, we will apply Strategy \ref{strat_global} to the resulting graphs. Notice that since the old internal molecules are connected only with deterministic edges, they will not be affected in the rest of the expansion process.

After each step of expansion, all non-$Q$ graphs satisfy the following properties: it is generalized doubly connected, and all molecules are connected together through $\dashed$ edges. 
Our expansions will stop when every graph is deterministic, of size $\le W^{-D}$, or a $Q$-graph. Then, taking expectation gives \eqref{mlevelTgdef pgon} with deterministic graphs \smash{$\cal G^{(\omega)}_{\fa_1,\cdots,\fa_{2p}}$} satisfying the assumptions of Lemma \ref{lemma_boundgen2net}.
Notice that throughout the expansion process, the number of ghost edges in each graph (including \smash{$\cal G^{(\omega)}_{\fa_1,\cdots,\fa_{2p}}$}) are under control due to the conditions \eqref{ordG1} and \eqref{ordG2}. Now, repeating the proof of Lemma \ref{lemma_boundgen2net}, we can bound that  
\be\label{ordG3}\frac{1}{\left(K^d\right)^{2p}}\sum_{\fa_1,\fa_2, \cdots, \fa_{2p}\in \cal I}   \cal G^{(\omega)}_{\fa_1,\cdots,\fa_{2p}} \prec \frac{1}{\left(K^d\right)^{2p}}\sum_{q=1}^{2p}\sum_{\omega,\gamma}\sum_{\fa_{1},\fa_2, \cdots, \fa_{q}\in \cal I} \cal G_{\gamma,aux}^{(\omega)} \left( \fa_1,\fa_2,\cdots, \fa_{q}\right),\ee
where we followed the notation in \eqref{bound 2net weakdeter2}: $\cal G_{\gamma,aux}^{(\omega)}$ are connected graphs consisting of external atoms $\fa_1, \cdots, \fa_q$, and double-line and purple solid edges. Using \eqref{bound_aux_wgamma} and the fact 
$$\ord\left( \cal G_{\gamma,aux}^{(\omega)}\right)=2\#\{\text{$\dashed$ and purple solid edges in }\cal G_{\gamma,aux}^{(\omega)}
\}\ge 2(q-1),$$
we get that 
$$\frac{1}{\left(K^d\right)^{2p}}\sum_{\fa_1,\fa_2, \cdots, \fa_{q}\in \cal I} \cal G_{\omega,\gamma,aux} \left( \fa_1,\fa_2,\cdots, \fa_{q}\right) \prec \left(\frac{1}{K^d}\right)^{2p-q}\left(\frac{W^{d-4}}{K^{d-4}}\right)^{q-1}  W^{-(q-1)d}  .
$$
Plugging it into \eqref{ordG3}, we conclude \eqref{ordGkappa}.
\end{proof}

\appendix

\section{Proof of Lemma \ref{hard lemmdot} and Lemma \ref{hard lemmQ}}\label{sec PDG}

All graphs in this section are \emph{molecular graphs without red solid edges}. 
We will say a diffusive edge is black (resp. blue) if this edge is used in the black (resp. blue) net.   
In the proof, whenever we say ``we can change a blue solid edge into a $\dashed$ edge", we actually mean ``this edge is redundant and hence we can change it into a $\dashed$ edge". Thus showing a subgraph is pre-deterministic is equivalent to showing that ``we can change the blue solid edges in this subgraph into $\dashed$ edges one by one". 

First, we state the following two claims, whose proofs are trivial.
\begin{claim}\label{trivial_graph1}
Given a graph $\cal G$, let $\cal S$ be a subset of molecules such that the subgraph induced on $\cal S$ is doubly connected. Then $\cal G$ is doubly connected if and only if the following quotient graph $\cal G/\cal S$ is doubly connected: $\cal G/\cal S$ is obtained by treating the subset of molecules $\cal S$ as one single vertex, and the edges connected with $\cal S$ in $\cal G$ are now connected with this vertex $\cal S$ in $\cal G/\cal S$. 
\end{claim}

\begin{claim}\label{trivial_graph3}
Suppose a blue solid edge $b_0$ is redundant in a doubly connected graph $\cal G$. We replace a $\dashed$ edge $b$ between two molecules $\cal M_1$ and $\cal M_2$ with a path of two $\dashed$ edges: $b_1$ from $\cal M_1$ to another molecule $\cal M$, and $b_2$ from $\cal M$ to $\cal M_{2}$. Then the edge $b_0$ is still redundant in the new graph. 
\end{claim}

Now Lemma \ref{hard lemmdot} follows from repeated applications of the above two claims. 

\begin{proof}[Proof of Lemma \ref{hard lemmdot}]
We replace $\Iso_{j+1}$ and its two external edges with a black $\dashed$ edge, and get the following graph (a) for $\Iso_i$:
	\begin{equation}\label{graph_adddot1}
	\parbox[c]{0.8\linewidth}{\center
		\includegraphics[width=13cm]{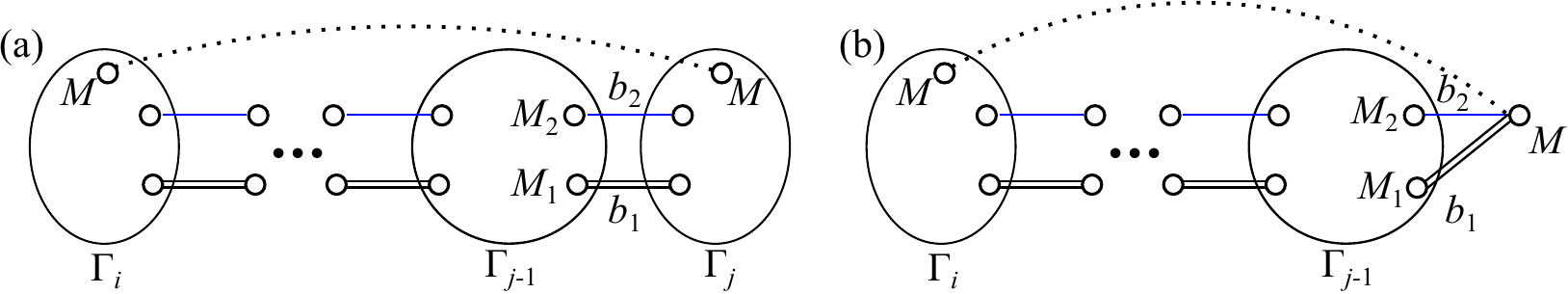}}
\end{equation}
Here we keep the names $\Gamma_l$, $i \le l \le j$, for the subgraph components inside the black circles. We also keep the dotted edge between $\Gamma_i$ and $\Gamma_j$ in order to have a clearer comparison with the original graph without this dotted edge. In particular, the ending molecules of this dotted edge should be understood as the same molecule, denoted by $\cal M$.

First, by the SPD property of the original graph $\cal G$, we know that the component $\Gamma_j$ in graph (a) of \eqref{graph_adddot1} is pre-deterministic, and hence we can change the blue solid edges in it into $\dashed$ edges one by one according to a pre-deterministic order. 
By Claim \ref{trivial_graph1}, graph (a) of \eqref{graph_adddot1} is equivalent (in the sense of doubly connected property) to graph (b) of \eqref{graph_adddot1} with $\Gamma_j$ replaced by a vertex. Furthermore, by merging the molecules connected by the dotted edge, we get the following graph (a): 
	\begin{equation}\label{graph_adddot2}
	\parbox[c]{0.8\linewidth}{\center
		\includegraphics[width=11cm]{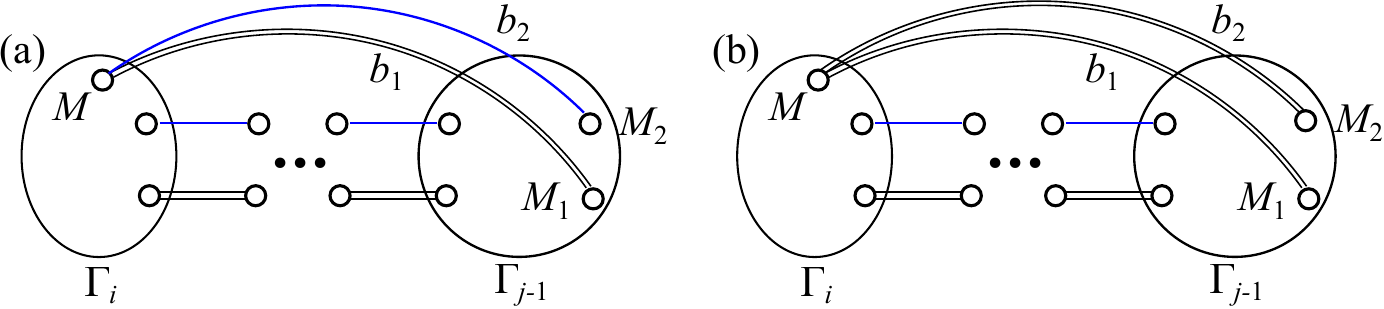}}
\end{equation}
We claim that $b_2$ is redundant in graph (a) of \eqref{graph_adddot1}. Due to the doubly connected property of the original graph $\cal G$, removing the subgraph $\Iso_{j}$ and its two external edges still gives a doubly connected graph. Correspondingly, in graph (a) of \eqref{graph_adddot2}, if we remove the edges $b_1$ and $b_2$, we still get a doubly connected graph, so the edge $b_2$ is redundant. Then we can change $b_2$ into a $\dashed$ edge and get graph (b) of \eqref{graph_adddot2}.
	
	
Next, by property (iii) of Definition \ref{def seqPDG}, if we replace $b_1$ and $b_2$ with a single $\dashed$ edge $b$ between $\cal M_1$ and $\cal M_2$ in graph (b) of \eqref{graph_adddot2}, the component $\Gamma_{j-1}$ in the resulting graph becomes pre-deterministic.  
Then using Claim \ref{trivial_graph3}, we get that the component $\Gamma_{j-1}$ in graph (b) of \eqref{graph_adddot2} is pre-deterministic, and we can change the blue solid edges in it into $\dashed$ edges one by one. 
After that, the subgraph induced on the molecule $\cal M$ and the molecules in $\Gamma_{j-1}$ is a doubly connected graph. Thus we can replace this subgraph with a vertex by Claim \ref{trivial_graph1} and obtain a graph with a similar structure as graph (a) of \eqref{graph_adddot2}.
Hence we can repeat exactly the same argument again. 
	
Continuing in this way, we can show that the blue solid edges in $\Iso_{i}$ can be changed into $\dashed$ edges one by one according to an order given by the above argument. Hence $\Iso_{i}$ is pre-deterministic. 
\end{proof}

The proof of the case (i) of Lemma \ref{hard lemmQ} uses Claim \ref{trivial_graph1} and the following three claims.   

\begin{claim}\label{trivial_graph4}
Given a doubly connected graph $\cal G$, we construct a new graph $\cal G_{new}$ as follows: we create a new molecule $\cal M_{new}$; we add a $\dashed$ edge $b_1$ between $\cal M_{new}$ and a molecule $\cal M_1$ in $\cal G$, and a blue solid edge $b_2$ between $\cal M_{new}$ and a molecule $\cal M_2$ in $\cal G$; we replace a blue solid edge $b$ between molecules $\cal M_3$ and $\cal M_4$ in $\cal G$ with two blue solid edges $b_3$ between $\cal M_3$ and $\cal M_{new}$ and $b_4$ between $\cal M_4$ and $\cal M_{new}$. Then the blue solid edge $b_2$ is redundant in $\cal G_{new}$.
\end{claim}
\begin{proof}
We illustrate the setting of this claim with the following figure, where graphs (a) and (b) represent $\cal G$ and $\cal G_{new}$, respectively:
\begin{equation} \label{eq_add_Q1}
\parbox[c]{0.5\linewidth}{\center \includegraphics[width=7cm]{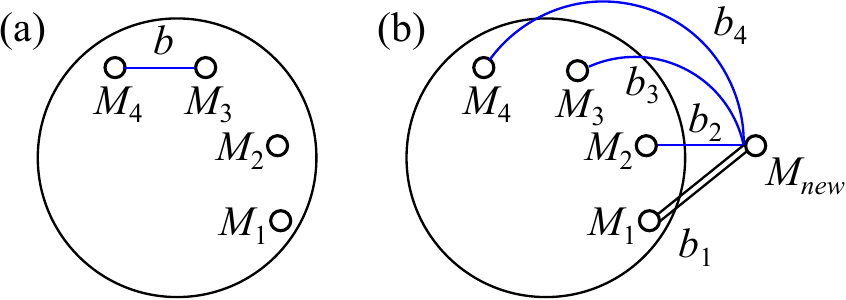}} 
\end{equation}
The original graph $\cal G$ is doubly connected with black and blue nets denoted by $\cal B_{black}$ and $\cal B_{blue}$. In $\cal G_{new}$, the edge $b_1$ connects the new molecule $\cal M_{new}$ to $\cal M_1$, and hence $\cal B_{black} \cup\{b_1\}$ is a black net in $\cal G_{new}$. Moreover, the path of edges $b_3$ and $b_4$ in $\cal G_{new}$ can replace the role of $b$ in the blue net $\cal B_{blue}$, and they also connect $\cal M_{new}$ to the molecules in $\cal G$. Hence $\cal B_{blue} \cup\{b_3,b_4\}\setminus \{b\}$ is a blue net in $\cal G_{new}$. Since $b_2$ is not used in this blue net, it is redundant in $\cal G_{new}$.
\end{proof}

\begin{claim}\label{trivial_graph5}
Given a doubly connected graph $\cal G$, we construct a new graph $\cal G_{new}$ as follows: we create a new molecule $\cal M_{new}$; we replace a $\dashed$ edge $b_0$ between molecules $\cal M_1$ and $\cal M_2$ in $\cal G$ with two $\dashed$ edges $b_1$ between $\cal M_1$ and $\cal M_{new}$, and $b_2$ between $\cal M_2$ and $\cal M_{new}$; we replace a blue solid edge $b$ between molecules $\cal M_3$ and $\cal M_4$ in $\cal G$ with two blue solid edges $b_3$ between $\cal M_3$ and $\cal M_{new}$, and $b_4$ between $\cal M_4$ and $\cal M_{new}$. Then we have that:
\begin{itemize}
\item[(i)] If a blue solid edge is not equal to $b$ and is redundant in $\cal G$, then it is also redundant in $\cal G_{new}$. 
			
\item[(ii)] If the edge $b$ is redundant in $\cal G$, then either $b_3$ or $b_4$ is redundant. Moreover, if we change one redundant edge of them into a $\dashed$ edge, the other edge also becomes redundant. 
\end{itemize}
\end{claim}
\begin{proof}
We first prove the statement (i). Suppose the redundant blue solid edge we are considering is $e=(\cM_5,\cM_6)$. We consider two cases: (A) the edge $b_0$ is included in the black net of $\cal G$ in order for $e$ to be redundant, and (B) the edge $b_0$ is included in the blue net of $\cal G$ in order for $e$ to be redundant. 
		
\vspace{5pt}
\noindent {\bf Case (A):} The graphs (a) and (b) of the following figure represent $\cal G$ and $\cal G_{new}$, respectively: 
\begin{equation} \label{eq_add_Q2}
\parbox[c]{0.8\linewidth}{\center\includegraphics[width=8.5cm]{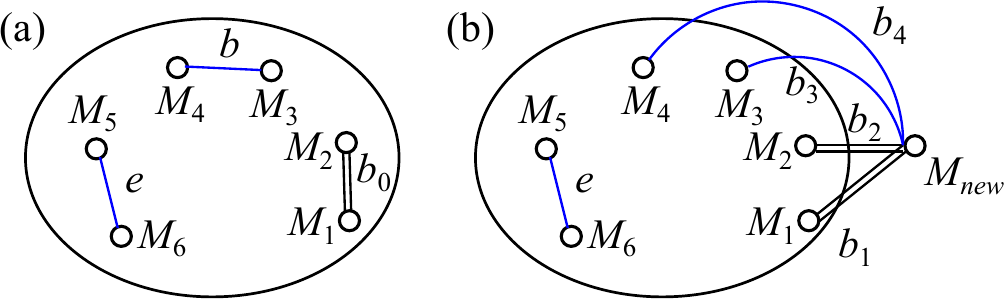}}
\end{equation}
In graph $\cal G$, the edge $e=(\cM_5,\cM_6)$ is redundant by putting $b_0$ into the black net. Let the corresponding black net and blue net be $\cal B_{black}$ and $\cal B_{blue}$, such that $b_0\in \cal B_{black}$ and $\cal B_{blue} \setminus \{e\}$ is still a blue net of $\cal G$. Then in $\cal G_{new}$, if we put edges $b_1$ and $b_2$ into the black net, they can replace the role of $b_0$ in the black net $\cal B_{black}$ and also connect $\cal M_{new}$ to other molecules. Hence $\cal B_{black} \cup\{b_1,b_2\}\setminus \{b_0\}$ is a black net in $\cal G_{new}$. Similarly, the path of edges $b_3$ and $b_4$ can replace the role of $b$ in the blue net $\cal B_{blue} \setminus \{e\}$ and also connect $\cal M_{new}$ to other molecules. Hence $\cal B_{blue} \cup\{b_3,b_4\}\setminus \{e, b\}$ is a blue net in $\cal G_{new}$. Since $e$ is not used in this blue net, it is redundant in $\cal G_{new}$.

\vspace{5pt}
\noindent {\bf Case (B):} The graphs (a) and (b) of the following figure represent $\cal G$ and $\cal G_{new}$, respectively:  
\begin{equation} \label{eq_add_Q3}
\parbox[c]{0.8\linewidth}{\center\includegraphics[width=8.5cm]{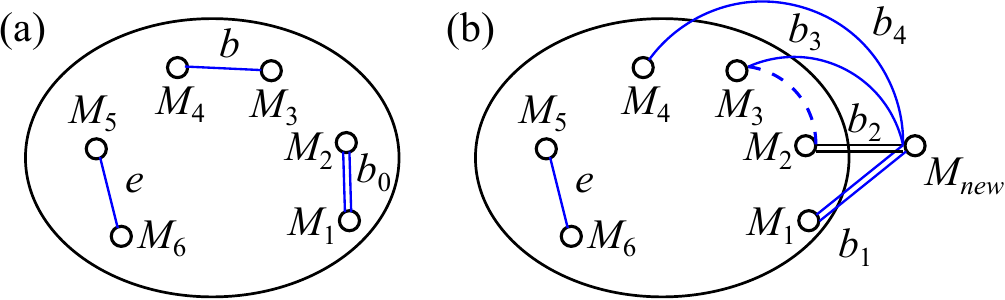}}
\end{equation}
In graph $\cal G$, the edge $e=(\cM_5,\cM_6)$ is redundant by putting $b_0$ into the blue net.  Let the corresponding black net and blue net be $\cal B_{black}$ and $\cal B_{blue}$, such that $b_0\in \cal B_{blue}$ and $\cal B_{blue}\setminus \{e\}$ is still a blue net of $\cal G$. Then either $\cal M_1$ or $\cal M_2$ is connected to molecules $\cal M_3$ and $\cal M_4$ through a path of edges in $\cal B_{blue}\setminus \{e,b_0\}$, since otherwise adding back the edge $b_0$ to $\cal B_{blue}\setminus \{e,b_0\}$ does not give a blue net, which is a contradiction. Without loss of generality, assume that $\cal M_2$ is connected to $\cal M_3$ through a path of edges in $\cal B_{blue}\setminus \{e,b_0,b\}$. In graph (b) of \eqref{eq_add_Q3}, the dashed edge represents such a path from $\cal M_2$ to $\cal M_3$.
 
		
Now in graph (b) of \eqref{eq_add_Q3}, we put $b_1$ into the blue net and $b_2$ into the black net. Then  $\cal B_{black}\cup\{b_2\}$ is a black net in $\cal G_{new}$, because $b_2$ connects the new molecule $\cal M_{new}$ to other molecules. On the other hand, we claim that $\cal B^{new}_{blue}:=\cal B_{blue}\cup \{b_1, b_3,b_4\}\setminus \{e,b_0,b\}$ is a blue net in $\cal G_{new}$. This can be derived from the following observations.
\begin{itemize}
\item[(1)] $\cM_3$ and $\cM_4$ are connected by the path of edges $b_3$ and $b_4$, which replaces the role of edge $b$ in the blue net $\cal B_{blue} \setminus \{e\}$ of $\cal G$.

\item[(2)] $\cM_1$ and $\cM_2$ is connected by a path of edges in $\cal B^{new}_{blue}$ consisting of $b_1$, $b_3$ and a path of edges connecting $\cal M_2$ to $\cal M_3$ in $\cal B_{blue}\setminus \{e,b_0, b\}$. This shows that the role of edge $b_0$ in the blue net $\cal B_{blue} \setminus \{e\}$ can be replaced by a path of edges in $\cal B^{new}_{blue}$.
			
\item[(3)] The new molecule $\cal M_{new}$ is connected to other molecules through $b_1$, $b_3$ and $b_4$.
\end{itemize}
Combing these observations with the fact that $\cal B_{blue}\setminus \{e\}$ is a blue net in $\cal G$, we conclude that $\cal B^{new}_{blue}$ is a blue net in $\cal G_{new}$. Since $e$ is not in this blue net, it is redundant in $\cal G_{new}$.
		
\medskip
		
Next we prove the statement (ii). Again there are two cases: (C) the edge $b_0$ is included in the black net of $\cal G$ in order for $b$ to be redundant, and (D) the edge $b_0$ is included in the blue net of $\cal G$ in order for $b$ to be redundant.

\vspace{5pt}
\noindent {\bf Case (C):} In this case, $\cal G$ and $\cal G_{new}$ are represented by the graphs (a) and (b) of \eqref{eq_add_Q2} (where the edge $e$ is irrelevant). The graph $\cal G$ has a black net $\cal B_{black}$ and a blue net $\cal B_{blue}$, such that $b_0\in \cal B_{black}$ and $\cal B_{blue} \setminus \{b\}$ is still a blue net of $\cal G$. Then in $\cal G_{new}$, 
we again have that $\cal B_{black} \cup\{b_1,b_2\}\setminus \{b_0\}$ is a black set in $\cal G_{new}$. On the other hand, $\cal B_{blue} \cup \{b_4\}\setminus \{b\}$ is a blue net in $\cal G_{new}$, because $\cal B_{blue} \setminus \{b\}$ is a blue net in $\cal G$ and $b_4$ connects $\cal M_{new}$ to other  molecules.  Since $b_3$ is not used in this blue net, it is redundant in $\cal G_{new}$. Hence we can change $b_3$ into a $\dashed$ edge. With the same argument, we can show that $b_4$ is also redundant in the new graph. This concludes statement (ii) for case (C).


\vspace{5pt}
\noindent {\bf Case (D):} In this case, $\cal G$ and $\cal G_{new}$ are represented by the graphs (a) and (b) of \eqref{eq_add_Q3} (where the edge $e$ is irrelevant). The graph $\cal G$ has a black net $\cal B_{black}$ and a blue net $\cal B_{blue}$, such that $b_0\in \cal B_{blue}$ and $\cal B_{blue} \setminus \{b\}$ is still a blue net in $\cal G$. By a similar argument as in case (B), either $\cal M_3$ or $\cal M_4$ is connected to one of $\cal M_1$ and $\cal M_2$ through a path of edges in $\cal B_{blue}\setminus \{b,b_0\}$. 
Without loss of generality, assume that $\cal M_3$ is connected to $\cal M_2$ through a path of edges in $\cal B_{blue}\setminus \{b,b_0\}$ (which is represented by the dashed edge in graph (b) of \eqref{eq_add_Q3}). 
First, we again have that $\cal B_{black}\cup\{b_2\}$ is a black net in $\cal G_{new}$. Then we claim that $\cal B^{new}_{blue}:=\cal B_{blue}\cup \{b_1, b_3\}\setminus \{b,b_0\}$ is a blue net in $\cal G_{new}$. This claim follows from the observation that $\cM_1$ and $\cM_2$ are connected by a path of edges in $\cal B^{new}_{blue}$ consisting of $b_1$, $b_3$ and a path of edges connecting $\cal M_2$ to $\cal M_3$ in $\cal B_{blue}\setminus \{b,b_0\}$, and hence the role of edge $b_0$ in $\cal B_{blue}\setminus \{b\}$ can be replaced by a path of edges in $\cal B^{new}_{blue}$. Since $b_4$ is not used in $\cal B^{new}_{blue}$, it is redundant in $\cal G_{new}$.
		
Now we change $b_4$ into a $\dashed$ edge and get the following graph, denoted by $\wt{\cal G}_{new}$: 
\begin{equation} \label{eq_add_Q5}
\parbox[c]{0.8\linewidth}{\center\includegraphics[width=4.7cm]{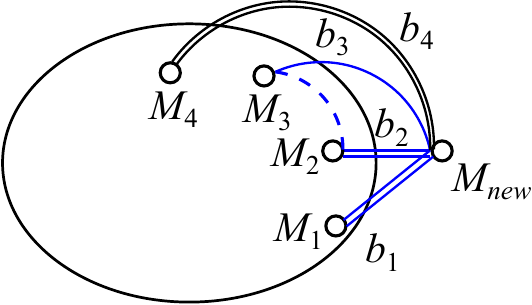}}
\end{equation}
We put $b_4$ into the black net, and $b_1$ and $b_2$ into the blue net. Then it is easy to see that $\cal B_{black}\cup\{b_4\}$ is a black net in $\cal G_{new}$. Moreover, $\cal B_{blue}\cup \{b_1, b_2\} \setminus \{b,b_0\}$ is a blue net in $\cal G_{new}$, because the path of edges $b_1$ and $b_2$ can replace the role of edge $b_0$ in the original blue net $\cal B_{blue}\setminus \{b\}$ of $\cal G$, and $b_1$ and $b_2$ also connect $\cal M_{new}$ to other molecules. Since $b_3$ is not used in this blue net, it is redundant in $\cal G_{new}$. This concludes the statement (ii) for case (D).
\end{proof}
	
\begin{claim}\label{trivial_graph6}
Given a doubly connected graph $\cal G$, we construct a new graph $\cal G_{new}$ as follows: we create a new molecule $\cal M_{new}$; we replace a $\dashed$ edge $b_0$ between molecules $\cal M_1$ and $\cal M_2$ in $\cal G$ with two $\dashed$ edges $b_1$ between $\cal M_1$ and $\cal M_{new}$, and $b_2$ between $\cal M_2$ and $\cal M_{new}$; we replace a $\dashed$ edge $b$ between molecules $\cal M_3$ and $\cal M_4$ in $\cal G$ with two $\dashed$ edges $b_3$ between $\cal M_3$ and $\cal M_{new}$, and $b_4$ between $\cal M_4$ and $\cal M_{new}$. If a blue solid edge is redundant in $\cal G$, then it is also redundant in $\cal G_{new}$.
\end{claim}
\begin{proof}
Suppose the redundant blue solid edge we are considering is $e=(\cM_5,\cM_6)$. If the edge $b$ is included into the blue net in order for $e$ to be redundant in $\cal G$, we put the edges $b_3$ and $b_4$ into the blue net, and the result then follows from Claim \ref{trivial_graph5} (because blue $\dashed$ edges and blue solid edges are equivalent in a blue net). Notice that $b_0$ and $b$ play symmetric roles. Hence if $b_0$ is included into the blue net in order for $e$ to be redundant in $\cal G$, we can conclude the proof with the same argument. 

It remains to consider the case where both $b$ and $b_0$ are included into the black net in order for $e$ to be redundant in $\cal G$. We have the following figure, where the graphs (a) and (b) respectively represent $\cal G$ and $\cal G_{new}$:
\begin{equation} \label{eq_add_Q6}
\parbox[c]{0.8\linewidth}{\center\includegraphics[width=8.5cm]{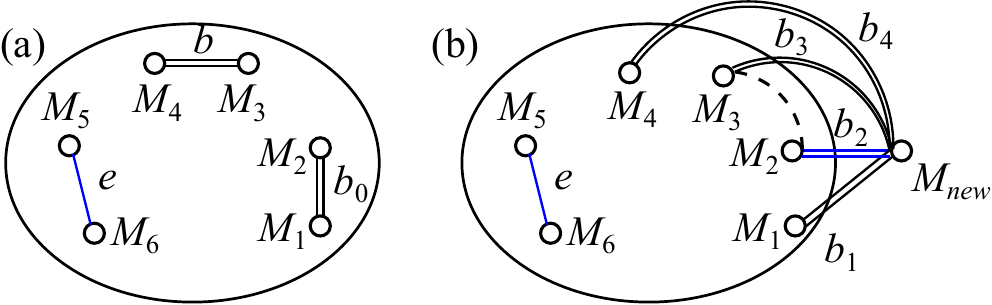}}
\end{equation}
The graph $\cal G$ has a black net $\cal B_{black}$ and a blue net $\cal B_{blue}$, such that $b_0, b\in \cal B_{black}$ and $\cal B_{blue} \setminus \{e\}$ is still a blue net of $\cal G$. Then either $\cal M_3$ or $\cal M_4$ is connected to one of $\cal M_1$ and $\cal M_2$ through a path of black edges in $\cal B_{black}\setminus \{b,b_0\}$, since otherwise adding back the edges $b$ and $b_0$ to $\cal B_{black}\setminus \{b,b_0\}$ does not give a black net, which is a contradiction. Without loss of generality, assume that $\cal M_3$ is connected to $\cal M_2$ through a path of edges in $\cal B_{black}\setminus \{b,b_0\}$, which is represented by the dashed edge in the graph (b) of \eqref{eq_add_Q6}. 
Then in $\cal G_{new}$, we put $b_2$ into the blue net, and put edges $b_1$, $b_3$ and $b_4$ into the black net. First, using the fact that $\cal B_{blue} \setminus \{e\}$ is a blue net in $\cal G$, it is easy to see that $\cal B_{blue} \cup\{b_2\}\setminus \{e\}$ is a blue net in $\cal G_{new}$. Then we claim that $\cal B^{new}_{black}:=\cal B_{black}\cup \{b_1, b_3,b_4\}\setminus \{b,b_0\}$ is a black net in $\cal G_{new}$. This is due to the following observations.
\begin{itemize}
\item[(1)] $\cM_3$ and $\cM_4$ are connected by the path of edges $b_3$ and $b_4$, which replaces the role of edge $b$ in the original black net $\cal B_{black}$.
			
\item[(2)] $\cM_1$ and $\cM_2$ are connected by a path of edges in $\cal B^{new}_{black}$ consisting of $b_1$, $b_3$ and a path of black edges connecting $\cal M_2$ to $\cal M_3$ in $\cal B_{black}\setminus \{b,b_0\}$. This shows that the role of edge $b_0$ in the original black net $\cal B_{black}$ can be replaced by a path of edges in $\cal B^{new}_{black}$.
			
\item[(3)] The new molecule $\cal M_{new}$ is connected to other molecules through $b_1$, $b_3$ and $b_4$.
\end{itemize}
Hence $\cal G_{new}$ has a black net $\cal B^{new}_{black}$ and a blue net $\cal B_{blue}\cup\{b_2\} \setminus \{e\}$. Since $e$ is not used in this blue net, it is redundant.
\end{proof}
	
Now we give the proof of the case (i) of Lemma \ref{hard lemmQ} using Claims \ref{trivial_graph4}--\ref{trivial_graph6}. Recall that the original graph is denoted by $\cal G$, and the new graph is denoted by $\wt{\cal G}$.
\begin{proof}[Proof of Lemma \ref{hard lemmQ}: case (i)]
The statement that $\Iso_{j+1}$ is the maximal isolated subgraph of $\Iso_i$ is trivial. We replace $\Iso_{j+1}$ and its two external edges with a $\dashed$ edge, and get the following graph (a) for the isolated subgraph $\Iso_i$ of $\wt{\cal G}$:
\begin{equation} \label{graph_add_Q7}
\parbox[c]{0.9\linewidth}{\center\includegraphics[width=14cm]{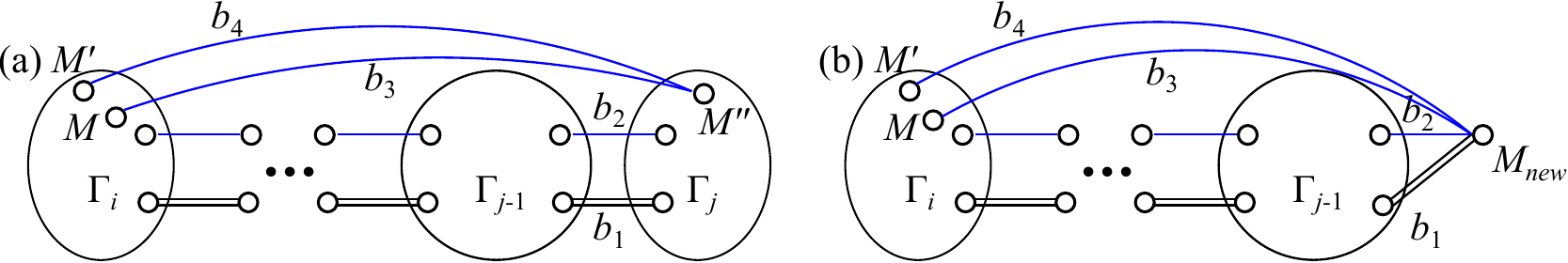}}
\end{equation}
Here we keep the names $\Gamma_l$, $i \le l \le j$, for the subgraph components inside the black circles. By the SPD property of $\cal G$, we know that the component $\Gamma_j$ in graph (a) is pre-deterministic, and hence we can change its blue solid edges into $\dashed$ edges one by one according to a pre-deterministic order. 
Now by Claim \ref{trivial_graph1}, 
it suffices to show that the graph (b) with $\Gamma_j$ replaced by a vertex is pre-deterministic. 

Now we show that the edge $b_2$ is redundant in the graph (b) of \eqref{graph_add_Q7}. By the doubly connected property of $\cal G$, removing the subgraph $\Iso_{j+1}$ and its two external edges still gives a doubly connected graph. Correspondingly, in the graph (b) of \eqref{graph_add_Q7}, if we remove the two edges $b_1$ and $b_2$ and replace the edges $b_3$ and $b_4$ with a single blue solid edge $b=(\cal M, \cal M')$, we still get a doubly connected graph. Then by Claim \ref{trivial_graph4}, the edge $b_2$ is redundant, so we can change it into a $\dashed$ edge and get the following graph: 
\begin{equation} \label{graph_add_Q75}
\parbox[c]{0.4\linewidth}{\center\includegraphics[width=6.5cm]{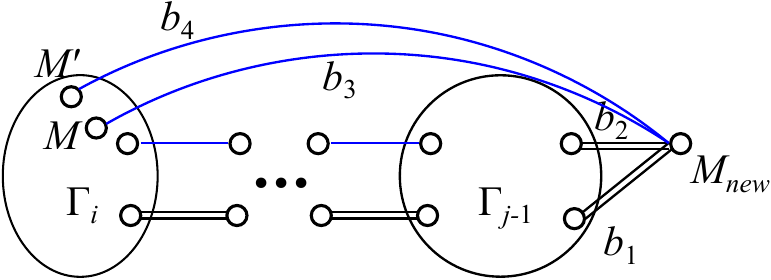}}
\end{equation}		
By the SPD property of $\cal G$, if we replace $\Iso_j$ and its two external edges with a $\dashed$ edge, then the blue solid edges in $\Gamma_{j-1}$ can be changed into $\dashed$ edges one by one according to a pre-deterministic order. Now using Claim \ref{trivial_graph5} (i), we get that the blue solid edges inside $\Gamma_{j-1}$ in \eqref{graph_add_Q75} can be changed into $\dashed$ edges according to the same pre-deterministic order. After that, since the subgraph induced on the molecules in $\Gamma_{j-1}$ and the molecule $\cal M_{new}$ is doubly connected, by Claim \ref{trivial_graph1} it is equivalent to replace them with a single molecule and get a graph that is similar to graph (b) of \eqref{graph_add_Q7}. 

Repeating the above argument, we finally reduced the graph to the following graph (b) with component $\Gamma_i$ only and a new molecule $\cal M_{new}$: 
\begin{equation} \label{graph_add_Q8}
\parbox[c]{0.9\linewidth}{\center\includegraphics[width=13.5cm]{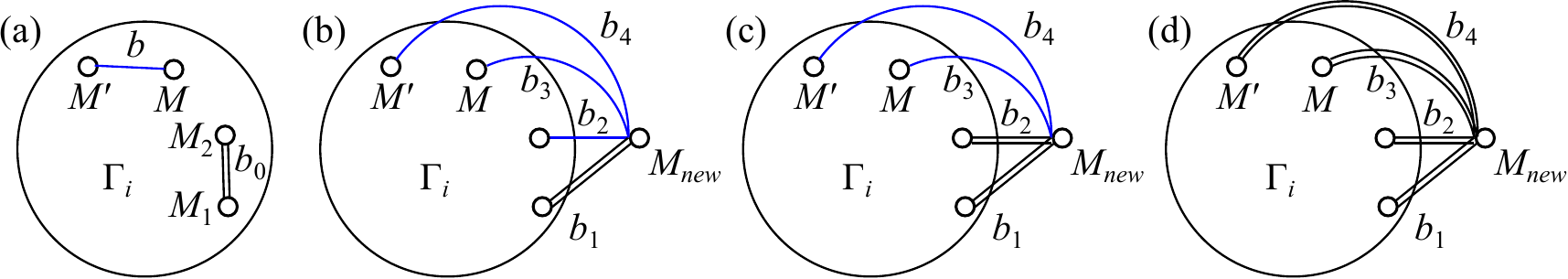}}
\end{equation}
Here the graph (a) is obtained from $\cal G$ by replacing $\Iso_{i+1}$ and its two external edges with a $\dashed$ edge $b_0$ between $\cal M_1$ and $\cal M_2$. By the property (iii) of Definition \ref{def seqPDG}, the component $\Gamma_i$ in graph (a) is pre-deterministic. Suppose a pre-deterministic order of the blue solid edges is $e_1\preceq \cdots \preceq e_{\ell-1}\preceq b\preceq e_{\ell+1}\preceq\cdots\preceq e_{k}$. First, by Claim \ref{trivial_graph4}, the edge $b_2$ in the graph (b) of \eqref{graph_add_Q8} is redundant, so we can change it into a $\dashed$ edge and get the graph (c). Second, by Claim \ref{trivial_graph5} (i), the blue solid edges $e_1, \cdots, e_{\ell-1}$ inside $\Gamma_{i}$ can be changed into $\dashed$ edges one by one. Third, by Claim \ref{trivial_graph5} (ii), the edges $b_3$ and $ b_4$ can be changed into $\dashed$ edges according to the order $b_3\preceq b_4$ or $b_4\preceq b_3$, which gives the graph (d) of \eqref{graph_add_Q8}. Finally, by Claim \ref{trivial_graph6}, the blue solid edges $e_{\ell+1},\cdots, e_{k}$ in $\Gamma_{i}$ can be changed into $\dashed$ edges one by one. 
		
In sum, we have shown that the blue solid edges in $\Iso_{i}$ can be changed into $\dashed$ edges one by one according to an order given by the above argument. Hence $\Iso_{i}$ is pre-deterministic. This concludes the case (i) of Lemma \ref{hard lemmQ}.
\end{proof}
	
The cases (ii)--(v) of Lemma \ref{hard lemmQ} are all easier to prove than case (i). We consider them case by case. 
	
\begin{proof}[Proof of Lemma \ref{hard lemmQ}: cases (ii) and (iv)] 
For the case (ii), the statement that $\Iso_{i+1}$ is the maximal isolated subgraph of $\Iso_j$ is trivial. We replace $\Iso_{i+1}$ and its two external edges with a $\dashed$ edge, and get the following graph (a) for the isolated subgraph $\Iso_j$ of \smash{$\wt{\cal G}$}:
\begin{equation} \label{eq_add_Q9}
\parbox[c]{0.9\linewidth}{\center\includegraphics[width=14.5cm]{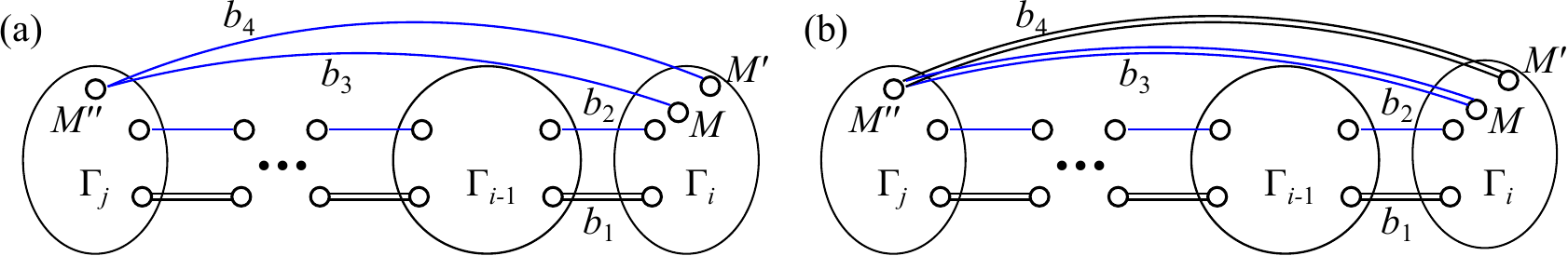}}
\end{equation}
By the SPD property of $\cal G$, we know that after replacing the edges $b_3$ and $b_4$ in the graph (a) with a blue solid edge $b:=(\cal M,\cal M')$, the graph component $\Gamma_j$ becomes pre-deterministic. Suppose a pre-deterministic order of blue solid edges in $\Gamma_i$ is $e_1\preceq \cdots \preceq e_{\ell-1}\preceq b\preceq e_{\ell+1}\preceq \cdots\preceq e_{k}$.
Then it is trivial to observe that in the graph (a), the edges $e_1, \cdots, e_{\ell-1}, b_3, b_4, e_{\ell+1},\cdots, e_{k}$ can be changed into $\dashed$ edges one by one, which gives the graph (b) of \eqref{eq_add_Q9}. Now if we put $b_3$ into the blue net and $b_4$ into the black net, the subgraph induced on the molecules in $\Gamma_i$ and the molecule $\cal M''$ is doubly connected. By Claim \ref{trivial_graph1}, it is equivalent to look at the quotient graph with this subgraph replaced by a single vertex. The quotient graph has the same structure as the graph (a) of \eqref{graph_adddot2}, which has been shown to be pre-deterministic. This concludes the case (ii) of Lemma \ref{hard lemmQ}.


For the case (iv), we replace $\Iso_{i+2}$ and its two external edges with a $\dashed$ edge, and get the following graph (a) for the isolated subgraph $\Iso_j$ of $\wt{\cal G}$:
\begin{equation} \label{eq_add_Q10}
\parbox[c]{0.9\linewidth}{\center\includegraphics[width=13cm]{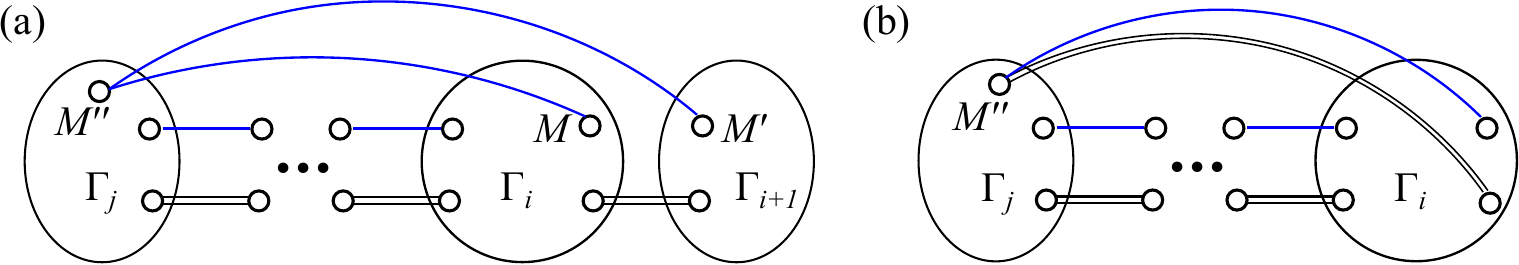}}
\end{equation}
The statement that $\Iso_{i+1}$ is the maximal isolated subgraph of $\Iso_j$ is trivial. Then we replace $\Iso_{i+1}$ and its two external edges with a $\dashed$ edge, and get the graph (b) in \eqref{eq_add_Q10}. Again this graph has the same structure as the graph (a) of \eqref{graph_adddot2}, which has been shown to be pre-deterministic. Hence we conclude the case (iv) of Lemma \ref{hard lemmQ}.
\end{proof}

The proofs of cases (iii) and (v) of Lemma \ref{hard lemmQ} use the following two claims.

\begin{claim}\label{trivial_graph2}
Given a doubly connected graph $\cal G$, we construct a new graph $\cal G_{new}$ as follows: 
\begin{center}
\includegraphics[width=4cm]{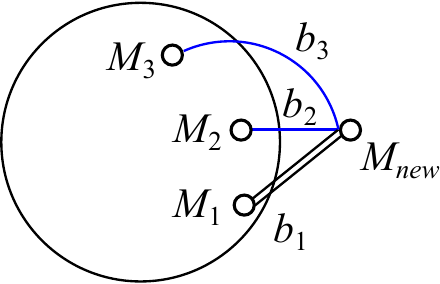}
\end{center}		
More precisely, we create a new molecule $\cal M_{new}$. Given any three molecules $\cal M_1$, $\cal M_2$ and $\cal M_3$ in $\cal G$, we add a $\dashed$ edge $b_1$ between $\cal M_{new}$ and $\cal M_1$, a blue solid edge $b_2$ between $\cal M_{new}$ and $\cal M_2$, and a blue solid edge $b_3$ between $\cal M_{new}$ and $\cal M_3$. Then the blue solid edges $b_2$ and $b_3$ are both redundant in $\cal G_{new}$.
\end{claim}
	
\begin{claim}\label{trivial_graph7}
Given a doubly connected graph $\cal G$, we construct a new graph $\cal G_{new}$ as follows: 
\begin{center}
\includegraphics[width=4cm]{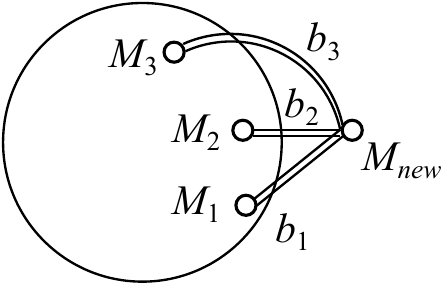}
\end{center}
More precisely, we create a new molecule $\cal M_{new}$, replace a $\dashed$ edge $b_0=(\cal M_1,\cal M_2)$ in $\cal G$ with two $\dashed$ edges $b_1$ between $\cal M_1$ and $\cal M_{new}$ and $b_2$ between $\cal M_2$ and $\cal M_{new}$, and add a $\dashed$ edge $b_3$ between $\cal M_{new}$ and a molecule $\cal M_3$ in $\cal G$. If a blue solid edge is redundant in $\cal G$, it is also redundant in $\cal G_{new}$.
\end{claim}
The proofs of the above two claims are simple, so we omit the details. 	
	
\begin{proof}[Proof of Lemma \ref{hard lemmQ}: cases (iii) and (v)] 
The case (v) can be regarded as a special case of the case (iii). For the case (iii), it is trivial to show that $\Iso_{i+1}$ is the maximal isolated subgraph of $\Iso_i$, and $\Iso_{j+1}$ is the maximal isolated subgraph of $\Iso_{i+1}$. We replace $\Iso_{j+1}$ and its two external edges with a $\dashed$ edge and get the following graph (a) for the isolated subgraph $\Iso_{i+1}$: 
\begin{equation} \label{eq_add_Q12}
	\parbox[c]{0.8\linewidth}{\center
		\includegraphics[width=14cm]{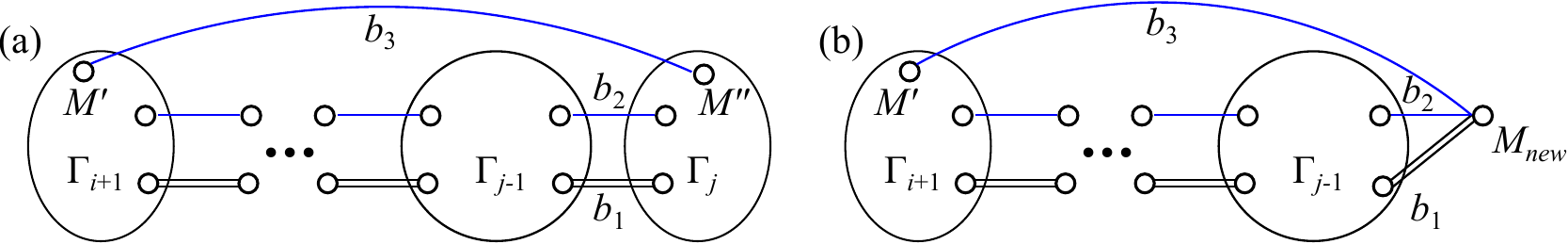}}
\end{equation}
By the SPD property of $\cal G$, the component $\Gamma_j$ in the graph (a) of \eqref{eq_add_Q12} is pre-deterministic, so we can change its blue solid edges into $\dashed$ edges one by one. Then by Claim \ref{trivial_graph1}, we can replace $\Gamma_j$ with a single molecule and get the graph (b) of \eqref{eq_add_Q12}. Now by Claim \ref{trivial_graph2}, both $b_2$ and $b_3$ are redundant, so we can change them into $\dashed$ edges. 
Next by Claim \ref{trivial_graph7}, we can change the blue solid edges in $\Gamma_{j-1}$ into $\dashed$ edges one by one. In the resulting graph, since the subgraph induced on the molecules in $\Gamma_{j-1}$ and the molecule $\cal M_{new}$ is a doubly connected subgraph, by Claim \ref{trivial_graph1} we can replace them with a single molecule and get a graph that is similar to the graph (b) of \eqref{eq_add_Q12}. Repeating the above argument, we can show that the blue solid edges in $\Iso_{i+1}$ can be changed into $\dashed$ edges one by one according to some order. Hence $\Iso_{i+1}$ is pre-deterministic, which concludes the case (iii) of Lemma \ref{hard lemmQ}.
\end{proof}

\section{Proof of Lemma \ref{def nonuni-T}}\label{app_pf_nonuniversal}

As discussed below Lemma \ref{def nonuni-T}, we need to expand the graphs in $(\PT^{(n)})_{\fa,\fb_1 \fb_2}$ and $(\AT^{(>n)})_{\fa,\fb_1\fb_2}$. First, we keep applying cases 1 and 2 of Strategy \ref{strat_global_weak} until all graphs satisfy the stopping rules (T1)--(T3) or the setting in cases 3 of Strategy \ref{strat_global_weak}. So far, all the resulting graphs do not contain any ghost edge. Moreover, as shown in Section \ref{sec expandlvl4}, the resulting graphs $\cal G_{\fa,\fb_1\fb_2}$ can be classified as follows.
	\begin{itemize}
		\item[(a)] If $\cal G_{\fa,\fb_1\fb_2}$ satisfies (T2), then it can be included into $\Err_{\fa,\fb_1\fb_2}$.
		
		\item[(b)] If $\cal G_{\fa,\fb_1\fb_2}$ is a $Q$-graph, then it satisfies the property (v) of Lemma \ref{def nonuni-T} with $k_{\gh}(\cal G)=0$. 
		
		\item[(c)] Suppose $\cal G_{\fa,\fb_1\fb_2}$ satisfies (T1) and is a $\oplus/\ominus$-recollision graph. Depending on whether each of $\oplus$ and $\ominus$ is connected with a dotted edge or not, we can write $\cal G_{\fa,\fb_1\fb_2}$ as 
		\be\label{one_dot}
		\cal G_{\fa,\fb_1\fb_2}= \sum_{x}\Theta_{\fa x} \cal D_{x\fb_1} \overline G_{\fb_1\fb_2} f(G_{\fb_1 \fb_1}),\quad \text{or}\quad \cal G_{\fa,\fb_1\fb_2}= \sum_{x}\Theta_{\fa x} \cal D_{x\fb_2} G_{\fb_2\fb_1} f(G_{\fb_2 \fb_2}), \ee
		or
		\be\label{two_dots}  \cal G_{\fa,\fb_1\fb_2}= \sum_{x}\Theta_{\fa x}\cal D_{x, \fb_1\fb_2} g(G_{\fb_1\fb_1},G_{\fb_2\fb_2},\overline G_{\fb_1\fb_2},  G_{\fb_2\fb_1}),
		\ee
		where $\cal D$ are deterministic doubly connected graphs with $x$, $\fa$, $\fb_1$ and $\fb_2$ regarded as internal atoms, and $f$ and $g$ are monomials satisfying (iv) of Lemma \ref{def nonuni-T}. 
		If $\cal G_{\fa,\fb_1\fb_2}$ takes the form \eqref{two_dots}, then it can be included into the fourth term on the right-hand side of \eqref{mlevelTgdef weak}. If $\cal G_{\fa,\fb_1\fb_2}$ takes one of the forms in \eqref{one_dot}, using Lemma \ref{lem Rdouble} we get that 
		\be\label{more_doubleline}\sum_x |\Theta_{\fa x} \cal D_{x\fb_i}| \prec \sum_x B_{\fa x} B_{x\fb_i}^2 \prec W^{-d}B_{\fa\fb_i},\quad i=1,2.\ee
		Hence $\cal G_{\fa,\fb_1\fb_2}$ can be included into the second and third terms on the right-hand side of \eqref{mlevelTgdef weak}.
		
		\item[(d)] If $\cal G_{\fa,\fb_1\fb_2}$ satisfies (T1) and is not a $\oplus/\ominus$-recollision graph, then we have $\ord(\cal G_{\fa,\fb_1\fb_2})> n$, because it must come from expansions of a graph in $(\AT^{(>n)})_{\fa,\fb_1\fb_2}$. Then $\cal G_{\fa,\fb_1\fb_2}$ can be written as  
		\be\label{cased_higher} \cal G_{\fa,\fb_1\fb_2}= \sum_{x,y} \Theta_{\fa x} \cal D_{x y} t_{y, \fb_1 \fb_2},\ee
		where $\cal D_{x y} $ is a deterministic doubly connected graph of scaling order $> n$. 
	\end{itemize}

%

Now we consider the remaining graphs that satisfy the setting in cases 3 of Strategy \ref{strat_global_weak} but not the stopping rules (T1)--(T3). So far, these graphs do not contain any ghost edge. We claim that all these graphs are GGS. First, by Lemma \ref{lem localgood2} and Lemma \ref{lem globalgood}, all these graphs are SPD. Then we show that these graphs satisfy \eqref{ord ghost0}. We have two cases. If a graph comes from expansions of \smash{$(\AT^{(>n)})_{\fa,\fb_1\fb_2}$}, then it trivially satisfies \eqref{ord ghost0} with $k_\gh=0$. On the other hand, suppose a graph, say $\cal G$, comes from expansions of a graph, say $\cal G_0$, in  \smash{$(\PT^{(n)})_{\fa,\fb_1\fb_2}$}. Since $\cal G_0$ is globally standard, it has no strongly isolated subgraphs. Furthermore, the only operation to generate a strongly isolated subgraph is to replace a $t_{x,y_1y_2}$ variable with a graph in \smash{$(\AT^{(>n)})_{x,y_1y_2}$} in Case 2 of Strategy \ref{strat_global_weak}. Hence $\cal G$ is of scaling order $> n$, and satisfies \eqref{ord ghost0} with $k_\gh=0$. 
%
%
	
In the next few lemmas, we show that Strategy \ref{strat_global_weak} preserve the GGS property. 
\begin{lemma}\label{lem localgood claim3}
Let $\cal G_{\fa,\fb_1\fb_2}$ be a GGS graph without $P/Q$ labels and satisfy the setting in Case 3 of Strategy \ref{strat_global_weak}. Applying Strategy \ref{strat_global_weak} to $\cal G_{\fa,\fb_1\fb_2}$, the resulting non-$Q$ graphs are still GGS, and the resulting $Q$-graphs satisfy the property (v) of Lemma \ref{def nonuni-T}.
\end{lemma}
	\begin{proof}
		$\cal G_{\fa,\fb_1\fb_2}$ satisfies either (A) or (B) of Definition \ref{defn gs weak}. If $\cal G$ satisfies (B), then, by the fact that $\Iso_k$ is deterministic, there exists a redundant ghost edge in $\Iso_k$. We can remove this edge, and get a new graph satisfying (A) of Definition \ref{defn gs weak}. Hence without loss of generality, we assume that $\cal G_{\fa,\fb_1\fb_2}$ satisfies \eqref{ord ghost0}.
		Now we add a ghost edge, denoted by $g_0$, between the ending atoms of $b_1$, and denote the resulting graph by $\wt{\cal G}$. Then we have 
		\be\label{trivial wtG} \ord(\wt{\cal G})\ge (n-1)\cdot  k_{\gh}(\wt{\cal G})  + 2, \ee
		and $b_1$ becomes a redundant edge in $\wt{\cal G}$. Using the generalized SPD property of $\cal G_{\fa,\fb_1\fb_2}$, it is trivial to see that $b_1$ is the first edge in a pre-deterministic order of the MIS containing $\Iso_k$.

		Now we apply \eqref{replaceT} (with $n-1$ replaced by $n$) to $t_{x,y_1 y_2}$, where $G_{\al y_1}$ is the edge $b_1$ and $\al$ is the standard neutral atom in $\Iso_k$. By \eqref{trivial wtG}, all the resulting graphs satisfy \eqref{ord ghost1}.  Moreover, using 
		Lemma \ref{lem globalgood}, we get that the resulting $Q$-graphs satisfy the property (v) of Lemma \ref{def nonuni-T}, and the resulting non-$Q$ graphs satisfy the generalized SPD property.
		It remains to prove the property (A) or (B1) of Definition \ref{defn gs weak} for the non-$Q$ graphs. Corresponding to the terms on the right-hand side of \eqref{replaceT}, we have the following four cases. 
		
		\medskip 
		
		\noindent{\bf Case I:} If we replace $t_{x,y_1y_2}$ with the first two terms on the right-hand side of \eqref{replaceT}, then it is trivial to see that the edge $g_0$ is redundant in each resulting graph, so that the property (B1) of Definition \ref{defn gs weak} holds.
		
		\medskip 
		
		\noindent{\bf Case II:} If we replace $t_{x,y_1y_2}$ with a graph $(\cal G_{A})_{x,y_1y_2}$ in $(\AT^{(>n)})_{x,y_1 y_2}$ or $ (\Err_{n,D})_{x,y_1y_2}$, then every resulting graph, say $\cal G_{new}$, satisfies that 
		$$ \ord({\cal G}_{new})\ge  \ord(\wt{\cal G}) + (n-1) \ge (n-1)\cdot [k_{\gh}(\wt{\cal G})+1]  + 2= (n-1)\cdot [k_{\gh}( {\cal G}_{new})+1]  + 2.$$
		Hence $\cal G_{new}$ satisfies the property (A) of Definition \ref{defn gs weak}.
		
		\medskip 
		
		\noindent{\bf Case III:} Suppose we replace $t_{x,y_1y_2}$ with a graph $(\cal G_{R})_{x,y_1 y_2}$ in $(\PT^{(n)})_{x,y_1 y_2}$, and get the following molecular graph (a) without red solid edges. We only show a case where there is a dotted edge connected with $y_2$ in $(\cal G_{R})_{x,y_1 y_2}$. 
		All the other cases can be analyzed in a similar way. With a slight abuse of notation, we use $x$, $y_1$ and $y_2$ to denote the molecules that contain atoms $x$, $y_1$ and $y_2$.
		\begin{center}
			\includegraphics[width=10cm]{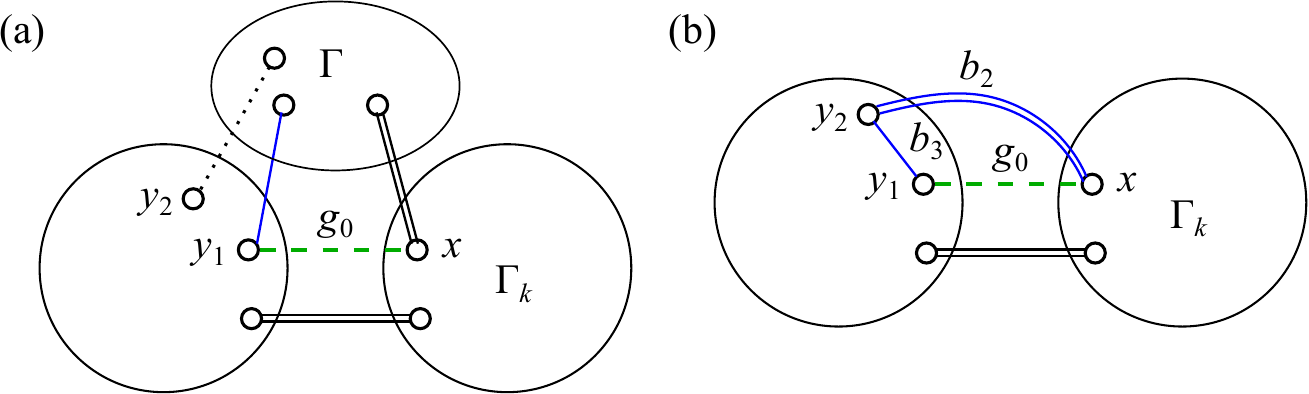}
		\end{center}
	Inside the back circles of graph (a) are some subgraphs, where $\Gamma_k$ contains the molecules in $\Iso_k$ of the original graph \smash{$\wt{\cal G}$}, and $\Gamma$ contains the molecules in $(\cal G_{R})_{x,y_1 y_2}$. We now show that $g_0$ is redundant in graph (a), which is stronger than the property (B1) of Definition \ref{defn gs weak}. By Claim \ref{trivial_graph1}, it is equivalent to consider the graph (b) obtained by merging all molecules in $\Gamma$ with molecule $y_2$. If we include the edges $b_2$ and $b_3$ into the blue net, then $g_0$ becomes redundant.

		\medskip 
		\noindent{\bf Case IV:} Suppose we replace $t_{x,y_1y_2}$ with a graph $(\cal G_{Q})_{x,y_1y_2}$ in $\cal Q^{(n)}_{x,y_1 y_2}$ and get a graph $(\cal G_{new})_{\fa,\fb_1\fb_2}$. Applying $Q$-expansions, we can expand it into a sum of $\OO(1)$ many new graphs:
		\be\nonumber 
		(\cal G_{new})_{\fa,\fb_1\fb_2}= \sum_\omega \cal G_\omega  + \cal Q  +\cal G_{err}, 
		\ee
		where $ \cal G_\omega $ are graphs without $P/Q$ labels, $Q$ is a sum of $Q$-graphs and $\cal G_{err}$ satisfies $|\cal G_{err}|\prec W^{-D}$. 
		Now we show that the graphs $ \cal G_\omega $ are all GGS. Suppose all the solid edges and weights in $(\cal G_{Q})_{x,y_1y_2}$ have the same $Q_y$ label for an atom $y$. Let $\cal G'$ be the complement of $\Iso_k$ in \smash{$\wt{\cal G}$}. By the property (iv) of Lemma \ref{Q_lemma}, $\cal G_{\omega}$ satisfies at least one of the following properties: 
		\begin{itemize}
			\item [(1)] there exist atoms in $\cal G'$ that have been merged with $y$ due to dotted edges;
			\item[(2)] the atoms in $\cal G'$ connect to $y$ only through red solid edges;
			\item[(3)] there exist atoms in $\cal G'$ that connect to $y$ through blue solid edges. 
		\end{itemize}
	    In the case (1), we can show that $g_0$ is redundant using the argument in Case III. In the case (2), if $(\cal G_{Q})_{x,y_1y_2}$ is a $\oplus/\ominus$-recollision graph, we can again use the argument in Case III to show that $g_0$ is redundant. Otherwise, \smash{$(\cal G_{Q})_{x,y_1y_2}$} is the closure of a weakly isolated subgraph, and replacing it with a $\dashed$ edge corresponds to replacing the edge $b_1$ in graph \smash{$\wt{\cal G}$} with a $\dashed$ edge, 
	    in which case the ghost edge $g_0$ becomes redundant. This shows that $\cal G_\omega$ satisfies the property (B1) of Definition \ref{defn gs weak} in the case (2).	
		Finally, we consider the case (3), where we have the following molecular graph (a) without red solid edges: 
		\begin{center}
			\includegraphics[width=14cm]{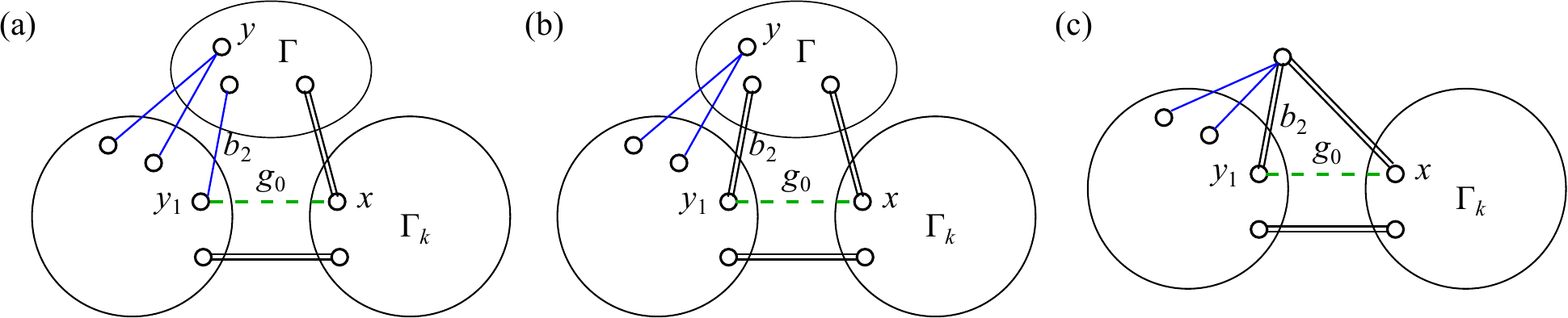}
		\end{center}
	    We change $b_2$ to a diffusive edge and get the graph (b). We claim that $g_0$ is redundant in the graph (b), which implies that $\cal G_{\omega}$ satisfies the property (B1) of Definition \ref{defn gs weak}. By Claim \ref{trivial_graph1}, it is equivalent to show that $g_0$ is redundant in the graph (c) obtained by merging the molecules in $\Gamma$. This follows directly from Claim \ref{trivial_graph5}.	    
		
		\medskip
		
		Combining the above four cases, we conclude that the resulting non-$Q$ graphs satisfy either (A) or (B) of Definition \ref{defn gs weak}, and hence are GGS.
	\end{proof}

\begin{lemma}\label{lem localgood claim1}
	Let $\cal G_{\fa,\fb_1\fb_2}$ be a GGS graph without $P/Q$ labels.  Applying any expansion in Definitions \ref{Ow-def}, \ref{multi-def}, \ref{GG-def} and \ref{GGbar-def} on an atom in the MIS with non-deterministic closure, the resulting non-$Q$ graphs are still GGS, and the resulting $Q$-graphs satisfy the property (v) of Lemma \ref{def nonuni-T}.

\end{lemma}
\begin{proof}
	Using Lemma \ref{lem localgood3}, we get that the resulting non-$Q$ graphs satisfy the generalized SPD property, and the resulting $Q$-graphs satisfy the property (v) of Lemma \ref{def nonuni-T}. It remains to prove that the new non-$Q$ graphs satisfy the property (A) or (B1) of Definition \ref{defn gs weak}. 
	If $\cal G$ satisfies \eqref{ord ghost0} (resp. \eqref{ord ghost1}), then the new graphs also satisfy \eqref{ord ghost0} (resp. \eqref{ord ghost1}), because they have scaling orders $\ge \ord(\cal G)$. Hence we only need to prove that if $\cal G_{\fa,\fb_1\fb_2}$ satisfies the property (B1) of Definition \ref{defn gs weak}, then the resulting non-$Q$ graphs also satisfy the property (B1). 
	
	Similar to the proof of Lemma \ref{lem localgood2}, we need to show that the following two operations will preserve the property (B1) of Definition \ref{defn gs weak}:
	\begin{itemize}
		\item[(I)] merging a pair of molecules due to a newly added dotted or waved edge between them;
		
		\item[(II)] replacing a plus $G$ edge between molecules $\cal M$ and $\cal M'$ with a path of two plus $G$ edges from $\cal M\to \cal M''\to \cal M'$ for a molecule $\cal M''$ in the MIS with non-deterministic closure. 
	\end{itemize}
Denote the maximal isolated subgraph of $\Iso'$ in $\cal G_{\fa,\fb_1\fb_2}$ by $\Iso_1'$. The proof of case (I) is non-trivial only when one molecule is inside $\Iso'_1$, while the other is not. The proof is the same as the one for Lemma \ref{hard lemmdot} by using Claims \ref{trivial_graph1} and \ref{trivial_graph3}, so we omit the details.  
The proof of case (II) is non-trivial only when molecules $\cal M$ and $\cal M'$ are not inside $\Iso'_1$, as shown in the graph (a) of the following figure: 
\begin{center}
\includegraphics[width=13cm]{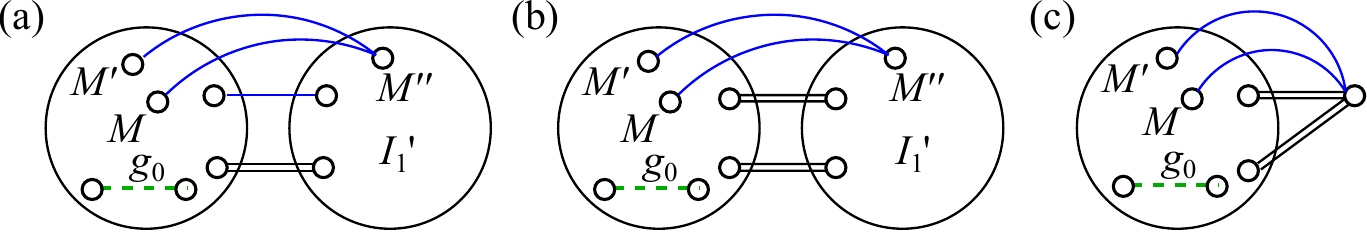}
\end{center}
We replace the blue solid edges inside $\Iso'$ with $\dashed$ edges and get the graph (b), where we have drawn a case such that $\cal M$ or $\cal M'$ does not belong to $\Iso'$. To show that $g_0$ is redundant in the graph (b), by Claim \ref{trivial_graph1} it is equivalent to consider the graph (c) obtained by merging the molecules in $\Iso'_1$. Now by Claim \ref{trivial_graph5}, $g_0$ is redundant in the graph (c). This implies that the graph (a) satisfies the property (B1) of Definition \ref{defn gs weak}. Finally, for the case where $\cal M$ and $\cal M'$ belong to $\Iso'$, we can conclude the proof with a similar argument by using Claim \ref{trivial_graph6}.
\end{proof}
	
	
	\begin{lemma}\label{lem leading claim}
		Let $\cal G_{\fa,\fb_1\fb_2}$ be a GGS graph without $P/Q$ labels. Applying the expansion in Case 2 of Strategy \ref{strat_global_weak}, the resulting non-$Q$ graphs are still GGS, and the resulting $Q$-graphs satisfy the property (v) of Lemma \ref{def nonuni-T}.
	\end{lemma}
	\begin{proof}
		This lemma can be proved using similar arguments as in the proofs of Lemmas \ref{lem localgood claim3} and \ref{lem localgood claim1}. We do not repeat them again. 
	\end{proof}
	
	
In sum, Lemmas \ref{lem localgood claim3}--\ref{lem leading claim} show that applying Strategy \ref{strat_global_weak} to any GGS input graph once, all the resulting non-$Q$ graphs are still GGS and all the $Q$-graphs satisfy the property (v) of Lemma \ref{def nonuni-T}. Hence we can apply Strategy \ref{strat_global_weak} repeatedly until all graphs satisfy the stopping rules (T1)--(T3). Now we claim that the expansion process will stop after $\OO(1)$ many iterations of Strategy \ref{strat_global_weak}.  	
	The proof of this claim is similar to the one for Lemma \ref{lemm expansionstrat}, so we will not write down all the details. 

If we only apply the Cases 1 and 2 of Strategy \ref{strat_global_weak}, then the expansions will stop after $\OO(1)$ many steps by Lemma \ref{lemm expansionstrat}. To complete the proof, it suffices to show that the Case 3 of Strategy \ref{strat_global_weak} is applied at most for $\OO(1)$ many times when all graphs satisfy the stopping rules (T1)--(T3).
%
Suppose a graph $\cal G$ satisfies the setting in the Case 3 of Strategy \ref{strat_global_weak}. As explained in the proof of Lemma \ref{lem localgood claim3}, we can remove a redundant ghost edge from $\cal G$ if necessary so that $\cal G$ satisfies the property (A) of Definition \ref{defn gs weak}. We expand $\cal G$ by applying Strategy \ref{strat_global_weak} repeatedly, and construct correspondingly a tree diagram of graphs as in the proof of Lemma \ref{lemm expansionstrat}. Suppose $\cal G_1$ is a graph on the tree so that it satisfies the setting in Case 3 of Strategy \ref{strat_global_weak}, but all the other graphs on the unique path between $\cal G$ and $\cal G_1$ do not. 
Following the expansion process, we find that at least one of the following properties holds for $\cal G_{1}$.
\begin{itemize}
			\item $\cal G_1$ contains strictly fewer blue solid edges than $\cal G$, and $\size(\cal G_1)\le \size( \cal G).$
			
			\item $\cal G_1$ has a strictly higher scaling order than $\cal G$, but not more ghost edges. Hence we have
			$$\size(\cal G_1)\le W^{-d/2}\size( \cal G).$$ 
			
			\item $\cal G_1$ has one more ghost edge than $\cal G$, and is of scaling order 
			\be\label{scaling_order_cond}
			\ord(\cal G_1)\ge \ord(\cal G) + (n-1). 
			\ee
			Here the extra ghost edge appears when we apply the Case 3 of Strategy \ref{strat_global_weak} to $\cal G$. Furthermore, in order for this ghost edge to be pivotal in $\cal G_1$, we must have replaced a $t_{x,y_1y_2}$ variable with a graph in \smash{$(\AT^{(>n)})_{x,y_1 y_2}$} at some step. This gives the condition \eqref{scaling_order_cond}. Using  \eqref{Lcondition1} and \eqref{defn sizeG}, we get that 
			$$\size(\cal G_1)\le W^{-c_0} \size( \cal G).$$ 
\end{itemize}
With the above facts, it is easy to show that the Case 3 of Strategy \ref{strat_global_weak} are applied at most for $\OO(1)$ many times when the expansion process stops.
	
	
	Applying Strategy \ref{strat_global_weak} repeatedly until all graphs satisfy the stopping rules (T1)--(T3), we can expand $T_{\fa,\fb_1 \fb_2}$ into a sum of $\OO(1)$ many graphs. If a graph satisfies the stopping rule (T2), then it can be included into $\Err_{\fa,\fb_1\fb_2}$ in \eqref{mlevelTgdef weak}. If a graph satisfies the stopping rule (T3), then it can be included into $\cal Q_{\fa,\fb_1\fb_2}$ in \eqref{mlevelTgdef weak}. If a graph $\cal G_{\fa,\fb_1\fb_2}$ satisfies the stopping rule (T1) and is a $\oplus/\ominus$-recollision graph, then it can be written into one of the forms in \eqref{one_dot} and \eqref{two_dots}, where $\cal D_{x\fb_1}$, $\cal D_{x\fb_2}$ and $\cal D_{x, \fb_1\fb_2}$ are deterministic doubly connected graphs satisfying \eqref{ord ghost1}. Then using Lemma \ref{no dot weak} and a similar calculation as in \eqref{more_doubleline}, we get that $\cal G_{\fa,\fb_1\fb_2}$ can be included into the second to fourth terms in \eqref{mlevelTgdef weak}. If a graph $\cal G_{\fa,\fb_1\fb_2}$ satisfies the stopping rule (T1) and is \emph{not} a $\oplus/\ominus$-recollision graph, then it is GGS and can be written into 
	$$\cal G_{\fa,\fb_1\fb_2}= \sum_{x,y} \Theta_{\fa x} \wt{\cal D}_{x y} t_{y, \fb_1 \fb_2} = \sum_{x,y} \Theta_{\fa x} \wt{\cal D}_{x y} T_{y, \fb_1 \fb_2} - \sum_{x,y} \Theta_{\fa x} \wt{\cal D}_{x y} |m|^2\sum_\al s_{y\al} G_{\al\fb_1}\overline G_{\al\fb_2}\left(1- \mathbf 1_{\al \ne \fb_1}\mathbf 1_{\al\ne \fb_2} \right),$$
	where $\wt{\cal D}_{xy}$ is a deterministic doubly connected graph satisfying \eqref{ord ghost0} (by removing a redundant ghost edge  if necessary). We include the second term of the above equation into the second to fourth terms in \eqref{mlevelTgdef weak}. In sum, we have obtained the following expansion:	 
	\be\label{mlevelTgdef weak2} 
	\begin{split}
		& T_{\fa,\fb_1 \fb_2} =m  \Theta_{\fa \fb_1}\overline G_{\fb_1\fb_2} +  \sum_{\mu} B^{(\mu)}_{\fa \fb_1}\overline G_{\fb_1\fb_2} f_\mu (G_{\fb_1\fb_1})+   \sum_{\nu} \wt B^{(\nu)}_{\fa \fb_2} G_{\fb_2\fb_1} \wt f_\nu(G_{\fb_2\fb_2})  \\
		&+\sum_{\mu} \sum_x \Theta_{\fa x}\cal D^{(\mu)}_{x, \fb_1 \fb_2}g_\mu(G_{\fb_1\fb_1},G_{\fb_2\fb_2},\overline G_{\fb_1\fb_2},  G_{\fb_2\fb_1})  +\sum_\mu \sum_{x,y} \Theta_{\fa x} \wt{\cal D}^{(\mu)}_{x y} T_{y, \fb_1 \fb_2}+  \cal Q_{\fa,\fb_1\fb_2}   + \Err_{\fa,\fb_1 \fb_2},  
	\end{split}
	\ee
	where 
	$\cal D^{(\mu)}_{x, \fb_1 \fb_2}$ are deterministic doubly connected graphs satisfying \eqref{ord ghost1}, and $\wt{\cal D}^{(\mu)}_{x y}$ are deterministic doubly connected graphs satisfying \eqref{ord ghost0}.
	
	Now similar to the proof of Corollary \ref{lem completeTexp} in Section \ref{sec lastglobal}, we solve \eqref{mlevelTgdef weak2} to get that 
	\be\label{mlevelTgdef weak22} 
	\begin{split}
		T_{\fa,\fb_1 \fb_2}&=\sum_x \left(1- \Theta \wt{\cal D}^{(\mu)}\right)^{-1}_{\fa x} \Big[m  \Theta_{x \fb_1}\overline G_{\fb_1\fb_2} +  \sum_{\mu} B^{(\mu)}_{x \fb_1}\overline G_{\fb_1\fb_2} f_\mu (G_{\fb_1\fb_1})+   \sum_{\nu} \wt B^{(\nu)}_{x \fb_2} G_{\fb_2\fb_1} \wt f_\nu(G_{\fb_2\fb_2}) \Big] \\
		&+\sum_x \left(1- \Theta \wt{\cal D}^{(\mu)}\right)^{-1}_{\fa x} \Big[\sum_{\mu} \sum_y \Theta_{xy}\cal D^{(\mu)}_{y, \fb_1 \fb_2}g_\mu(G_{\fb_1\fb_1},G_{\fb_2\fb_2},\overline G_{\fb_1\fb_2},  G_{\fb_2\fb_1})  \Big] \\
		&+\sum_x \left(1- \Theta \wt{\cal D}^{(\mu)}\right)^{-1}_{\fa x} \Big[\cal Q_{x,\fb_1\fb_2}   + \Err_{x,\fb_1 \fb_2}  \Big]. 
	\end{split}
	\ee
	Using \eqref{bound 2net1 weak2}, we can get that 
	$$  \sum_{\al} \Theta_{x\al} \wt{\cal D}^{(\mu)}_{\al y}\prec \left( \frac{L^2}{W^2}\right)^{k_{\gh}(\wt{\cal D}^{(\mu)}_{\al y})} W^{ - \left[\ord(\wt{\cal D}^{(\mu)} _{\al y})-2\right]d/2 } B_{xy} \le \frac{W^{-c_0}}{\langle x- y\rangle^d}, $$
	where we used \eqref{ord ghost0}, \eqref{Lcondition1} and $W^2/L^2 \cdot B_{xy}\le \langle x- y\rangle^d$ in the second step. Using this estimate and the Taylor expansion of $(1- \Theta \wt{\cal D}^{(\mu)})^{-1}$, we can get that 
	\be\label{key_taylor} \left|(1- \Theta \wt{\cal D}^{(\mu)})^{-1}_{xy} -\delta_{xy}\right| \prec \frac{W^{-c_0}}{\langle x- y\rangle^d} + W^{-D},\ee
	for any large constant $D>0$. By \eqref{key_taylor}, we see that 
	\begin{itemize}
		\item matrix products of (labelled) $\dashed$ edges with $(1- \Theta \wt{\cal D}^{(\mu)})^{-1}$ give deterministic graphs $\wh B^{(\mu)}$ and $\cal B^{(\omega)}$ satisfying \eqref{est_BwtB1} and \eqref{est_BwtB2};
		
		\item matrix products of $B^{(\mu)}$ and $\wt B^{(\nu)}$ in \eqref{mlevelTgdef weak22} with $(1- \Theta \wt{\cal D}^{(\mu)})^{-1}$ give deterministic graphs satisfying \eqref{est_BwtB} (where we have used the same notations $B^{(\mu)}$ and $\wt B^{(\nu)}$ for convenience); 
		
		\item $[(1- \Theta \wt{\cal D}^{(\mu)})^{-1} \Theta]_{xy}=\Theta_{xy} + B^{(\mu)}_{x\fb_1} $ for a deterministic graph $B^{(\mu)}_{x\fb_1}$ satisfying \eqref{est_BwtB}. 
	\end{itemize}
	Hence \eqref{mlevelTgdef weak22} can be written into \eqref{mlevelTgdef weak}, which concludes the proof of Lemma \ref{def nonuni-T}.

\end{document}